\documentclass[sn-mathphys]{sn-jnl}

\usepackage{blkarray}
\usepackage{subfigure}
\usepackage{bm}

\theoremstyle{thmstyleone}%
\newtheorem{theorem}{Theorem}
\newtheorem{proposition}{Proposition}%
\makeatletter
\renewcommand{\maketag@@@}[1]{\hbox{\m@th\normalsize\normalfont#1}}%
\makeatother

\begin{document}

\title[Article Title]{Multiple-Periods Locally-Facet-Based MIP Formulations for the Unit Commitment Problem}

\author*[1,2]{\fnm{Linfeng} \sur{Yang}}\email{ylf@gxu.edu.cn}

\author[3]{\fnm{Shifei} \sur{Chen}}\email{sfchen2021@163.com}

\author[4]{\fnm{Zhaoyang} \sur{Dong}}\email{zydong@ieee.org}

\affil[1]{\orgdiv{School of Computer Electronics and Information}, \orgname{Guangxi University}, \orgaddress{\city{Nanning}, \postcode{530004}, \country{China}}}
\affil*[2]{\orgdiv{The Guangxi Key Laboratory of Multimedia Communication and Network Technology}, \orgname{Guangxi University}, \orgaddress{\city{Nanning}, \postcode{530004}, \country{China}}}
\affil[3]{\orgdiv{School of Electrical Engineering}, \orgname{Guangxi University}, \orgaddress{\city{Nanning}, \postcode{530004}, \country{China}}}
\affil[4]{\orgdiv{School of Electrical and Electronics Engineering}, \orgname{Nanyang Technological University}, \orgaddress{\city{Singapore}, \postcode{639798}, \country{Singapore}}}

\maketitle

{\noindent{\textbf{Abstract}\\}
	The thermal unit commitment (UC) problem has historically been formulated as a mixed integer quadratic programming (MIQP), which is difficult to solve efficiently, especially for large-scale systems. The tighter characteristic reduces the search space, therefore, as a natural consequence, significantly reduces the computational burden. In literatures, many tightened formulations for a single unit with parts of constraints were reported without presenting explicitly how they were derived. In this paper, a systematic approach is developed to formulate tight formulations. The idea is to use more binary variables to represent the state of the unit so as to obtain the tightest upper bound of power generation limits and ramping constraints for a single unit. In this way, we propose a multi-period formulation based on sliding windows which may have different sizes for each unit in the system. Furthermore, a multi-period model taking historical status into consideration is obtained. Besides, sufficient and necessary conditions for the facets of single-unit constraints polytope are provided and redundant inequalities are eliminated. The proposed models and three other state-of-the-art models are tested on 73 instances with a scheduling time of 24 hours. The number of generators in the test systems ranges from 10 to 1080. The simulation results show that our proposed multi-period formulations are tighter than the other three state-of-the-art models when the window size of the multi-period formulation is greater than 2.}

\vspace{1ex}

\section{Introduction}\label{sec1}

The unit commitment (UC) has been receiving significant attention from both industry and academia. In general, the UC problem is formulated as a mixed integer nonlinear programming (MINLP) problem \cite{anjos2017unit} to determine the operational schedule of the generating units at each time period with varying loads under different operating constraints and environments.

Mixed-Integer Programming (MIP) problems can be handled with commercial solvers. The UC problem can be formulated as a mixed integer quadratic programming (MIQP) problem and directly solved by using such solvers \cite{yang2017novel}. We can also convert the quadratic objective function of the UC problem into a piecewise linear function by accurately approximating the quadratic production cost function with a set of piecewise blocks \cite{carrion2006computationally}. Then we obtain a mixed-integer linear programming (MILP) problem which can be solved by using MILP solvers. The numerical  results reported in \cite{frangioni2009computational} show that solving the approximate MILP problem is more competitive than simply solving the MIQP problem directly.

The quality of the MIP model, which mainly depends on the tightness and compactness of the model, seriously affects the performance of solver \cite{williams2013model}. The tightness of an MIP formulation is defined as the difference between the optimal values for the MIP problem and its continuous relaxation problem \cite{wolsey2020integer}. Tightening an MIP formulation, usually by adding cutting planes, can reduce the search space that the solver requires to explore in order to find the optimal integer solution \cite{wolsey2003strong}. The compactness of an MIP formulation refers to the quantity of data that must be processed when solving the problem. A more compact formulation can speed up the search for the optimal solution. How to build a high-quality MIP model has become a hot topic in recent years. Researchers have made a lot of efforts to construct better formulations for the UC problem in terms of tightness or compactness.

In order to describe the physical constraints for the unit more accurately, researchers have made a lot of effort, such as trying to express the unit state with different binary variables. \cite{garver1962power} first proposes using three binary variables to represent the commitment, startup and shutdown status of the unit. \cite{carrion2006computationally} and \cite{frangioni2009tighter} omit two sets of binary variables from the three-binary-variable formulation presented in \cite{garver1962power} and presents a more compact formulation only with commitment variable. \cite{yang2017novel} proposes a two-binary-variable MIQP formulation for the UC problem with considering compactness and tightness simultaneously.

Since the optimal solution of an MILP problem can always be found at the vertices of the convex hull for its feasible region, we can obtain the optimal solution of the MILP problem directly by solving the relaxation problem. We should note that describing the convex hull of an MILP problem’s feasible set needs an enormous number of inequalities. Finding the explicit convex hull representation for the feasible set of an entire UC problem could be difficult and maybe even unrealistic. \cite{padberg1998location} provides the definition of the locally ideal MILP formulation. Many existing UC literatures provide a locally ideal or locally tighter formulation for a subset of the constraints of the UC problem. \cite{rajan2005minimum} successfully describes the convex hull of a minimum up/down time polytope. \cite{gentile2017tight} provides a convex hull for the polytope subject to minimum up/down time and generation limits. \cite{damci2016polyhedral} provides a convex hull for a two-period ramping polytope and some facets of multi-period ramping polytope. \cite{frangioni2015new} describes the convex hull of the feasible solutions of single-unit satisfying minimum up/down time constraints, minimum and maximum power output, and ramping constraints (including start-up and shut-down limits).

The main works and innovations of this paper are:
\begin{enumerate}
	\item We borrow the variable definition method in \cite{frangioni2015new} to describe the state of unit so that we could easily construct stronger valid inequalities for constraints of a single-unit within any period. By this way, we extend the modeling method from \cite{yang2021two} to multi-period. A new and more direct method is proposed to construct strong valid inequalities (facets) for the feasible region of the UC problem satisfying minimum up/down time constraints, generation limits and ramping limits simultaneously. Actually, the proposed method can be extended to construct facets for other UC constraints, for instance, startup/shutdown cost constraints.
	\item We provide more complete theoretical results on facets of the multi-period polytope which is subject to a subset of constraints for the UC problem. The constraints obtained by the proposed method are all proved to be facets under certain conditions. We provide sufficient and necessary conditions for the facets of the multi-period polytope and eliminate redundant inequalities. In addition, we also provide the improved version of the formulations when considering the generation’s history, and discuss the tightness of the formulations. Besides, the convex hull of feasible solutions for unit commitment problem with partial constraints and some related theories are given.
	\item Based on sliding window, the method of building a tight multi-period formulation for a unit commitment problem in high dimensional space is provided. The window size of each unit in the system can be different and adjusted according to actual needs. Furthermore, a history-dependent multi-period formulation is proposed. In order to improve the efficiency of the model, we make a series of transformations to the model.
\end{enumerate}

The remainder of this paper is organized as follows. In Section \ref{sec2}, we recall three recent well-performed MIQP formulations of the UC problem. In Section \ref{sec3}, we extend the method of constructing inequalities in Section \ref{sec2} to multi-period and present the procedure for constructing multi-period locally-facet constraints of single-unit based on sliding window. In Section \ref{sec4}, we improve the formulations in Section \ref{sec3} and propose some history-dependent multi-period formulations. In Section \ref{sec5}, we perform some computational experiments to demonstrate the behaviors of our models and three recent well-performed formulations on some data set used in other literatures. And we shed further light on the classes of instances for which our formulation outperforms. In Section \ref{sec6}, we draw the conclusions of our study and suggest future research.

\section{State-of-the-art UC formulations for two-period and three-period}\label{sec2}

In this section, we will review three state-of-the-art formulations of UC problem. $N$ and $T$ represent the total number of units in the system and scheduling time respectively. Most of the constraints for a single unit described in this article could be applied to every generator $i \in \{1,\cdots,N\}$ and every period $t \in \{1,\cdots,T\}$. If it is not the case mentioned above, we will specifically point out.

\subsection{State-of-the-art three-binary formulations for UC}\label{subsec1}

Let $P^i_t$ be the power output of unit $i$ in period $t$, and $u^i_t$ be the binary variable denoting the schedule of unit $i$ in period $t$. The objective function is usually defined as the minimization of power system operation costs, which mainly depends on the production cost and the startup cost.
\begin{equation}
	F_c=\sum\nolimits_{i=1}^{N}\sum\nolimits_{t=1}^{T}[f_i(P^i_t)+S^i_t]\label{eq:objective-function}
\end{equation}
where $f_i(P^i_t)=\alpha_iu^i_t+\beta_iP_t^i+\gamma_i(P_t^i )^2$ is the production cost ($\alpha_i,\beta_i,\gamma_i$ are the coefficients), and $S_t^i$ is the startup cost.

One of the most common system-wide constraints that link the schedules of different units in the system is power balance constraint. Let $P_{D,t}$ be the system load demand in period $t$.
\begin{equation}
	\sum\nolimits_{i=1}^{N}P^i_t-P_{D,t}=0\label{eq:power-balance}
\end{equation}
Another common system-wide constraints is system spinning reserve requirement. Let $\overline{P}^i$ be the maximum power output of unit $i$, and $R_t$ be the spinning reserve requirement in period $t$.
\begin{equation}
	\sum\nolimits_{i=1}^Nu^i_t\overline{P}^i\ge P_{D,t}+R_t\label{eq:spinning-reserve}
\end{equation}

Let $v^i_t$ be 1 if unit $i$ is started up in period $t$ (i.e., $u^i_t=1$ and $u^i_{t-1}=0$), and $w^i_t$ be 1 if unit $i$ is shut down in period $t$ (i.e., $u^i_t=0$ and $u^i_{t-1}=1$). \cite{garver1962power} uses three binary variables ($u^i_t$, $v^i_t$, $w^i_t$) to represent the unit status and introduces a logical constraint to relate the three binary variables.
\begin{equation}
	v^i_t-w^i_t=u^i_t-u^i_{t-1}\label{eq:logical}
\end{equation}

Let $C^i_{hot}$ and $C^i_{cold}$ be the hot and cold startup cost of unit $i$ respectively. Furthermore, let $\underline{T}^i_{on}$ and $\underline{T}^i_{off}$ be the minimum up and down time periods of unit $i$ respectively. $T_{cold}^i$ represents cold startup time of unit $i$. $T_0^i$ is the number of periods unit $i$ has been online (+) or offline (-) prior to the first period of the scheduling time span. We define the operator $[\cdot]^+$ as $\max(0,\cdot)$. \cite{jabr2012tight} and \cite{atakan2018state} express the startup cost $S^i_t$ as an MILP formulation:
\begin{align}
	&S^i_t\ge C^i_{hot}v^i_t\label{eq:hot-start}\\
	&S^i_t\ge C^i_{cold}[v^i_t-\sum\nolimits_{\pi=\max(1,t-\underline{T}^i_{off}-T^i_{cold})}^{t-1}w^i_{\pi}-m^i_t]\label{eq:cold-start}
\end{align}
where $m_t^i=1$ if $t-\underline{T}_{off}^i-T_{cold}^i\le 0$ and ${[-T_0^i ]}^+<\lvert t-\underline{T}_{off}^i-T_{cold}^i-1\rvert+1$; $m_t^i=0$ otherwise.

Rajan and Takriti use $v^i_t$ and $w^i_t$ to formulate the minimum up/down time constraints in \cite{rajan2005minimum}.
\begin{align}
	&\sum\nolimits_{\overline{\omega}=[t-\underline{T}^i_{on}]^++1}^tv^i_{\overline{\omega}}\le u^i_t, \qquad t \in [W^i+1,T]\label{eq:min-up-time}\\
	&\sum\nolimits_{\overline{\omega}=[t-\underline{T}^i_{off}]^++1}^tw^i_{\overline{\omega}}\le 1-u^i_t, \qquad t \in [L^i+1,T]\label{eq:min-down-time}
\end{align}
where $W^i=[\min(T,u^i_0(\underline{T}^i_{on}-T^i_0)))]^+$ and $L^i=[\min(T,(1-u^i_0)(\underline{T}^i_{off}+T^i_0))]^+$.
With the minimum up/down time constraints, taking the generator’s history into account, the initial status of the unit is subject to the following constraints \cite{yang2017novel}:
\begin{equation}
	u^i_t=u^i_0, \qquad t \in [0,W^i+L^i]\label{eq:initial-status}
\end{equation}

Let $\underline{P}^i$ be the minimum power output of unit $i$. The simplest generation limits are provided by \cite{carrion2006computationally}.
\begin{align}
	&u^i_t{\underline{P}}^i\le P^i_t\label{eq:min-output}\\
	&P^i_t\le u^i_t{\overline{P}}^i\label{eq:max-output}
\end{align}
If startup $(P^i_{start})$ and shutdown $(P^i_{shut})$ ramping limits are taken into account, the upper bounds can be tighten \cite{morales2013tight}. For unit $i \in \mathcal{J}^{>1}:=\{i\vert\underline{T}^i_{on}>1\}$:
\begin{equation}
	P^i_t\le u^i_t{\overline{P}}^i-v^i_t({\overline{P}}^i-P^i_{start})-w^i_{t+1}({\overline{P}}^i-P^i_{shut}), \qquad t \in [1,T-1]\label{eq:max-output-12}
\end{equation}
And for unit $i \in \mathcal{J}^1:=\{i\vert\underline{T}^i_{on}=1\}$, \cite{gentile2017tight} proposes tighter upper bound constraints:
\begin{align}
	&P^i_t\le u^i_t{\overline{P}}^i-v^i_t({\overline{P}}^i-P^i_{start})-w^i_{t+1}[P^i_{start}-P^i_{shut}]^+, \qquad t \in [1,T-1]\label{eq:max-output-21}\\
	&P^i_t\le u^i_t{\overline{P}}^i-w^i_{t+1}({\overline{P}}^i-P^i_{shut})-v^i_t[P^i_{shut}-P^i_{start}]^+, \qquad t \in [1,T-1]\label{eq:max-output-22}
\end{align}

Let $P^i_{up}$ and $P^i_{down}$ be the ramp-up and ramp-down limits for unit $i$ respectively. \cite{damci2016polyhedral} provides the ramping constraints which are proved to be facets of two-period ramping polytope.
\begin{align}
	&P^i_t-P^i_{t-1} \le u^i_t(P^i_{up}+\underline{P}^i)-u^i_{t-1}\underline{P}^i+v^i_t(P^i_{start}-P^i_{up}-\underline{P}^i)\label{eq:ramp-up-1}\\
	&P^i_{t-1}-P^i_t\le u^i_{t-1}(P^i_{down}+\underline{P}^i)-u^i_t\underline{P}^i+w^i_t(P^i_{shut}-P^i_{down}-\underline{P}^i)\label{eq:ramp-down-1}
\end{align}
For unit $i \in \mathcal{J}^{>1}\cap\mathcal{L}$, where $\mathcal{L}:=\{i\vert P_{up}^i>P_{shut}^i-\underline{P}^i\}, $ \cite{ostrowski2012tight} proposes a strengthened ramp-up inequality.
\begin{align}
	&P^i_t-P^i_{t-1} \le u^i_tP^i_{up}-w^i_t\underline{P}^i-w^i_{t+1}(P^i_{up}-P^i_{shut}+\underline{P}^i)+v^i_t(P^i_{start}-P^i_{up}),\notag\\
	&t \in [1,T-1]\label{eq:ramp-up-2}
\end{align}
A strengthened ramp-down inequality for unit $i \in \mathcal{J}^{>1}\cap\underline{\mathcal{L}}$, where $\underline{\mathcal{L}}:=\{i\vert P_{down}^i>P_{start}^i-\underline{P}^i\}$, is also proposed.
\begin{align}
	P^i_{t-1}-P^i_t \le &u^i_tP^i_{down}+w^i_tP^i_{shut}-v^i_{t-1}(P^i_{down}-P^i_{start}+\underline{P}^i)\notag\\
	&-v^i_t(P^i_{down}+\underline{P}^i), \qquad t \in [2,T]\label{eq:ramp-down-2}
\end{align}
For unit $i \in \mathcal{J}^{>1}\cap\underline{\mathcal{J}}^{>1}\cap\mathcal{L}$, bounded on three periods, the other ramping constraint is
\begin{align}
	P^i_{t+1}-P^i_{t-1}\le &2u^i_{t+1}P^i_{up}-w^i_t\underline{P}^i-w^i_{t+1}\underline{P}^i+v^i_t(P^i_{start}-P^i_{up})+\notag\\
	&v^i_{t+1}(P^i_{start}-2P^i_{up}), \qquad t \in [1,T-1]\label{eq:ramp-up-3}
\end{align}
where $\underline{\mathcal{J}}^{>1}:=\{i\vert\underline{T}^i_{off}>1\}$.

When compactness and tightness are taken into account simultaneously, the state-of-the-art two-period MIQP UC formulation \cite{damci2016polyhedral}, denoted as 2-period model (2P), is
\begin{align}
	&\min \quad(\ref{eq:objective-function})\notag\\
	&s.t.\left\{
	\begin{aligned}
		&(\ref{eq:power-balance})(\ref{eq:spinning-reserve})(\ref{eq:logical})(\ref{eq:hot-start})(\ref{eq:cold-start})(\ref{eq:min-up-time})(\ref{eq:min-down-time})\\
		&(\ref{eq:initial-status})(\ref{eq:min-output})(\ref{eq:max-output})(\ref{eq:ramp-up-1})(\ref{eq:ramp-down-1})\\
		&u^i_t,v^i_t,w^i_t\in \{0,1\};P^i_t,S^i_t\in R_+
	\end{aligned}
	\right.
\end{align}

The state-of-the-art UC formulation with tightest unit generation and ramping constraints within three periods \cite{gentile2017tight}\cite{damci2016polyhedral}\cite{morales2013tight}\cite{ostrowski2012tight}, denoted as 3-period model (3P), is
\begin{align}
	&\min \quad(\ref{eq:objective-function})\notag\\
	&s.t.\left\{
	\begin{aligned}
		&(\ref{eq:power-balance})(\ref{eq:spinning-reserve})(\ref{eq:logical})(\ref{eq:hot-start})(\ref{eq:cold-start})(\ref{eq:min-up-time})(\ref{eq:min-down-time})(\ref{eq:initial-status})(\ref{eq:min-output})\\
		&(\ref{eq:max-output})\text{ for } t=T;(\ref{eq:max-output-12})\text{ for } i \in \mathcal{J}^{>1};(\ref{eq:max-output-21})(\ref{eq:max-output-22})\text{ for } i \in \mathcal{J}^1\\
		&(\ref{eq:ramp-up-1}) \text{ for } i \in (\mathbb{N}-(\mathcal{J}^{>1}\cap\mathcal{L}))\cup((\mathcal{J}^{>1}\cap\mathcal{L})\land(t=T))\\
		&(\ref{eq:ramp-down-1}) \text{ for } i \in (\mathbb{N}-(\mathcal{J}^{>1}\cap\underline{\mathcal{L}})\cup((\mathcal{J}^{>1}\cap\underline{\mathcal{L}})\land(t=1))\\
		&(\ref{eq:ramp-up-2}) \text{ for } i \in \mathcal{J}^{>1}\cap\mathcal{L};(\ref{eq:ramp-down-2}) \text{ for } i \in \mathcal{J}^{>1}\cap\underline{\mathcal{L}}\\
		&(\ref{eq:ramp-up-3})\text{ for } i \in \mathcal{J}^{>1}\cap\underline{\mathcal{J}}^{>1}\cap\mathcal{L}\\
		&u^i_t,v^i_t,w^i_t\in \{0,1\};P^i_t,S^i_t\in R_+
	\end{aligned}
	\right.
\end{align}
where $\mathbb{N}:=\{1,\cdots,N\}$.

\subsection{Three-period locally ideal model}\label{subsec2}

In order to simplify the proof, \cite{yang2017novel} proposes to project $P_t^i$ onto $[0,1]$.
\begin{equation}
	\widetilde{P}^i_t=\frac{P^i_t-u^i_t\underline{P}^i}{\overline{P}^i-\underline{P}^i}
\end{equation}
where $\widetilde{P}^i_t \in [0,1]$. By this way, semi-continuous variable can be eliminated from the model. Similarly, we let $\widetilde{P}^i_{up}=\frac{P^i_{up}}{\overline{P}^i-\underline{P}^i}$, $\widetilde{P}^i_{down}=\frac{P^i_{down}}{\overline{P}^i-\underline{P}^i}$ , $\widetilde{P}^i_{start}=\frac{P^i_{start}-\underline{P}^i}{\overline{P}^i-\underline{P}^i}$, $\widetilde{P}^i_{shut}=\frac{P^i_{shut}-\underline{P}^i}{\overline{P}^i-\underline{P}^i}$. The production cost $f_i(P^i_t)$ can be transformed to $\widetilde{f}_i(\widetilde{P}^i_t)=\widetilde{\alpha}_iu^i_t+\widetilde{\beta}_i\widetilde{P}_t^i+\widetilde{\gamma}_i(\widetilde{P}_t^i )^2$, where $\widetilde{\alpha}_i=\alpha_i+\beta_i\underline{P}^i+\gamma_i(\underline{P}^i)^2$, $\widetilde{\beta}_i=(\overline{P}^i-\underline{P}^i)(\beta_i+2\gamma_i\underline{P}^i)$,  $\widetilde{\gamma}_i=\gamma_i(\overline{P}^i-\underline{P}^i)^2$ \cite{yang2017novel}.

The power balance constraint can be reformulated as \cite{yang2017novel}:
\begin{equation}
	\sum\nolimits_{i=1}^{N}[\widetilde{P}^i_t(\overline{P}^i-\underline{P}^i)+u^i_t\underline{P}^i]-P_{D,t}=0\label{eq:new-power-balance}
\end{equation}

\cite{yang2017novel} introduces $\widetilde{S}_t^i$ to represent the part of startup cost that exceeds $C_{hot}^i$, then (\ref{eq:hot-start}) can be discarded. And (\ref{eq:cold-start}) should be reformulated as follows \cite{yang2021two}.
\begin{equation}
	\widetilde{S}^i_t\ge (C^i_{cold}-C^i_{hot})[v^i_t-\sum\nolimits_{\pi=\max(1,t-\underline{T}^i_{off}-T^i_{cold})}^{t-1}w^i_{\pi}-m^i_t]\label{eq:start-cost}
\end{equation}
\cite{yang2021two} proves that (\ref{eq:start-cost}) is more compact than (\ref{eq:hot-start})-(\ref{eq:cold-start}). The objective function (\ref{eq:objective-function}) can be reformulated as
\begin{equation}
	F_c=\sum\nolimits_{i=1}^{N}\sum\nolimits_{t=1}^{T}[\widetilde{f}_i(\widetilde{P}^i_t)+C^i_{hot}v^i_t+\widetilde{S}^i_t]\label{eq:new-objective-function}
\end{equation}

For notational simplicity, we drop the superscript $i$ for the generator when considering the single-unit constraints in the following description of this paper. \cite{yang2021two} introduces eight binary variables $\tau_{i,t}^1\sim\tau_{i,t}^8$, illustrated in Table \ref{tab:state-variable}, to represent the commitment of a single unit within three periods.  According to the linear relations between $u_{t-1}^i$, $u_t^i$, $u_{t+1}^i$, $v_t^i$, $v_{t+1}^i$, $w_t^i$, $w_{t+1}^i$, $\tau_{i,t}^1\sim\tau_{i,t}^8$, the new variables except $\tau^3_{i,t}$ can be eliminated from the following constraints in this section. $\tau^3_{i,t}$ can be determined by the following inequalities:
\begin{align}
	&\tau^3_t\ge v_t+w_{t+1}-u_t, \qquad t \in [1,T-1]\label{eq:tao3-in1} \\
	&\tau^3_t\le w_{t+1}, \qquad t \in [1,T-1]\label{eq:tao3-in2}\\
	&\tau^3_t\le v_t, \qquad t \in [1,T-1]\label{eq:tao3-in3}
\end{align}
\begin{table}[t]
	\begin{center}
		\begin{minipage}{\textwidth}
			\caption{Illustration for state variables.}\label{tab:state-variable}
			\resizebox{\textwidth}{!}{
			\begin{tabular}{ccccccccccccccccc}
				\toprule
				$u_{t-1}$ & $u_t$ & $u_{t+1}$ & $\tau_t^1$ & $\tau_t^2$ & $\tau_t^3$ & $\tau_t^4$ & $\tau_t^5$ & $\tau_t^6$ & $\tau_t^7$ & $\tau_t^8$ & $\tau^{t-1}_{t-1,t-1}$ & $\tau^{t-1}_{t-1,t}$ & $\tau^{t-1}_{t-1,t+1}$ & $\tau^{t-1}_{t,t}$ & $\tau^{t-1}_{t,t+1}$ & $\tau^{t-1}_{t+1,t+1}$ \\
				\midrule
				0 & 0 & 0 & 1 & 0 & 0 & 0 & 0 & 0 & 0 & 0 & 0 & 0 & 0 & 0 & 0 & 0 \\
				0 & 0 & 1 & 0 & 1 & 0 & 0 & 0 & 0 & 0 & 0 & 0 & 0 & 0 & 0 & 0 & 1 \\
				0 & 1 & 0 & 0 & 0 & 1 & 0 & 0 & 0 & 0 & 0 & 0 & 0 & 0 & 1 & 0 & 0 \\
				0 & 1 & 1 & 0 & 0 & 0 & 1 & 0 & 0 & 0 & 0 & 0 & 0 & 0 & 0 & 1 & 0 \\
				1 & 0 & 0 & 0 & 0 & 0 & 0 & 1 & 0 & 0 & 0 & 1 & 0 & 0 & 0 & 0 & 0 \\
				1 & 0 & 1 & 0 & 0 & 0 & 0 & 0 & 1 & 0 & 0 & 1 & 0 & 0 & 0 & 0 & 1 \\
				1 & 1 & 0 & 0 & 0 & 0 & 0 & 0 & 0 & 1 & 0 & 0 & 1 & 0 & 0 & 0 & 0 \\
				1 & 1 & 1 & 0 & 0 & 0 & 0 & 0 & 0 & 0 & 1 & 0 & 0 & 1 & 0 & 0 & 0 \\
				\botrule
			\end{tabular}
		    }
		\end{minipage}
	\end{center}
\end{table}

With the introducing of new state variables, \cite{yang2021two} obtains a strong valid inequality for the unit generation limits.
\begin{align}
	\widetilde{P}_{t-1}\le& u_{t-1}-w_t(1-\widetilde{P}_{shut})-w_{t+1}(1-\widetilde{P}_{down}-\widetilde{P}_{shut})\notag\\&+\tau^3_t(1-\widetilde{P}_{down}-\widetilde{P}_{shut}), \qquad t \in [1,T-1]\label{eq:max-output-31}\\
	\widetilde{P}_t\le& u_t-v_t(1-\widetilde{P}_{start})-w_{t+1}(1-\widetilde{P}_{shut})\notag\\&+\tau^3_t[1-\max(\widetilde{P}_{start},\widetilde{P}_{shut})], \qquad t \in [1,T-1]\label{eq:max-output-32}\\
	\widetilde{P}_{t+1}\le& u_{t+1}-v_t(1-\widetilde{P}_{up}-\widetilde{P}_{start})-v_{t+1}(1-\widetilde{P}_{start})\notag\\&+\tau^3_t(1-\widetilde{P}_{up}-\widetilde{P}_{start}), \qquad t \in [1,T-1]\label{eq:max-output-33}
\end{align}
\cite{yang2021two} has proved that (\ref{eq:max-output-31})-(\ref{eq:max-output-33}) are facets of three-period ramping polytope on the assumption that $\widetilde{P}_{shut}+\widetilde{P}_{down}<1$ and  $\widetilde{P}_{start}+\widetilde{P}_{up}<1$. Similar to the upper bound limit for the unit generation, \cite{yang2021two} provides tighter ramping constraints with taking startup and shutdown ramping limits into account.
\begin{align}
	\widetilde{P}_t-\widetilde{P}_{t-1}\le& v_t(\widetilde{P}_{start}-\widetilde{P}_{up})+\tau^3_t([\widetilde{P}_{up}-\widetilde{P}_{shut}]^+-[\widetilde{P}_{start}-\widetilde{P}_{shut}]^+)\notag\\
	&+u_t\widetilde{P}_{up}-w_{t+1}[\widetilde{P}_{up}-\widetilde{P}_{shut}]^+, \qquad t \in [1,T-1]\label{eq:ramp-up-31}\\
	\widetilde{P}_{t+1}-\widetilde{P}_t\le& u_{t+1}\widetilde{P}_{up}+v_{t+1}(\widetilde{P}_{start}-\widetilde{P}_{up}), \qquad t \in [1,T-1]\label{eq:ramp-up-32}\\
	\widetilde{P}_{t+1}-\widetilde{P}_{t-1}\le& 2u_{t+1}\widetilde{P}_{up}+v_t(\widetilde{P}_{start}-\widetilde{P}_{up})+v_{t+1}(\widetilde{P}_{start}-2\widetilde{P}_{up})\notag\\
	&+\tau^3_t(\widetilde{P}_{up}-\widetilde{P}_{start}), \qquad t \in [1,T-1]\label{eq:ramp-up-33}\\
	\widetilde{P}_{t-1}-\widetilde{P}_t\le& u_{t-1}\widetilde{P}_{down}+w_t(\widetilde{P}_{shut}-\widetilde{P}_{down})\label{eq:ramp-down-31}\\
	\widetilde{P}_t-\widetilde{P}_{t+1}\le& w_{t+1}(\widetilde{P}_{shut}-\widetilde{P}_{down})+\tau^3_t([\widetilde{P}_{down}-\widetilde{P}_{start}]^+-[\widetilde{P}_{shut}-\widetilde{P}_{start}]^+)\notag\\
	&+u_t\widetilde{P}_{down}-v_t[\widetilde{P}_{down}-\widetilde{P}_{start}]^+, \qquad t \in [1,T-1]\label{eq:ramp-down-32}\\
	\widetilde{P}_{t-1}-\widetilde{P}_{t+1}\le& 2u_{t-1}\widetilde{P}_{down}+w_t(\widetilde{P}_{shut}-2\widetilde{P}_{down})+w_{t+1}(\widetilde{P}_{shut}-\widetilde{P}_{down})\notag\\
	&+\tau^3_t(\widetilde{P}_{down}-\widetilde{P}_{shut}), \qquad t \in [1,T-1]\label{eq:ramp-down-33}
\end{align}
In \cite{yang2021two}, (\ref{eq:ramp-up-31})-(\ref{eq:ramp-down-32}) are proved to be facets of three-period ramping polytope on the assumption that $\widetilde{P}_{shut}+\widetilde{P}_{down}<1$, $\widetilde{P}_{start}+\widetilde{P}_{up}<1$, $2\widetilde{P}_{down}<1$, $2\widetilde{P}_{up}<1$.

Then the tight and compact MIQP UC formulation within three periods \cite{yang2021two}, denoted as 3-period high dimensional model (3P-HD), is 
\begin{align}
	&\min \quad(\ref{eq:new-objective-function})\notag\\
	&s.t.\left\{
	\begin{aligned}
		&(\ref{eq:spinning-reserve})(\ref{eq:logical})(\ref{eq:min-up-time})(\ref{eq:min-down-time})(\ref{eq:initial-status})(\ref{eq:new-power-balance})(\ref{eq:start-cost})\\
		&(\ref{eq:tao3-in1})(\ref{eq:tao3-in2})(\ref{eq:tao3-in3})(\ref{eq:max-output-31})(\ref{eq:max-output-32})(\ref{eq:max-output-33})\\
		&(\ref{eq:ramp-up-31})(\ref{eq:ramp-up-33})(\ref{eq:ramp-down-32})(\ref{eq:ramp-down-33})\\
		&(\ref{eq:ramp-up-32}) \text{ for } t=T-1,(\ref{eq:ramp-down-31}) \text{ for } t=1\\
		&\tau^3_t=0 \text{ for } i \in \mathcal{J}^{>1}\\
		&u^i_t,v^i_t,w^i_t,{\tau^3_t}^i\in \{0,1\},\widetilde{P}^i_t\in [0,1],\widetilde{S}^i_t\in R_+
	\end{aligned}
	\right.
\end{align}

\section{Multi-period locally ideal model based on sliding window}\label{sec3}

Many advances have been made in the research on tightness of model in the last few decades. Some literatures give facets of feasible solutions of the UC problem under partial constraints. However, most of them only give expressions of facets, without presenting the construction method and process. Although some papers have given the construction methods, which are all in the case of two-period or three-period, they are difficult to be generalized to multi-period.

In this paper, we extend the method of constructing inequalities used in \cite{yang2021two} to multi-period. Then we obtain the facets of ramping polytope within any time periods, and the corresponding theoretical results.

We use $\textup{UB}(\cdot)$ to denote upper bound on ``$\cdot$", $\textup{LB}(\cdot)$ to denote lower bound on ``$\cdot$", $\textup{RHS}$ to represent right hand side, $\textup{LHS}$ to represent left hand side. Consider the following inequality for UC.
\begin{equation}
	\textup{LHS expression}\le \textup{RHS expression}\label{eq:inequality-expression}
\end{equation}

Traditionally, the LHS expression can be constructed from the physical significance of UC problem. For instance, according to the physical significance of generation limit, it is obvious that the ``LHS expression" can be set to be ``$\widetilde{P}^i_t$". Similarly, ``LHS expression" can be set to be ``$\widetilde{P}^i_t-\widetilde{P}^i_{t-1}$" for ramping constraints. Certainly, some ``LHS expression" without clear physical significance also can be used for constructing the tight UC constrains, for instance, equality (13) in \cite{pan2016polyhedral} where ``$P^i_{t-1}-P^i_t+P^i_{t+1}$" was used as ``LHS expression" of ramping constraints.

``RHS expression" of (\ref{eq:inequality-expression}) also can be constructed according to the physical significance of thermal unit. However, one should try to compress the upper bound on left hand side (LHS) of (\ref{eq:inequality-expression}). Other physics constraints can be considered in this procedure, and construct ``RHS expression" equaling to UB(LHS). Then stronger possible inequality would be obtained.

Based on the above method, we could improve the single-unit constraints, bounded in any multi-period, of every unit in the system next. We should note that in this section we talk about any M-period without taking history status into consideration.

\subsection{Sliding window}\label{subsec3}

As illustrated in Fig. \ref{fig:sliding-window}, the scheduling span (including $T$ periods in the planning horizon and one period prior to the planning horizon) for each unit $i$ can be divided into into $T-M^i+2$ M-periods that intersect each other, denoted as sliding window. $M^i$ represents the size of sliding window for unit $i$. Obviously, $M^i\in [2,T+1]$. We should note that the window size for each unit $i$ in the system can be different.
\begin{figure}[t]
	\centering
	\includegraphics[width=\textwidth,height=5cm]{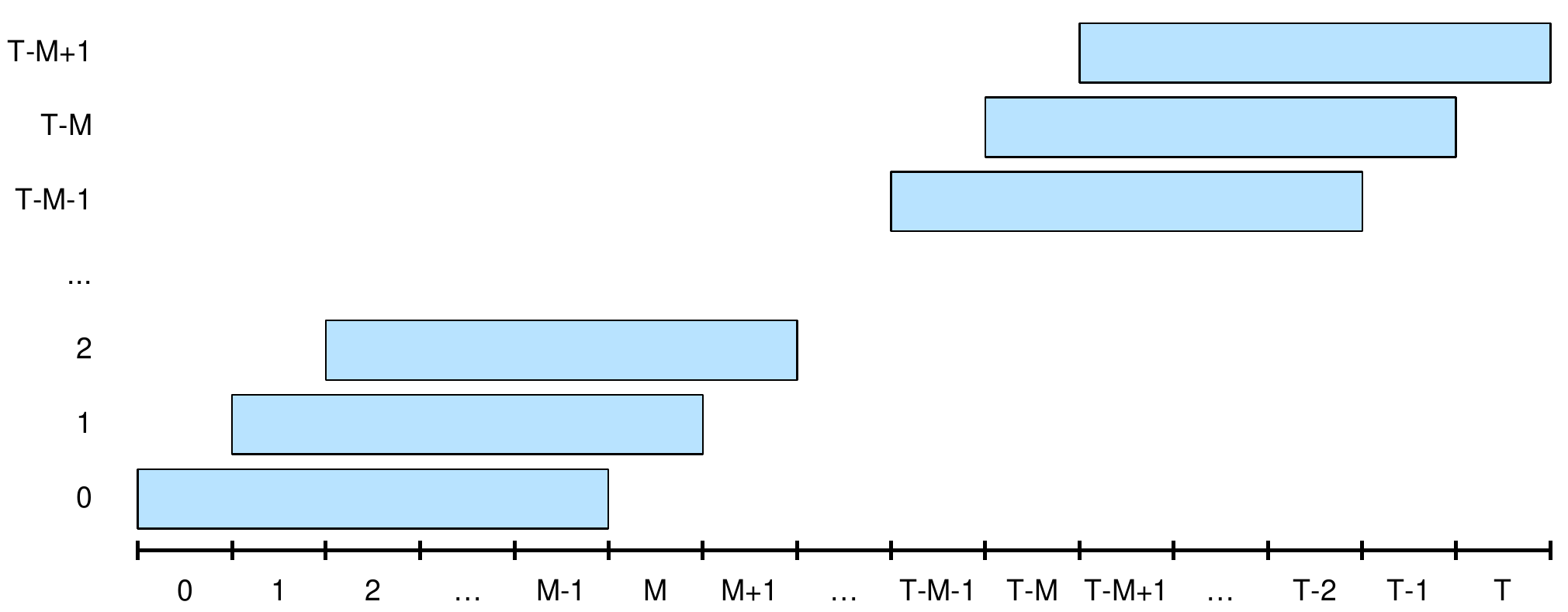}
	\caption{Sliding window for the scheduling span}
	\label{fig:sliding-window}
\end{figure}

By introducing new binary variables to represent feasible combinations of unit commitments in periods $\{t-1,t,t+1\}$, \cite{yang2021two} provides a method to construct facets of three-period polytope. Theoretically, this method can  be directly extended to multi-period, but the state variables that needs to be introduced are going to be $O(2^T)$ in this case.

In order to reduce the number of state variables, we draw on the definition of variables in \cite{frangioni2015new} to present our high-dimensional formulation for UC problem. We introduce new state variables $\tau_{h,k}^{i,m}$ to denote the schedule of unit $i$ in periods $[h,k]$ of M-period which ranks $m$. For any M-period $m$ , we let $\tau_{h,k}^{i,m}$ be 1 if the unit $i$ is turned on at time period $h$ (i.e., $u^i_{h-1}=0$ when $h>m$, $u^i_h=1$) and remains operational until the end of time period $k$ (i.e., $u^i_k=1$, $u^i_{k+1}=0$ when $k<m+M^i-1$), and 0 otherwise. It is worth noting that in \cite{frangioni2015new} and \cite{bacci2019new} the generator is offline in period $h-1$ and $k+1$. While in our formulations, the commitments of the generator in period $h-1$ and $k+1$ are uncertain when $h=m$ and $k=m+M-1$ respectively.

We let $A^i_m$ be the set of all feasible continuous operating intervals for the unit $i$ in the M-period $m$. We define $A^i_m$ as $A^i_m:=\{[h,k]\vert h=m,k\in [m,m+M^i-1],\text{ or }h\in[m+1,m+M^i-1],k\in [\min(h+\underline{T}^i_{on}-1,m+M^i-1),m+M^i-1]\}$. And let $\lvert A^i_m\rvert$ be the number of elements in the set $A^i_m$. For example, we let $m=t-1$, $M^i=3$, $T^i_{on}=1$. As illustrated in Table \ref{tab:state-variable}, \cite{anjos2017unit} introduces eight state variables ${\tau_t^1}^i\sim{\tau_t^8}^i$ to represent the unit state and thus construct some stronger inequalities for generation limits and ramping limits. Here we replace ${\tau_t^1}^i\sim{\tau_t^8}^i$ with $\tau^{i,m}_{h,k}$.

According to the definition of $\tau^m_{h,k}$, we can easily represent $u_t$, $v_t$, $w_t$ by using $\tau^m_{h,k}$.
\begin{align}
	&\sum\nolimits_{\{[h,k] \in A_m,t \in [h,k]\}}\tau^m_{h,k}=u_t, \qquad t \in [m,m+M-1] 
	\label{eq:m-taou} \\
	&\sum\nolimits_{\{[h,k] \in A_m,h=t\}}\tau^m_{h,k}=v_t, \qquad t \in [m+1,m+M-1]
	\label{eq:m-taov} \\
	&\sum\nolimits_{\{[h,k] \in A_m,k=t-1\}}\tau^m_{h,k}=w_t, \qquad t \in [m+1,m+M-1]
	\label{eq:m-taow}
\end{align}
\eqref{eq:logical} can be obtained by using constraints \eqref{eq:m-taou}-\eqref{eq:m-taow}.

For each M-period, period $t$ either falls into the continuous online interval or the continuous offline interval. When minimum-down time constraints are taken into consideration, according to the physical significance of $\tau_{h,k}^m$, we obtain the following inequalities \cite{knueven2018ramping}.
\begin{equation}
	\sum\nolimits_{\{[h,k] \in A_m,t \in [h,k+\underline{T}_{off}]\}}\tau^m_{h,k}\le 1, \qquad t \in [m,m+M-1] 
	\label{eq:inequality-taotao}
\end{equation}

Two adjacent M-period are not independent, but intersecting and interrelated. The value of $u_t,v_t,w_t$ in two adjacent M-period $m$ and $m+1$ should be consistent. According to (\ref{eq:m-taou})-(\ref{eq:m-taow}), we have the following equalities.
\begin{align}
	&\sum\nolimits_{\{[h,k] \in A_m,t \in [h,k]\}}\tau^m_{h,k}=\sum\nolimits_{\{[h,k] \in A_{m+1},t \in [h,k]\}}\tau^{m+1}_{h,k},\notag\\
	&\qquad m \in [0,T-M],t \in [m+1,m+M-1] 
	\label{eq:m-taou-m} \\
	&\sum\nolimits_{\{[h,k] \in A_m,h=t\}}\tau^m_{h,k}=\sum\nolimits_{\{[h,k] \in A_{m+1},h=t\}}\tau^{m+1}_{h,k},\notag\\
	&\qquad m \in [0,T-M],t \in [m+2,m+M-1]
	\label{eq:m-taov-m} \\
	&\sum\nolimits_{\{[h,k] \in A_m,k=t-1\}}\tau^m_{h,k}=\sum\nolimits_{\{[h,k] \in A_{m+1},k=t-1\}}\tau^{m+1}_{h,k},\notag\\
	&\qquad m \in [0,T-M],t \in [m+2,m+M-1]
	\label{eq:m-taow-m}
\end{align}
Any one of (\ref{eq:m-taov-m}) and (\ref{eq:m-taow-m}) can be derived from the other one and (\ref{eq:m-taou-m}). By this way, we could use the new binary variables $\tau^m_{h,k}$ to replace the state variable $u_t,v_t,w_t$ in the UC formulation.

We let $\prod$ denote the Cartesian product throughout this paper.
\begin{theorem}\label{binary-integer-0}
	For any M-period $m$ we define the polyhedron $\mathcal{B}^R_m=\{\prod_{[h,k]\in A_m}\tau^m_{h,k}\in[0,1]^{\lvert A_m\rvert}:(\ref{eq:inequality-taotao})\}$. Let $\mathcal{B}^I_m=\mathcal{B}^R_m\cap\{\prod_{[h,k]\in A_m}\tau^m_{h,k}\in\{0,1\}^{\lvert A_m\rvert}\}$. Then $\mathcal{B}^R_m=\textup{conv}(\mathcal{B}^I_m)$.
\end{theorem}
\begin{proof}[\textbf{Proof}]
	$$
	\bm{C_h}=\begin{blockarray}{ccccc}
		\tau^m_{h,k} & \tau^m_{h,k+1} & \cdots & \tau^m_{h,m+M-1} & \\
		\begin{block}{(cccc)c}
			0 & 0 & \cdots & 0 & t=m \\
			\vdots & \vdots & \vdots & \vdots & \vdots \\
			0 & 0 & \cdots & 0 & t=h-1 \\
			1 & 1 & \cdots & 1 & t=h \\
			\vdots & \vdots & \vdots & \vdots & \vdots \\		
			1 & 1 & \cdots & 1 & t=k+\underline{T}_{off} \\
			0 & 1 & \cdots & 1 & t=k+\underline{T}_{off}+1 \\	
			\vdots & \ddots & \ddots & \vdots & \vdots \\
			0 & \cdots & 0 & 1 & t=m+M-1 \\	
		\end{block}
	\end{blockarray}
	$$
	We let $\bm{B_m}$ be the coefficient matrix of inequalities (\ref{eq:inequality-taotao}) for M-period $m$, and $\bm{B_m}=[\bm{C_m},\bm{C_{m+1}},\cdots,\bm{C_{m+M-1}}]$. For any column, if any time periods $t_1$, $t_2$ exist such that $t_1+1<t_2$, $t_1,t_2\in [m,m+M-1]$, $t_1\in [h,k+\underline{T}_{off}]$, $t_2\in [h,k+\underline{T}_{off}]$ (i.e. the elements of the row corresponding to $t=t_1$, $t=t_2$ in this column are equal to 1), then we have $t\in [h,k+\underline{T}_{off}]$ for all $t$ with $t_1<t<t_2$. Clearly, $\bm{B_m}$ is an interval matrix(Definition 2.2. of §$\uppercase\expandafter{\romannumeral3}.1.2$ in \cite{wolsey1988integer}). According to the Corollary 2.10 of §$\uppercase\expandafter{\romannumeral3}.1.2$ in \cite{wolsey1988integer}, $\bm{B_m}$ is totally unimodular. Let $\bm{d'}=\bm{0}$, $\bm{d}=\bm{1}$, $\bm{b'}=\bm{0}$, $\bm{b}=\bm{1}$, and then we know that $\mathcal{B}^R_m$ is an integral polyhedron according to Proposition 2.3 of §$\uppercase\expandafter{\romannumeral3}.1.2$ in \cite{wolsey1988integer}. Analogous theorem, using different proof technique, is proposed in \cite{knueven2018ramping}.
\end{proof}

\subsection{The multi-period locally-facet-based constraints}\label{subsec4}

In this subsection we will talk about the generation limits and ramping constraints in any M-period $m\in [0,T-M+1]$.

Facet is an important concept for the description of polyhedral and can be used to address the question of what a strong inequality means for a polyhedral. If $P$ is a polyhedron, $F=\{x\in P\vert\pi^Tx=\pi_0\}$ defines a face of the polyhedron $P$ if $\pi^Tx=\pi_0$ is a valid inequality of $P$. Furthermore, $F$ is a facet of $P$ if $F$ is a face of $P$ and $\textup{dim}(F)=\textup{dim}(P)-1$. Proposition 9.1 in \cite{wolsey2020integer} points out that, if $P$ is full-dimensional, a valid inequality $\pi^Tx=\pi_0$ is necessary in the description of $P$ if and only if it defines a facet of $P$.

However, this is unrealistic for directly constructing the facet in most cases, because an enormous number of affine independent points and state variables are needed to describe the facet of an MILP problem’s feasible set.

Since UC is an NP-hard problem, finding the explicit facet-based representation of the entire UC problem could be difficult and maybe even unrealistic. Alternatively, one may possibly resemble the facet-based formulation for a well-structured subset of the original MILP problem’s feasible set. We borrow the term ``locally" from \cite{padberg1998location}, and use ``locally" to refers that the facet-based formulation is sought for a specifically selected portion, but not the entire mathematical programming problem. Although including such a locally ideal MILP formulation into the problem does not guarantee the facet-based formulation of the entire UC problem, it could help tighten the lower bound and in turn reduce the search effort of the branch\&cut algorithm.

Now we will show the details of our method to construct the locally-facet-based expressions for unit generation limits and ramping constraints.

\begin{table}[t]
	\begin{center}
		\begin{minipage}{\textwidth}
			\caption{Illustration for upper bound of generation limits.}\label{tab:max-output}
			\begin{tabular*}{\textwidth}{ccc}
				\toprule
				case & $\tau^m_{h,k}=1$ & $\textup{UB}(\widetilde{P}_t)$ \\
				\midrule
				1  & $m<h\le t\le k<m+M-1$ & $\min[1,\widetilde{P}_{start}+(t-h)\widetilde{P}_{up},\widetilde{P}_{shut}+(k-t)\widetilde{P}_{down}]$\\
				2  & $m<h\le t\le k=m+M-1$ & $\min[1,\widetilde{P}_{start}+(t-h)\widetilde{P}_{up}]$\\
				3  & $m=h\le t\le k<m+M-1$ & $\min[1,\widetilde{P}_{shut}+(k-t)\widetilde{P}_{down}]$\\
				4  & $m=h\le t\le k=m+M-1$ & 1\\
				\botrule
			\end{tabular*}
		\end{minipage}
	\end{center}
\end{table}
First we consider the generation limits in $M$-period $m$. When the ramping limits are taken into consideration, we list $\textup{UB}(\widetilde{P}_t)$ in Table \ref{tab:max-output}. If period $t$ is not in any operating interval, Clearly the maximum power output in period $t$ is equal to zero. According to the Table \ref{tab:max-output}, The maximum power output in period $t$ varies with the continuous operating interval in which $t$ is located. Then we obtain the following upper bound limit for the power of unit:
\begin{align}
	\widetilde{P}_t\le &\sum\nolimits_{\{[h,k] \in A_m,m<h\le t\le k<m+M-1\}}\tau^m_{h,k}\times\notag\\
	&\min[1,\widetilde{P}_{start}+(t-h)\widetilde{P}_{up},\widetilde{P}_{shut}+(k-t)\widetilde{P}_{down}]\notag\\
	&+\sum\nolimits_{\{[h,k]\in A_m,m<h\le t\le k=m+M-1\}}\tau^m_{h,k}\min[1,\widetilde{P}_{start}+(t-h)\widetilde{P}_{up}]\notag\\
	&+\sum\nolimits_{\{[h,k]\in A_m,m=h\le t\le k<m+M-1\}}\tau^m_{h,k}\min[1,\widetilde{P}_{shut}+(k-t)\widetilde{P}_{down}]\notag\\
	&+\sum\nolimits_{\{[h,k]\in A_m,m=h\le t\le k=m+M-1\}}\tau^m_{h,k},\quad t\in[m,m+M-1]\label{eq:tight-max-output-0}
\end{align}
Analogous results are proposed in \cite{bacci2019new}. But our formulation takes all possible cases into account and is more general, and our inequalities are strictly tighter than (55)(56) in \cite{bacci2019new} when $[h,k] \in A_m$ exists such that $\widetilde{P}_{start}+(t-h)\widetilde{P}_{up}<\widetilde{P}_{shut}$ or $\widetilde{P}_{shut}+(k-t)\widetilde{P}_{down}<\widetilde{P}_{start}$.

When the startup/shutdown ramping limits and generation limits are taken into consideration, similar to the constructing of unit generation limits, we list $\textup{UB}(\widetilde{P}_t-\widetilde{P}_{t-a})$ in Table \ref{tab:ramp}. According to Table \ref{tab:ramp}, we can obtain the following ramp-up constraints:
\begin{align}
	\widetilde{P}_t-\widetilde{P}_{t-a}\le &\sum\nolimits_{\{[h,k]\in A_m,t-a<h\le t\le k=m+M-1\}}\tau^m_{h,k}\min[1,\widetilde{P}_{start}+(t-h)\widetilde{P}_{up}]\notag\\
	&+\sum\nolimits_{\{[h,k] \in A_m,t-a<h\le t\le k<m+M-1\}}\tau^m_{h,k}\times\notag\\
	&\min[1,\widetilde{P}_{start}+(t-h)\widetilde{P}_{up},\widetilde{P}_{shut}+(k-t)\widetilde{P}_{down}]\notag\\
	&+\sum\nolimits_{\{[h,k] \in A_m,h\le t-a<t\le k<m+M-1\}}\tau^m_{h,k}\times\notag\\
	&\min[1,a\widetilde{P}_{up},\widetilde{P}_{shut}+(k-t)\widetilde{P}_{down}]\notag\\
	&+\sum\nolimits_{\{[h,k] \in A_m,h\le t-a<t\le k=m+M-1\}}\tau^m_{h,k}\min(1,a\widetilde{P}_{up}),\notag\\
	&t\in[m+1,m+M-1],a\in[1,t-m]\label{eq:tight-ramp-up-0}
\end{align}
Similarly, we have ramp-down constraints as follows:
\begin{align}
	\widetilde{P}_{t-a}-\widetilde{P}_t\le &\sum\nolimits_{\{[h,k]\in A_m,m=h\le t-a\le k<t\}}\tau^m_{h,k}\min[1,\widetilde{P}_{shut}+(k-t+a)\widetilde{P}_{down}]\notag\\
	&+\sum\nolimits_{\{[h,k] \in A_m,m<h\le t-a\le k<t\}}\tau^m_{h,k}\times\notag\\
	&\min[1,\widetilde{P}_{start}+(t-a-h)\widetilde{P}_{up},\widetilde{P}_{shut}+(k-t+a)\widetilde{P}_{down}]\notag\\
	&+\sum\nolimits_{\{[h,k] \in A_m,m<h\le t-a<t\le k\}}\tau^m_{h,k}\times\notag\\
	&\min[1,\widetilde{P}_{start}+(t-a-h)\widetilde{P}_{up},a\widetilde{P}_{down}]\notag\\
	&+\sum\nolimits_{\{[h,k] \in A_m,m=h\le t-a<t\le k\}}\tau^m_{h,k}\min(1,a\widetilde{P}_{down}),\notag\\
	&t\in[m+1,m+M-1],a\in[1,t-m]\label{eq:tight-ramp-down-0}
\end{align}
Analogous results with $a=1$ are proposed in \cite{bacci2019new}. But our formulation takes generation limits into consideration and is more general with ramping constraints spanning more than two periods (i.e., $a>1$) being considered. (\ref{eq:tight-ramp-up-0}) is strictly tighter than (50) in \cite{bacci2019new} when $[h,k] \in A_m$ exists such that $\widetilde{P}_{shut}+(k-t)\widetilde{P}_{down}<\max(\widetilde{P}_{up},\widetilde{P}_{start})$. (\ref{eq:tight-ramp-down-0}) is strictly tighter than (51) in \cite{bacci2019new} when $[h,k]\in A_m$ exists such that $\widetilde{P}_{start}+(t-a-h)\widetilde{P}_{up}<\max(\widetilde{P}_{down},\widetilde{P}_{shut})$.
\begin{table}[t]
	\begin{center}
		\begin{minipage}{\textwidth}
			\caption{Illustration for upper bound of ramping limits}\label{tab:ramp}
			\resizebox{\textwidth}{!}{
			\begin{tabular}{ccc}
				\toprule
				case & $\tau^m_{h,k}=1$ & $\textup{UB}(\widetilde{P}_t-\widetilde{P}_{t-a})$ \\
				\midrule
				1 & $h\le t-a<t\le k<m+M-1$ & $\min[1,a\widetilde{P}_{up},\widetilde{P}_{shut}+(k-t)\widetilde{P}_{down}]$ \\
				2 & $h\le t-a<t\le k=m+M-1$ & $\min(1,a\widetilde{P}_{up})$ \\
				3 & $t-a<h\le t\le k<m+M-1$ & $\min[1,\widetilde{P}_{start}+(t-h)\widetilde{P}_{up},\widetilde{P}_{shut}+(k-t)\widetilde{P}_{down}]$ \\
				4 & $t-a<h\le t\le k=m+M-1$ & $\min[1,\widetilde{P}_{start}+(t-h)\widetilde{P}_{up}]$ \\
				\botrule
			\end{tabular}
		    }
	        \resizebox{\textwidth}{!}{
			\begin{tabular}{ccc}
				\toprule
				case & $\tau^m_{h,k}=1$ & $\textup{UB}(\widetilde{P}_{t-a}-\widetilde{P}_t)$ \\
				\midrule
				1 & $m<h\le t-a<t\le k$ & $\min[1,\widetilde{P}_{start}+(t-a-h)\widetilde{P}_{up},a\widetilde{P}_{down}]$\\
				2 & $m=h\le t-a<t\le k$ & $\min(1,a\widetilde{P}_{down})$\\
				3 & $m<h\le t-a\le k<t$ & $\min[1,\widetilde{P}_{start}+(t-a-h)\widetilde{P}_{up},\widetilde{P}_{shut}+(k-t+a)\widetilde{P}_{down}]$\\
				4 & $m=h\le t-a\le k<t$ & $\min[1,\widetilde{P}_{shut}+(k-t+a)\widetilde{P}_{down}]$\\
				\botrule
			\end{tabular}
		    }
		\end{minipage}
	\end{center}
\end{table}

We note that similar expressions had been proposed in \cite{bacci2019new}. Inspired by dynamic programming algorithm, \cite{bacci2019new} presents a new formulation for single-unit commitment problem and prove that it describes the convex hull of the solutions of the single-unit commitment problem. Then they transform the DP convex hull model to obtain several similar expressions for unit generation limits and ramping constraints which are formed by linear combination of some constraints of DP model.

However:
\begin{enumerate}
	\item Our procedure to construct the expressions is more directly and more simple than \cite{bacci2019new}. And some comprehensive and rigorous facet theory are proved in this paper.
	\item Expressions in \cite{bacci2019new} were obtained by relaxing the DP model, the theoretical tightness of resulted expressions have not been given. 
	\item As we mentioned above, our model is tighter than that in \cite{bacci2019new}, because we consider the constraints more comprehensively. For instance, we consider the upper bounds of power generations of the generator when constructing the ramping constraints.
	\item \cite{bacci2019new} constructs constraint inequalities on the solution space of the whole UC problem, while our model constructs tight constraint inequalites on multi-period single-unit constraint polytopes for the UC problem based on ``locally ideal" and then combines them together.
\end{enumerate}

Therefore, the model we construct in this paper satisfies all the typical constraints for thermal units: minimum up and down time, power generation limits, ramping (including ramp up, ramp down, start-up and shut-down) limits, which are concerned when constructing convex hull of polytope for the UC problem in most literatures \cite{damci2016polyhedral}\cite{frangioni2015new}\cite{pan2016polyhedral}.

In fact, by introducing new state variables $\varrho_{c,d}$ for unit $i$ which remains off in periods $[c,d]$, the method proposed above can be extended to multi-period constraint polytope with other single-unit constraints such as startup/shutdown cost. Lack of space, we won't give further treatment of them here.

\subsection{Discussion of tightness of proposed constraints}\label{subsec5}

We name the facet of polytope which considers a subset of the constraints for UC model (i.e., minimum up and down time, power generation limits, ramping limits) only in a few periods of time as ``locally-facet". We have derived some tight expressions for a subset of the constraints for unit commitment problem in the previous subsection. Here we discuss the tightness of these inequalities. 

We refer to the feasible set defined by constraints (\ref{eq:inequality-taotao})(\ref{eq:tight-max-output-0}) as $\mathcal{P}^R_m$, i.e. $\mathcal{P}^R_m=\{(\prod_{[h,k]\in A_m}\tau^m_{h,k},\widetilde{P}_m,\cdots,\widetilde{P}_{m+M-1})\in [0,1]^{\lvert A_m\rvert +M}:\eqref{eq:inequality-taotao}(\ref{eq:tight-max-output-0})\}$. Let $\mathcal{P}^I_m=\mathcal{P}^R_m\cap\{(\prod_{[h,k]\in A_m}\tau^m_{h,k},\widetilde{P}_m,\cdots,\widetilde{P}_{m+M-1})\in\{0,1\}^{\lvert A_m\rvert}\times[0,1]^M\}$.

\begin{proposition}\label{prp1}
	Inequality (\ref{eq:tight-max-output-0}) defines a facet of $\textup{conv}(\mathcal{P}^I_m)$.
\end{proposition}
\begin{proof}[\textbf{Proof}]
	See ``Appendix 2".
\end{proof}

\begin{theorem}
	$\mathcal{P}^R_m=\textup{conv}(\mathcal{P}^I_m)$.
\end{theorem}
\begin{proof}[\textbf{Proof}]
	We have proved in Theorem \ref{binary-integer-0} that $\mathcal{B}^R_m$ is an integral polyhedron. According to the Lemma 4 in \cite{gentile2017tight}, it is not difficult to prove $\mathcal{P}^R_m=\textup{conv}(\mathcal{P}^I_m)$.
\end{proof}

We refer to the feasible set defined by constraints (\ref{eq:inequality-taotao})(\ref{eq:tight-max-output-0})-(\ref{eq:tight-ramp-down-0}) as $\mathcal{Q}^R_m$, i.e. $\mathcal{Q}^R_m=\{(\prod_{[h,k]\in A_m}\tau^m_{h,k},\widetilde{P}_m,\cdots,\widetilde{P}_{m+M-1})\in [0,1]^{\lvert A_m\rvert+M}:\eqref{eq:inequality-taotao}(\ref{eq:tight-max-output-0})-(\ref{eq:tight-ramp-down-0})\}$. Let $\mathcal{Q}^I_m=\mathcal{Q}^R_m\cap\{(\prod_{[h,k]\in A_m}\tau^m_{h,k},\widetilde{P}_m,\cdots,\widetilde{P}_{m+M-1})\in\{0,1\}^{\lvert A_m\rvert}\times[0,1]^M\}$.

\begin{proposition}\label{prp2}
	Inequality (\ref{eq:tight-max-output-0}) defines a facet of $\textup{conv}(\mathcal{Q}^I_m)$.
\end{proposition}
\begin{proof}[\textbf{Proof}]
	See ``Appendix 2".
\end{proof}

\begin{proposition}\label{prp3}
	Inequality (\ref{eq:tight-ramp-up-0}) defines a facet of $\textup{conv}(\mathcal{Q}^I_m)$ if and only if $a<\frac{1}{\widetilde{P}_{up}}$.
\end{proposition}
\begin{proof}[\textbf{Proof}]
	See ``Appendix 2".
\end{proof}

\begin{proposition}\label{prp4}
	Inequality (\ref{eq:tight-ramp-down-0}) defines a facet of $\textup{conv}(\mathcal{Q}^I_m)$ if and only if $a<\frac{1}{\widetilde{P}_{down}}$.
\end{proposition}
\begin{proof}[\textbf{Proof}]
	See ``Appendix 2".
\end{proof}

We let $G_1=\min(\lceil\frac{1}{\widetilde{P}_{up}}\rceil,M)$ and $G_2=\min(\lceil\frac{1}{\widetilde{P}_{down}}\rceil,M)$. Our work identifies facets and removes a large number of redundant inequalities (roughly accounts for   $\frac{2M(M-1)-(G_1-1)(2M-G_1)-(G_2-1)(2M-G_2)}{2M(M-1)}\times100\%)$, as shown in Fig. \ref{fig:facet-constraints} in the numerical experiment, which can significantly improve the speed of modeling and solving.

\subsection{Multi-period locally-facet-based UC formulation}\label{subsec6}

Based on sliding window, we can obtain Multi-Period Locally-facet-based UC formulation. Since we omit the state variable $u_t^i$, we have to transform some of the original constraints.

\begin{equation}
	\sum\nolimits_{\{[h,k] \in A_{\min(t,T-M+1)},t \in [h,k]\}}\tau^{\min(t,T-M+1)}_{h,k}=u_0,\qquad t\in[0,W+L]\label{eq:new-initial-status-0}
\end{equation}
We should note that qualities (\ref{eq:new-initial-status-0}) can replace (\ref{eq:initial-status}). 

For the convenience of processing, we classify the units as follows. We let $J^{>x}:=\{i\vert\underline{T}_{on}^i>x\} $ and $K^{>x}:=\{i\vert\underline{T}_{off}^i>x\}$ where $x\in[1,T]$.
We should note that the minimum up time constraints are actually implicit in the definition of $A^i_m$. But it's not complete in the global space. According to linear relations (\ref{eq:m-taou})-(\ref{eq:m-taow}), we can reformulate minimum up/down time constraints (\ref{eq:min-up-time})(\ref{eq:min-down-time}) as follows:
\begin{align}
	&\sum\nolimits_{\overline{\omega}=[t-\underline{T}_{on}]^++1}^t\sum\nolimits_{\{[\overline{\omega},k] \in A_{\max(0,\overline{\omega}-M+1)}\}}\tau^{\max(0,\overline{\omega}-M+1)}_{\overline{\omega},k}\le\notag\\
	&\sum\nolimits_{\{[h,k] \in A_{\min(t,T-M+1)},t \in [h,k]\}}\tau^{\min(t,T-M+1)}_{h,k}, \quad t \in [W+1,T]\label{eq:tao-min-up-time}\\
	&\sum\nolimits_{\overline{\omega}=[t-\underline{T}_{off}]^++1}^t\sum\nolimits_{\{[h,\overline{\omega}-1] \in A_{\min(\overline{\omega}-1,T-M+1)}\}}\tau^{\min(\overline{\omega}-1,T-M+1)}_{h,\overline{\omega}-1}\le\notag\\
	&1-\sum\nolimits_{\{[h,k] \in A_{\min(t,T-M+1)},t \in [h,k]\}}\tau^{\min(t,T-M+1)}_{h,k}, \quad t \in [L+1,T]\label{eq:tao-min-down-time}
\end{align}
If $\underline{T}^i_{on}<M^i$, then (\ref{eq:tao-min-up-time}) are redundant and can be removed. If $\underline{T}^i_{off}<M^i$, then (\ref{eq:tao-min-down-time}) are redundant and can be removed.

Using (\ref{eq:m-taou})-(\ref{eq:m-taow}), we can reformulate (\ref{eq:spinning-reserve}), (\ref{eq:new-power-balance}) and (\ref{eq:start-cost}) as:
\begin{align}
	&\widetilde{S}_t\ge (C_{cold}-C_{hot})[\sum\nolimits_{\{[t,k] \in A_{\max(0,t-M+1)}\}}\tau^{\max(0,t-M+1)}_{t,k}-\notag\\
	&\sum\nolimits_{\pi=\max(1,t-\underline{T}_{off}-T_{cold})}^{t-1}\sum\nolimits_{\{[h,\pi-1] \in A_{\min(\pi-1,T-M+1)}\}}\tau^{\min(\pi-1,T-M+1)}_{h,\pi-1}-m_t]\label{eq:tao-start-cost}\\
	&\sum\nolimits_{i=1}^{N}[\widetilde{P}^i_t(\overline{P}^i-\underline{P}^i)+\underline{P}^i\sum\nolimits_{\{[h,k] \in A^i_{\min(t,T-M^i+1)},t \in [h,k]\}}\tau^{i,\min(t,T-M^i+1)}_{h,k}]\notag\\
	&-P_{D,t}=0\label{eq:tao-power-balance}\\
	&\sum\nolimits_{i=1}^N\overline{P}^i\sum\nolimits_{\{[h,k] \in A^i_{\min(t,T-M^i+1)},t \in [h,k]\}}\tau^{i,\min(t,T-M^i+1)}_{h,k}\ge P_{D,t}+R_t\label{eq:tao-spinning-reserve}
\end{align}

We reformulate $\widetilde{f}^i(\widetilde{P}^i_t)$ as $\widetilde{f}^i(\widetilde{P}^i_t)=\widetilde{\alpha}_i\sum\nolimits_{\{[h,k]\in A^i_m,t\in[h,k]\}}\tau^{i,m}_{h,k}+\widetilde{\beta}_i\widetilde{P}_t^i+\widetilde{\gamma}_i(\widetilde{P}_t^i )^2,m=\min(t,T-M^i+1)$. A multi-period tight MIQP UC formulation, denoted as MP-1, that contains only binary variables $\tau^{i,m}_{h,k}$ could be obtained directly.
\begin{align}
	&\min \quad\sum_{i=1}^{N}\sum_{t=1}^{T}[\widetilde{f}_i(\widetilde{P}^i_t)+C^i_{hot}\sum\nolimits_{\{[t,k] \in A^i_{\max(0,t-M^i+1)}\}}\tau^{i,\max(0,t-M^i+1)}_{t,k}+\widetilde{S}^i_t]\notag\\
	&s.t.\left\{
	\begin{aligned}
		&(\ref{eq:inequality-taotao})(\ref{eq:m-taou-m})(\ref{eq:m-taov-m})(\ref{eq:m-taow-m})(\ref{eq:tight-max-output-0})(\ref{eq:new-initial-status-0})(\ref{eq:tao-spinning-reserve})(\ref{eq:tao-power-balance})(\ref{eq:tao-start-cost})\\
		&(\ref{eq:tight-ramp-up-0})\text{ for }m=T-M^i+1\text{ or }a=t-m\\
		&(\ref{eq:tight-ramp-down-0})\text{ for }m=0\text{ or }t=m+M^i-1\\
		&(\ref{eq:tao-min-up-time})\text{ for } i\in J^{>M^i-1},(\ref{eq:tao-min-down-time})\text{ for } i\in K^{>M^i-1}\\
		&\widetilde{P}^i_t\in [0,1],\widetilde{S}^i_t\in R_+\\
		&\tau^{i,m}_{h,k}\in \{0,1\},[h,k]\in A^i_m,m\in [0,T-M^i+1]
	\end{aligned}
	\right.
\end{align}
We should note that $M^i$ can take on any value belonging to $[2,T+1]$. The variables $u^i_t, v^i_t, w^i_t$ in UC formulation are replaced with $\tau^{i,m}_{h,k}$ by using (\ref{eq:m-taou})-(\ref{eq:m-taow}) and adding (\ref{eq:m-taou-m})-(\ref{eq:m-taow-m}) to the model. In this way, the number of non-zero elements in the formulation increase dramatically. For any unit $i$ in $M^i$ periods, we note that, when only considering state variable constraints, Minimum up/down time constraints, and unit generation limits, our model gives the ideal formulation (convex hull of the solutions of the single-unit commitment problem). And when ramping constraints are considered furthermore, our model gives the facets of unit generation limits and ramping constraints.

In addition, compared to DP-based model, our model can segment the UC problem which is based on sliding window. For example, for a 24-period scheduling span, we can take every six periods as a group to build the multi-period model presented in this paper and then connect them together. As to the size of sliding window, we need to consider comprehensively to get the most efficient model.

\begin{table}[t]
	\begin{center}
		\begin{minipage}{\textwidth}
			\caption{Number of variables and constraints.}\label{tab:variable-num}
			\resizebox{\textwidth}{!}{
				\begin{tabular}{ccccccccccccccccccccccccc}
					\toprule
					M & 2 & 3 & 4 & 5 & 6 & 7 & 8 & 9 & 10 & 11 & 12 & 13 & 14 & 15 & 16 & 17 & 18 & 19 & 20 & 21 & 22 & 23 & 24 & 25 \\
					\midrule
					$u_t,v_t,w_t$ & 72 & 72 & 72 & 72 & 72 & 72 & 72 & 72 & 72 & 72 & 72 & 72 & 72 & 72 & 72 & 72 & 72 & 72 & 72 & 72 & 72 & 72 & 72 & 72 \\
					$n_1$ & 72 & 138 & 220 & 315 & 420 & 532 & 648 & 765 & 880 & 990 & 1092 & 1183 & 1260 & 1320 & 1360 & 1377 & 1368 & 1330 & 1260 & 1155 & 1012 & 828 & 600 & 325 \\
					$n_2$ & 72 & 115 & 154 & 189 & 220 & 247 & 270 & 289 & 304 & 315 & 322 & 325 & 324 & 319 & 310 & 297 & 280 & 259 & 234 & 205 & 172 & 135 & 94 & 49 \\
					Upper bounds & 47 & 68 & 87 & 104 & 119 & 132 & 143 & 152 & 159 & 164 & 167 & 168 & 167 & 164 & 159 & 152 & 143 & 132 & 119 & 104 & 87 & 68 & 47 & 24 \\
					Ramping constraints & 48 & 138 & 264 & 420 & 600 & 798 & 1008 & 1224 & 1440 & 1650 & 1848 & 2028 & 2184 & 2310 & 2400 & 2448 & 2448 & 2394 & 2280 & 2100 & 1848 & 1518 & 1104 & 600 \\
					\botrule
				\end{tabular}
			}
		\end{minipage}
	\end{center}
\end{table}

If the minimum up/down time constraints and generator’s history state are not taken into consideration in the definition of $A^i_m$, the number of variables $\tau^{i,m}_{h,k}$ is equal to $n_1=\frac{M^i(M^i+1)(T-M^i+2)}{2}$. As the value of $M^i$ increases, the number of variables first increases and then decreases. If we take minimum up time constraints into account, the number of variables $\tau^{i,m}_{h,k}$ is equal to $n_2=\{2M^i-1+\frac{{[M^i-\underline{T}^i_{on}-1]}^+(M^i-\underline{T}^i_{on})}{2}\}(T-M^i+2)$. With $M^i$ fixed, $n_2$ reaches minimum when $\underline{T}^i_{on}\ge M^i-1$. The number of generation limits and ramping constraints are equal to $M^i(T-M^i+2)-1$ and $M^i(M^i-1)(T-M^i+2)$ respectively. Here we let $T=24$, $\underline{T}^i_{on}=M^i-1$ and list the number of variables and constraints for a single unit in the Table \ref{tab:variable-num}.

In order to construct a compact and tight model, we can adjust the values of $M^i$ for different unit $i$. When $M^i\ge \max(\lceil\frac{1-\widetilde{P}^i_{start}}{\widetilde{P}^i_{up}}\rceil+1,\lceil\frac{1-\widetilde{P}^i_{shut}}{\widetilde{P}^i_{down}}\rceil+1,\lceil\frac{1}{\widetilde{P}^i_{up}}\rceil,\lceil\frac{1}{\widetilde{P}^i_{down}}\rceil,2)$, we could get tightest generation upper bounds and ramping constraints in the global space.

\subsection{Improvements of multi-period UC formulation}\label{subsec7}

The number of nonzero elements of model MP-1 increases as all of the variables $u^i_t,v^i_t,w^i_t$ in the model are replaced with $\tau^{i,m}_{h,k}$. By adding the variables $u^i_t,v^i_t,w^i_t$, model MP-1 can be transformed to
\begin{align}
	&\min \quad(\ref{eq:new-objective-function})\notag\\
	&s.t.\left\{
	\begin{aligned}
		&(\ref{eq:spinning-reserve})(\ref{eq:logical})(\ref{eq:min-up-time})(\ref{eq:min-down-time})(\ref{eq:initial-status})(\ref{eq:new-power-balance})(\ref{eq:start-cost})(\ref{eq:m-taou})(\ref{eq:m-taov})(\ref{eq:m-taow})(\ref{eq:tight-max-output-0})\\
		&(\ref{eq:tight-ramp-up-0})\text{ for } m=T-M^i+1\text{ or }a=t-m\\
		&(\ref{eq:tight-ramp-down-0})\text{ for } m=0\text{ or }t=m+M^i-1\\
		&u^i_t,v^i_t,w^i_t\in\{0,1\},\widetilde{P}^i_t\in [0,1],\widetilde{S}^i_t\in R_+\\
		&\tau^{i,m}_{h,k}\in\{0,1\},[h,k]\in A^i_m,m\in[0,T-M^i+1]
	\end{aligned}
	\right.
\end{align}
which is denoted as MP-2.

In general, models that are more compact and tighter are more computationally efficient. However, when the tightness of the model reaches a certain level, the increase of tightness is of little help to the improvement of its computational efficiency. On the other hand, too many binary variables will also increase the computational burden.

This formulation is not compact enough because it contains too many variables. Since variables $u_t,v_t,w_t,\tau^m_{h,k}$ are linearly dependent, some of these variables can be represented by linear combinations of the rest. In order to reduce the number of variables without adding too many non-zero elements, $u_t,v_t,w_t$ should be preserved. And some of $\tau^m_{h,k}$ should be omitted instead. For this purpose, we try to represent $\tau^m_{h,k}$ with $u_t,v_t,w_t,\tau^m_{h,k}$ in the following.

We assume, without loss of generality, that $\underline{T}_{on}=1$. Firstly, $\tau^m_{h,k}\in\{\tau^m_{h,k}\vert h \in [m+1,m+M-2],k \in [h,m+M-2]\}$ can be viewed as basic variables.
\begin{align}
	&\left\{
	\begin{aligned}
		&\sum\nolimits_{\{t \in [h,k],h \in [m,m+M-1],k \in [h,m+M-1]\}}\tau^m_{h,k}=u_t, \quad t \in [m,m+M-1] \\ 
		&\sum\nolimits_{\{h=t,h \in [m,m+M-1],k \in [h,m+M-1]\}}\tau^m_{h,k}=v_t, \quad t \in [m+1,m+M-1] \\ 
		&\tau^m_{h,k}=\tau^m_{h,k}, \quad h \in [m+1,m+M-2],k \in [h,m+M-2]
	\end{aligned}
	\right.\notag\\
	&\Longrightarrow\qquad \bm{E}\bm{x}=\bm{b}
\end{align}
where $\bm{E}$ is an $\frac{M(M+1)}{2}\times\frac{M(M+1)}{2}$ matrix, $\bm{x}=(\tau^m_{m,m},\cdots,\tau^m_{m,m+M-1},\\\tau^m_{m+1,m+M-1},\cdots,\tau^m_{m+M-1,m+M-1},
\tau^m_{m+1,m+1},\cdots,\tau^m_{m+M-2,m+M-2})$, \\$\bm{b}=(u_m,\cdots,u_{m+M-1},v_{m+1},\cdots,v_{m+M-1},\tau^m_{m+1,m+1},\cdots,\tau^m_{m+M-2,m+M-2})$. Matrix $\bm{E}$ can be cut into four blocks.  
\begin{align}
	\bm{E}=
	\left[\begin{array}{c|c}
		\bm{C} & \bm{B} \\ 
		\hline 
		\bm{0} & \bm{I}
	\end{array}\right]
\end{align} 
where $\bm{I}$ is an identity matrix, and matrix $\bm{C}$ have the structure as follows.
$$
C=\begin{blockarray}{cccccccccc}
	1 & 2 & \cdots & M-1 & M & M+1 & \cdots & 2M-2 & 2M-1 & \\
	\begin{block}{(ccccccccc)c}
		1 & 1 & \cdots & 1 & 1 & 0 & \cdots & 0 & 0 & 1 \\	
		0 & 1 & \cdots & 1 & 1 & 1 & \ddots & \vdots & \vdots & 2 \\
		\vdots & 0 & \ddots & \vdots & \vdots & \vdots & \ddots & 0 & \vdots & \vdots \\
		0 & \vdots & \ddots & 1 & 1 & 1 & \cdots & 1 & 0 & M-1 \\			
		0 & 0 & \cdots & 0 & 1 & 1 & \cdots & 1 & 1 & M \\
		0 & 0 & \cdots & 0 & 0 & 1 & 0 & \cdots & 0 & M+1 \\	
		\vdots & \vdots & \ddots & \vdots & \vdots & 0 & \ddots & \ddots & \vdots & \vdots \\	
		0 & 0 & \cdots & 0 & 0 & \vdots & \ddots & 1 & 0 & 2M-2 \\	
		0 & 0 & \cdots & 0 & 0 & 0 & \cdots & 0 & 1 & 2M-1 \\	
	\end{block}
\end{blockarray}
$$

Matrix $\bm{E}$ can be easily transformed to identity matrix by elementary transformation, which means that matrix $\bm{E}$ is equivalent to identity matrix. Hence, matrix $\bm{E}$ is nonsingular. It is not difficult to verify that state variables $u_m$, $\cdots$, $u_{m+M-1}$, $v_{m+1}$, $\cdots$, $v_{m+M-1}$, $\tau^m_{h,k}\in \{\tau^m_{h,k}\vert m<h\le k<m+M-1\}$, are linearly independent. And all the new state variables $\tau^m_{h,k}$ can be expressed as linear combinations of these base variables.
\begin{equation}
	\bm{x}=\bm{E}^{-1}\bm{b}\label{eq:tao-taobase}
\end{equation}
According to the definition of $\tau^m_{h,k}$, we have the following inequalities.
\begin{align}
	\bm{E}^{-1}\bm{b}\ge \bm{0}\label{eq:taom-in-0}\\
	\bm{E}^{-1}\bm{b}\le \bm{1}\label{eq:taom-in-1}
\end{align}
And inequalities corresponding to variables $\tau^m_{h,k}\in \{\tau^m_{h,k}\vert m<h\le k<m+M-1\}$ in inequality groups (\ref{eq:taom-in-0})(\ref{eq:taom-in-1}) can be deleted. If $\underline{T}_{on}>1$, (\ref{eq:taom-in-0})(\ref{eq:taom-in-1}) still hold with $\tau^m_{h,k}=0,k\in[h,h+\underline{T}_{on}-2]$.

Now, we present our multi-period tight and compact MIQP UC formulation, denoted as M-period tight model (MP-3), in high-dimensional space:
\begin{align}
	&\min \quad(\ref{eq:new-objective-function})\notag\\
	&s.t.\left\{
	\begin{aligned}
		&(\ref{eq:spinning-reserve})(\ref{eq:logical})(\ref{eq:min-up-time})(\ref{eq:min-down-time})(\ref{eq:initial-status})(\ref{eq:new-power-balance})(\ref{eq:start-cost})(\ref{eq:tight-max-output-0})(\ref{eq:taom-in-0})(\ref{eq:taom-in-1})\\
		&(\ref{eq:tight-ramp-up-0})\text{ for } m=T-M^i+1\text{ or }a=t-m\\
		&(\ref{eq:tight-ramp-down-0})\text{ for } m=0\text{ or }t=m+M^i-1\\
		&u^i_t,v^i_t,w^i_t\in\{0,1\},\widetilde{P}^i_t\in [0,1],\widetilde{S}^i_t\in R_+,m\in[0,T-M^i+1]\\
		&\tau^{i,m}_{h,k}\in\{0,1\},[h,k]\in A^i_m, m<h\le k<m+M^i-1
	\end{aligned}
	\right.
\end{align}
We should note that all of the ariables $\tau^{i,m}_{h,k}\in \{\tau^{i,m}_{h,k}\vert m\le h\le k\le m+M^i-1,h=m\text{ or }k=m+M^i-1\text{ or }[h,k]\notin A^i_m\}$ in this model have been eliminated by using (\ref{eq:tao-taobase}). In fact, our 3P-3 formulation is equivalent to the 3P-HD formulation. However, 3P-HD formulation is simplified by using (\ref{eq:logical}) to make the model more compact with less nonzeros.

\section{History-dependent multi-period formulation}\label{sec4}

In the discussion in Section \ref{sec3}, we did not consider the influence of the initial state on the definition of $A^i_m$, the upper bound of power output and the ramping constraints. Although the constraints obtained in this way is ideal in the M-period, they are not tight enough in the global space. Therefore, we will make some improvements to them.

\subsection{Tighter constraints based on historical status}\label{subsec8}

If the unit $i$ has been offline for $-T^i_0$ periods prior to the first period of the time span, then $A^i_m:=\{[h,k]\vert h=m\ge L^i+1,k\in [\min(\max(m,L^i+\underline{T}^i_{on}),m+M^i-1),m+M^i-1],\text{ or }h\in[\max(m+1,L^i+1),m+M^i-1],k\in [\min(h+\underline{T}^i_{on}-1,m+M^i-1),m+M^i-1]\}$.

With ramping limits taken into consideration , we let $U^i=\min\{T,u^i_0[\max(\lceil \frac{\max(0,\widetilde{P}^i_0-\widetilde{P}^i_{shut})}{\widetilde{P}^i_{down}}\rceil,\underline{T}^i_{on}-\underline{T}^i_0)]\}$. If the unit $i$ has been online for $T^i_0$ periods prior to the first period of the time span, then $A^i_m:=\{[h,k]\vert h=m,k\in [\min(\max(m,U^i),m+M^i-1),m+M^i-1],\text{ or }h\in [\max(m+1,U^i+\underline{T}^i_{off}+1),m+M^i-1],k\in [\min(h+\underline{T}^i_{on}-1,m+M^i-1),m+M^i-1]\}$.
If the unit $i$ has been online prior to the first period of the time span, there must be one and only one continuous operating interval $[h,k]$ in each M-period $m$ where $m\le\min(U^i,T-M^i+1)$. Then we have
\begin{equation}
	\sum\nolimits_{[m,k]\in A_m}\tau^m_{m,k}=1,\qquad m\in[0,\min(U,T-M+1)]\label{eq:new-initial-status-1}
\end{equation}

In the following discussion of this section, we will adopt this new definition of $A^i_m$. Similar to Theorem \ref{binary-integer-0}, we can derive the following theorem.
\begin{theorem}\label{binary-integer-1}
	For any M-period $m$ we define the polyhedron $\widetilde{\mathcal{B}}^R_m=\{\\\prod_{[h,k]\in A_m}\tau^m_{h,k}\in [0,1]^{\lvert A_m\rvert}:(\ref{eq:inequality-taotao})(\ref{eq:new-initial-status-1})\}$. Let $\widetilde{\mathcal{B}}^I_m=\widetilde{\mathcal{B}}^R_m\cap\{\prod_{[h,k]\in A_m}\tau^m_{h,k}\in\{0,1\}^{\lvert A_m\rvert}\}$. Then $\widetilde{\mathcal{B}}^R_m=\textup{conv}(\widetilde{\mathcal{B}}^I_m)$.
\end{theorem}
\begin{proof}[\textbf{Proof}]
	If $m\in[0,\min(U,T-M+1)]$, we get $\sum\nolimits_{[m,k]\in A_m}\tau^m_{m,k}\le 1$ and \\$-\sum\nolimits_{[m,k]\in A_m}\tau^m_{m,k}\le -1$ from  (\ref{eq:new-initial-status-1}). We let $\bm{B_m}$ be the coefficient matrix of inequalities (\ref{eq:inequality-taotao}), and $\bm{D_m}$ be the coefficient matrix of the inequalities obtained from (\ref{eq:new-initial-status-1}).
	$$\bm{E_m}=\begin{bmatrix}
		\bm{B_m} \\
		\bm{D_m} \\
	\end{bmatrix} \textup{can be obtained by a series of pivot operations on} \begin{bmatrix}
		\bm{B_m} \\
		\bm{0} \\
	\end{bmatrix}$$
	We have shown that $\bm{B_m}$ is totally unimodular in Theorem \ref{binary-integer-0}. According to the Proposition 2.1 of §$\uppercase\expandafter{\romannumeral3}.1.2$ in \cite{wolsey1988integer}, $\bm{E_m}$ is totally unimodular. Let $\bm{d'}=\bm{0}$, $\bm{d}=\bm{1}$, $\bm{b'}=(\bm{0},\bm{-1})$, $\bm{b}=(\bm{1},\bm{-1})$, and then we know that $\widetilde{\mathcal{B}}^R_m$ is an integral polyhedron according to Proposition 2.3 of §$\uppercase\expandafter{\romannumeral3}.1.2$ in \cite{wolsey1988integer}.
\end{proof}
We define $\mathcal{D}$ as $\mathcal{D}:=\{i\vert u_0^i=0\}$, $\mathcal{U}$ as $\mathcal{U}:=\{i\vert u_0^i=1\}$, and $\mathcal{V}$ as $\mathcal{V}:=\{i\vert U^i+\underline{T}^i_{off}<m\}$.The theorems and propositions in Section \ref{sec3} still hold with new definition of $A^i_m$ for unit $i\in \mathcal{D}\cup(\mathcal{U}\cap\mathcal{V})$ in any M-period $m$. Hence, we won't repeat them here. For unit $i\in (\mathbb{N}-(\mathcal{D}\cup(\mathcal{U}\cap\mathcal{V})))$, we can further tighten (\ref{eq:tight-max-output-0})-(\ref{eq:tight-ramp-down-0}) and lower bound of generation limits with considering $\widetilde{P}_0$.
\begin{table}[t]
	\begin{center}
		\begin{minipage}{\textwidth}
			\caption{Illustration for upper bound of generation limits with historical status.}\label{tab:max-output-1}
			\resizebox{\textwidth}{!}{
				\begin{tabular}{ccc}
					\toprule
					$\tau_{h,k}=1$ & $\textup{UB}(\widetilde{P}_t)$ \\
					\midrule
					$m=h\le t\le k<m+M-1$ & $\min[1,\widetilde{P}_0+t\widetilde{P}_{up},\widetilde{P}_{shut}+(k-t)\widetilde{P}_{down}]$\\
					$m=h\le t\le k=m+M-1$ & $\min(1,\widetilde{P}_0+t\widetilde{P}_{up})$\\
					$m<h\le t\le k<m+M-1$ & $\min[1,\widetilde{P}_{start}+(t-h)\widetilde{P}_{up},\widetilde{P}_{shut}+(k-t)\widetilde{P}_{down}]$\\
					$m<h\le t\le k=m+M-1$ & $\min[1,\widetilde{P}_{start}+(t-h)\widetilde{P}_{up}]$\\
					\botrule
				\end{tabular}
			}
		\end{minipage}
	\end{center}
\end{table}

Taking historical status into consideration, the lower bound of generation limits for the unit in period $t$ is:
\begin{equation}
	\textup{LB}(\widetilde{P}_t)=\sum\nolimits_{\{[m,k] \in A_m,t\le k\}}\tau^m_{m,k}\max(0,\widetilde{P}_0-t\widetilde{P}_{down})
\end{equation}
Then, we obtain the following inequalities:
\begin{align}
	\widetilde{P}_t\ge\sum\nolimits_{\{[m,k] \in A_m,t\le k\}}\tau^m_{m,k}\max(0,\widetilde{P}_0-t\widetilde{P}_{down}),\quad t\in[m,m+M-1]\label{eq:tight-min-output-1}
\end{align}
As inllustrated in Table \ref{tab:max-output-1}, we get tighter inequalities for upper bound of generation limits.
\begin{align}
	\widetilde{P}_t\le &\sum\nolimits_{\{[h,m+M-1] \in A_m,m<h\le t\}}\tau^m_{h,m+M-1}\min[1,\widetilde{P}_{start}+(t-h)\widetilde{P}_{up}]\notag\\
	&+\sum\nolimits_{\{[h,k] \in A_m,m<h\le t\le k<m+M-1\}}\tau^m_{h,k}\times\notag\\
	&\min[1,\widetilde{P}_{start}+(t-h)\widetilde{P}_{up},\widetilde{P}_{shut}+(k-t)\widetilde{P}_{down}]\notag\\
	&+\sum\nolimits_{\{[m,k] \in A_m,t\le k<m+M-1\}}\tau^m_{m,k}\min[1,\widetilde{P}_0+t\widetilde{P}_{up},\widetilde{P}_{shut}+(k-t)\widetilde{P}_{down}]\notag\\
	&+\sum\nolimits_{\{[m,m+M-1]\in A_m\}}\tau^m_{m,m+M-1}\min(1,\widetilde{P}_0+t\widetilde{P}_{up}),\notag\\
	&t\in[m,m+M-1]\label{eq:tight-max-output-1}
\end{align}
Similarly, we obtain tighter inequalities for ramping constraints according to Table \ref{tab:tight-ramp-1}.
\begin{align}
	\widetilde{P}_t-\widetilde{P}_{t-a}\le &\sum\nolimits_{\{[h,m+M-1]\in A_m,t-a<h\le t\}}\tau^m_{h,m+M-1}\min[1,\widetilde{P}_{start}+(t-h)\widetilde{P}_{up}]\notag\\
	&+\sum\nolimits_{\{[h,k] \in A_m,t-a<h\le t\le k<m+M-1\}}\tau^m_{h,k}\times\notag\\
	&\min[1,\widetilde{P}_{start}+(t-h)\widetilde{P}_{up},\widetilde{P}_{shut}+(k-t)\widetilde{P}_{down}]\notag\\
	&+\sum\nolimits_{\{[h,k] \in A_m,m<h\le t-a<t\le k<m+M-1\}}\tau^m_{h,k}\times\notag\\
	&\min[1,a\widetilde{P}_{up},\widetilde{P}_{shut}+(k-t)\widetilde{P}_{down}]\notag\\
	&+\sum\nolimits_{\{[h,m+M-1] \in A_m,m<h\le t-a\}}\tau^m_{h,m+M-1}\min(1,a\widetilde{P}_{up})\notag\\
	&+\sum\nolimits_{\{[m,k] \in A_m,t\le k<m+M-1\}}\tau^m_{m,k}\min\{a\widetilde{P}_{up},1-\max[0,\widetilde{P}_0-\notag\\
	&(t-a)\widetilde{P}_{down}],\widetilde{P}_{shut}+(k-t)\widetilde{P}_{down}-\max[0,\widetilde{P}_0-(t-a)\widetilde{P}_{down}]\}\notag\\
	&+\sum\nolimits_{\{[m,m+M-1] \in A_m\}}\tau^m_{m,m+M-1}\times\notag\\
	&\min\{1-\max[0,\widetilde{P}_0-(t-a)\widetilde{P}_{down}],a\widetilde{P}_{up}\}\notag\\
	&-\sum\nolimits_{\{[m,k] \in A_m,t-a\le k<t\}}\tau^m_{m,k}\max[0,\widetilde{P}_0-(t-a)\widetilde{P}_{down}],\notag\\
	&t\in[m+1,m+M-1],a\in[1,t-m]\label{eq:tight-ramp-up-1}\\
	\widetilde{P}_{t-a}-\widetilde{P}_t\le &\sum\nolimits_{\{[h,k] \in A_m,m<h\le t-a<t\le k\}}\tau^m_{h,k}\times\notag\\
	&\min[1,\widetilde{P}_{start}+(t-a-h)\widetilde{P}_{up},a\widetilde{P}_{down}]\notag\\
	&+\sum\nolimits_{\{[h,k] \in A_m,m<h\le t-a\le k<t\}}\tau^m_{h,k}\times\notag\\
	&\min[1,\widetilde{P}_{start}+(t-a-h)\widetilde{P}_{up},\widetilde{P}_{shut}+(k-t+a)\widetilde{P}_{down}]\notag\\
	&+\sum\nolimits_{\{[m,k] \in A_m,t\le k\}}\tau^m_{m,k}\min[1,\widetilde{P}_0+(t-a)\widetilde{P}_{up},a\widetilde{P}_{down}]\notag\\
	&+\sum\nolimits_{\{[m,k] \in A_m,t-a\le k<t\}}\tau^m_{m,k}\times\notag\\
	&\min[1,\widetilde{P}_0+(t-a)\widetilde{P}_{up},\widetilde{P}_{shut}+(k-t+a)\widetilde{P}_{down}],\notag\\
	&t\in[m+1,m+M-1],a\in[1,t-m]\label{eq:tight-ramp-down-1}
\end{align}
\begin{table}[t]
	\begin{center}
		\begin{minipage}{\textwidth}
			\caption{Illustration for upper bound of ramping limits with historical status.}\label{tab:tight-ramp-1}
			\resizebox{\textwidth}{!}{
				\begin{tabular}{cc}
					\toprule
					$\tau_{h,k}=1$ & $\textup{UB}(\widetilde{P}_t-\widetilde{P}_{t-a})$ \\
					\midrule
					$m=h\le t-a<t\le k<m+M-1$ & $\min[1-\textup{LB}(\widetilde{P}_{t-a}),a\widetilde{P}_{up},\widetilde{P}_{shut}+(k-t)\widetilde{P}_{down}-\textup{LB}(\widetilde{P}_{t-a})]$ \\
					$m=h\le t-a<t\le k=m+M-1$ & $\min[1-\textup{LB}(\widetilde{P}_{t-a}),a\widetilde{P}_{up}]$ \\
					$m=h\le t-a\le k<t$ & $\textup{UB}(\widetilde{P}_t)-\textup{LB}(\widetilde{P}_{t-a})$ \\
					$m<h\le t-a<t\le k<m+M-1$ & $\min[1,a\widetilde{P}_{up},\widetilde{P}_{shut}+(k-t)\widetilde{P}_{down}]$ \\
					$m<h\le t-a<t\le k=m+M-1$ & $\min(1,a\widetilde{P}_{up})$ \\
					$t-a<h\le t\le k<m+M-1$ & $\min[1,\widetilde{P}_{start}+(t-h)\widetilde{P}_{up},\widetilde{P}_{shut}+(k-t)\widetilde{P}_{down}]-\textup{LB}(\widetilde{P}_{t-a})$ \\
					$t-a<h\le t\le k=m+M-1$ & $\min[1,\widetilde{P}_{start}+(t-h)\widetilde{P}_{up}]-\textup{LB}(\widetilde{P}_{t-a})$ \\
					\botrule
				\end{tabular}
			}
			\resizebox{\textwidth}{!}{
				\begin{tabular}{cc}
					\toprule
					$\tau_{h,k}=1$ & $\textup{UB}(\widetilde{P}_{t-a}-\widetilde{P}_t)$ \\
					\midrule
					$m=h\le t-a<t\le k$ & $\min[1,\widetilde{P}_0+(t-a)\widetilde{P}_{up},a\widetilde{P}_{down}]$\\
					$m=h\le t-a\le k<t$ & $\min[1,\widetilde{P}_0+(t-a)\widetilde{P}_{up},\widetilde{P}_{shut}+(k-t+a)\widetilde{P}_{down}]$\\
					$m<h\le t-a<t\le k$ & $\min[1,\widetilde{P}_{start}+(t-a-h)\widetilde{P}_{up},a\widetilde{P}_{down}]$\\
					$m<h\le t-a\le k<t$ & $\min[1,\widetilde{P}_{start}+(t-a-h)\widetilde{P}_{up},\widetilde{P}_{shut}+(k-t+a)\widetilde{P}_{down}]$\\
					\botrule
				\end{tabular}
			}
		\end{minipage}
	\end{center}
\end{table}

Next, we will talk about the tightness of (\ref{eq:tight-min-output-1})-(\ref{eq:tight-ramp-down-1}).

We let $\widetilde{\mathcal{P}}^R_m:=\{(\prod_{[h,k]\in A_m}\tau^m_{h,k},\widetilde{P}_{\max(1,m)},\cdots,\widetilde{P}_{m+M-1})\in {[0,1]}^{\lvert A_m\rvert+M+\min(0,m-1)}:(\ref{eq:inequality-taotao})(\ref{eq:new-initial-status-1})(\ref{eq:tight-min-output-1})(\ref{eq:tight-max-output-1})\}$. Let $\widetilde{\mathcal{P}}^I_m:=\widetilde{\mathcal{P}}^R_m\cap\{(\prod_{[h,k]\in A_m}\tau^m_{h,k},\widetilde{P}_{\max(1,m)},\cdots,\widetilde{P}_{m+M-1})\in\{0,1\}^{\lvert A_m\rvert }\times[0,1]^{M+\min(0,m-1)}\}$.
\begin{proposition}\label{prp5}
	Inequality (\ref{eq:tight-min-output-1}) defines a facet of $\textup{conv}(\widetilde{\mathcal{P}}^I_m)$.
\end{proposition}
\begin{proof}[\textbf{Proof}]
	See ``Appendix 2".
\end{proof}

\begin{proposition}\label{prp6}
	Inequality (\ref{eq:tight-max-output-1}) defines a facet of $\textup{conv}(\widetilde{\mathcal{P}}^I_m)$.
\end{proposition}
\begin{proof}[\textbf{Proof}]
	See ``Appendix 2".
\end{proof}

\begin{theorem}
	$\widetilde{\mathcal{P}}^R_m=\textup{conv}(\widetilde{\mathcal{P}}^I_m)$.
\end{theorem}
\begin{proof}[\textbf{Proof}]
	We have proved in Theorem \ref{binary-integer-1} that $\widetilde{\mathcal{B}}^R_m$ is an integral polyhedron. According to the Lemma 4 in \cite{gentile2017tight}, it is not difficult to prove $\widetilde{\mathcal{P}}^R_m=\textup{conv}(\widetilde{\mathcal{P}}^I_m)$.
\end{proof}
We refer to the feasible set defined by constraints (\ref{eq:inequality-taotao})(\ref{eq:new-initial-status-1})(\ref{eq:tight-min-output-1})-(\ref{eq:tight-ramp-down-1}) as $\widetilde{\mathcal{Q}}^R_m$, i.e. $\widetilde{\mathcal{Q}}^R_m:=\{(\prod_{[h,k]\in A_m}\tau^m_{h,k},\widetilde{P}_{\max(1,m)},\cdots,\widetilde{P}_{m+M-1})\in{[0,1]}^{\lvert A_m\rvert+M+\min(0,m-1)}:(\ref{eq:inequality-taotao})(\ref{eq:new-initial-status-1})(\ref{eq:tight-min-output-1})-(\ref{eq:tight-ramp-down-1})\}$. Let $\widetilde{\mathcal{Q}}^I_m:=\widetilde{\mathcal{Q}}^R_m\cap\{(\prod_{[h,k]\in A_m}\tau^m_{h,k},\widetilde{P}_{\max(1,m)},\cdots,\widetilde{P}_{m+M-1})\in\{0,1\}^{\lvert A_m\rvert }\times[0,1]^{M+\min(0,m-1)}\}$, $K=\min\{T,u_0[\max(\max(0,\lfloor \frac{\widetilde{P}_0-\widetilde{P}_{shut}}{\widetilde{P}_{down}}\rfloor+1),\underline{T}_{on}-\underline{T}_0)]\}$.
\begin{proposition}\label{prp7}
	Inequality (\ref{eq:tight-min-output-1}) defines a facet of $\textup{conv}(\widetilde{\mathcal{Q}}^I_m)$ if and only if one of the following conditions hold:
	\begin{enumerate}
		\item $t=\max(1,m)$.
		\item $\min(K,\frac{\widetilde{P}_0}{\widetilde{P}_{down}})<t$.
	\end{enumerate}
\end{proposition}
\begin{proof}[\textbf{Proof}]
	See ``Appendix 2".
\end{proof}
\begin{proposition}\label{prp8}
	Inequality (\ref{eq:tight-max-output-1}) defines a facet of $\textup{conv}(\widetilde{\mathcal{Q}}^I_m)$ if and only if one of the following conditions hold:
	\begin{enumerate}
		\item $t=\max(1,m)$.
		\item $\min(K,\frac{1-\widetilde{P}_0}{\widetilde{P}_{up}})<t$.
		\item $K<m+M-1,\max(\frac{\widetilde{P}_{shut}-\widetilde{P}_0}{\widetilde{P}_{up}} ,\frac{K\widetilde{P}_{down}+\widetilde{P}_{shut}-\widetilde{P}_0}{\widetilde{P}_{up}+\widetilde{P}_{down}})<t<m+M-1$.
	\end{enumerate}
\end{proposition}
\begin{proof}[\textbf{Proof}]
	See ``Appendix 2".
\end{proof}

\begin{proposition}\label{prp9}
	We define condition $\mathcal{M}_1$ as one of the following conditions hold:
	\begin{enumerate}
		\item $a=1$.
		\item $t>K$.
		\item $\min(\frac{1}{\widetilde{P}_{up}},\frac{1-\widetilde{P}_0+t\widetilde{P}_{down}}{\widetilde{P}_{up}+\widetilde{P}_{down}})<a$.
		\item $\max(t,K)<m+M-1,\min\{\frac{\widetilde{P}_{shut}+[K-t]^+ \widetilde{P}_{down}}{\widetilde{P}_{up}},\frac{\widetilde{P}_{shut}+\max(t,K)×\widetilde{P}_{down}-\widetilde{P}_0}{\widetilde{P}_{up}+\widetilde{P}_{down}}\}<a$.
		\item $\min(\underline{T}_{on},\frac{1-\widetilde{P}_{start}}{\widetilde{P}_{up}}+1)<a\le t-U-\underline{T}_{off}$.
		\item $t<m+M-1,\max[\frac{\widetilde{P}_{shut}-\widetilde{P}_{start}}{\widetilde{P}_{up}},\frac{(\underline{T}_{on}-1)\widetilde{P}_{down}+\widetilde{P}_{shut}-\widetilde{P}_{start}}{\widetilde{P}_{up}+\widetilde{P}_{down}},t+\underline{T}_{on}-m-M]+1<a\le t-U-\underline{T}_{off}$.
	\end{enumerate}
	We define condition $\mathcal{M}_2$ as one of the following conditions hold:
	\begin{enumerate}
		\item $t-a=0$.
		\item $a<\min(\frac{1}{\widetilde{P}_{up}},\frac{1-\widetilde{P}_0+t\widetilde{P}_{down}}{\widetilde{P}_{up}+\widetilde{P}_{down}})$.
		\item $a<\min(\frac{1}{\widetilde{P}_{up}},t-U-\underline{T}_{off})$.
	\end{enumerate}
	Inequality (\ref{eq:tight-ramp-up-1}) defines a facet of $\textup{conv}(\widetilde{\mathcal{Q}}^I_m)$ if and only if condition $\mathcal{M}_1$ and condition $\mathcal{M}_2$ both hold.
\end{proposition}
\begin{proof}[\textbf{Proof}]
	See ``Appendix 2".
\end{proof}

\begin{proposition}\label{prp10}
	We define condition $\mathcal{M}_3$ as one of the following conditions hold:
	\begin{enumerate}
		\item $a=1$.
		\item $a\le t-U-\underline{T}_{off}$.
		\item $\min(\frac{1}{\widetilde{P}_{down}},\frac{\widetilde{P}_0+t\widetilde{P}_{up}}{\widetilde{P}_{up}+\widetilde{P}_{down}})<a$.
		\item $U<t,\min(\frac{1-\widetilde{P}_{shut}}{\widetilde{P}_{down}}+1,\frac{\widetilde{P}_0+t\widetilde{P}_{up}+\widetilde{P}_{down}-\widetilde{P}_{shut}}{\widetilde{P}_{up}+\widetilde{P}_{down}})<a$.
	\end{enumerate}
	We define condition $\mathcal{M}_4$ as one of the following conditions hold:
	\begin{enumerate}
		\item $t-a=0$.
		\item $a<\min(\frac{1}{\widetilde{P}_{down}},\frac{\widetilde{P}_0+t\widetilde{P}_{up}}{\widetilde{P}_{up}+\widetilde{P}_{down}})$.
		\item $a<\min(\frac{1}{\widetilde{P}_{down}},t-U-\underline{T}_{off},\frac{\widetilde{P}_{start}+(t-U-\underline{T}_{off}-1)\widetilde{P}_{up}}{\widetilde{P}_{up}+\widetilde{P}_{down}})$.
	\end{enumerate}
	Inequality (\ref{eq:tight-ramp-down-1}) defines a facet of $\textup{conv}(\widetilde{\mathcal{Q}}^I_m)$ if and only if condition $\mathcal{M}_3$ and condition $\mathcal{M}_4$ both hold.
\end{proposition}
\begin{proof}[\textbf{Proof}]
	See ``Appendix 2".
\end{proof}

Then, model MP-3 can be transformed to history-dependent multi-period tight MIQP UC formulation which takes generator’s history into account, denoted as M-period tight model (MP-Ti):
\begin{align}
	&\min \quad(\ref{eq:new-objective-function})\notag\\
	&s.t.\left\{
	\begin{aligned}
		&(\ref{eq:spinning-reserve})(\ref{eq:logical})(\ref{eq:min-up-time})(\ref{eq:min-down-time})(\ref{eq:initial-status})(\ref{eq:new-power-balance})(\ref{eq:start-cost})(\ref{eq:taom-in-0})(\ref{eq:taom-in-1})\\
		&(\ref{eq:tight-min-output-1})\text{ for } i\in (\mathbb{N}-(\mathcal{D}\cup(\mathcal{U}\cap\mathcal{V})))\\
		&(\ref{eq:tight-max-output-0})\text{ for } i\in \mathcal{D}\cup(\mathcal{U}\cap\mathcal{V}),(\ref{eq:tight-max-output-1})\text{ for } i\in (\mathbb{N}-(\mathcal{D}\cup(\mathcal{U}\cap\mathcal{V})))\\
		&(\ref{eq:tight-ramp-up-0})\text{ for } i\in \mathcal{D}\cup(\mathcal{U}\cap\mathcal{V}),m=T-M^i+1\text{ or }a=t-m\\
		&(\ref{eq:tight-ramp-down-0})\text{ for } i\in \mathcal{D}\cup(\mathcal{U}\cap\mathcal{V}),m=0\text{ or }t=m+M^i-1\\
		&(\ref{eq:tight-ramp-up-1})\text{ for } i\in (\mathbb{N}-(\mathcal{D}\cup(\mathcal{U}\cap\mathcal{V}))),m=T-M^i+1\text{ or }a=t-m\\
		&(\ref{eq:tight-ramp-down-1})\text{ for } i\in (\mathbb{N}-(\mathcal{D}\cup(\mathcal{U}\cap\mathcal{V}))),m=0\text{ or }t=m+M^i-1\\
		&u^i_t,v^i_t,w^i_t\in\{0,1\},\widetilde{P}^i_t\in [0,1],\widetilde{S}^i_t\in R_+,m\in[0,T-M^i+1]\\
		&\tau^{i,m}_{h,k}\in\{0,1\},[h,k]\in A^i_m, m<h\le k<m+M^i-1
	\end{aligned}
	\right.
\end{align}

\subsection{MILP approximations}\label{subsec9}

Solving MILP is easier than solving MIQP, so it is very popular to approximate MIQP as MILP and then use MILP solver to solve the UC problem.

Assume that $L$ is a given parameter, let $p^i_l=\underline{P}^i+l(\overline{P}^i-\underline{P}^i)/L$ and 
$l=0,1,2,\cdots,L$. For 2P and 3P, after replacing $\gamma_i(P^i_t)^2$ in the objective function with a corresponding new variable $z^i_t$ and adding the following linear constraints to the formulation.
\begin{equation}
	z^i_t\ge 2\gamma_ip^i_lP^i_t-\gamma_i(p^i_l)^2
\end{equation}
we obtain the MILP UC models that approximate the original MIQP models.

Similarly, let $\widetilde{p}^i_l=l/L$. Replace $\widetilde{\gamma_i}(\widetilde{P}^i_t)^2$ in the 3P-HD and MP models with $z^i_t$ and add the following constraints:
\begin{equation}
	z^i_t\ge 2\widetilde{\gamma_i}\widetilde{p}^i_l\widetilde{P}^i_t-\widetilde{\gamma_i}(\widetilde{p}^i_l)^2
\end{equation}
then we obtain the MILP approximations of 3P-HD and MP models.

\section{Numerical results and analysis}\label{sec5}

In this section, two data sets are used to test the tightness and computational performances of the new formulations proposed in Section \ref{sec3} and Section \ref{sec4}, compared with the state-of-the-art 2-period and 3-period formulations.

The first data set stems from the synthetic instances of \cite{ostrowski2012tight}, which are replicated dataes from \cite{carrion2006computationally}. The second data set, including 5 different data values for 10-, 20-, 50-, 75-, 100-, 150-unit system, and 12 different data values for 200-unit system, is published at \url{http://groups.di.unipi.it/optimize/Data/UC.html}. In total, seventy-three realistic instances with 10-1080 units running for a time span of 24 hours are used in our experiments. The machine on which we perform all of our computations is a desktop with Intel i7-8700K 3.7 GHz CPU and 8 GB of RAM, running MS-Windows 10 (64-bit) and MATLAB 2016b. We use MATLAB to call GUROBI 9.1.1 to solve the MILP problems. The time limit for the solver is set to 3600 seconds. All the codes and instances for the simulations in this article are available from \url{https://github.com/linfengYang/multi-period-UC-model}.

For MP-1 (MP-2, MP-3, MP-Ti) formulations, we choose four different sizes of sliding windows to construct the models, where $M$ equals to 2, 3, $H^i=\max(\lceil\frac{1-\widetilde{P}^i_{start}}{\widetilde{P}^i_{up}}\rceil+1,\lceil\frac{1-\widetilde{P}^i_{shut}}{\widetilde{P}^i_{down}}\rceil+1,\lceil\frac{1}{\widetilde{P}^i_{up}}\rceil,\lceil\frac{1}{\widetilde{P}^i_{down}}\rceil,2)$, $T+1$ respectively, named 2P-1 (2P-2, 2P-3, 2P-Ti), 3P-1 (3P-2, 3P-3, 3P-Ti), HP-1 (HP-2, HP-3, HP-Ti) and TP-1 (TP-2, TP-3, TP-Ti) respectively.

We have made a series of improvements to improve the computational efficiency of our models.

First, we eliminate some redundant inequalities by using Propositions \ref{prp3}-\ref{prp4}. ``MP-1-all" represent the models that retain all proposed inequalities. ``MP-1" represent the models that exclude redundant inequalities.

In order to reflect the difference in computational efficiency of each model, we introduce relative time ``rTime", which can be defined as $(\textup{time}-\textup{reftime})/\textup{reftime}$. In this expression, ``time" represents the solving time of each model, and ``reftime" represents the solving time of reference model. Here, we choose 3P model as reference model. We compare MP-1-all and MP-1 models in terms of rTime in Fig. \ref{fig:facet-runtime}. There is no obvious difference between MP-1-all model and MP-1 model in 2-period, 3-period and H-period. However, TP-1 model performs better than TP-1-all model in general.
\begin{figure}[t]
	\begin{center}
		\subfigure[2-period]{
			\includegraphics[width=0.48\textwidth,height=4cm]{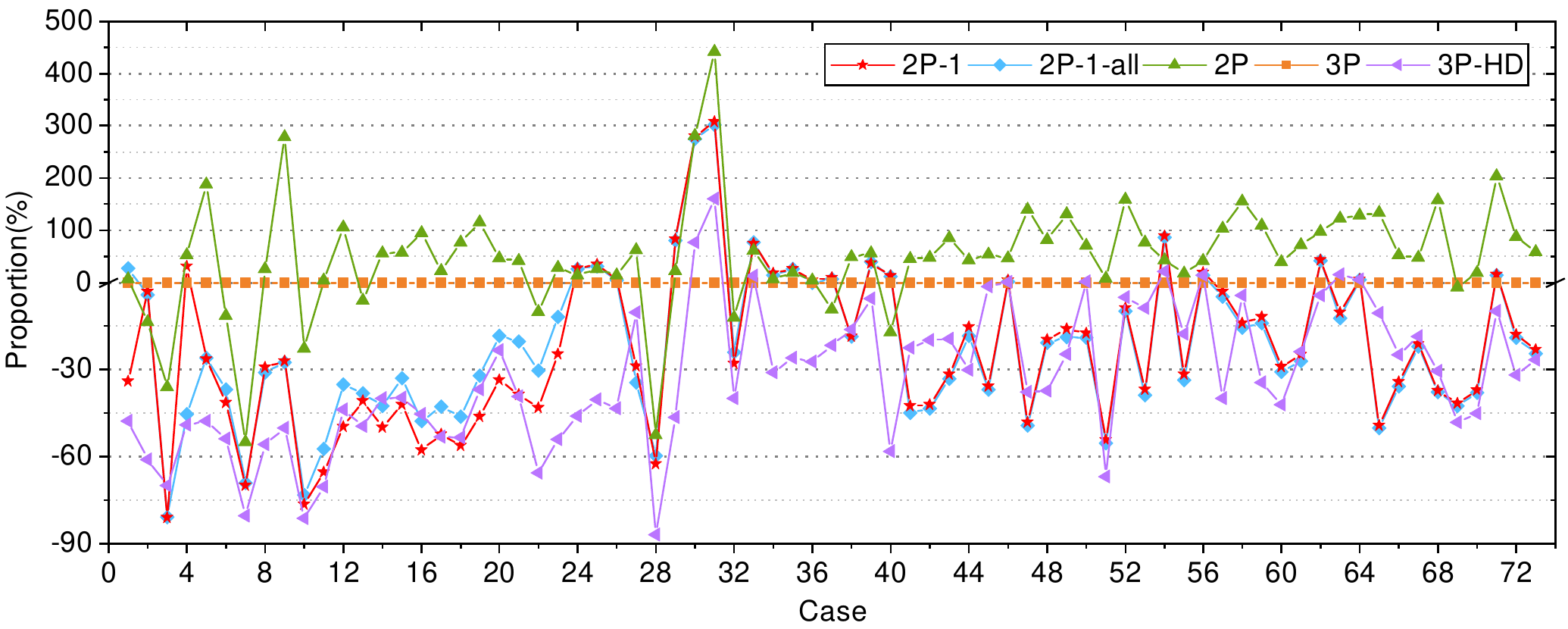}}\subfigure[3-period]{
			\includegraphics[width=0.48\textwidth,height=4cm]{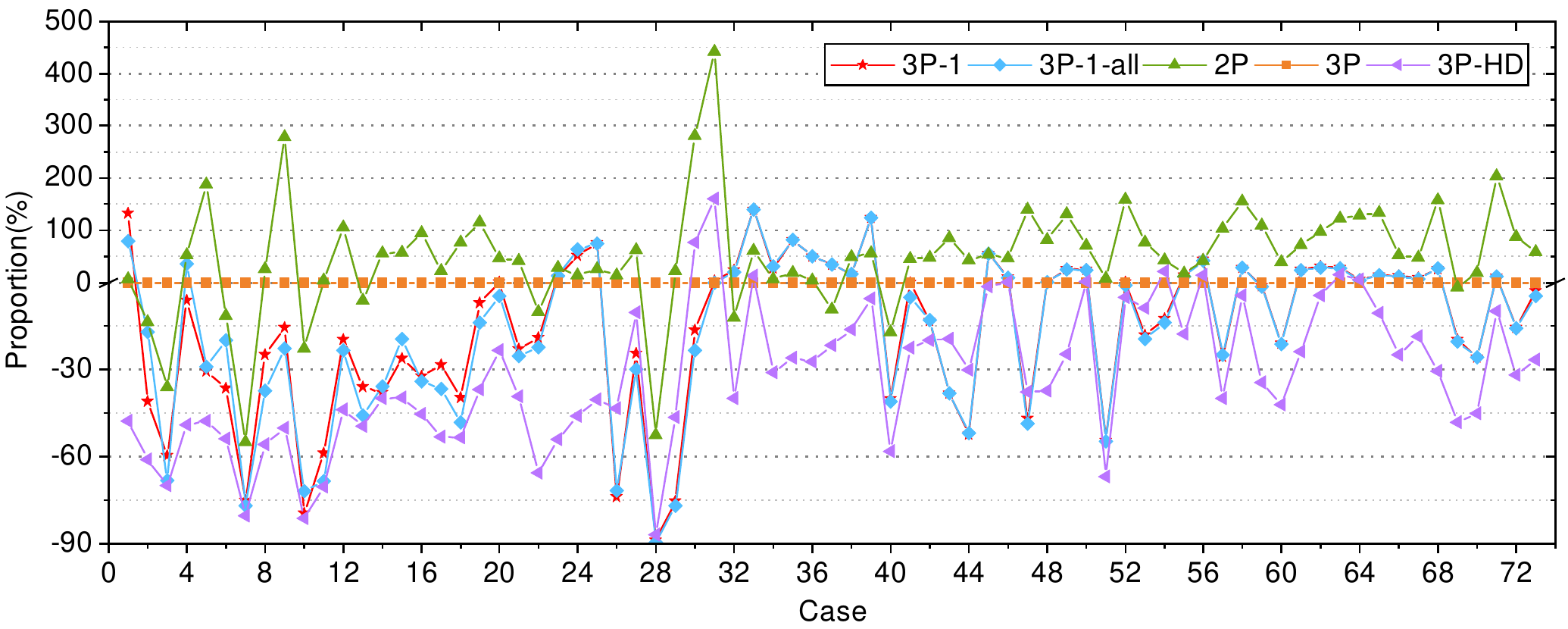}}
		\subfigure[H-period]{
			\includegraphics[width=0.48\textwidth,height=4cm]{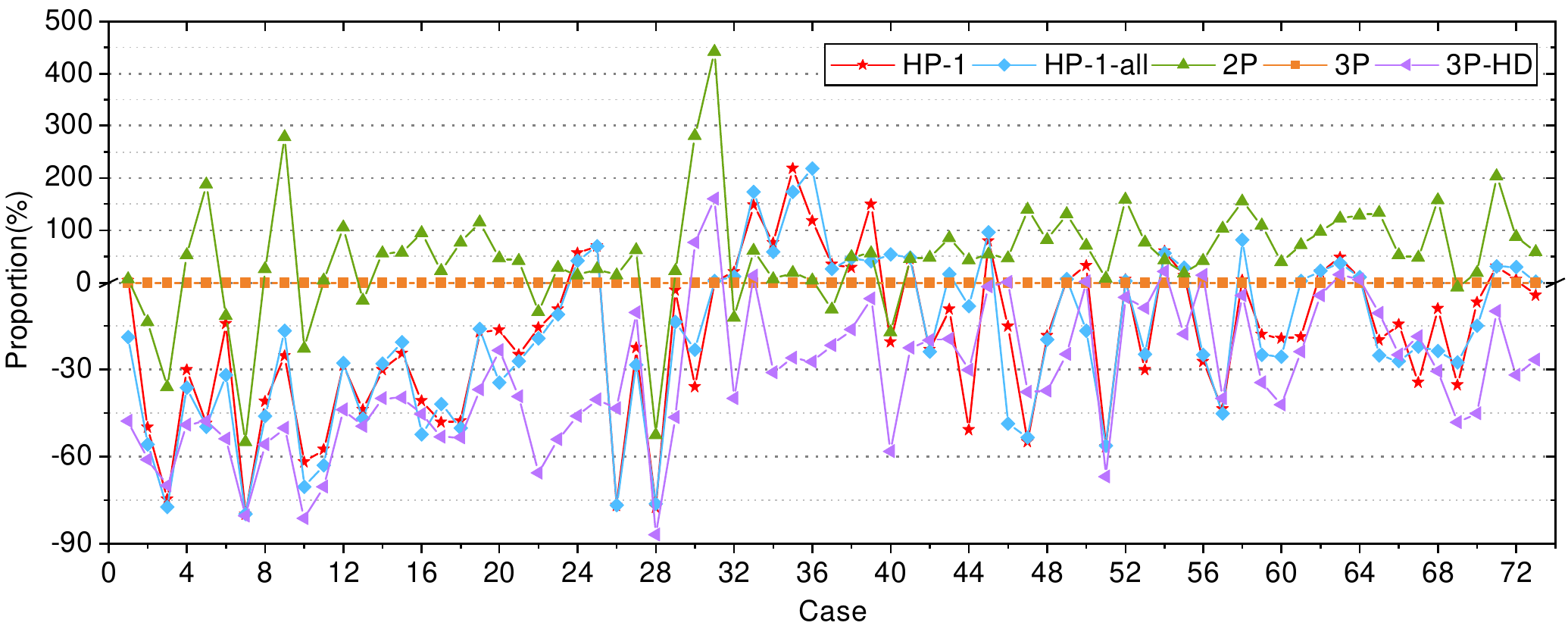}}\subfigure[T-period]{
			\includegraphics[width=0.48\textwidth,height=4cm]{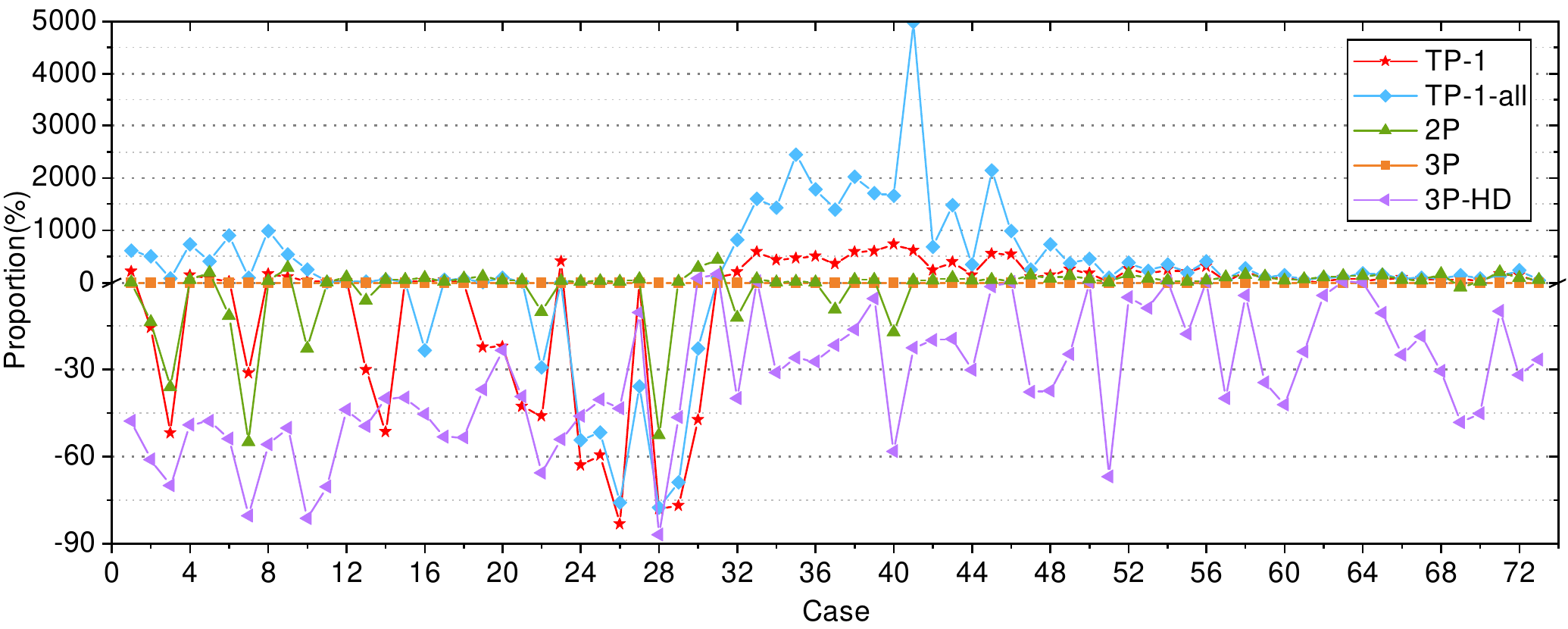}}
		\caption{Comparison of MP-1, MP-1-all and the other three state-of-the-art MILP formulations in terms of rTime.}
		\label{fig:facet-runtime}
	\end{center}
\end{figure}

We adopt the performance profiles to show the difference in computational efficiency of each model more clearly \cite{dolan2002benchmarking}. we define $t_{p,s}$ as the computing time required to solve problem $p$ by model $s$.
\begin{equation}
	\rho_s(\tau)=\frac{1}{n_p}\textup{size}\{p\in\mathscr{P}:\frac{t_{p,s}}{\min\{t_{p,s}:s\in\mathscr{S}\}}\le\tau\}
\end{equation}
where $\mathscr{S}$ is the set of models, $\mathscr{P}$ is the test set, $n_p$ is the total number of problems. As shown in Fig. \ref{fig:facet-profile}, model MP-1 has excellent performance in T-period.
\begin{figure}[t]
	\begin{center}
		\subfigure[2-period]{
			\includegraphics[width=0.48\textwidth,height=4cm]{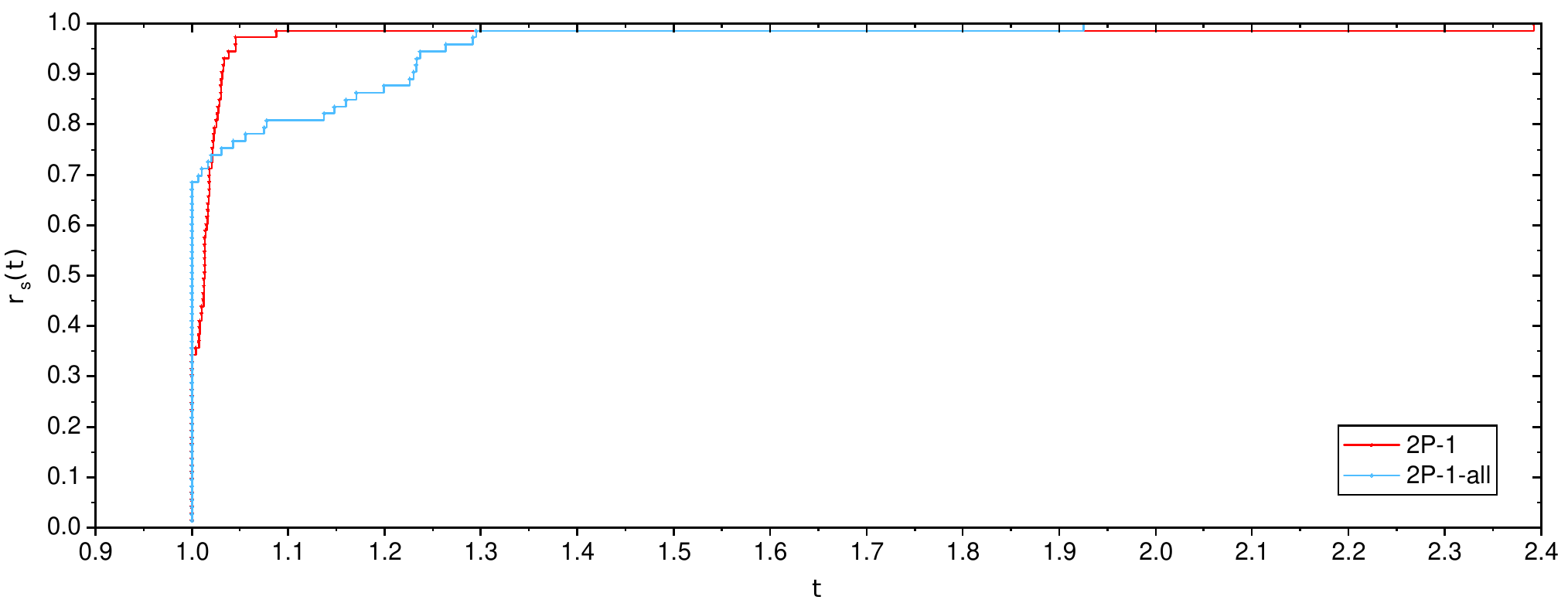}}\subfigure[3-period]{
			\includegraphics[width=0.48\textwidth,height=4cm]{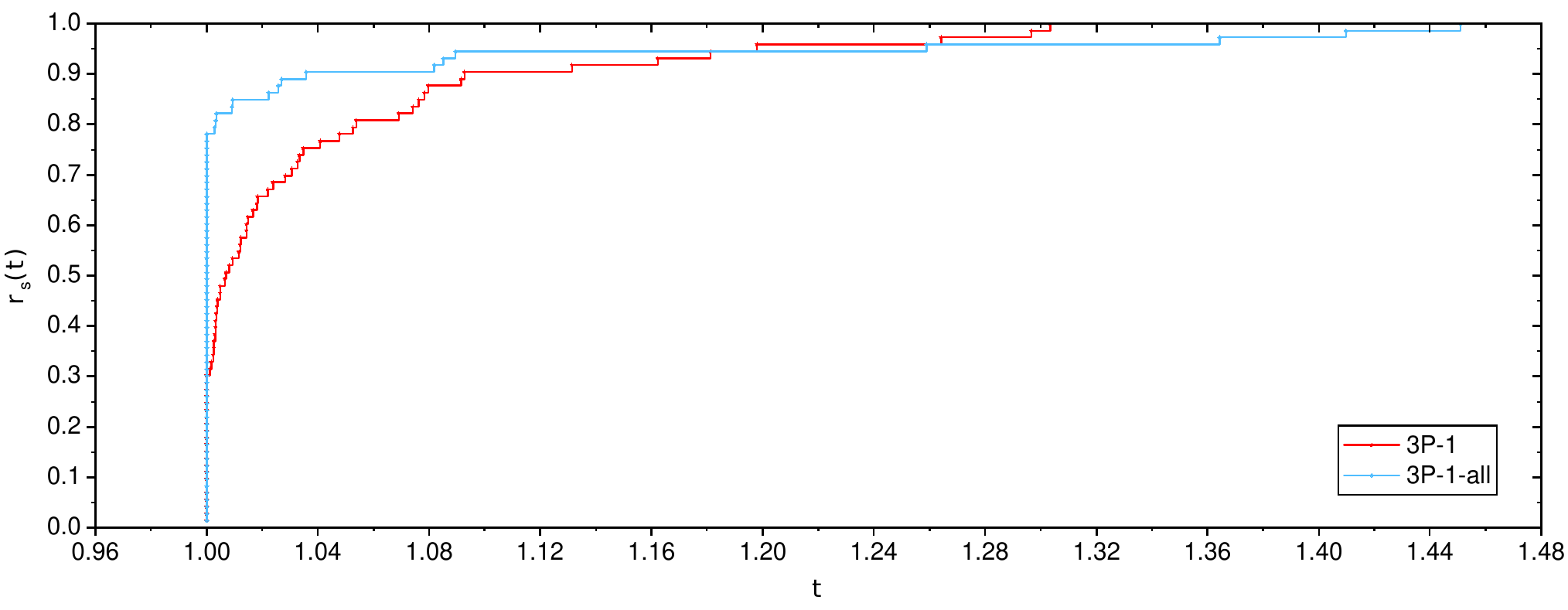}}
		\subfigure[H-period]{
			\includegraphics[width=0.48\textwidth,height=4cm]{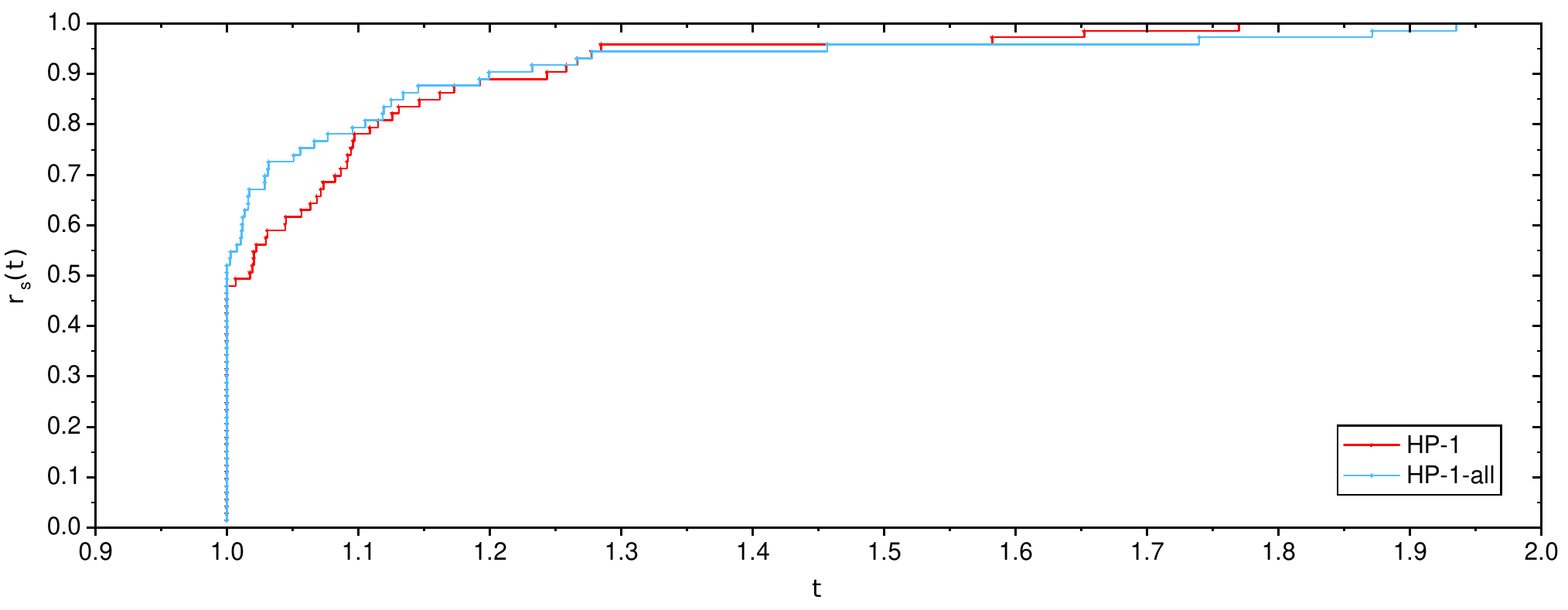}}\subfigure[T-period]{
			\includegraphics[width=0.48\textwidth,height=4cm]{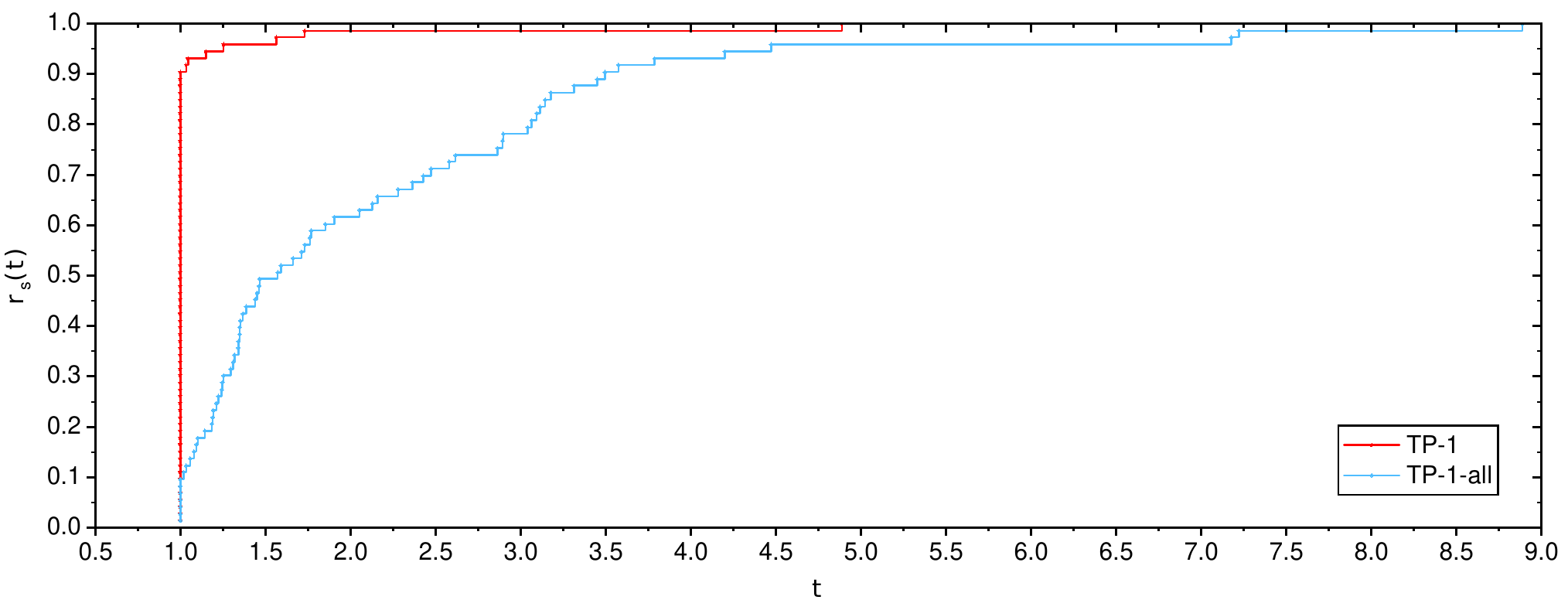}}
		\caption{Performance profiles on CPU time for MP-1 and MP-1-all formulations.}
		\label{fig:facet-profile}
	\end{center}
\end{figure}

We define ``redu\_con" as the percentage decreases in the number of constraints of model MP-1 compared to model MP-1-all. As shown in Fig. \ref{fig:facet-constraints}, compared with MP-1-all model, the number of constraints in MP-1 model decreases more significantly for T-period.
\begin{figure}[t]
 	\centering
 	\includegraphics[width=\textwidth,height=6cm]{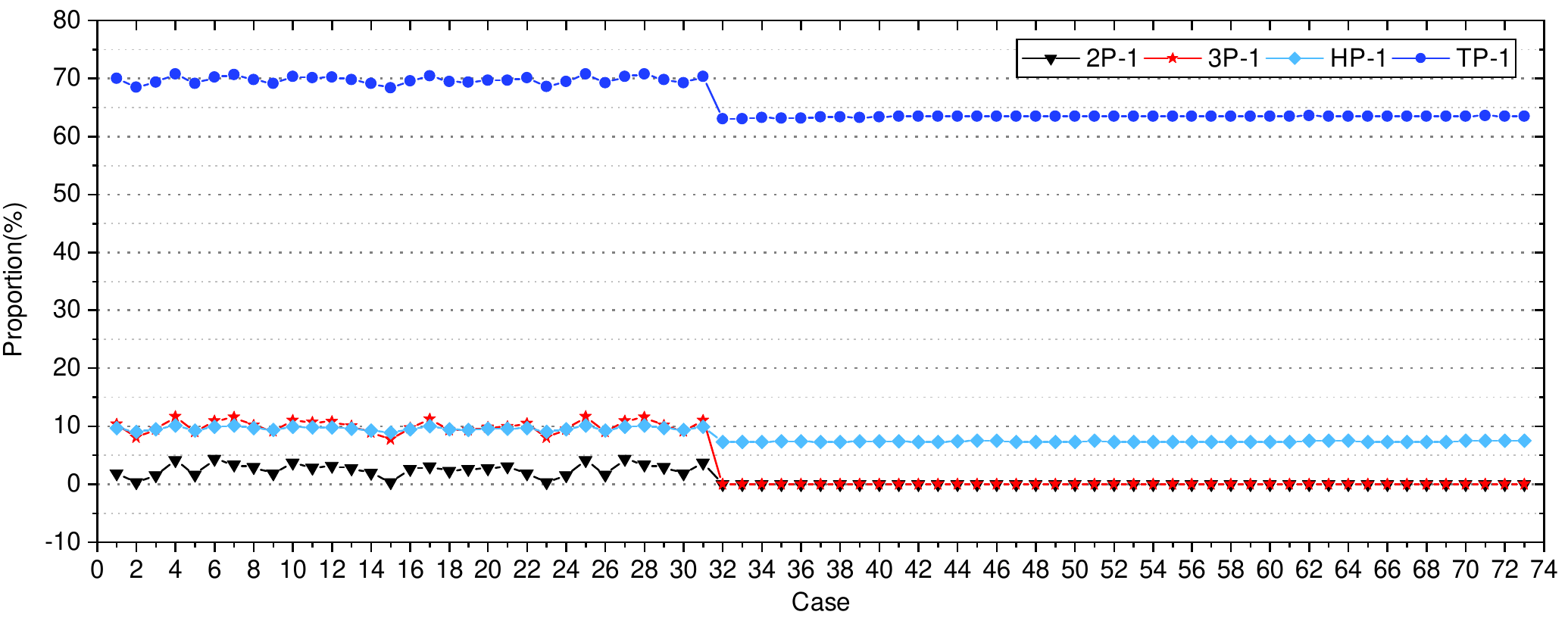}
 	\caption{The redu\_con of four MP-1 formulations.}
 	\label{fig:facet-constraints}
\end{figure}

The solving process was frequently interrupted due to insufficient memory, when we used TP-1-all model to solve the continuous relaxation of the MIP problem with more than 900 units. And finally we got the solutions after several attempts. To simplify the model, for lower/upper bound of generation limits and ramping constraints, we select only inequalities that satisfy the conditions of Propositions \ref{prp3}-\ref{prp4} for models MP-1, MP-2, MP-3. And only inequalities that satisfy the conditions of Propositions \ref{prp7}-\ref{prp10} are selected for models MP-Ti.

Next, considering the compactness of the models, we improved MP-1 models with respect to binary variables, and get two groups of MP models ``MP-2, MP-3". As shown in Fig. \ref{fig:Variable-runtime} and Fig. \ref{fig:Variable-performance-profile}, MP-3 models always perform best.
\begin{figure}[t]
	\begin{center}
		\subfigure[2-period]{
			\includegraphics[width=0.48\textwidth,height=4cm]{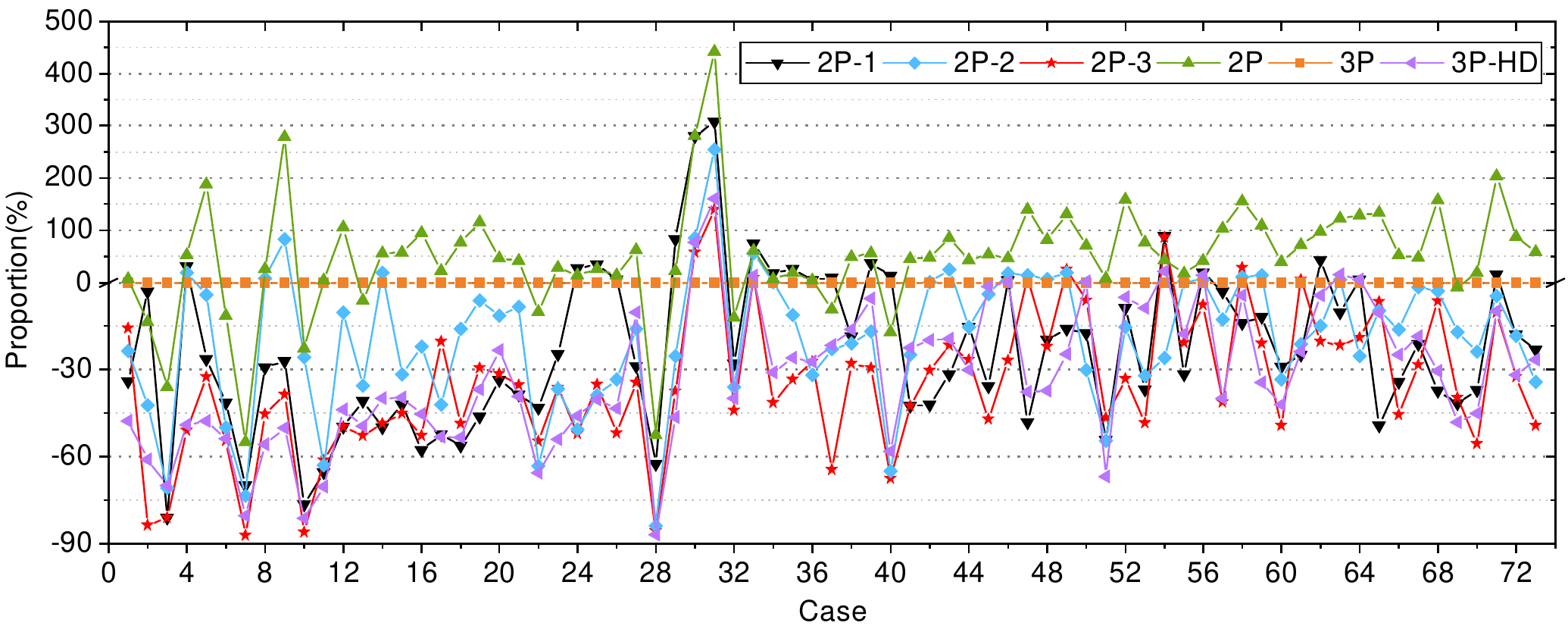}}\subfigure[3-period]{
			\includegraphics[width=0.48\textwidth,height=4cm]{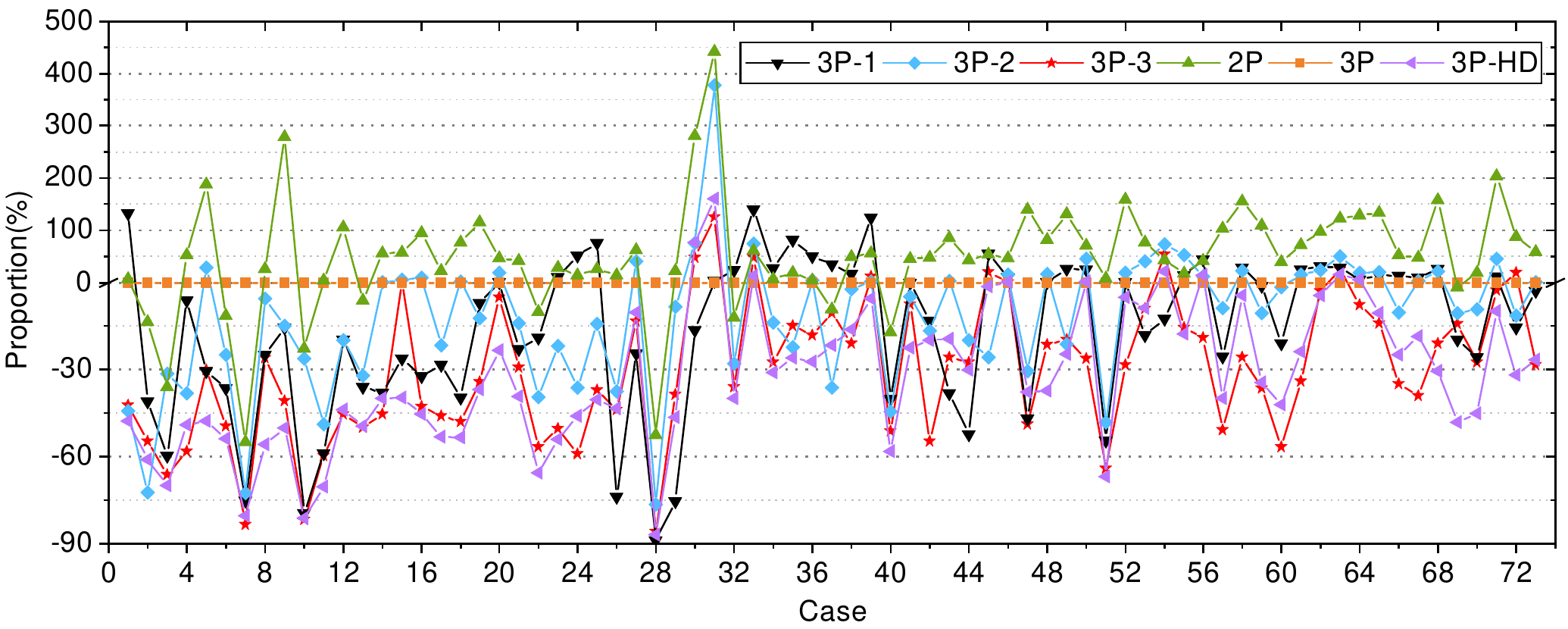}}
		\subfigure[H-period]{
			\includegraphics[width=0.48\textwidth,height=4cm]{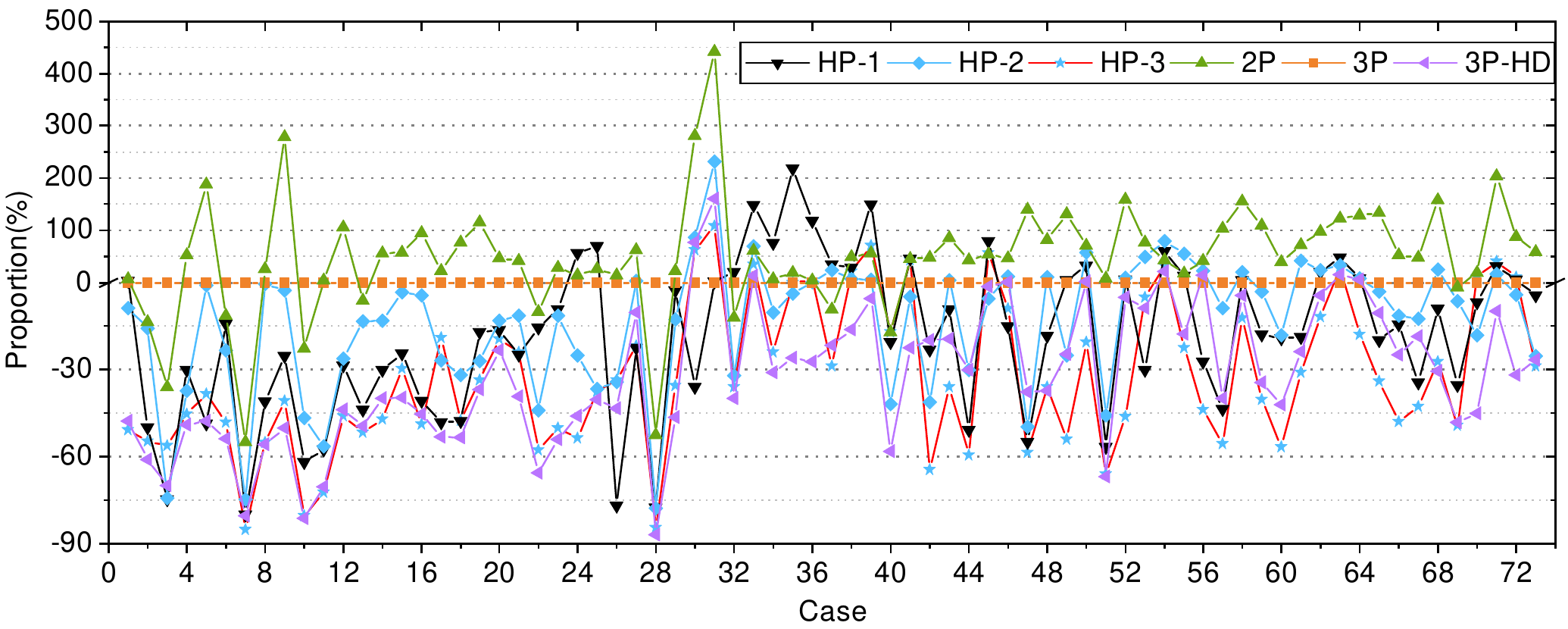}}\subfigure[T-period]{
			\includegraphics[width=0.48\textwidth,height=4cm]{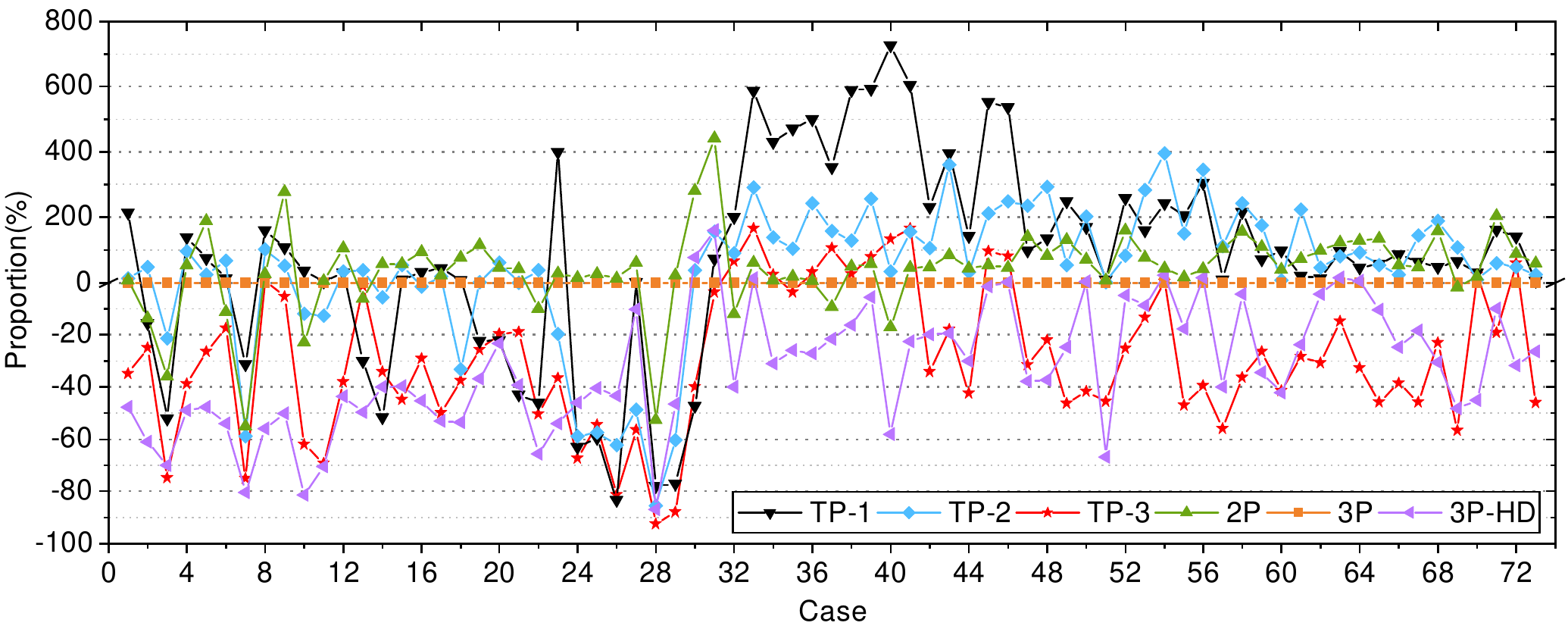}}
		\caption{Comparison of MP-1, MP-2, MP-3 and the other three state-of-the-art MILP formulations in terms of rTime.}
		\label{fig:Variable-runtime}
	\end{center}
\end{figure}
\begin{figure}[t]
	\begin{center}
		\subfigure[2-period]{
			\includegraphics[width=0.48\textwidth,height=4cm]{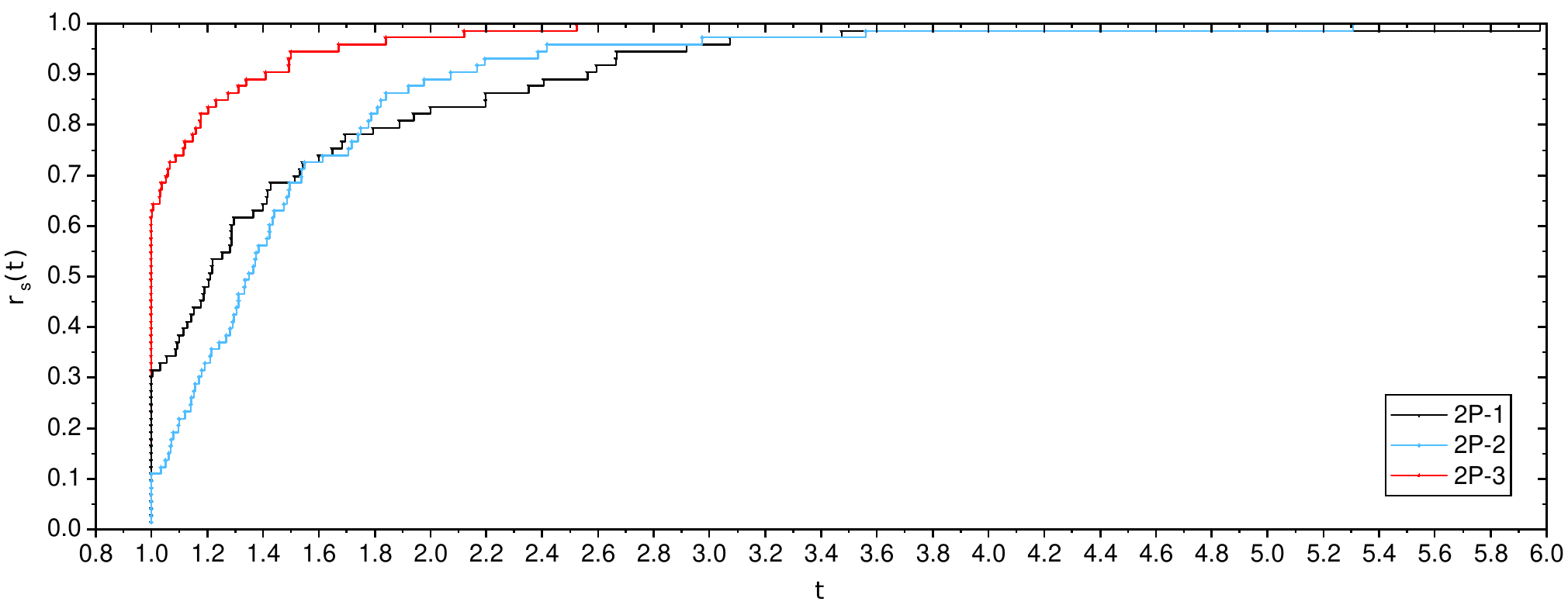}}\subfigure[3-period]{
			\includegraphics[width=0.48\textwidth,height=4cm]{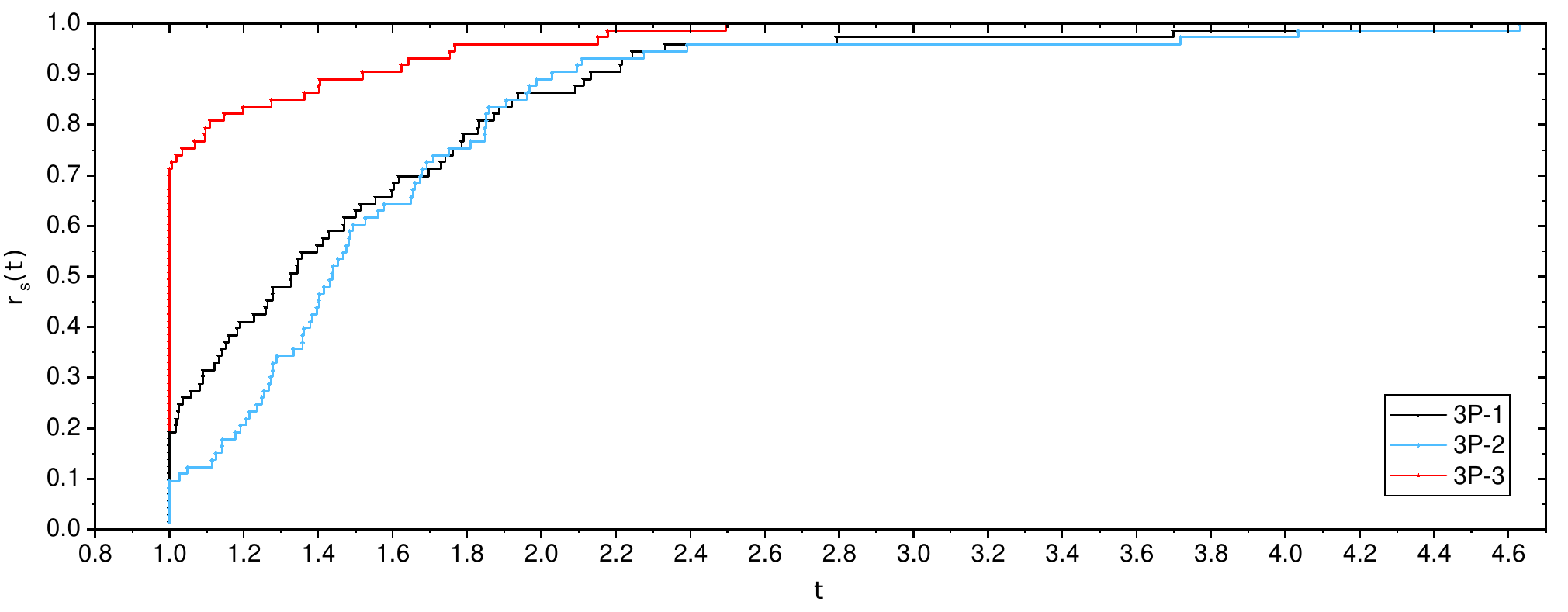}}
		\subfigure[H-period]{
			\includegraphics[width=0.48\textwidth,height=4cm]{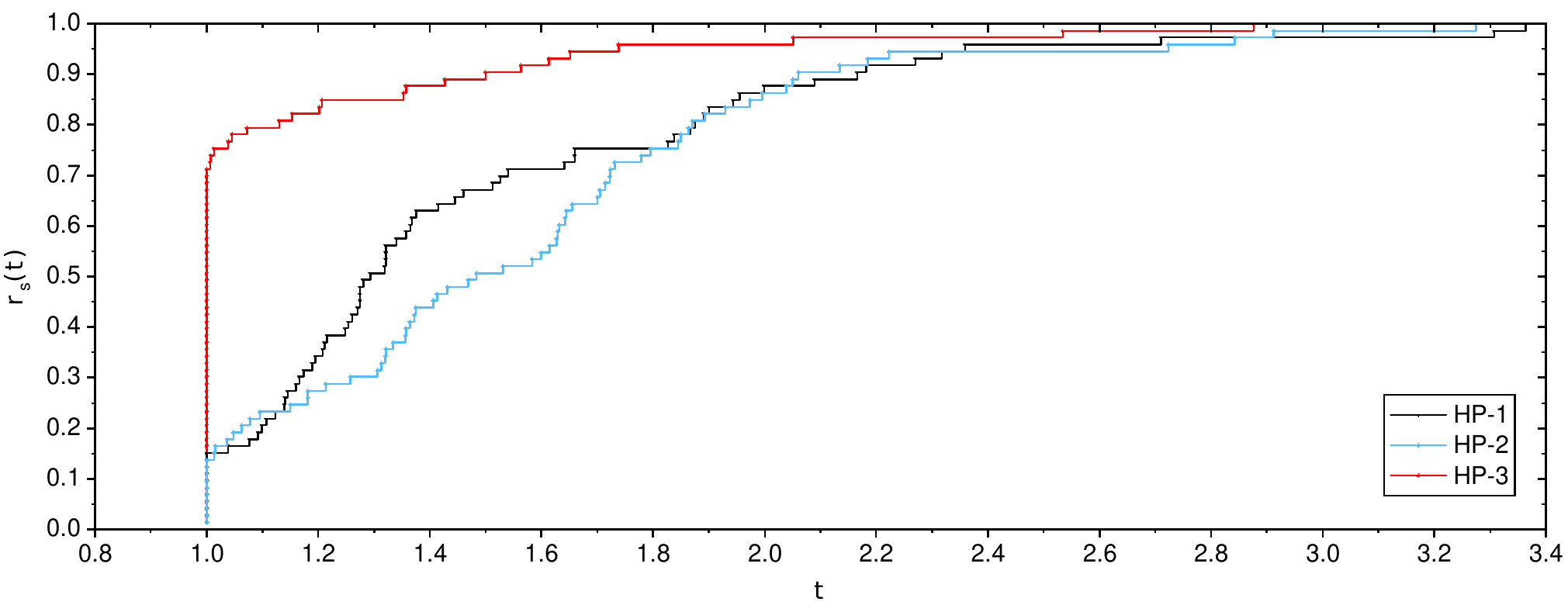}}\subfigure[T-period]{
			\includegraphics[width=0.48\textwidth,height=4cm]{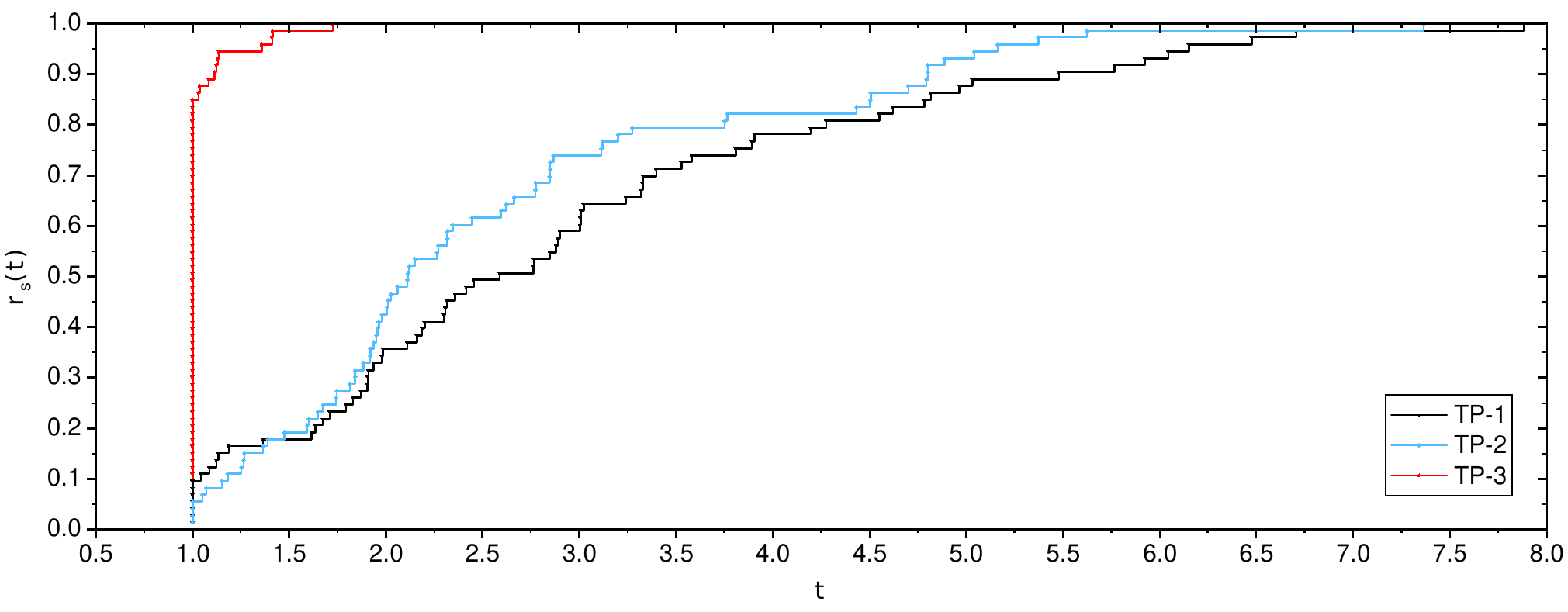}}
		\caption{Performance profiles on CPU time for MP-1, MP-2 and MP-3 formulations.}
		\label{fig:Variable-performance-profile}
	\end{center}
\end{figure}

Finally, we took historical status into consideration in order to make the models tighter and got MP-Ti. We compare the compactness of MP-3 and MP-Ti MILP formulations in Table \ref{tab:HistoryCompactness1} and Table \ref{tab:HistoryCompactness2}. The columns ``pre\_cons", ``pre\_vars", and ``pre\_nozs" represent the numbers of constraints, variables, and nonzeros respectively for the corresponding problem after being presolved by Gurobi. The columns ``redu\_con", ``redu\_var", and ``redu\_noz" represent the percentage decreases in the number of constraints, variables, and nonzeros respectively of MP-Ti compared to MP-3. As seen in Table \ref{tab:HistoryCompactness1} and Table \ref{tab:HistoryCompactness2}, there is little difference in the number of constraints and variables between MP-3 and MP-Ti. The number of nonzeros in TP-Ti model increases greatly compared with TP-3 model.
\begin{sidewaystable}
	\begin{center}
		\begin{minipage}{\textheight}
			\caption{Comparison of MP-3 and MP-Ti MILP formulations in terms of compactness for the first data set.}\label{tab:HistoryCompactness1}
			\resizebox{\textheight}{!}{
				\begin{tabular}{@{\extracolsep{\fill}}lccccccccccccccccccccccccc@{\extracolsep{\fill}}}
					\toprule
					&&\multicolumn{6}{@{}c@{}}{2P-Ti}&\multicolumn{6}{@{}c@{}}{3P-Ti}&\multicolumn{6}{@{}c@{}}{HP-Ti}&\multicolumn{6}{@{}c@{}}{TP-Ti}\\
					\cmidrule(lr){3-8}\cmidrule(lr){9-14}\cmidrule(lr){15-20}\cmidrule(lr){21-26}
					Case & Units &  pre\_cons & redu\_con &  pre\_vars & redu\_var &  pre\_nozs & redu\_noz &  pre\_cons & redu\_con &  pre\_vars & redu\_var &  pre\_nozs & redu\_noz &  pre\_cons & redu\_con &  pre\_vars & redu\_var &  pre\_nozs & redu\_noz &  pre\_cons & redu\_con &  pre\_vars & redu\_var &  pre\_nozs & redu\_noz \\
					& & & (\%) & & (\%) & & (\%) & & (\%) & & (\%) & & (\%) & & (\%) & & (\%) & & (\%) & & (\%) & & (\%) & & (\%) \\
					\midrule
					1 & 28 & 7213  & 0.00  & 3378  & 0.00  & 48395  & 0.00  & 7749  & 0.09  & 3378  & 0.00  & 51491  & 0.05  & 7773  & 0.10  & 3378  & 0.00  & 51686  & 0.07  & 7023  & 0.10  & 4689  & 0.00  & 43711  & -1.05  \\
					2 & 35 & 9153  & 0.00  & 4199  & 0.00  & 62161  & 0.00  & 10146  & 0.39  & 4219  & 0.00  & 68071  & 0.23  & 10205  & 0.45  & 4199  & 0.00  & 68933  & 0.29  & 9074  & 0.44  & 5590  & 0.00  & 55975  & -0.90  \\
					3 & 44 & 11317  & 0.00  & 5352  & 0.00  & 73226  & 0.00  & 12499  & 0.26  & 5412  & 0.00  & 79896  & 0.16  & 12418  & 0.31  & 5352  & 0.00  & 80219  & 0.21  & 11330  & 0.29  & 7924  & 0.00  & 73420  & -0.78  \\
					4 & 45 & 11216  & 0.00  & 5554  & 0.00  & 70109  & 0.00  & 12183  & 0.23  & 5614  & 0.00  & 75410  & 0.15  & 12091  & 0.26  & 5554  & 0.00  & 75618  & 0.18  & 11309  & 0.25  & 8956  & 0.00  & 69607  & -0.56  \\
					5 & 49 & 12605  & 0.00  & 5958  & 0.00  & 81200  & 0.00  & 13855  & 0.38  & 5978  & 0.00  & 88834  & 0.24  & 13959  & 0.44  & 5958  & 0.00  & 90060  & 0.29  & 12686  & 0.42  & 8793  & 0.00  & 81600  & -0.64  \\
					6 & 50 & 12406  & 0.00  & 6215  & 0.00  & 74658  & 0.00  & 13740  & 0.44  & 6315  & 0.00  & 82175  & 0.30  & 13626  & 0.51  & 6215  & 0.00  & 82882  & 0.37  & 12847  & 0.47  & 10544  & 0.00  & 79070  & -0.24  \\
					7 & 51 & 12813  & 0.00  & 6262  & 0.00  & 81373  & 0.00  & 13919  & 0.09  & 6342  & 0.00  & 87179  & 0.06  & 13724  & 0.11  & 6262  & 0.00  & 86823  & 0.07  & 12781  & 0.10  & 9854  & 0.00  & 81959  & -0.76  \\
					8 & 51 & 12854  & 0.00  & 6293  & 0.00  & 79658  & 0.00  & 14212  & 0.35  & 6433  & 0.00  & 78675  & 0.25  & 14148  & 0.41  & 6293  & 0.00  & 88143  & 0.28  & 13192  & 0.38  & 10246  & 0.00  & 85132  & -0.37  \\
					9 & 52 & 13296  & 0.00  & 6354  & 0.00  & 84563  & 0.00  & 14872  & 0.40  & 6454  & 0.00  & 93478  & 0.25  & 14749  & 0.46  & 6354  & 0.00  & 94121  & 0.31  & 13530  & 0.43  & 9760  & 0.00  & 86867  & -0.58  \\
					10 & 54 & 13485  & 0.00  & 6693  & 0.00  & 81980  & 0.00  & 14918  & 0.29  & 6813  & 0.00  & 89800  & 0.20  & 14701  & 0.35  & 6693  & 0.00  & 89915  & 0.24  & 13837  & 0.32  & 11202  & 0.00  & 84870  & -0.46  \\
					11 & 132 & 33269  & 0.00  & 16194  & 0.00  & 213763  & 0.00  & 36893  & 0.22  & 16514  & 0.00  & 232707  & 0.14  & 36189  & 0.26  & 16194  & 0.00  & 232026  & 0.18  & 33479  & 0.25  & 25173  & 0.00  & 198776  & -0.29  \\
					12 & 156 & 39199  & 0.00  & 19174  & 0.00  & 247908  & 0.00  & 43093  & 0.24  & 19434  & 0.00  & 269364  & 0.16  & 42664  & 0.29  & 19174  & 0.00  & 269915  & 0.19  & 39602  & 0.26  & 30362  & 0.00  & 233730  & -0.21  \\
					13 & 156 & 39423  & 0.00  & 19162  & 0.00  & 249182  & 0.00  & 43794  & 0.32  & 19462  & 0.00  & 273556  & 0.21  & 43355  & 0.38  & 19162  & 0.00  & 274711  & 0.26  & 40083  & 0.35  & 30183  & 0.00  & 237059  & -0.31  \\
					14 & 160 & 40802  & 0.00  & 19537  & 0.00  & 261028  & 0.00  & 45397  & 0.47  & 19717  & 0.00  & 288411  & 0.30  & 45523  & 0.54  & 19537  & 0.00  & 292484  & 0.36  & 41564  & 0.51  & 29701  & 0.00  & 245005  & -0.05  \\
					15 & 165 & 43033  & 0.00  & 19866  & 0.00  & 286426  & 0.00  & 47677  & 0.45  & 19886  & 0.00  & 315410  & 0.27  & 48279  & 0.51  & 19866  & 0.00  & 320907  & 0.33  & 43086  & 0.49  & 27296  & 0.00  & 250126  & -0.05  \\
					16 & 167 & 42225  & 0.00  & 20486  & 0.00  & 266401  & 0.00  & 46847  & 0.41  & 20726  & 0.00  & 293251  & 0.26  & 46736  & 0.47  & 20486  & 0.00  & 296259  & 0.32  & 43089  & 0.45  & 32174  & 0.00  & 254236  & -0.12  \\
					17 & 172 & 43252  & 0.00  & 21118  & 0.00  & 274647  & 0.00  & 47392  & 0.12  & 21458  & 0.00  & 296262  & 0.07  & 46557  & 0.14  & 21118  & 0.00  & 294721  & 0.09  & 43326  & 0.13  & 33312  & 0.00  & 253008  & -0.27  \\
					18 & 182 & 46249  & 0.00  & 22271  & 0.00  & 294854  & 0.00  & 51338  & 0.42  & 22511  & 0.00  & 324558  & 0.26  & 51285  & 0.48  & 22271  & 0.00  & 328095  & 0.33  & 47021  & 0.45  & 34210  & 0.00  & 274945  & -0.16  \\
					19 & 182 & 46075  & 0.00  & 22272  & 0.00  & 293823  & 0.00  & 51196  & 0.48  & 22492  & 0.00  & 324165  & 0.30  & 51309  & 0.55  & 22272  & 0.00  & 328773  & 0.37  & 47041  & 0.52  & 34414  & 0.00  & 276345  & -0.12  \\
					20 & 183 & 46255  & 0.00  & 22402  & 0.00  & 295570  & 0.00  & 51159  & 0.39  & 22642  & 0.00  & 323946  & 0.25  & 51072  & 0.45  & 22402  & 0.00  & 327177  & 0.31  & 46941  & 0.43  & 34619  & 0.00  & 274911  & -0.20  \\
					21 & 187 & 47074  & 0.00  & 22961  & 0.00  & 297810  & 0.00  & 52051  & 0.45  & 23181  & 0.00  & 327159  & 0.29  & 52132  & 0.52  & 22961  & 0.00  & 331455  & 0.35  & 48052  & 0.48  & 36066  & 0.00  & 281953  & -0.03  \\
					22 & 560 & 143367  & 0.00  & 67560  & 0.00  & 967982  & 0.00  & 154087  & 0.09  & 67560  & 0.00  & 1029904  & 0.05  & 154567  & 0.10  & 67560  & 0.00  & 1033804  & 0.07  & 139567  & 2.15  & 93780  & 0.00  & 927750  & 19.23  \\
					23 & 700 & 182167  & 0.00  & 83980  & 0.00  & 1243299  & 0.00  & 202027  & 0.39  & 84380  & 0.00  & 1361500  & 0.23  & 203207  & 0.45  & 83980  & 0.00  & 1378731  & 0.29  & 184387  & 0.43  & 111800  & 0.00  & 1465883  & 0.21  \\
					24 & 880 & 225447  & 0.00  & 107040  & 0.00  & 1464515  & 0.00  & 249087  & 0.26  & 108240  & 0.00  & 1597920  & 0.16  & 247467  & 0.31  & 107040  & 0.00  & 1604367  & 0.21  & 229787  & 0.30  & 158480  & 0.00  & 1865371  & 0.14  \\
					25 & 900 & 223427  & 0.00  & 111080  & 0.00  & 1402179  & 0.00  & 242767  & 0.23  & 112280  & 0.00  & 1508200  & 0.15  & 240927  & 0.26  & 111080  & 0.00  & 1512347  & 0.19  & 228787  & 0.34  & 179120  & 0.00  & 1917874  & 0.17  \\
					26 & 980 & 251207  & 0.00  & 119160  & 0.00  & 1623991  & 0.00  & 276207  & 0.38  & 119560  & 0.00  & 1776680  & 0.24  & 278287  & 0.44  & 119160  & 0.00  & 1801171  & 0.29  & 257187  & 0.45  & 175860  & 0.00  & 2075852  & 0.23  \\
					27 & 1000 & 247227  & 0.00  & 124300  & 0.00  & 1493160  & 0.00  & 273907  & 0.44  & 126300  & 0.00  & 1643500  & 0.30  & 271627  & 0.51  & 124300  & 0.00  & 1657623  & 0.37  & 259427  & 0.55  & 210880  & 0.00  & 2159739  & 1.62  \\
					28 & 1020 & 255367  & 0.00  & 125240  & 0.00  & 1627460  & 0.00  & 277487  & 0.09  & 126840  & 0.00  & 1743580  & 0.06  & 273587  & 0.11  & 125240  & 0.00  & 1736423  & 0.08  & 259007  & 0.17  & 197080  & 0.00  & 2144855  & 0.64  \\
					29 & 1020 & 256187  & 0.00  & 125860  & 0.00  & 1593160  & 0.00  & 287327  & 0.35  & 128660  & 0.00  & 1761340  & 0.23  & 282067  & 0.41  & 125860  & 0.00  & 1762851  & 0.28  & 267007  & 0.40  & 204920  & 0.00  & 2158174  & 3.51  \\
					30 & 1040 & 265027  & 0.00  & 127080  & 0.00  & 1691247  & 0.00  & 296547  & 0.40  & 129080  & 0.00  & 1869560  & 0.25  & 294087  & 0.46  & 127080  & 0.00  & 1882407  & 0.31  & 274247  & 0.44  & 195200  & 0.00  & 2245151  & 0.22  \\
					31 & 1080 & 268807  & 0.00  & 133860  & 0.00  & 1639600  & 0.00  & 297467  & 0.29  & 136260  & 0.00  & 1796000  & 0.20  & 293127  & 0.35  & 133860  & 0.00  & 1798275  & 0.24  & 279727  & 0.38  & 224040  & 0.00  & 2307866  & 2.10  \\
					\botrule
				\end{tabular}
			}
		\end{minipage}
	\end{center}
\end{sidewaystable}
\begin{sidewaystable}
	\begin{center}
		\begin{minipage}{\textheight}
			\caption{Comparison of MP-3 and MP-Ti MILP formulations in terms of compactness for the second data set.}\label{tab:HistoryCompactness2}
			\resizebox{\textheight}{!}{
				\begin{tabular}{@{\extracolsep{\fill}}lccccccccccccccccccccccccc@{\extracolsep{\fill}}}
					\toprule
					&&\multicolumn{6}{@{}c@{}}{2P-Ti}&\multicolumn{6}{@{}c@{}}{3P-Ti}&\multicolumn{6}{@{}c@{}}{HP-Ti}&\multicolumn{6}{@{}c@{}}{TP-Ti}\\
					\cmidrule(lr){3-8}\cmidrule(lr){9-14}\cmidrule(lr){15-20}\cmidrule(lr){21-26}
					Case & Units  &  pre\_cons & redu\_con &  pre\_vars & redu\_var &  pre\_nozs & redu\_noz &  pre\_cons & redu\_con &  pre\_vars & redu\_var &  pre\_nozs & redu\_noz &  pre\_cons & redu\_con &  pre\_vars & redu\_var &  pre\_nozs & redu\_noz &  pre\_cons & redu\_con &  pre\_vars & redu\_var &  pre\_nozs & redu\_noz \\
					& & & (\%) & & (\%) & & (\%) & & (\%) & & (\%) & & (\%) & & (\%) & & (\%) & & (\%) & & (\%) & & (\%) & & (\%) \\
					\midrule
					32 & 10 & 2207  & 0.05  & 953  & 0.00  & 13295  & 0.03  & 2714  & 2.30  & 952  & 0.00  & 17098  & 2.75  & 2887  & 2.40  & 950  & 0.21  & 18560  & 2.80  & 2580  & 2.68  & 1636  & 0.00  & 33061  & -83.61  \\
					33 & 10 & 2148  & 0.00  & 937  & -0.11  & 12308  & 0.00  & 2669  & 1.33  & 934  & 0.00  & 15642  & 0.79  & 2847  & 1.39  & 936  & -0.21  & 17143  & 0.96  & 2575  & 1.79  & 1846  & -0.44  & 46400  & -144.75  \\
					34 & 10 & 2439  & 0.08  & 1070  & 0.00  & 15234  & 0.05  & 3009  & 1.86  & 1069  & 0.00  & 20764  & 2.10  & 3207  & 2.23  & 1068  & 0.19  & 22498  & 2.58  & 2846  & 2.43  & 2055  & -0.15  & 44978  & -118.57  \\
					35 & 10 & 2306  & 0.00  & 1053  & 0.00  & 12574  & 0.27  & 3079  & 1.75  & 1052  & -0.10  & 19458  & 2.16  & 3573  & 1.79  & 1253  & -0.24  & 23750  & 1.99  & 2839  & 2.20  & 2234  & -0.27  & 24328  & -9.62  \\
					36 & 10 & 2180  & 0.00  & 978  & 0.00  & 10473  & 0.00  & 2918  & 2.18  & 977  & 0.00  & 15008  & 2.38  & 3200  & 1.99  & 1078  & 0.09  & 19873  & 1.84  & 2728  & 2.29  & 2124  & -0.14  & 33775  & -55.80  \\
					37 & 20 & 4535  & 0.00  & 2017  & 0.00  & 28566  & 0.00  & 5619  & 1.21  & 2010  & 0.00  & 36195  & 0.69  & 5984  & 1.42  & 2006  & 0.30  & 39547  & 0.68  & 5363  & 1.83  & 3803  & -0.13  & 74438  & -100.33  \\
					38 & 20 & 4576  & 0.00  & 2017  & 0.00  & 26416  & 0.00  & 5689  & 1.32  & 2013  & 0.00  & 35159  & 1.04  & 6057  & 1.78  & 2011  & 0.15  & 38367  & 1.48  & 5407  & 1.71  & 3621  & -0.19  & 59189  & -63.15  \\
					39 & 20 & 4446  & 0.00  & 1979  & 0.00  & 26253  & 0.00  & 5495  & 1.52  & 1972  & 0.00  & 33520  & 0.85  & 5856  & 1.70  & 1970  & 0.20  & 36759  & 0.96  & 5318  & 1.70  & 4039  & -0.40  & 102512  & -183.06  \\
					40 & 20 & 4349  & 0.05  & 2052  & 0.00  & 23762  & 0.03  & 5904  & 0.64  & 2049  & -0.39  & 32995  & 0.34  & 6492  & 1.96  & 2259  & 1.40  & 43474  & 1.79  & 5557  & 2.23  & 4893  & 2.76  & 134478  & -204.25  \\
					41 & 20 & 4461  & 0.02  & 2073  & 0.00  & 23723  & 0.01  & 6063  & 1.16  & 2072  & 0.00  & 32671  & 0.91  & 7079  & 4.45  & 2495  & 4.41  & 44191  & 5.57  & 5644  & 5.81  & 4703  & 11.91  & 118228  & -153.01  \\
					42 & 50 & 11536  & 0.00  & 5179  & 0.00  & 57211  & 0.36  & 14287  & 1.22  & 5159  & 0.00  & 79867  & 0.96  & 15215  & 1.67  & 5141  & 0.39  & 89946  & 1.04  & 13694  & 1.60  & 10301  & -0.23  & 217748  & -136.76  \\
					43 & 50 & 11656  & 0.00  & 5225  & 0.00  & 56654  & 0.08  & 14418  & 1.50  & 5208  & 0.00  & 79097  & 1.14  & 15350  & 2.02  & 5193  & 0.33  & 89278  & 1.42  & 13895  & 1.40  & 10661  & -0.46  & 247907  & -168.32  \\
					44 & 50 & 11328  & -0.01  & 5091  & 0.00  & 56373  & 0.34  & 13990  & 1.54  & 5070  & 0.00  & 78473  & 1.34  & 14888  & 1.94  & 5053  & 0.34  & 88387  & 1.49  & 13483  & 1.61  & 10487  & -0.46  & 249359  & -170.76  \\
					45 & 50 & 11050  & -0.01  & 5229  & 0.00  & 48980  & 0.30  & 14944  & 1.33  & 5218  & -0.04  & 75564  & 1.18  & 17443  & 1.78  & 6271  & 0.21  & 95646  & 3.04  & 13963  & 1.90  & 12420  & -0.07  & 290565  & -166.55  \\
					46 & 50 & 11311  & -0.02  & 5329  & 0.00  & 50313  & 0.08  & 15254  & 1.49  & 5319  & 0.00  & 76549  & 1.42  & 17790  & 1.90  & 6369  & 0.22  & 97586  & 3.21  & 14253  & 2.07  & 12635  & -0.09  & 299209  & -169.54  \\
					47 & 75 & 16935  & 0.00  & 7677  & 0.00  & 79713  & 0.08  & 20983  & 1.13  & 7637  & 0.00  & 111907  & 1.37  & 22281  & 1.67  & 7601  & 0.47  & 127657  & 1.13  & 20327  & 1.04  & 16416  & -0.37  & 394974  & -182.75  \\
					48 & 75 & 17350  & 0.00  & 7832  & 0.00  & 81491  & 0.14  & 21476  & 1.17  & 7796  & 0.00  & 113951  & 1.06  & 22835  & 1.90  & 7768  & 0.36  & 129339  & 1.51  & 20719  & 1.46  & 16431  & -0.31  & 385202  & -173.77  \\
					49 & 75 & 17224  & 0.01  & 7744  & 0.00  & 81162  & 0.11  & 21313  & 1.81  & 7714  & 0.00  & 113002  & 1.65  & 22714  & 1.98  & 7692  & 0.29  & 128408  & 1.50  & 20491  & 1.60  & 15696  & -0.34  & 352952  & -156.95  \\
					50 & 75 & 17194  & 0.00  & 7733  & 0.00  & 80472  & 0.23  & 21269  & 1.84  & 7706  & 0.00  & 112069  & 1.80  & 22636  & 2.08  & 7683  & 0.30  & 127357  & 1.68  & 20496  & 1.58  & 15975  & -0.48  & 373117  & -173.00  \\
					51 & 75 & 16589  & 0.00  & 7890  & 0.00  & 69644  & 0.34  & 22497  & 1.20  & 7872  & -0.04  & 107831  & 1.18  & 26327  & 1.56  & 9476  & 0.21  & 142681  & 1.54  & 21038  & 2.09  & 19129  & -0.07  & 489668  & -195.24  \\
					52 & 100 & 23057  & 0.00  & 10375  & 0.00  & 105032  & 0.18  & 28471  & 1.68  & 10378  & 0.27  & 146059  & 1.57  & 30277  & 2.10  & 10334  & 0.60  & 166452  & 1.57  & 27391  & 1.74  & 21220  & -0.25  & 490400  & -170.12  \\
					53 & 100 & 22744  & 0.00  & 10256  & 0.00  & 103288  & 0.10  & 28125  & 1.53  & 10242  & 0.16  & 144024  & 1.58  & 29903  & 1.95  & 10195  & 0.49  & 164430  & 1.49  & 27096  & 1.59  & 21090  & -0.30  & 478637  & -165.05  \\
					54 & 100 & 22878  & 0.00  & 10301  & 0.00  & 103417  & 0.04  & 28248  & 1.58  & 10315  & 0.38  & 143736  & 1.50  & 30039  & 1.99  & 10276  & 0.58  & 163693  & 1.63  & 27274  & 1.68  & 21609  & -0.38  & 514915  & -180.28  \\
					55 & 100 & 22758  & -0.00  & 10244  & 0.00  & 103415  & 0.22  & 28119  & 1.54  & 10242  & 0.29  & 144454  & 1.51  & 29917  & 1.91  & 10199  & 0.56  & 164639  & 1.40  & 27151  & 1.55  & 21150  & -0.31  & 490836  & -169.01  \\
					56 & 100 & 22949  & -0.00  & 10325  & 0.00  & 104097  & 0.29  & 28399  & 1.59  & 10310  & 0.14  & 145324  & 1.65  & 30220  & 1.98  & 10269  & 0.39  & 165785  & 1.43  & 27286  & 1.51  & 20628  & -0.31  & 427577  & -137.37  \\
					57 & 150 & 34234  & -0.01  & 15562  & 0.00  & 149229  & 0.20  & 42342  & 1.82  & 15462  & 0.46  & 206781  & 1.67  & 45046  & 2.19  & 15397  & 0.73  & 236090  & 1.72  & 40768  & 1.90  & 31574  & -0.29  & 694568  & -160.55  \\
					58 & 150 & 34164  & -0.01  & 15559  & 0.00  & 149167  & 0.07  & 42317  & 1.61  & 15457  & 0.35  & 207201  & 1.44  & 45016  & 2.00  & 15390  & 0.59  & 237072  & 1.50  & 40818  & 1.57  & 32023  & -0.38  & 758207  & -182.89  \\
					59 & 150 & 34007  & -0.01  & 15501  & 0.00  & 148146  & 0.11  & 42098  & 1.64  & 15403  & 0.46  & 205463  & 1.55  & 44793  & 2.08  & 15336  & 0.65  & 235073  & 1.59  & 40650  & 1.65  & 32105  & -0.35  & 773638  & -186.52  \\
					60 & 150 & 34142  & -0.01  & 15561  & 0.00  & 149225  & 0.06  & 42293  & 1.48  & 15452  & 0.41  & 207869  & 1.39  & 44961  & 1.93  & 15379  & 0.74  & 238184  & 1.41  & 40887  & 1.59  & 32472  & -0.29  & 780618  & -187.18  \\
					61 & 150 & 33841  & -0.01  & 15450  & 0.00  & 147221  & 0.18  & 41888  & 1.56  & 15344  & 0.50  & 204771  & 1.42  & 44557  & 1.97  & 15271  & 0.77  & 234371  & 1.50  & 40557  & 1.49  & 32330  & -0.42  & 780435  & -189.98  \\
					62 & 200 & 44731  & 0.00  & 21357  & 0.00  & 167475  & 0.04  & 60483  & 1.30  & 21263  & 0.30  & 257349  & 1.39  & 70586  & 1.54  & 25410  & 0.33  & 370863  & 1.42  & 56629  & 2.01  & 50463  & -0.12  & 1133378  & -163.07  \\
					63 & 200 & 44203  & 0.00  & 21186  & 0.00  & 165529  & 0.05  & 59849  & 1.17  & 21093  & 0.37  & 255332  & 1.20  & 69903  & 1.51  & 25267  & 0.42  & 369229  & 1.40  & 56190  & 1.95  & 51025  & -0.09  & 1223444  & -182.20  \\
					64 & 200 & 43858  & 0.00  & 20993  & 0.00  & 164254  & 0.08  & 59467  & 1.25  & 20897  & 0.25  & 254063  & 1.33  & 69485  & 1.42  & 25059  & 0.35  & 367603  & 1.30  & 55708  & 1.82  & 50023  & -0.08  & 1154743  & -169.64  \\
					65 & 200 & 45840  & 0.01  & 20808  & 0.01  & 195239  & 0.10  & 56815  & 1.69  & 20674  & 0.26  & 272026  & 1.55  & 60454  & 2.06  & 20594  & 0.52  & 312792  & 1.52  & 54588  & 1.67  & 41449  & -0.32  & 872550  & -145.99  \\
					66 & 200 & 45750  & -0.01  & 20821  & 0.00  & 194451  & 0.11  & 56714  & 1.57  & 20688  & 0.30  & 271524  & 1.42  & 60331  & 1.97  & 20606  & 0.57  & 312305  & 1.51  & 54706  & 1.52  & 42760  & -0.39  & 996296  & -177.77  \\
					67 & 200 & 45314  & -0.01  & 20642  & 0.00  & 192382  & 0.20  & 56153  & 1.47  & 20506  & 0.29  & 268985  & 1.36  & 59694  & 1.94  & 20414  & 0.59  & 309434  & 1.43  & 54209  & 1.44  & 42641  & -0.30  & 985777  & -175.81  \\
					68 & 200 & 45113  & -0.00  & 20507  & 0.00  & 191704  & 0.05  & 55933  & 1.48  & 20372  & 0.17  & 267790  & 1.37  & 59454  & 1.93  & 20286  & 0.49  & 307691  & 1.47  & 53868  & 1.47  & 41497  & -0.31  & 908161  & -158.94  \\
					69 & 200 & 45847  & -0.00  & 20891  & 0.00  & 194728  & 0.07  & 56802  & 1.64  & 20759  & 0.35  & 271331  & 1.46  & 60426  & 2.05  & 20666  & 0.60  & 311804  & 1.55  & 54818  & 1.58  & 43272  & -0.35  & 1020516  & -183.33  \\
					70 & 200 & 43877  & 0.00  & 21039  & 0.00  & 164491  & 0.09  & 59510  & 1.19  & 20945  & 0.26  & 254759  & 1.26  & 69495  & 1.39  & 25121  & 0.37  & 368583  & 1.27  & 55828  & 1.85  & 50702  & -0.09  & 1199085  & -177.97  \\
					71 & 200 & 44706  & 0.00  & 21348  & 0.00  & 167841  & 0.04  & 60528  & 1.16  & 21254  & 0.16  & 258721  & 1.25  & 70615  & 1.41  & 25418  & 0.20  & 372703  & 1.26  & 56612  & 1.77  & 50409  & -0.11  & 1094816  & -152.40  \\
					72 & 200 & 44226  & 0.00  & 21235  & 0.00  & 165590  & 0.04  & 59929  & 1.16  & 21137  & 0.31  & 255669  & 1.08  & 70049  & 1.48  & 25332  & 0.39  & 369875  & 1.37  & 56203  & 1.95  & 51323  & -0.08  & 1279255  & -194.51  \\
					73 & 200 & 44347  & 0.00  & 21187  & 0.00  & 165972  & 0.04  & 59975  & 1.26  & 21091  & 0.35  & 255558  & 1.33  & 69983  & 1.46  & 25234  & 0.30  & 368469  & 1.30  & 56191  & 1.90  & 50077  & -0.15  & 1053273  & -143.82  \\
					\botrule
				\end{tabular}
			}
		\end{minipage}
	\end{center}
\end{sidewaystable}

The tightness of an MIP formulation is usually characterized by using the relative integrality gap $(Z_{MIP}-Z_{CR})/Z_{MIP}$ \cite{wolsey2020integer}, which is denoted as ``iGap". In this expression, $Z_{CR}$ represents the optimal value for the continuous relaxation of the MIP problem, which will vary with the formulation we construct for the MIP problem. And $Z_{MIP}$ represents the optimal value for the MIP problem. Due to the limitation of solving algorithm and solving time, it is usually difficult to get the real optimal value of MIP problem. The solving process usually terminates with an optimal value satisfying the preset optimal tolerance. To ensure fairness, we use the same $Z_{MIP}$ for all of the MILP formulations, which is equal to the minimum of the optimal values obtained by using GUROBI with an accuracy of 0.1\% among the aforementioned MILP formulations. In this way, the difference in iGap could accurately reflect the difference in tightness of MIP model’s continuous relaxation.

Fig. \ref{fig:History-iGap} shows the iGap for MP-3, MP-Ti and the other three state-of-the-art MILP formulations on all of the instances in comparison with 3P (using ratios), where the ratio for the 3P model is always 100\%. The differences in gap among MP-3 and MP-Ti models is large for the first dataset, but small for the second dataset. In additional, we have the following tightness relationship among the eleven formulations:
\begin{align}
	&2P\precsim 3P\precsim 2P-3\precsim 2P-Ti\\
	&2P\precsim 3P\precsim 2P-3\precsim 3P-HD=3P-3\precsim 3P-Ti\\
	&2P\precsim 3P\precsim 2P-3\precsim 3P-HD\precsim HP-3\precsim HP-Ti\\
	&2P\precsim 3P\precsim 2P-3\precsim 3P-HD\precsim TP-3\precsim TP-Ti
\end{align}

It is worth mentioning that the tightness relationship between 2P-Ti formualion and 3P-HD formualion is ambiguous. As shown in Fig. \ref{fig:History-iGap}, 2P-Ti formualion is more tighter than 3P-HD formualion for instances in the first data set. However, 3P-HD formualion is more tighter than 2P-Ti formualion for all instances in the second data set except three instances. The main reason is that MP-Ti formulations improve tightness mainly by considering the historical status in the lower/upper bound of generation limits and ramping constraints (as analyzed in Section \ref{sec4}), and the historical status plays a key role in some instances.
\begin{figure}[t]
	\begin{center}
		\subfigure[2-period]{
			\includegraphics[width=0.48\textwidth,height=4cm]{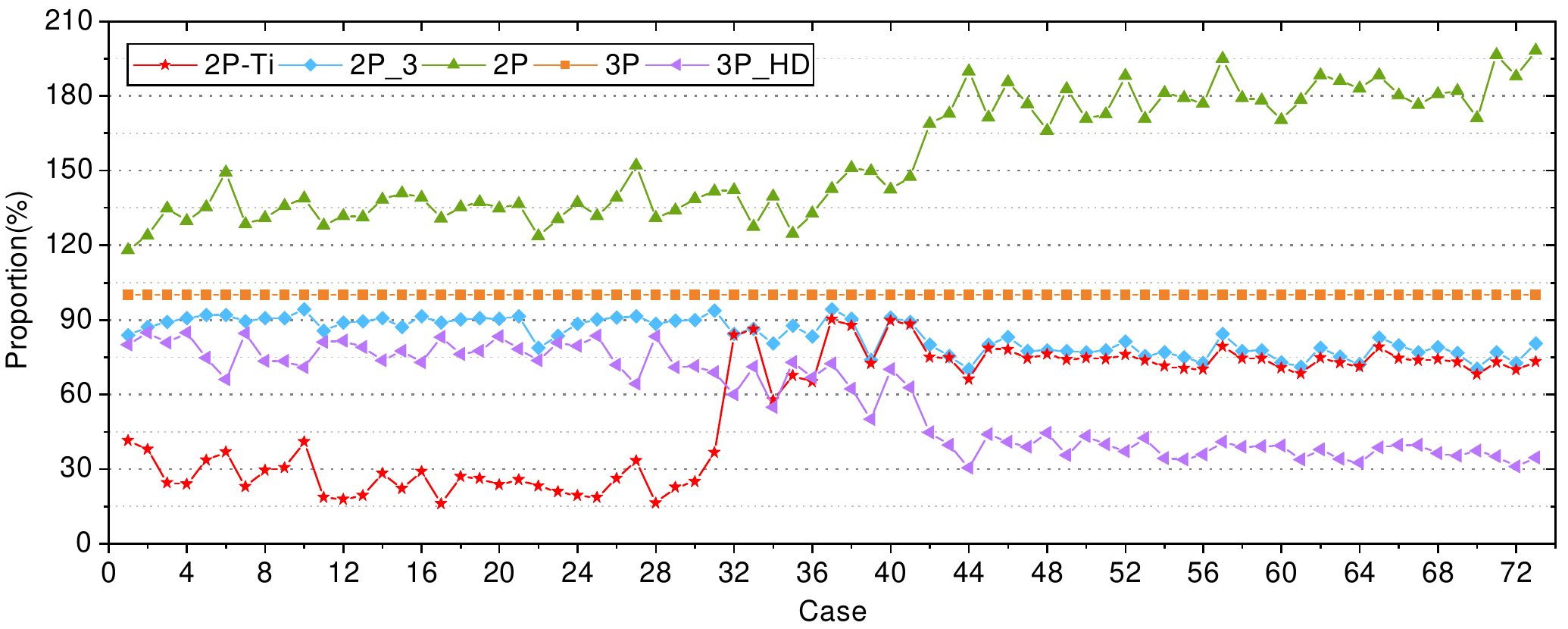}}\subfigure[3-period]{
			\includegraphics[width=0.48\textwidth,height=4cm]{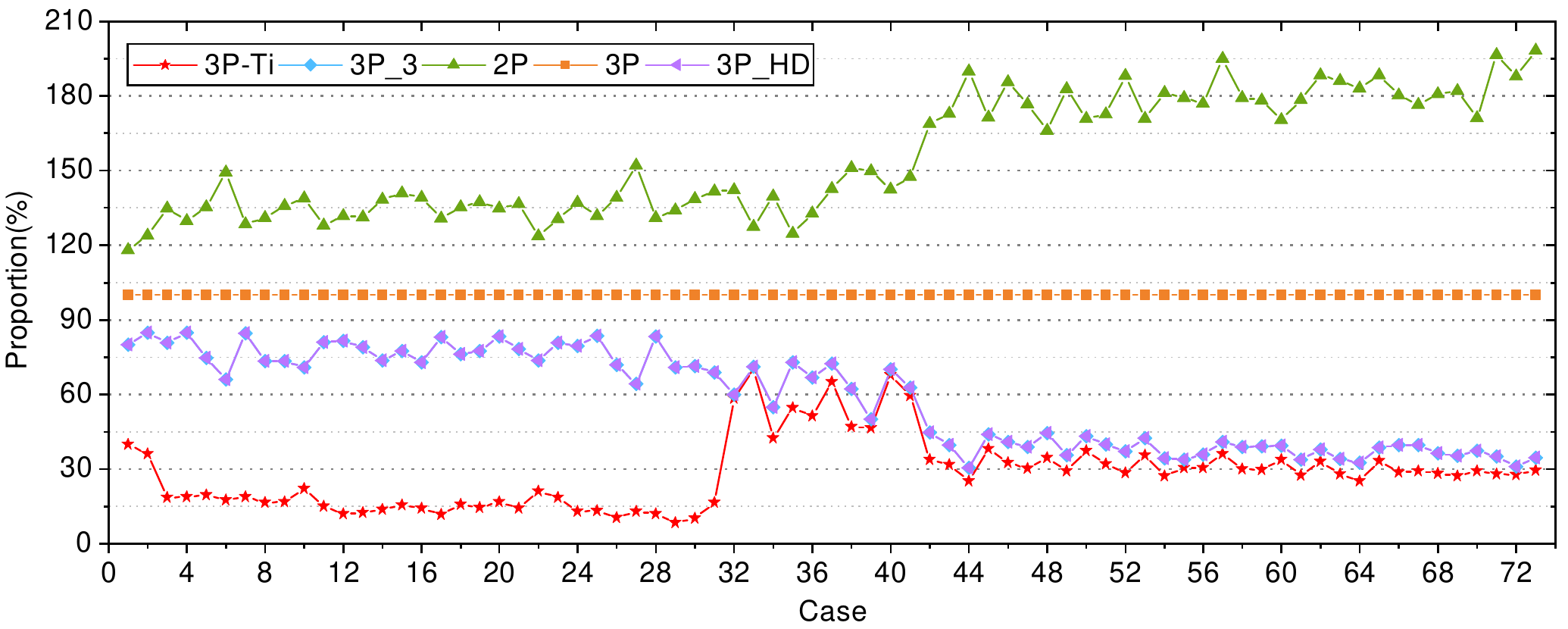}}
		\subfigure[H-period]{
			\includegraphics[width=0.48\textwidth,height=4cm]{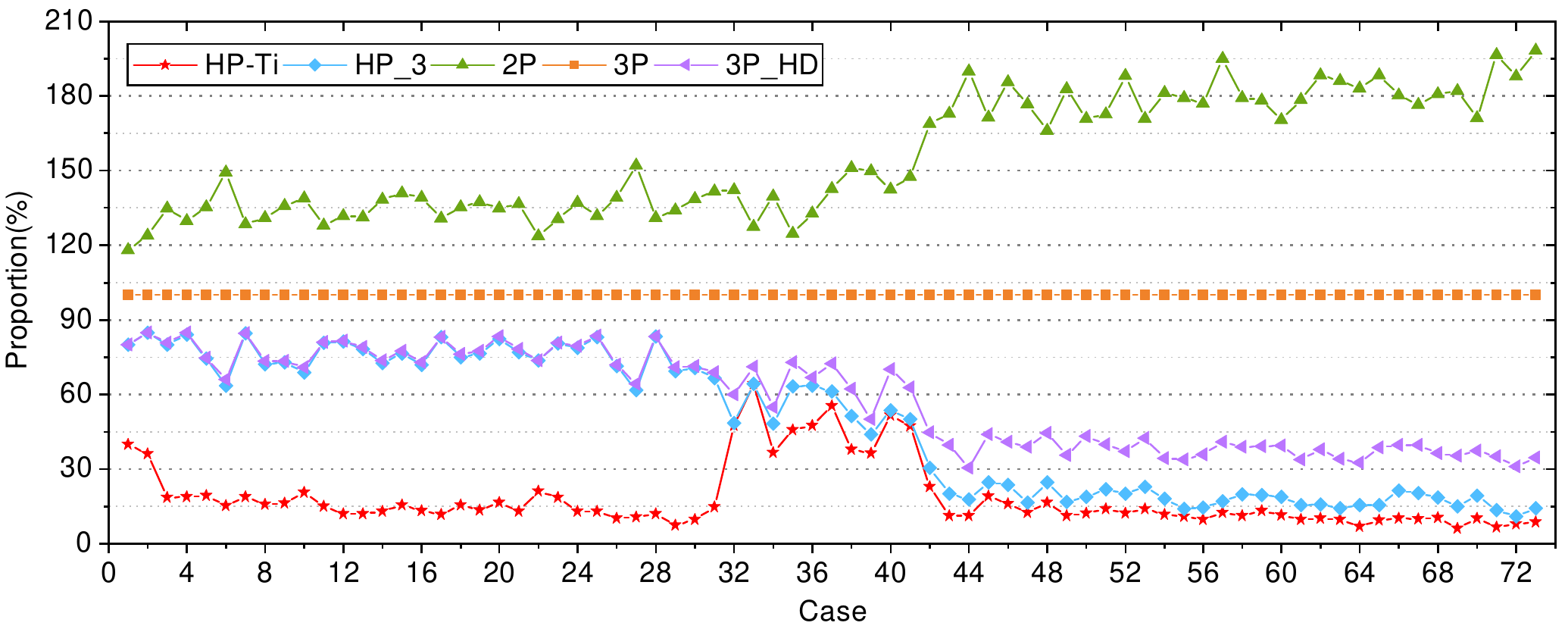}}\subfigure[T-period]{
			\includegraphics[width=0.48\textwidth,height=4cm]{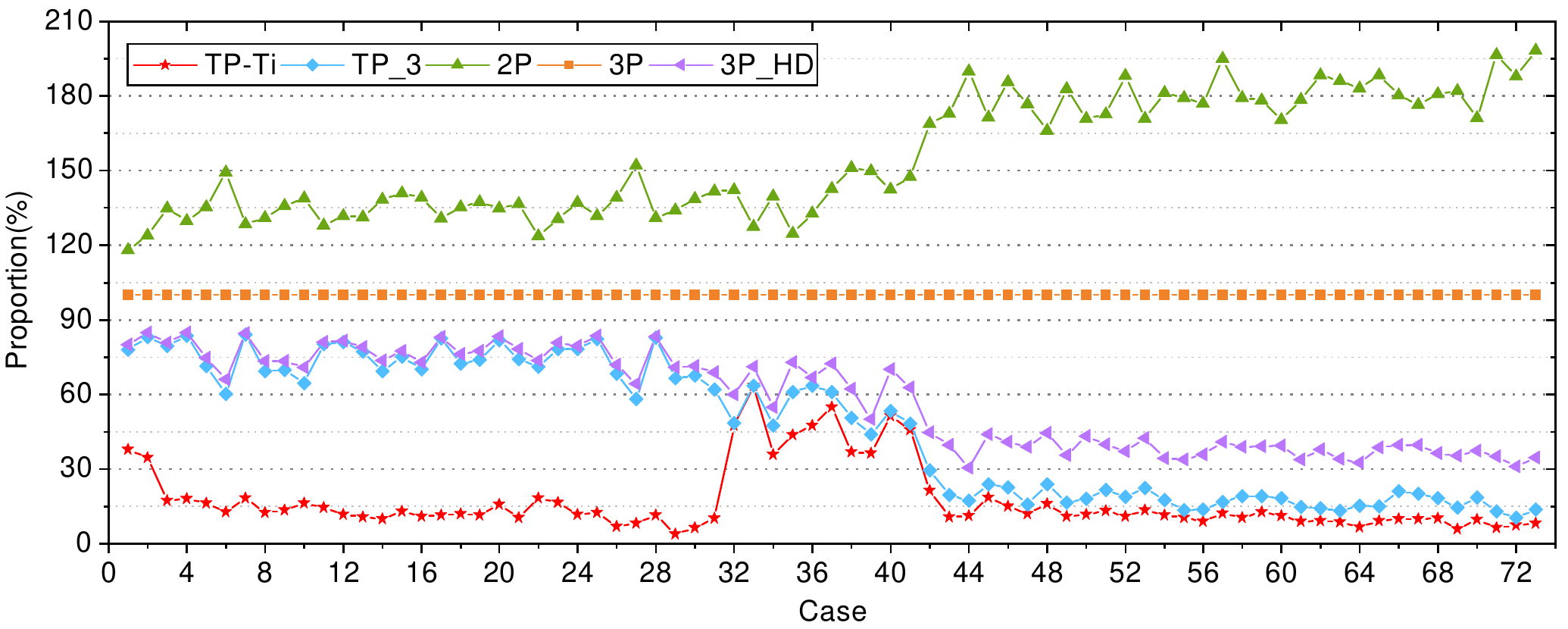}}
		\caption{Comparison of MP-3, MP-Ti and the other three state-of-the-art MILP formulations in terms of iGap.}
		\label{fig:History-iGap}
	\end{center}
\end{figure}

As shown in Fig. \ref{fig:History-runtime} and Fig. \ref{fig:History-profile-1}, the performance of MP-Ti models are slightly better than that of MP-3 models for the first data set. Fig. \ref{fig:History-runtime} and Fig. \ref{fig:History-profile-2} show that the performance of MP-Ti models are slightly worse than that of MP-3 models for 2-period, 3-period and H-period. Furthermore, MP-3 models perform significantly better than MP-Ti for T-period.
\begin{figure}[t]
	\begin{center}
		\subfigure[2-period]{
			\includegraphics[width=0.48\textwidth,height=4cm]{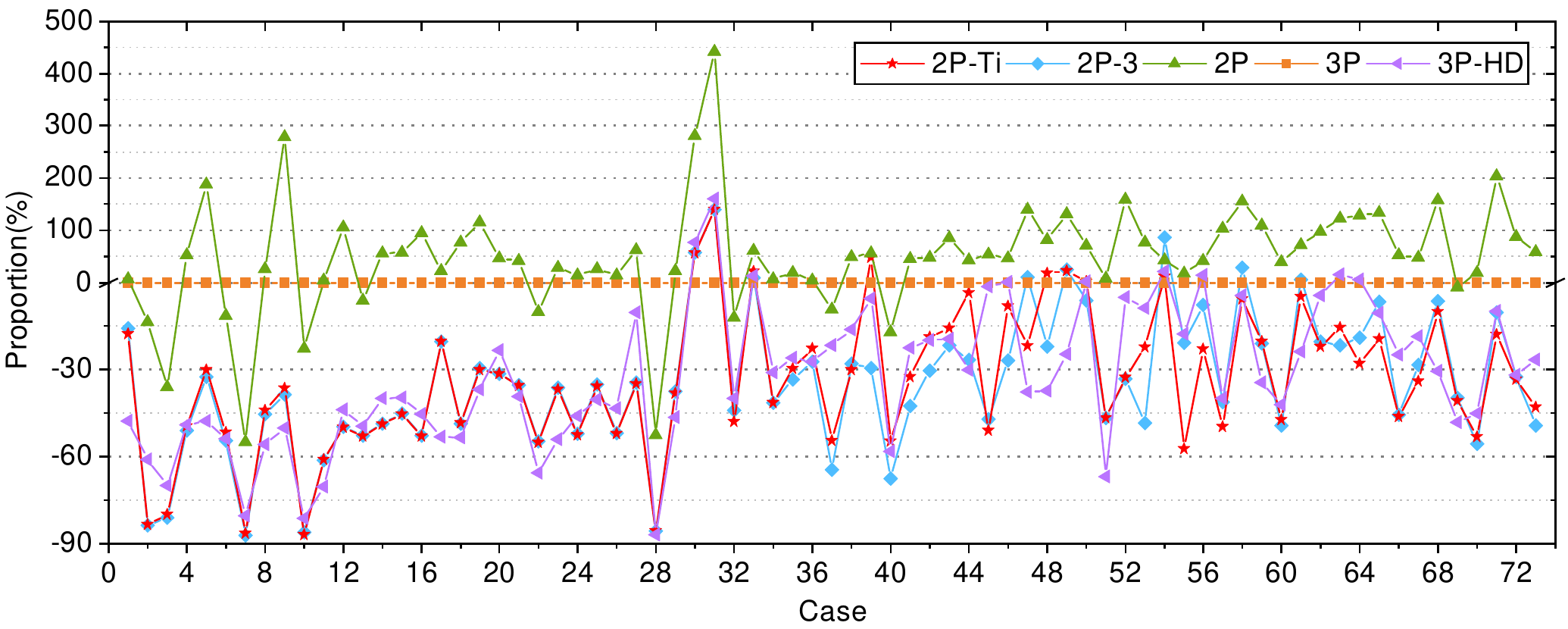}}\subfigure[3-period]{
			\includegraphics[width=0.48\textwidth,height=4cm]{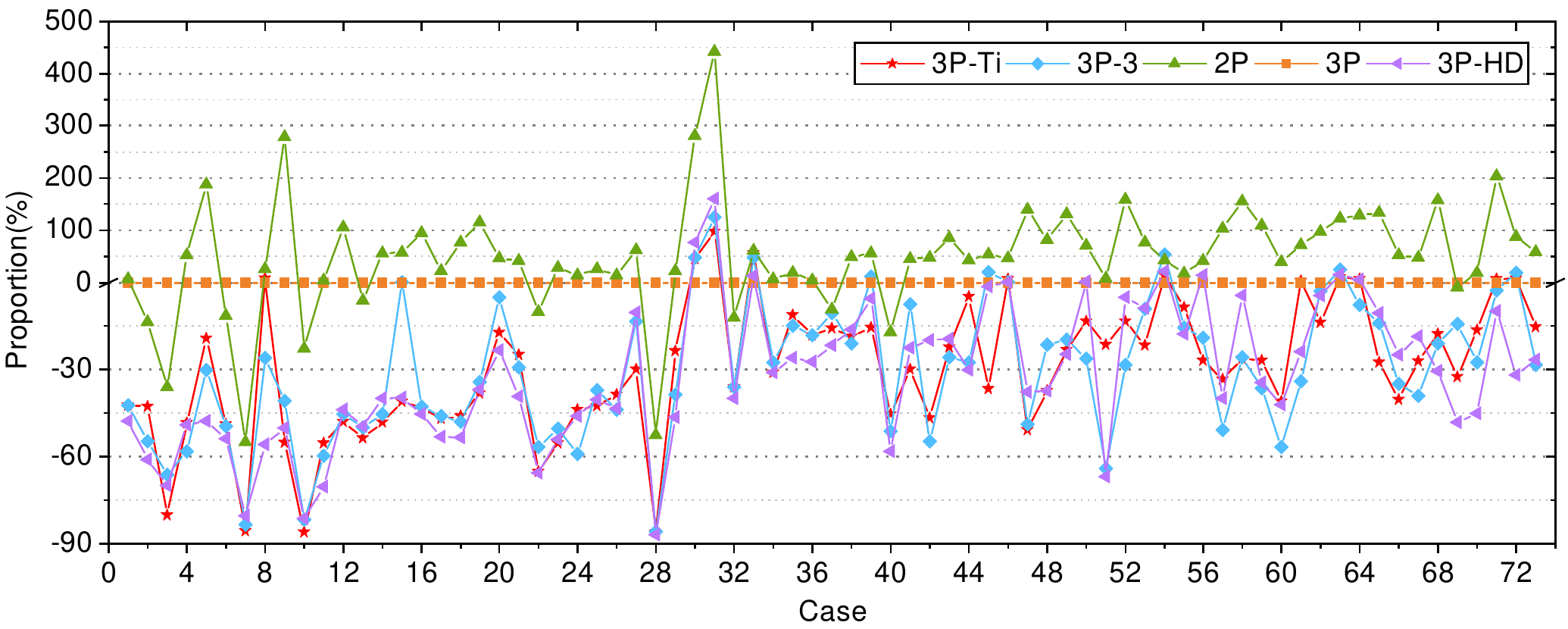}}
		\subfigure[H-period]{
			\includegraphics[width=0.48\textwidth,height=4cm]{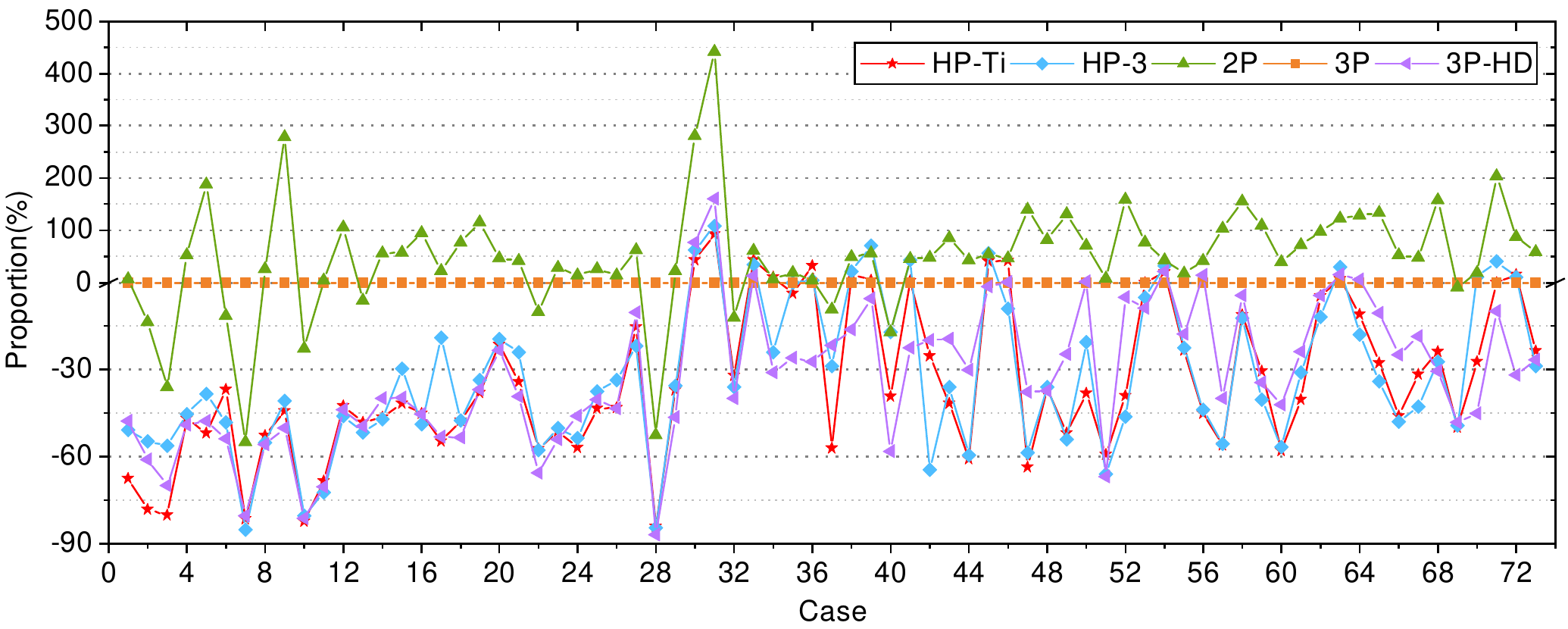}}\subfigure[T-period]{
			\includegraphics[width=0.48\textwidth,height=4cm]{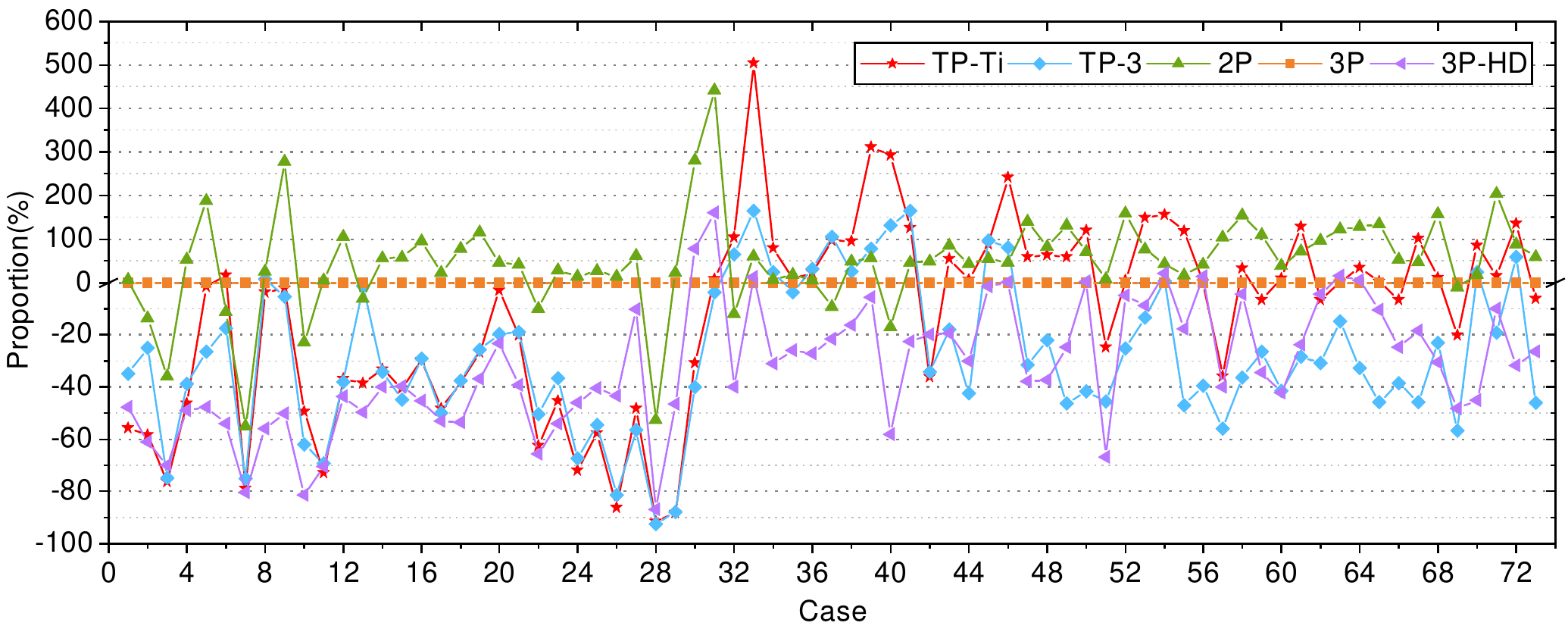}}
		\caption{Comparison of MP-3, MP-Ti and the other three state-of-the-art MILP formulations in terms of rTime.}
		\label{fig:History-runtime}
	\end{center}
\end{figure}
\begin{figure}[t]
	\begin{center}
		\subfigure[2-period]{
			\includegraphics[width=0.48\textwidth,height=4cm]{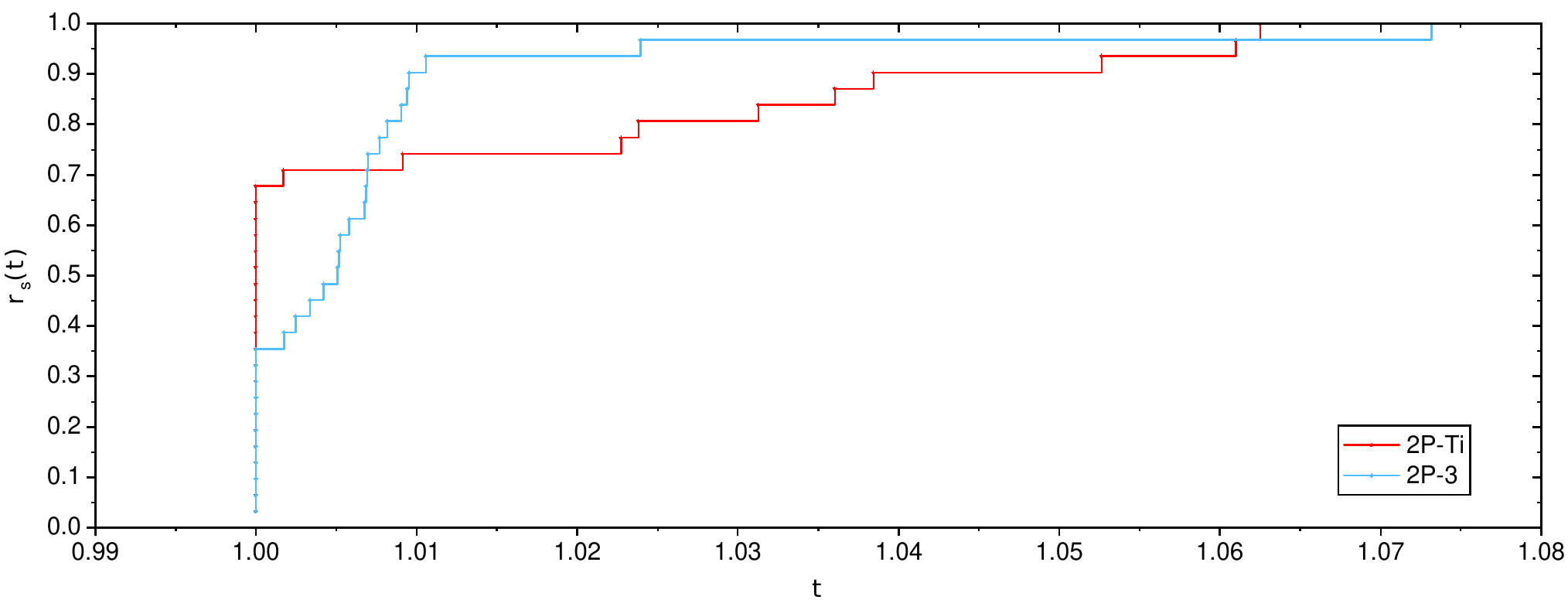}}\subfigure[3-period]{
			\includegraphics[width=0.48\textwidth,height=4cm]{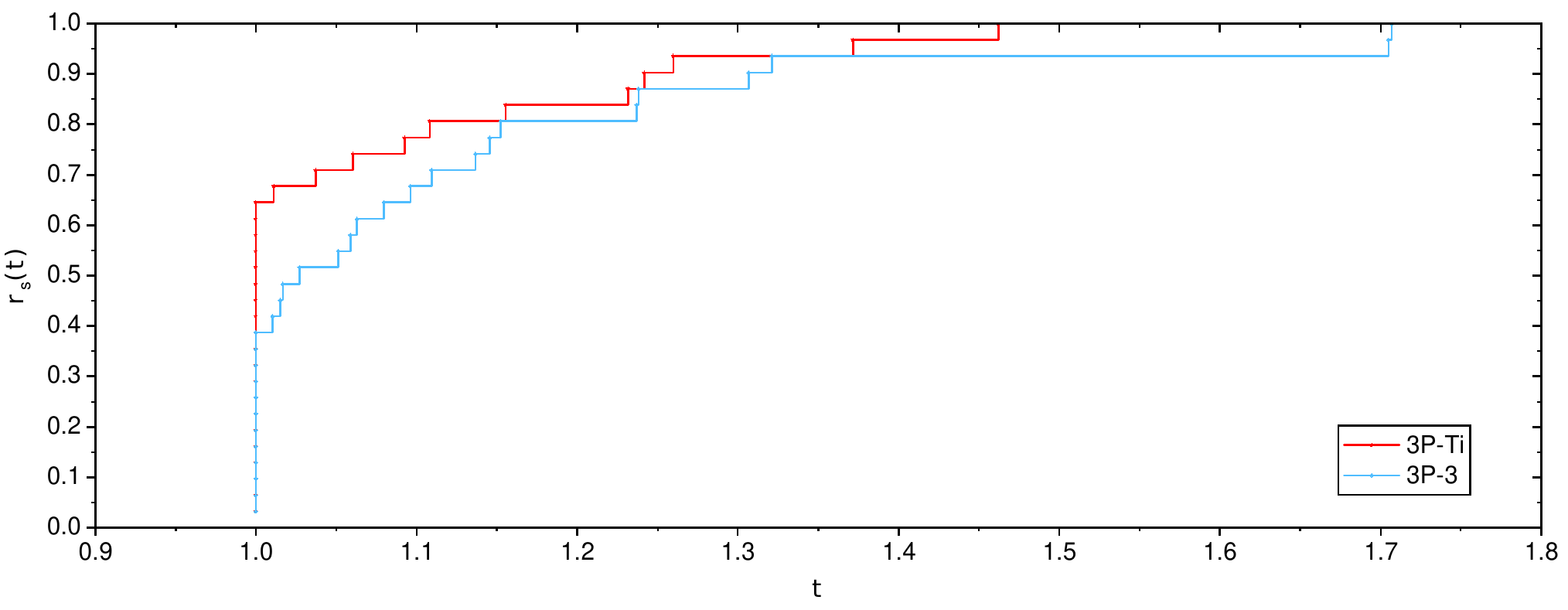}}
		\subfigure[H-period]{
			\includegraphics[width=0.48\textwidth,height=4cm]{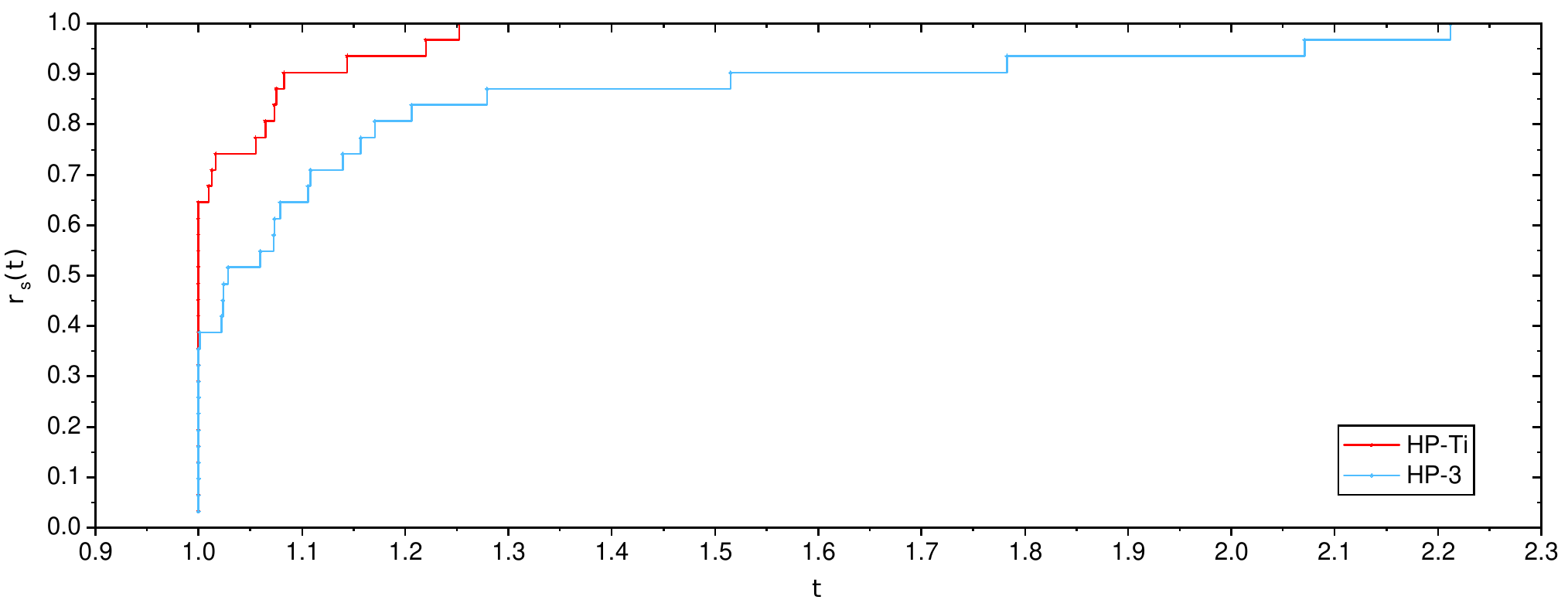}}\subfigure[T-period]{
			\includegraphics[width=0.48\textwidth,height=4cm]{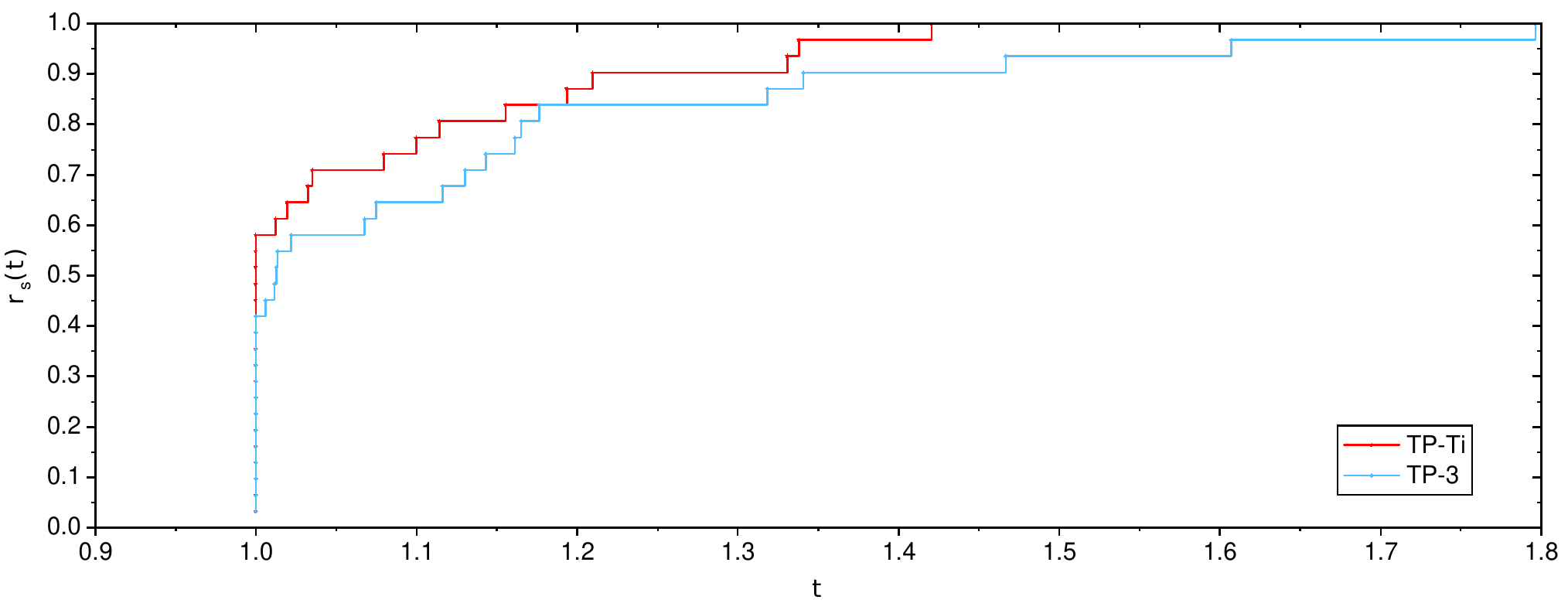}}
		\caption{Performance profiles on CPU time for MP-3 and MP-Ti formulations for the first data set.}
		\label{fig:History-profile-1}
	\end{center}
\end{figure}
\begin{figure}[t]
	\begin{center}
		\subfigure[2-period]{
			\includegraphics[width=0.48\textwidth,height=4cm]{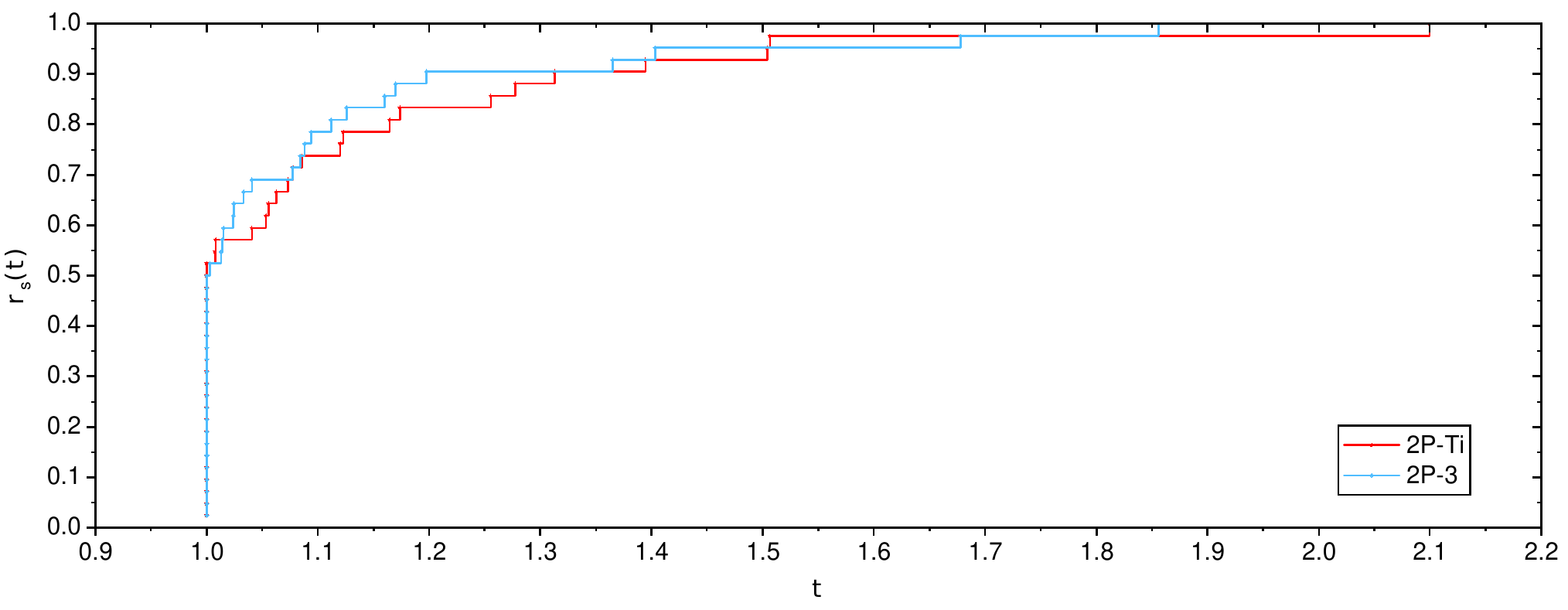}}\subfigure[3-period]{
			\includegraphics[width=0.48\textwidth,height=4cm]{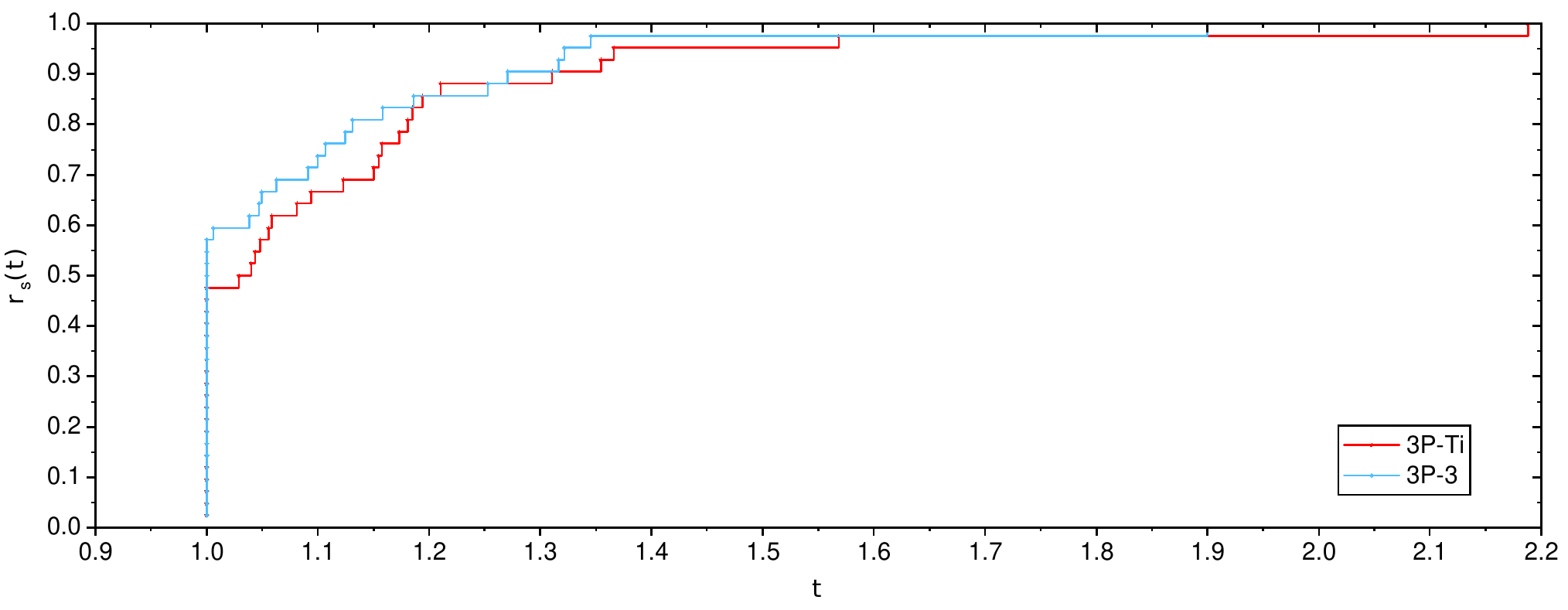}}
		\subfigure[H-period]{
			\includegraphics[width=0.48\textwidth,height=4cm]{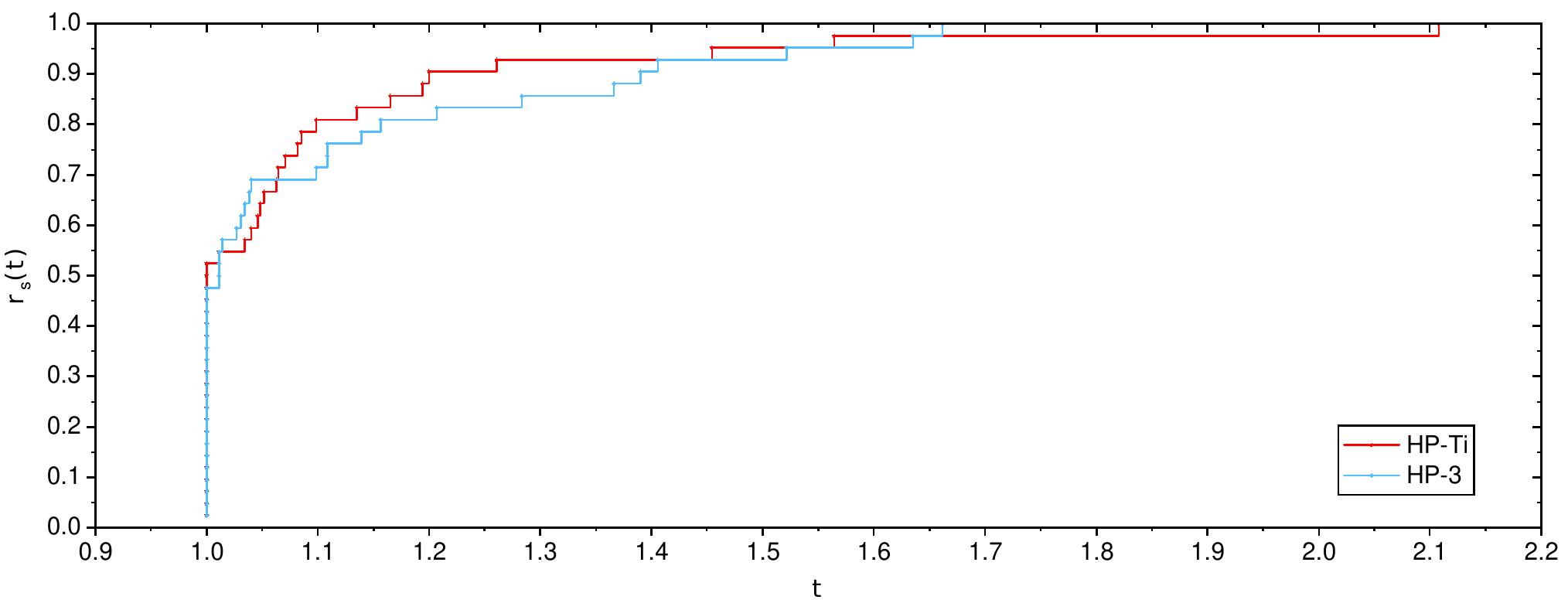}}\subfigure[T-period]{
			\includegraphics[width=0.48\textwidth,height=4cm]{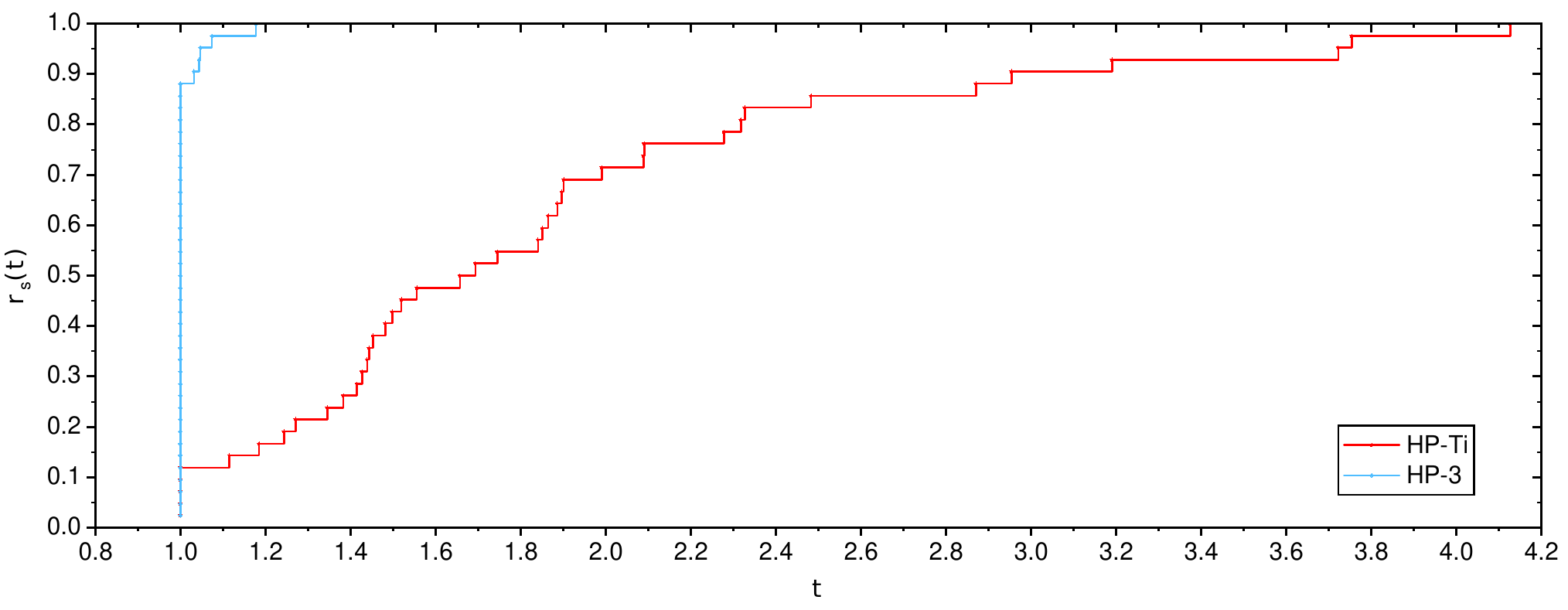}}
		\caption{Performance profiles on CPU time for MP-3 and MP-Ti formulations for the second data set.}
		\label{fig:History-profile-2}
	\end{center}
\end{figure}

Fig. \ref{fig:MP3iGap} shows the iGap for MP-3 and the other three state-of-the-art MILP formulations and all of the instances in comparison with 3P (using ratios). Fig. \ref{fig:MPTiiGap} shows the iGap for MP-Ti and the other three state-of-the-art MILP formulations and all of the instances in comparison with 3P (using ratios). Our TP-3 and TP-Ti models always provide the best gaps. In additional, we have the following tightness relationship among the models:
\begin{align}
	2P\precsim 3P\precsim 2P-3\precsim 3P-HD=3P-3\precsim HP-3\precsim TP-3\\
	2P\precsim 3P\precsim 2P-Ti\precsim3P-Ti\precsim HP-Ti\precsim TP-Ti
\end{align}
\begin{figure}[t]
	\centering
	\includegraphics[width=\textwidth,height=5cm]{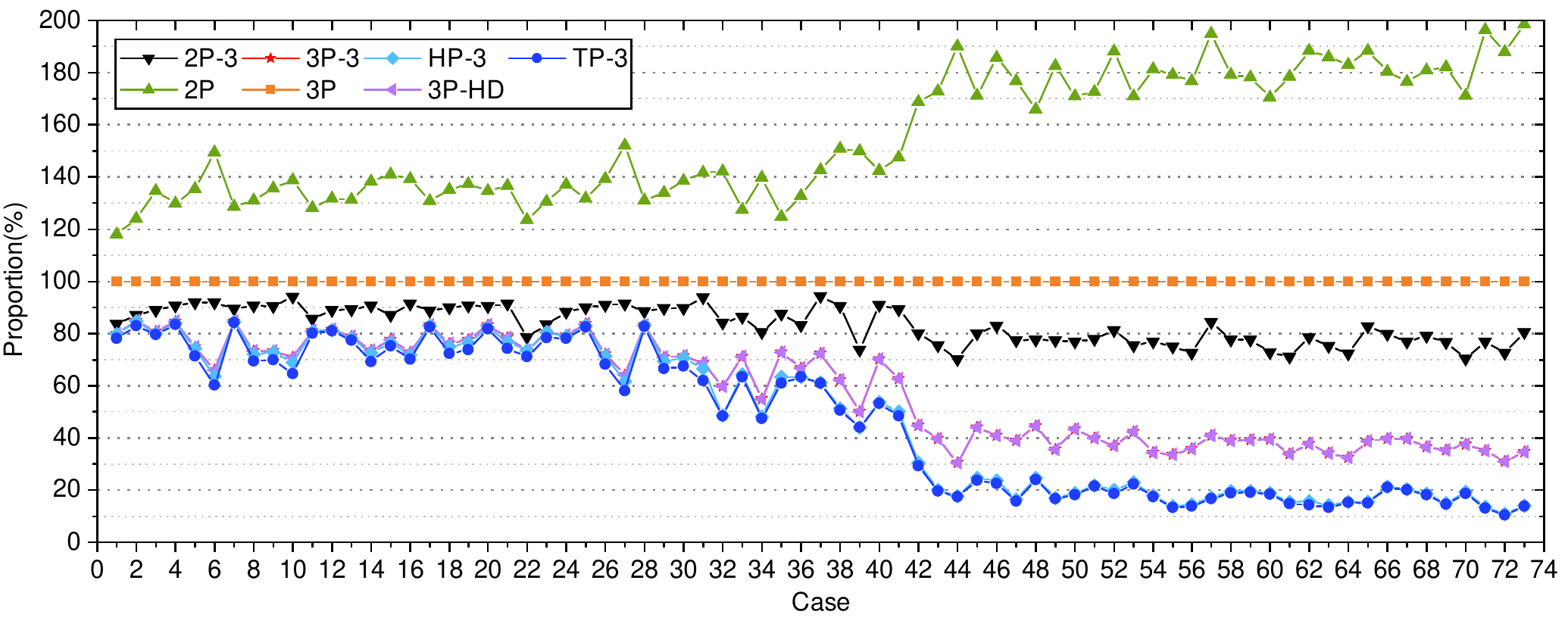}
	\caption{Comparison of MP-3 and the other three state-of-the-art MILP formulations in terms of iGap.}
	\label{fig:MP3iGap}
\end{figure}
\begin{figure}[t]
	\centering
	\includegraphics[width=\textwidth,height=5cm]{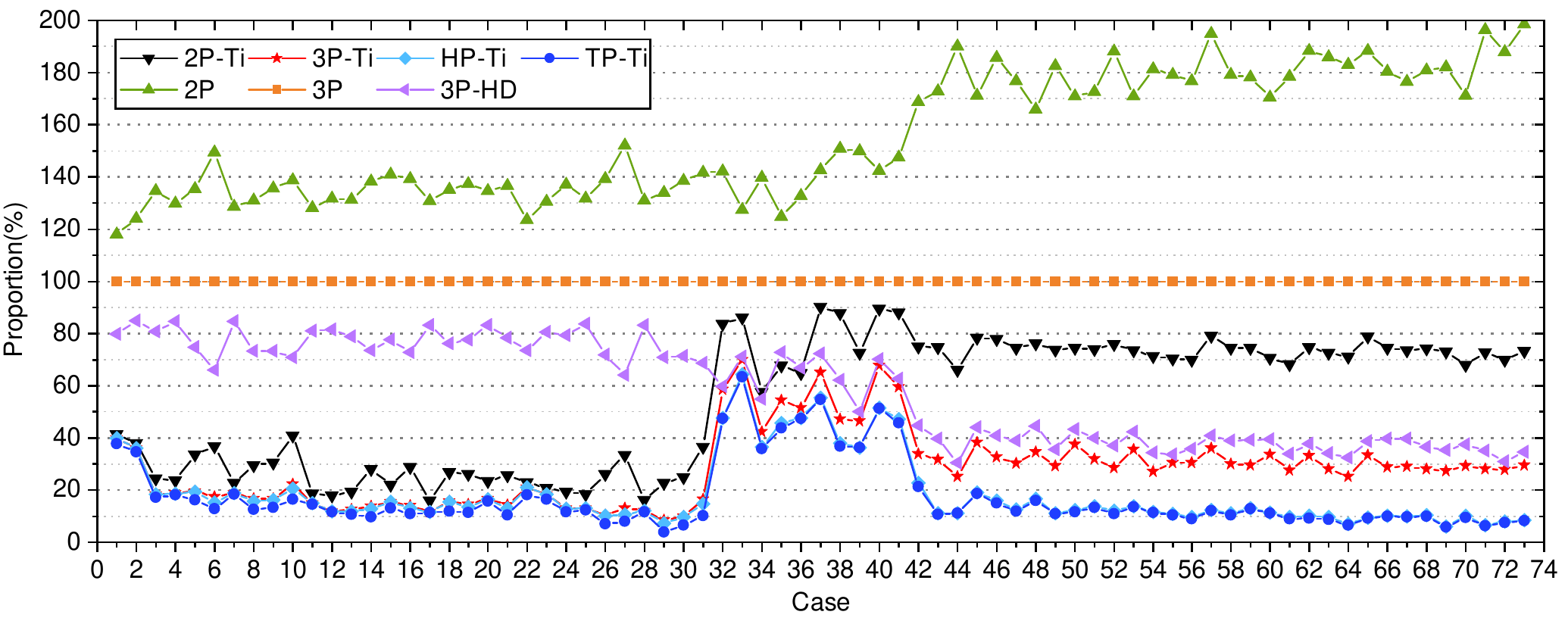}
	\caption{Comparison of MP-Ti and the other three state-of-the-art MILP formulations in terms of iGap.}
	\label{fig:MPTiiGap}
\end{figure}

It can be seen from Fig. \ref{fig:MP3PB} that 2P model and 3P model always give the lowest proportion of integers in binary variables for the relaxation solution. On the contrary, TP-3 model always provides the highest one. For data set 1, the differences among 2P-3, 3P-3, HP-3 and 3P-HD are not significant. For data set 2, HP-3 model performs better than the other three. Fig. \ref{fig:MPTiPB} shows the difference in the proportion of integers in binary variables among MP-Ti and the other three state-of-the-art MILP formulations. The differences among 2P-Ti, 3P-Ti, HP-Ti and 3P-HD are slight for the first data set.
\begin{figure}[t]
	\centering
	\includegraphics[width=\textwidth,height=5cm]{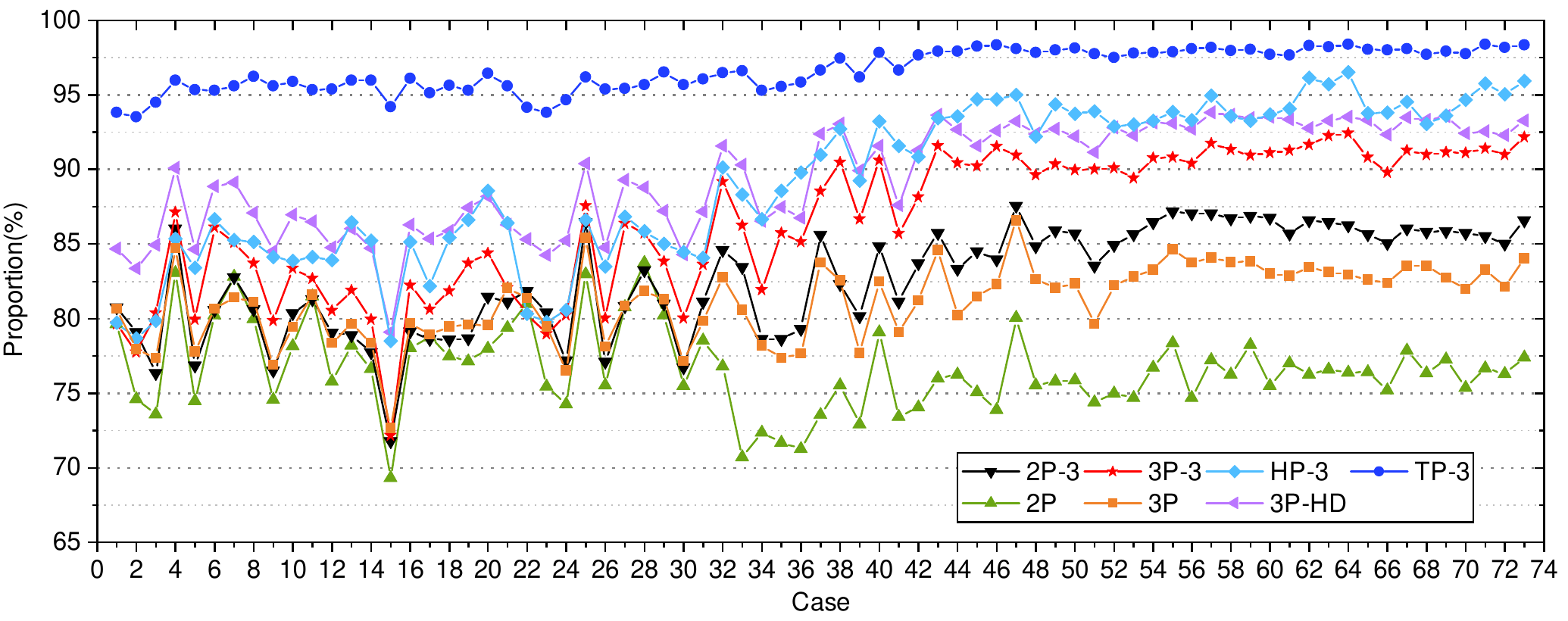}
	\caption{Comparison of MP-3 and the other three state-of-the-art MILP formulations in terms of the proportion of integers in binary variables.}
	\label{fig:MP3PB}
\end{figure}
\begin{figure}[t]
	\centering
	\includegraphics[width=\textwidth,height=5cm]{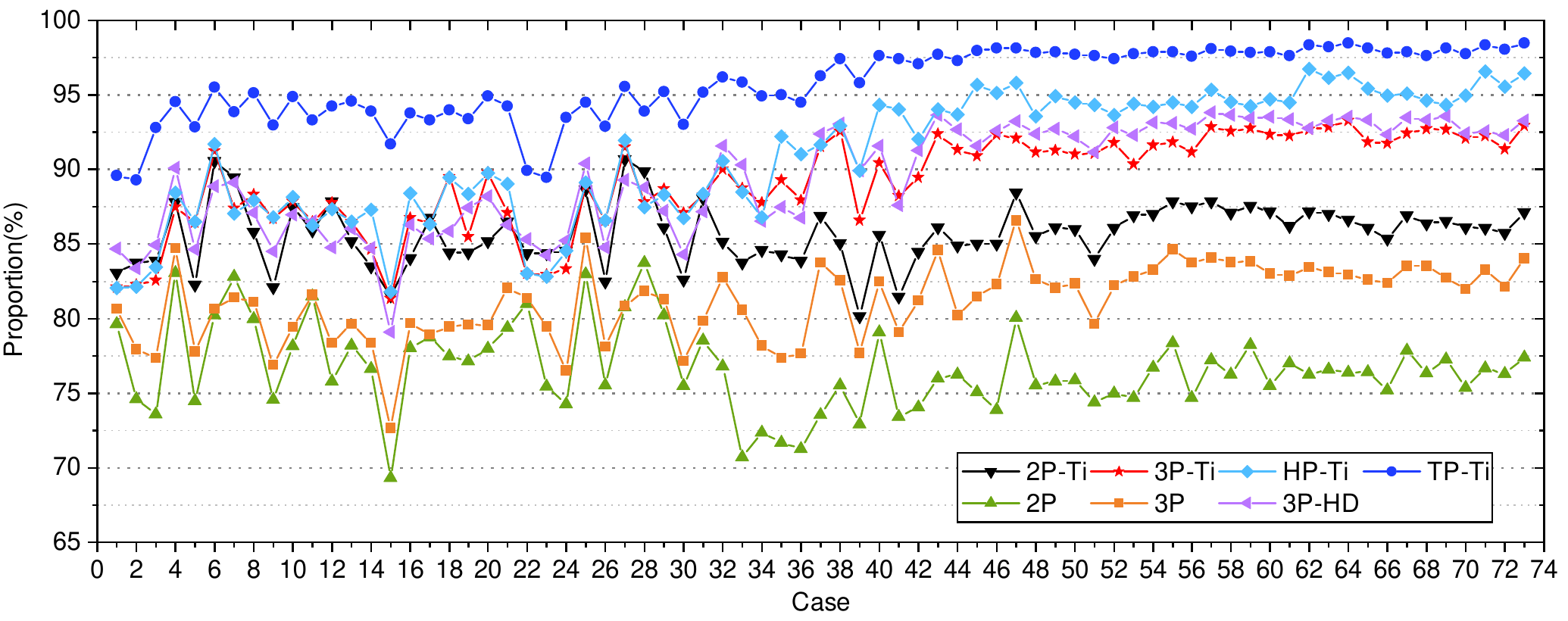}
	\caption{Comparison of MP-Ti and the other three state-of-the-art MILP formulations in terms of the proportion of integers in binary variables.}
	\label{fig:MPTiPB}
\end{figure}

Fig. \ref{fig:MP3Runtime} shows the rTime of MP-3 and the other three state-of-the-art MILP formulations for all test instances. Overall, 2P model performs worst, followed by 3P model , then TP-3 model. Other models perform equally well. In particular, TP-3 model performs well for instances 24-31 in the first data set with units ranging from 560 to 1080. Fig. \ref{fig:MPTiRuntime} shows the rTime of MP-Ti and the other three state-of-the-art MILP formulations for all test instances.
\begin{figure}[t]
	\centering
	\includegraphics[width=\textwidth,height=5cm]{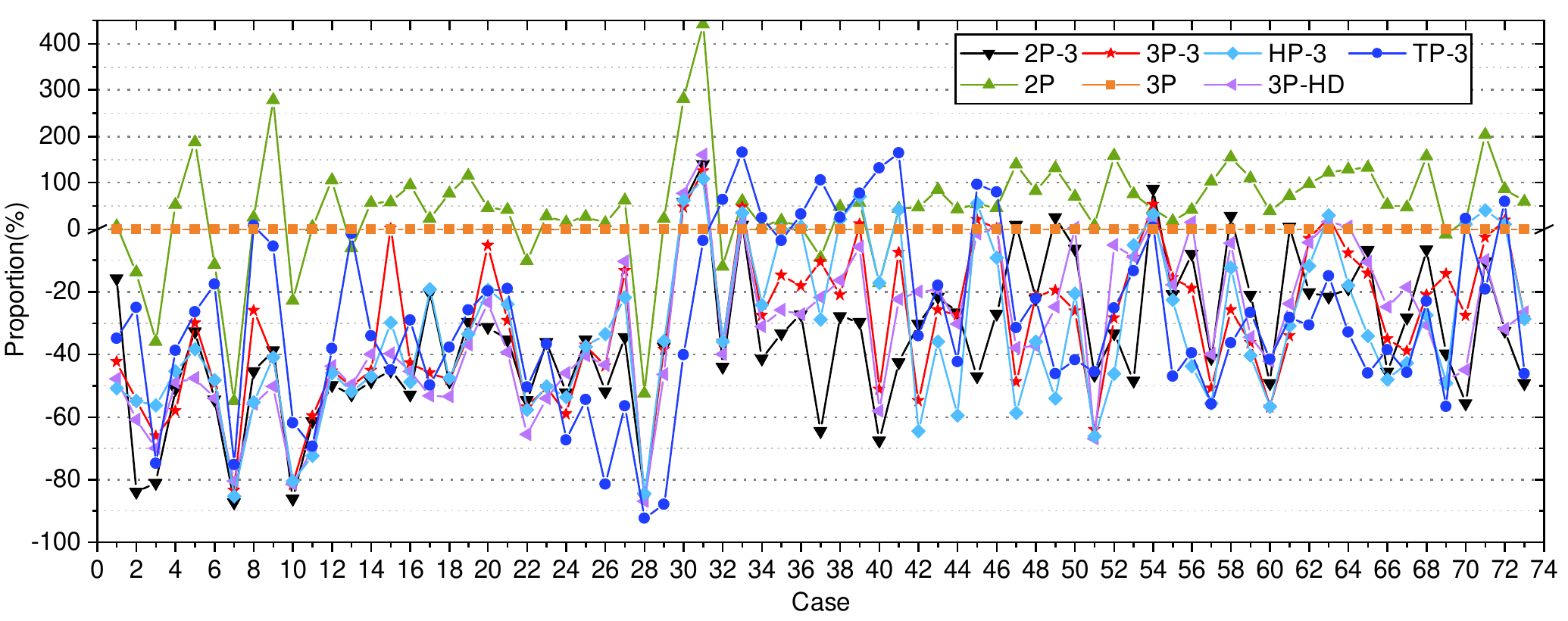}
	\caption{Comparison of MP-3 and the other three state-of-the-art MILP formulations in terms of rTime.}
	\label{fig:MP3Runtime}
\end{figure}
\begin{figure}[t]
	\centering
	\includegraphics[width=\textwidth,height=5cm]{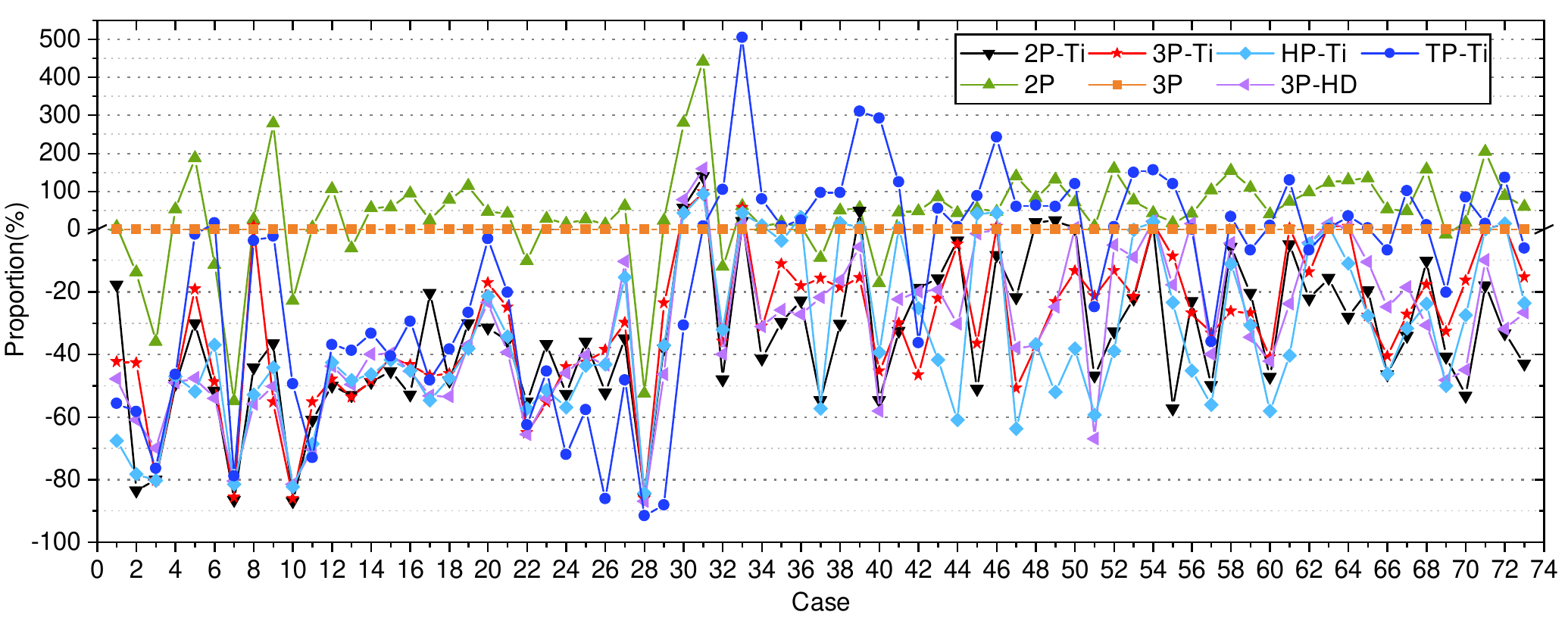}
	\caption{Comparison of MP-Ti and the other three state-of-the-art MILP formulations in terms of rTime.}
	\label{fig:MPTiRuntime}
\end{figure}

As shown in Fig. \ref{fig:performance-profile}, our MP-3, MP-Ti formulations as well as 3P-HD, have a relatively excellent performance compared with the other two state-of-the-art MILP formulations for the first data set. For the seconde data set, 3P model peforms better than TP-Ti. The poor convergences of TP-Ti for the second data set are mainly due to the large number of variables and constraints.  Fig. \ref{fig:performance-profile-clear} shows the differences among the well-performed models more precisely.
\begin{figure}[t]
	\begin{center}
		\subfigure[the first data set]{
			\includegraphics[width=0.48\textwidth,height=4cm]{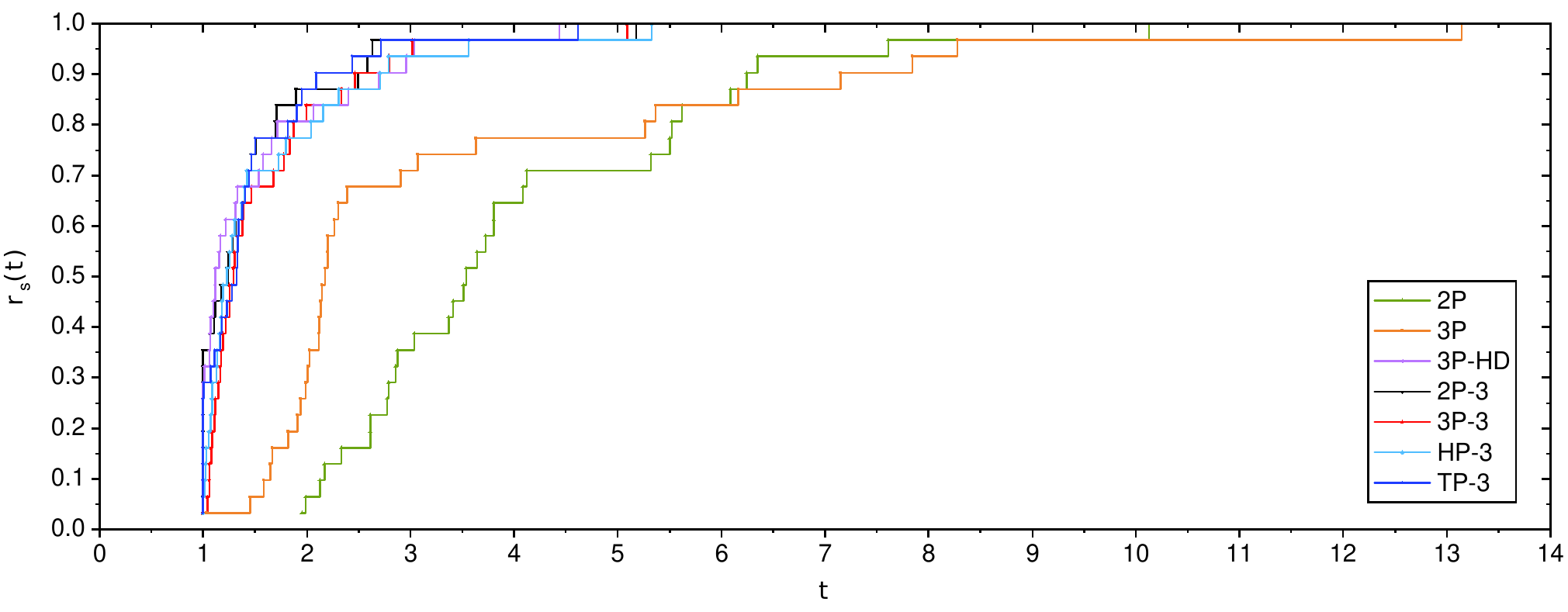}}\subfigure[the first data set]{
			\includegraphics[width=0.48\textwidth,height=4cm]{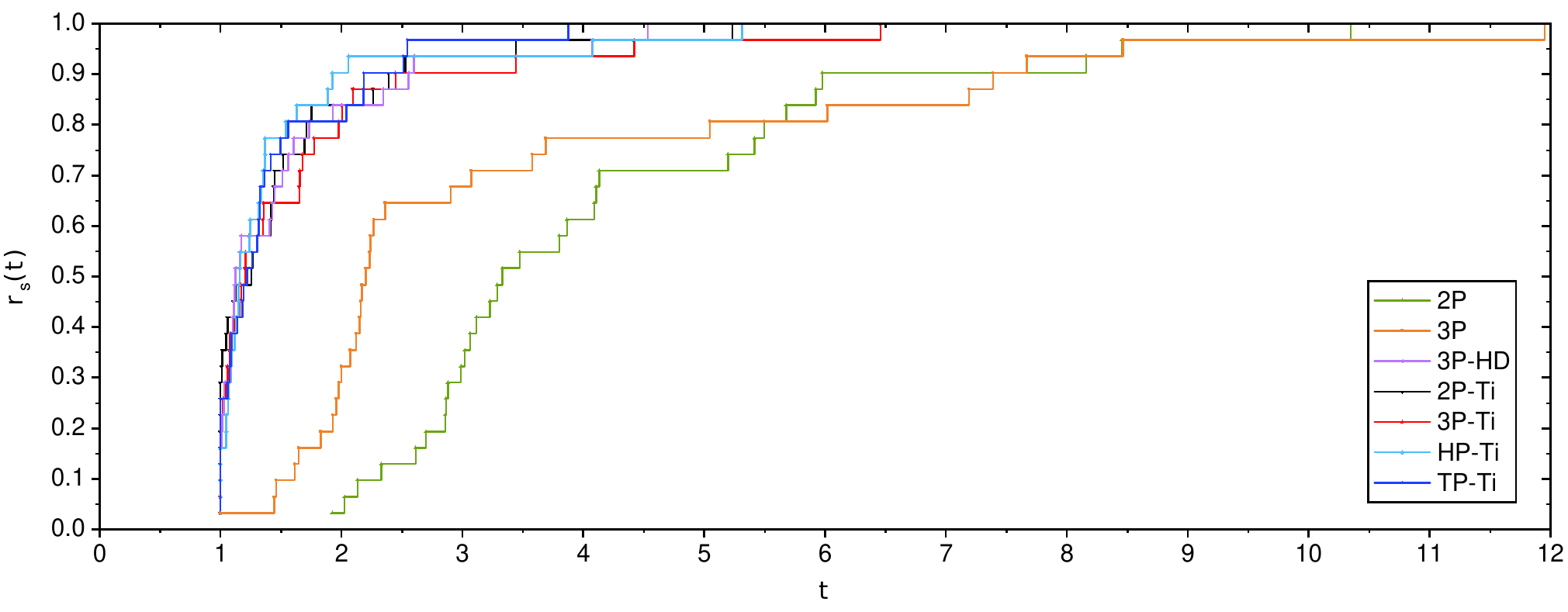}}
		\subfigure[the second data set]{
			\includegraphics[width=0.48\textwidth,height=4cm]{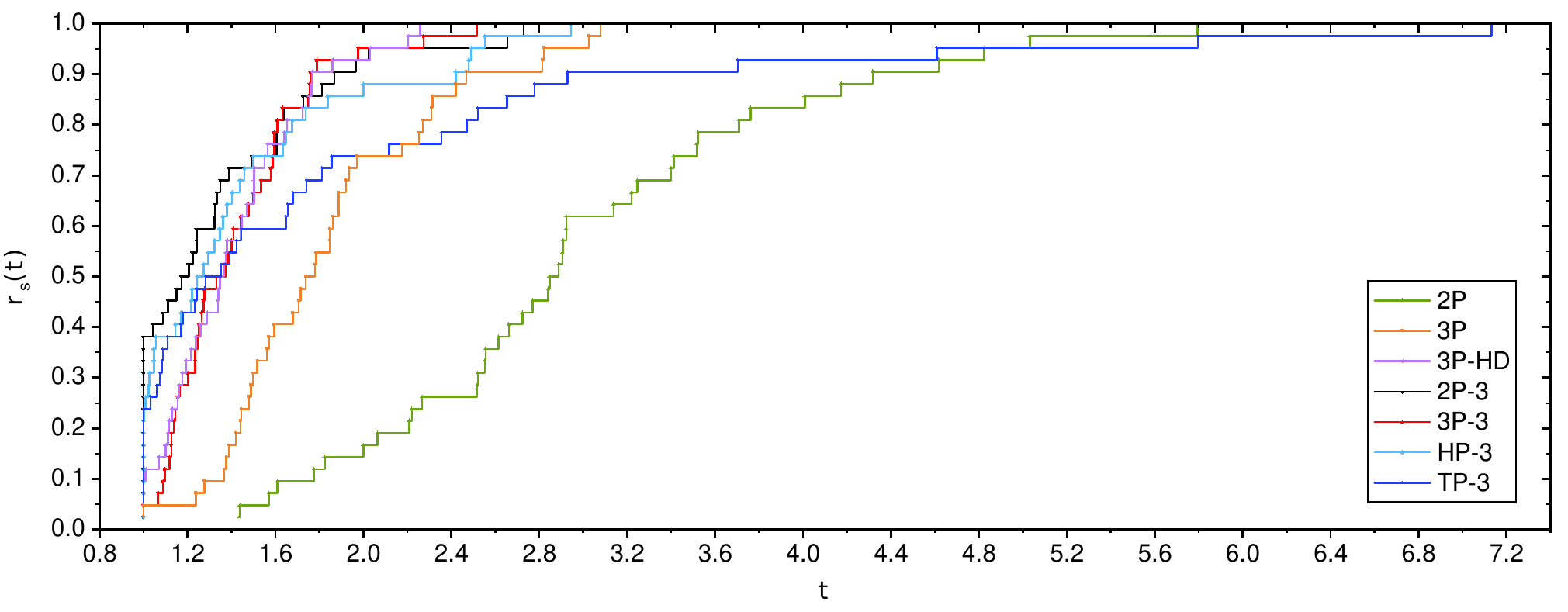}}\subfigure[the second data set]{
			\includegraphics[width=0.48\textwidth,height=4cm]{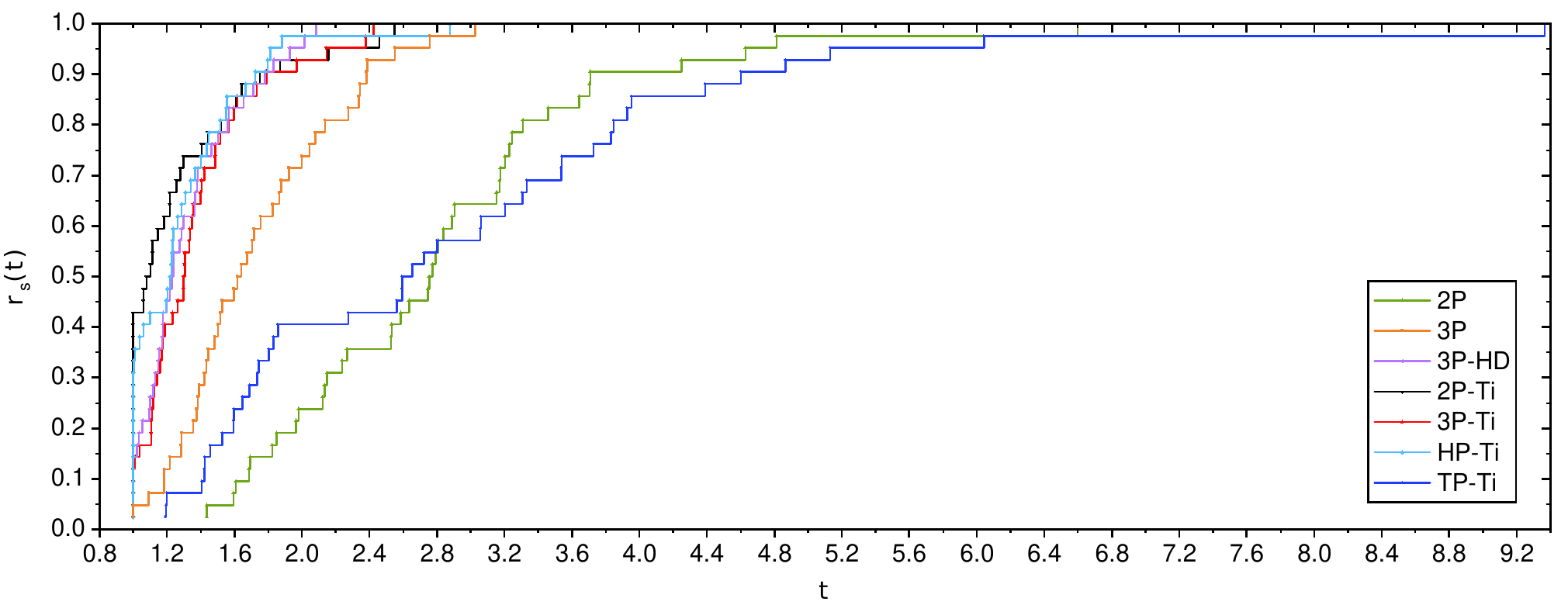}}
		\caption{Performance profiles on CPU time for MP-3, MP-Ti and the other three state-of-the-art MILP formulations.}
		\label{fig:performance-profile}
	\end{center}
\end{figure}
\begin{figure}[t]
	\begin{center}
		\subfigure[the first data set]{
			\includegraphics[width=0.48\textwidth,height=4cm]{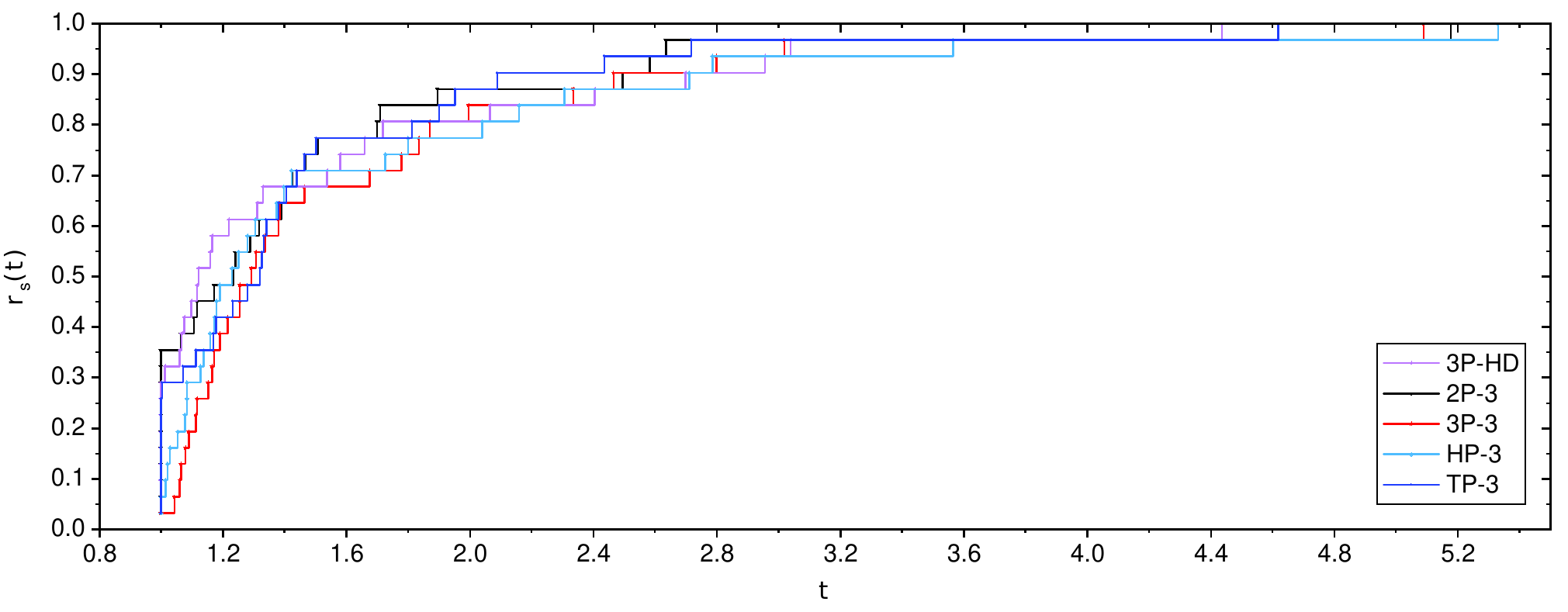}}\subfigure[the first data set]{
			\includegraphics[width=0.48\textwidth,height=4cm]{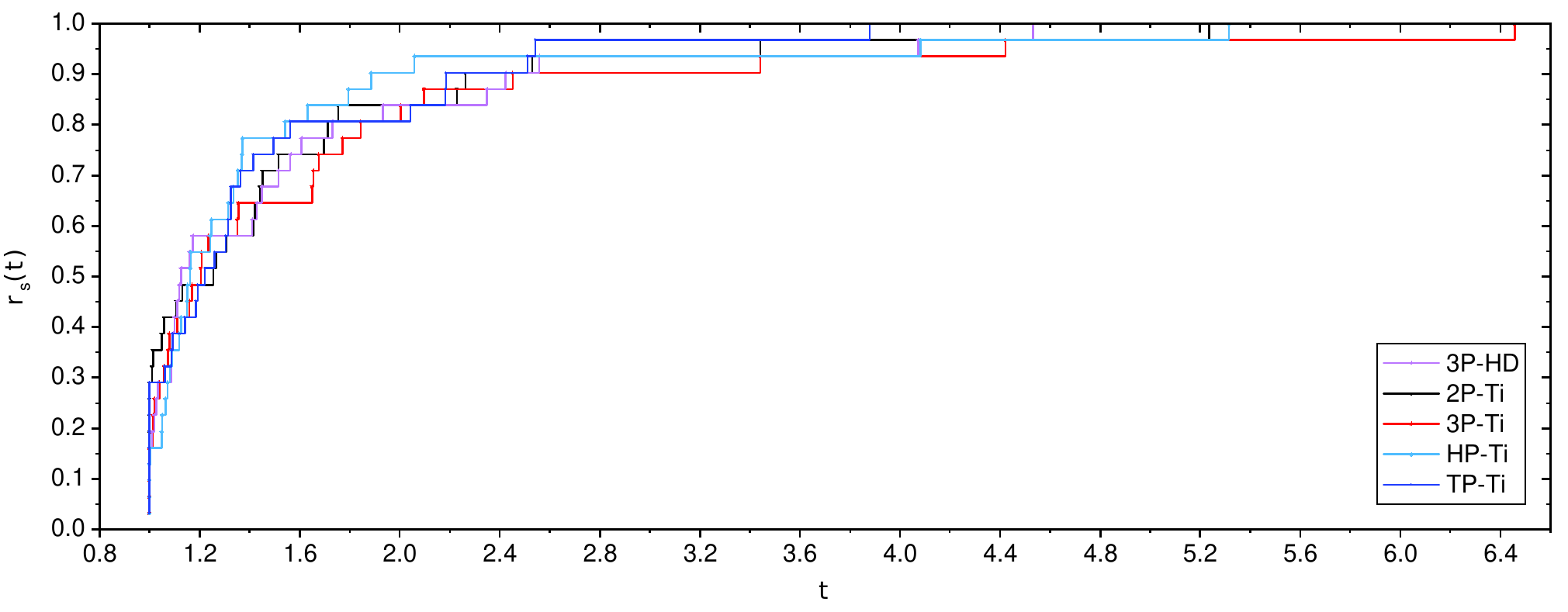}}
		\subfigure[the second data set]{
			\includegraphics[width=0.48\textwidth,height=4cm]{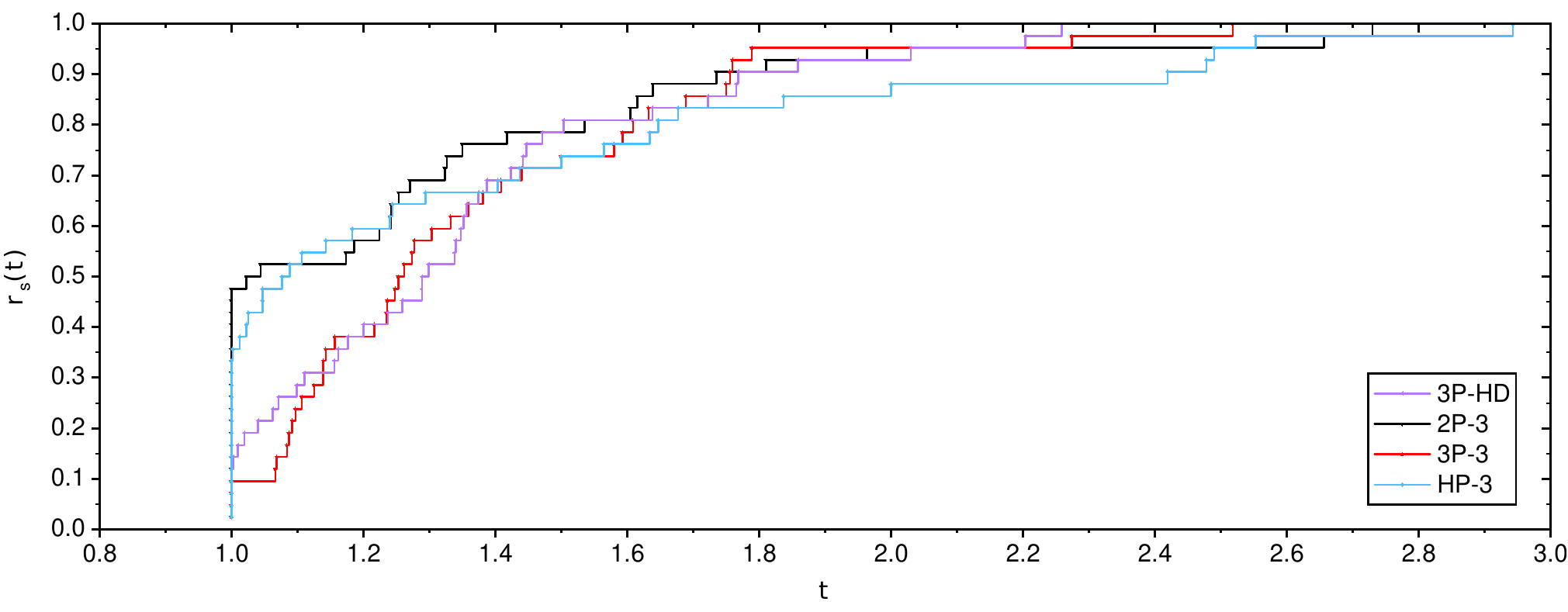}}\subfigure[the second data set]{
			\includegraphics[width=0.48\textwidth,height=4cm]{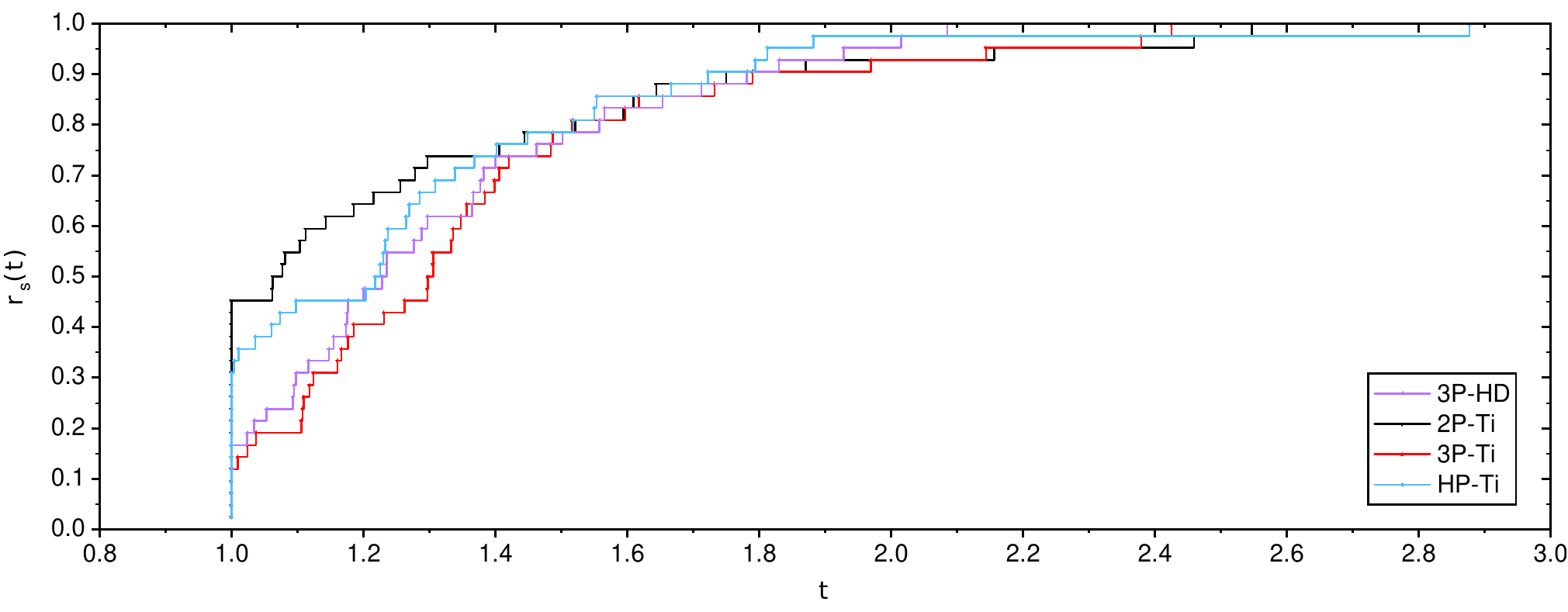}}
		\caption{Performance profiles on CPU time for MP-3, MP-Ti and 3P-HD formulations.}
		\label{fig:performance-profile-clear}
	\end{center}
\end{figure}
In Table \ref{tab:data1} and Table \ref{tab:data2}, we report more detailed results for the two data set respectively, where ``Units" represents the number of units in the system, ``time" represents the execution time for solving, ``$N_b$" represents the proportion of integers in binary variables for the solution of continuous relaxation problem. All of the test instances can be solved to optimality within the time limit of 3600 seconds.
\begin{sidewaystable}
	\begin{center}
		\begin{minipage}{\textheight}
			\caption{Comparison of MP-3 and the other three state-of-the-art MILP formulations in computational performance for the first data set.}\label{tab:data1}
			\resizebox{\textheight}{!}{
				\begin{tabular}{@{}lcccccccccccccccccccccc@{}}
					\toprule
					&&\multicolumn{3}{@{}c@{}}{2P}&\multicolumn{3}{@{}c@{}}{2P-3}&\multicolumn{3}{@{}c@{}}{3P}&\multicolumn{3}{@{}c@{}}{3P-3}&\multicolumn{3}{@{}c@{}}{3P-HD}&\multicolumn{3}{@{}c@{}}{HP-3}&\multicolumn{3}{@{}c@{}}{TP-3}\\
					\cmidrule(lr){3-5}\cmidrule(lr){6-8}\cmidrule(lr){9-11}\cmidrule(lr){12-14}\cmidrule(lr){15-17}\cmidrule(lr){18-20}\cmidrule(lr){21-23}
					Case&Units&$N_b$(\%)&iGap(\%)&time(s)&$N_b$(\%)&iGap(\%)&time(s)&$N_b$(\%)&iGap(\%)&time(s)&$N_b$(\%)&iGap(\%)&time(s)&$N_b$(\%)&iGap(\%)&time(s)&$N_b$(\%)&iGap(\%)&time(s)&$N_b$(\%)&iGap(\%)&time(s)\\
					\midrule
					1 & 28 & 79.66  & 0.67  & 3.39  & 80.75  & 0.47  & 2.67  & 80.65  & 0.57  & 3.17  & 79.66  & 0.45  & 1.83  & 84.66  & 0.45  & 1.66  & 79.71  & 0.45  & 1.56  & 93.83  & 0.44  & 2.06  \\
					2 & 35 & 74.60  & 0.63  & 6.98  & 79.09  & 0.44  & 1.31  & 77.94  & 0.51  & 8.09  & 77.74  & 0.43  & 3.67  & 83.37  & 0.43  & 3.16  & 78.69  & 0.43  & 3.66  & 93.52  & 0.42  & 6.06  \\
					3 & 44 & 73.58  & 0.64  & 5.00  & 76.36  & 0.42  & 1.48  & 77.37  & 0.48  & 7.81  & 80.38  & 0.38  & 2.64  & 84.93  & 0.38  & 2.34  & 79.86  & 0.38  & 3.42  & 94.51  & 0.38  & 1.97  \\
					4 & 45 & 83.06  & 0.57  & 4.67  & 86.05  & 0.40  & 1.50  & 84.72  & 0.44  & 3.06  & 87.13  & 0.37  & 1.28  & 90.08  & 0.37  & 1.56  & 85.37  & 0.37  & 1.67  & 95.98  & 0.36  & 1.87  \\
					5 & 49 & 74.46  & 0.73  & 8.69  & 76.87  & 0.50  & 2.03  & 77.78  & 0.54  & 3.01  & 79.92  & 0.40  & 2.11  & 84.64  & 0.40  & 1.58  & 83.39  & 0.40  & 1.86  & 95.33  & 0.39  & 2.22  \\
					6 & 50 & 80.22  & 0.84  & 2.44  & 80.58  & 0.52  & 1.25  & 80.64  & 0.56  & 2.75  & 86.14  & 0.37  & 1.39  & 88.88  & 0.37  & 1.27  & 86.67  & 0.36  & 1.42  & 95.32  & 0.34  & 2.27  \\
					7 & 51 & 82.82  & 0.56  & 4.53  & 82.76  & 0.39  & 1.28  & 81.40  & 0.44  & 10.04  & 85.12  & 0.37  & 1.66  & 89.12  & 0.37  & 1.97  & 85.27  & 0.37  & 1.48  & 95.60  & 0.37  & 2.50  \\
					8 & 51 & 79.98  & 0.76  & 3.17  & 80.64  & 0.53  & 1.37  & 81.13  & 0.58  & 2.52  & 83.69  & 0.43  & 1.86  & 87.08  & 0.43  & 1.11  & 85.16  & 0.42  & 1.12  & 96.23  & 0.40  & 2.70  \\
					9 & 52 & 74.57  & 0.72  & 10.70  & 76.52  & 0.48  & 1.73  & 76.90  & 0.53  & 2.83  & 79.87  & 0.39  & 1.67  & 84.51  & 0.39  & 1.41  & 84.13  & 0.39  & 1.67  & 95.58  & 0.37  & 2.67  \\
					10 & 54 & 78.16  & 0.92  & 7.59  & 80.35  & 0.63  & 1.37  & 79.45  & 0.67  & 9.83  & 83.38  & 0.47  & 1.80  & 86.98  & 0.47  & 1.83  & 83.87  & 0.46  & 1.92  & 95.90  & 0.43  & 3.73  \\
					11 & 132 & 81.52  & 0.55  & 37.08  & 81.28  & 0.37  & 13.70  & 81.61  & 0.43  & 35.40  & 82.68  & 0.35  & 14.28  & 86.49  & 0.35  & 10.47  & 84.13  & 0.35  & 9.75  & 95.36  & 0.34  & 10.84  \\
					12 & 156 & 75.78  & 0.57  & 50.63  & 79.03  & 0.38  & 12.39  & 78.37  & 0.43  & 24.67  & 80.55  & 0.35  & 13.48  & 84.74  & 0.35  & 13.87  & 83.91  & 0.35  & 13.32  & 95.37  & 0.35  & 15.25  \\
					13 & 156 & 78.21  & 0.62  & 24.21  & 78.86  & 0.42  & 12.18  & 79.66  & 0.47  & 25.79  & 81.89  & 0.37  & 12.90  & 86.03  & 0.37  & 13.00  & 86.44  & 0.37  & 12.43  & 95.99  & 0.37  & 25.43  \\
					14 & 160 & 76.63  & 0.70  & 41.94  & 77.75  & 0.46  & 13.79  & 78.37  & 0.50  & 26.81  & 79.94  & 0.37  & 14.67  & 84.74  & 0.37  & 16.07  & 85.22  & 0.37  & 14.18  & 95.99  & 0.35  & 17.65  \\
					15 & 165 & 69.30  & 0.64  & 54.03  & 71.77  & 0.40  & 18.79  & 72.66  & 0.45  & 34.24  & 72.13  & 0.35  & 34.49  & 79.09  & 0.35  & 20.62  & 78.51  & 0.35  & 24.04  & 94.22  & 0.34  & 18.85  \\
					16 & 167 & 78.03  & 0.72  & 52.02  & 79.14  & 0.47  & 12.62  & 79.67  & 0.51  & 26.74  & 82.23  & 0.37  & 15.32  & 86.28  & 0.37  & 14.61  & 85.15  & 0.37  & 13.67  & 96.11  & 0.36  & 18.95  \\
					17 & 172 & 78.75  & 0.55  & 40.01  & 78.64  & 0.38  & 26.01  & 78.91  & 0.42  & 32.68  & 80.63  & 0.35  & 17.62  & 85.37  & 0.35  & 15.31  & 82.19  & 0.35  & 26.43  & 95.14  & 0.35  & 16.39  \\
					18 & 182 & 77.47  & 0.69  & 63.09  & 78.58  & 0.46  & 18.34  & 79.49  & 0.51  & 35.65  & 81.86  & 0.39  & 18.53  & 85.88  & 0.39  & 16.59  & 85.45  & 0.38  & 18.70  & 95.63  & 0.37  & 22.23  \\
					19 & 182 & 77.15  & 0.63  & 64.03  & 78.66  & 0.42  & 20.96  & 79.59  & 0.46  & 29.77  & 83.69  & 0.36  & 19.59  & 87.45  & 0.36  & 18.78  & 86.63  & 0.35  & 19.79  & 95.31  & 0.34  & 22.10  \\
					20 & 183 & 77.99  & 0.59  & 38.32  & 81.45  & 0.40  & 18.00  & 79.54  & 0.44  & 26.21  & 84.40  & 0.37  & 24.85  & 88.21  & 0.37  & 20.09  & 88.56  & 0.36  & 21.10  & 96.43  & 0.36  & 21.03  \\
					21 & 187 & 79.39  & 0.64  & 39.33  & 81.11  & 0.43  & 17.95  & 82.04  & 0.47  & 27.74  & 82.04  & 0.36  & 19.62  & 86.28  & 0.36  & 16.86  & 86.40  & 0.36  & 21.09  & 95.61  & 0.35  & 22.46  \\
					22 & 560 & 80.98  & 0.53  & 583.02  & 81.84  & 0.34  & 293.73  & 81.38  & 0.43  & 648.07  & 80.34  & 0.32  & 279.95  & 85.30  & 0.32  & 222.98  & 80.31  & 0.32  & 273.94  & 94.14  & 0.31  & 320.83  \\
					23 & 700 & 75.47  & 0.52  & 827.01  & 80.41  & 0.33  & 412.00  & 79.46  & 0.40  & 645.93  & 78.95  & 0.32  & 319.83  & 84.25  & 0.32  & 296.57  & 79.77  & 0.32  & 321.24  & 93.82  & 0.31  & 409.06  \\
					24 & 880 & 74.28  & 0.61  & 938.12  & 77.20  & 0.39  & 392.27  & 76.52  & 0.45  & 820.82  & 80.25  & 0.35  & 335.37  & 85.22  & 0.35  & 443.24  & 80.59  & 0.35  & 379.96  & 94.69  & 0.35  & 267.25  \\
					25 & 900 & 82.97  & 0.54  & 757.57  & 86.44  & 0.37  & 388.36  & 85.40  & 0.41  & 599.64  & 87.57  & 0.34  & 377.10  & 90.40  & 0.34  & 357.37  & 86.62  & 0.34  & 374.68  & 96.17  & 0.34  & 272.59  \\
					26 & 980 & 75.55  & 0.68  & 990.17  & 77.10  & 0.44  & 419.98  & 78.11  & 0.49  & 872.87  & 80.04  & 0.35  & 490.98  & 84.78  & 0.35  & 494.16  & 83.51  & 0.35  & 579.74  & 95.40  & 0.33  & 162.63  \\
					27 & 1000 & 80.79  & 0.81  & 912.14  & 80.81  & 0.49  & 369.38  & 80.88  & 0.53  & 563.73  & 86.39  & 0.34  & 488.71  & 89.31  & 0.34  & 505.82  & 86.85  & 0.33  & 440.66  & 95.41  & 0.31  & 244.96  \\
					28 & 1020 & 83.73  & 0.53  & 1431.49  & 83.26  & 0.36  & 434.44  & 81.84  & 0.40  & 3013.43  & 85.74  & 0.34  & 428.87  & 88.79  & 0.34  & 393.72  & 85.85  & 0.34  & 467.47  & 95.69  & 0.33  & 229.24  \\
					29 & 1020 & 80.24  & 0.71  & 871.92  & 81.00  & 0.48  & 445.75  & 81.29  & 0.53  & 712.96  & 83.84  & 0.38  & 438.08  & 87.24  & 0.38  & 381.91  & 85.01  & 0.37  & 458.86  & 96.54  & 0.35  & 86.09  \\
					30 & 1040 & 75.48  & 0.68  & 1244.44  & 76.73  & 0.44  & 516.05  & 77.16  & 0.49  & 326.78  & 80.01  & 0.35  & 482.78  & 84.31  & 0.35  & 579.05  & 84.49  & 0.35  & 530.72  & 95.67  & 0.33  & 195.88  \\
					31 & 1080 & 78.52  & 0.88  & 1108.14  & 81.11  & 0.58  & 491.57  & 79.86  & 0.62  & 204.61  & 83.63  & 0.43  & 460.09  & 87.19  & 0.43  & 531.69  & 84.08  & 0.41  & 425.42  & 96.05  & 0.38  & 197.05  \\
					\botrule
				\end{tabular}
			}
		\end{minipage}
	\end{center}
\end{sidewaystable}
\begin{sidewaystable}
	\begin{center}
		\begin{minipage}{\textheight}
			\caption{Comparison of MP-3 and the other three state-of-the-art MILP formulations in computational performance for the second data set.}\label{tab:data2}
			\resizebox{\textheight}{!}{
				\begin{tabular}{@{}lcccccccccccccccccccccc@{}}
					\toprule
					&&\multicolumn{3}{@{}c@{}}{2P}&\multicolumn{3}{@{}c@{}}{2P-3}&\multicolumn{3}{@{}c@{}}{3P}&\multicolumn{3}{@{}c@{}}{3P-3}&\multicolumn{3}{@{}c@{}}{3P-HD}&\multicolumn{3}{@{}c@{}}{HP-3}&\multicolumn{3}{@{}c@{}}{TP-3}\\
					\cmidrule(lr){3-5}\cmidrule(lr){6-8}\cmidrule(lr){9-11}\cmidrule(lr){12-14}\cmidrule(lr){15-17}\cmidrule(lr){18-20}\cmidrule(lr){21-23}
					Case&Units&$N_b$(\%)&iGap(\%)&time(s)&$N_b$(\%)&iGap(\%)&time(s)&$N_b$(\%)&iGap(\%)&time(s)&$N_b$(\%)&iGap(\%)&time(s)&$N_b$(\%)&iGap(\%)&time(s)&$N_b$(\%)&iGap(\%)&time(s)&$N_b$(\%)&iGap(\%)&time(s)\\
					\midrule
					32 & 10 & 76.81  & 1.49  & 0.34  & 84.58  & 0.88  & 0.22  & 82.78  & 1.05  & 0.39  & 89.17  & 0.63  & 0.25  & 91.58  & 0.63  & 0.23  & 90.14  & 0.51  & 0.25  & 96.47  & 0.51  & 0.64  \\
					33 & 10 & 70.69  & 2.35  & 0.58  & 83.47  & 1.59  & 0.39  & 80.56  & 1.84  & 0.36  & 86.25  & 1.31  & 0.53  & 90.32  & 1.31  & 0.41  & 88.33  & 1.18  & 0.48  & 96.62  & 1.17  & 0.95  \\
					34 & 10 & 72.36  & 1.91  & 0.48  & 78.61  & 1.10  & 0.27  & 78.19  & 1.36  & 0.45  & 81.94  & 0.75  & 0.33  & 86.53  & 0.75  & 0.31  & 86.67  & 0.66  & 0.34  & 95.31  & 0.65  & 0.56  \\
					35 & 10 & 71.67  & 2.24  & 0.50  & 78.61  & 1.57  & 0.28  & 77.36  & 1.80  & 0.42  & 85.75  & 1.31  & 0.36  & 87.47  & 1.31  & 0.31  & 88.57  & 1.14  & 0.42  & 95.55  & 1.10  & 0.41  \\
					36 & 10 & 71.25  & 1.95  & 0.36  & 79.31  & 1.22  & 0.25  & 77.64  & 1.47  & 0.34  & 85.15  & 0.98  & 0.28  & 86.74  & 0.98  & 0.25  & 89.81  & 0.93  & 0.36  & 95.86  & 0.93  & 0.45  \\
					37 & 20 & 73.54  & 2.06  & 2.16  & 85.63  & 1.36  & 0.84  & 83.75  & 1.44  & 2.37  & 88.54  & 1.04  & 2.12  & 92.37  & 1.04  & 1.86  & 90.97  & 0.88  & 1.69  & 96.65  & 0.88  & 4.89  \\
					38 & 20 & 75.56  & 1.59  & 1.00  & 82.29  & 0.95  & 0.48  & 82.57  & 1.05  & 0.67  & 90.49  & 0.66  & 0.53  & 93.05  & 0.66  & 0.56  & 92.71  & 0.54  & 0.81  & 97.47  & 0.53  & 0.84  \\
					39 & 20 & 72.92  & 1.58  & 1.73  & 80.14  & 0.78  & 0.78  & 77.71  & 1.05  & 1.11  & 86.67  & 0.53  & 1.23  & 89.95  & 0.53  & 1.05  & 89.24  & 0.46  & 1.89  & 96.21  & 0.46  & 1.97  \\
					40 & 20 & 79.10  & 1.87  & 1.52  & 84.86  & 1.19  & 0.59  & 82.50  & 1.31  & 1.83  & 90.60  & 0.92  & 0.89  & 91.58  & 0.92  & 0.77  & 93.24  & 0.70  & 1.52  & 97.85  & 0.70  & 4.23  \\
					41 & 20 & 73.40  & 1.85  & 0.91  & 81.11  & 1.12  & 0.36  & 79.10  & 1.25  & 0.62  & 85.69  & 0.79  & 0.58  & 87.58  & 0.79  & 0.48  & 91.57  & 0.63  & 0.89  & 96.66  & 0.61  & 1.66  \\
					42 & 50 & 74.06  & 1.49  & 9.06  & 83.72  & 0.71  & 4.26  & 81.22  & 0.88  & 6.12  & 88.14  & 0.40  & 2.76  & 91.28  & 0.40  & 4.91  & 90.86  & 0.27  & 2.17  & 97.67  & 0.26  & 4.03  \\
					43 & 50 & 76.00  & 1.13  & 4.01  & 85.72  & 0.49  & 1.70  & 84.58  & 0.65  & 2.17  & 91.58  & 0.26  & 1.61  & 93.64  & 0.26  & 1.75  & 93.44  & 0.13  & 1.39  & 97.91  & 0.13  & 1.78  \\
					44 & 50 & 76.25  & 1.21  & 4.95  & 83.31  & 0.45  & 2.55  & 80.25  & 0.64  & 3.47  & 90.42  & 0.19  & 2.51  & 92.67  & 0.19  & 2.42  & 93.58  & 0.11  & 1.41  & 97.91  & 0.11  & 2.00  \\
					45 & 50 & 75.08  & 1.28  & 6.45  & 84.50  & 0.60  & 2.22  & 81.47  & 0.75  & 4.19  & 90.20  & 0.33  & 5.05  & 91.58  & 0.33  & 4.12  & 94.72  & 0.18  & 6.53  & 98.28  & 0.18  & 8.22  \\
					46 & 50 & 73.89  & 1.22  & 5.25  & 83.94  & 0.54  & 2.62  & 82.31  & 0.65  & 3.59  & 91.54  & 0.27  & 3.62  & 92.61  & 0.27  & 3.61  & 94.70  & 0.15  & 3.26  & 98.35  & 0.15  & 6.48  \\
					47 & 75 & 80.07  & 1.31  & 11.86  & 87.57  & 0.57  & 5.44  & 86.59  & 0.74  & 4.95  & 90.94  & 0.29  & 2.53  & 93.24  & 0.29  & 3.08  & 95.02  & 0.12  & 2.05  & 98.11  & 0.12  & 3.39  \\
					48 & 75 & 75.54  & 1.46  & 8.08  & 84.85  & 0.68  & 3.45  & 82.63  & 0.88  & 4.44  & 89.63  & 0.39  & 3.48  & 92.36  & 0.39  & 2.78  & 92.22  & 0.22  & 2.84  & 97.86  & 0.21  & 3.45  \\
					49 & 75 & 75.80  & 1.27  & 11.95  & 85.93  & 0.54  & 6.48  & 82.06  & 0.70  & 5.17  & 90.33  & 0.25  & 4.16  & 92.72  & 0.25  & 3.89  & 94.39  & 0.12  & 2.37  & 98.01  & 0.12  & 2.78  \\
					50 & 75 & 75.89  & 1.31  & 8.95  & 85.70  & 0.59  & 4.92  & 82.35  & 0.77  & 5.25  & 89.98  & 0.33  & 3.87  & 92.21  & 0.33  & 5.37  & 93.72  & 0.14  & 4.17  & 98.13  & 0.14  & 3.06  \\
					51 & 75 & 74.41  & 1.38  & 9.89  & 83.54  & 0.62  & 4.92  & 79.65  & 0.80  & 9.22  & 90.02  & 0.32  & 3.31  & 91.17  & 0.32  & 3.05  & 93.90  & 0.17  & 3.12  & 97.77  & 0.17  & 5.01  \\
					52 & 100 & 75.00  & 1.29  & 21.78  & 84.93  & 0.56  & 5.61  & 82.24  & 0.68  & 8.40  & 90.08  & 0.25  & 6.01  & 92.80  & 0.25  & 7.98  & 92.88  & 0.14  & 4.51  & 97.51  & 0.13  & 6.28  \\
					53 & 100 & 74.69  & 1.39  & 18.39  & 85.64  & 0.61  & 5.39  & 82.81  & 0.81  & 10.44  & 89.40  & 0.34  & 9.48  & 92.29  & 0.34  & 9.51  & 93.01  & 0.19  & 9.90  & 97.82  & 0.18  & 9.04  \\
					54 & 100 & 76.71  & 1.30  & 12.12  & 86.44  & 0.55  & 15.78  & 83.28  & 0.72  & 8.45  & 90.79  & 0.25  & 12.97  & 93.13  & 0.25  & 10.28  & 93.25  & 0.13  & 11.18  & 97.86  & 0.13  & 8.72  \\
					55 & 100 & 78.36  & 1.27  & 12.53  & 87.19  & 0.53  & 8.47  & 84.63  & 0.71  & 10.70  & 90.82  & 0.24  & 9.04  & 93.11  & 0.24  & 8.79  & 93.86  & 0.10  & 8.28  & 97.86  & 0.09  & 5.67  \\
					56 & 100 & 74.71  & 1.24  & 10.47  & 87.06  & 0.51  & 6.81  & 83.75  & 0.70  & 7.39  & 90.39  & 0.25  & 5.98  & 92.72  & 0.25  & 8.44  & 93.31  & 0.10  & 4.16  & 98.11  & 0.10  & 4.47  \\
					57 & 150 & 77.22  & 1.18  & 32.74  & 87.06  & 0.51  & 9.47  & 84.09  & 0.61  & 16.11  & 91.76  & 0.25  & 7.92  & 93.84  & 0.25  & 9.67  & 94.95  & 0.10  & 7.15  & 98.18  & 0.10  & 7.09  \\
					58 & 150 & 76.28  & 1.22  & 35.51  & 86.70  & 0.53  & 17.92  & 83.81  & 0.68  & 13.92  & 91.32  & 0.26  & 10.33  & 93.66  & 0.26  & 13.31  & 93.58  & 0.13  & 12.22  & 97.95  & 0.13  & 8.86  \\
					59 & 150 & 78.24  & 1.21  & 32.87  & 86.89  & 0.53  & 12.39  & 83.82  & 0.68  & 15.68  & 90.94  & 0.27  & 9.98  & 93.42  & 0.27  & 10.26  & 93.29  & 0.13  & 9.34  & 98.06  & 0.13  & 11.51  \\
					60 & 150 & 75.48  & 1.27  & 25.24  & 86.74  & 0.54  & 9.20  & 83.01  & 0.75  & 18.15  & 91.09  & 0.29  & 7.84  & 93.52  & 0.29  & 10.50  & 93.69  & 0.14  & 7.86  & 97.72  & 0.14  & 10.61  \\
					61 & 150 & 77.01  & 1.30  & 26.79  & 85.68  & 0.52  & 16.45  & 82.85  & 0.73  & 15.54  & 91.30  & 0.25  & 10.25  & 93.35  & 0.25  & 11.84  & 94.06  & 0.11  & 10.73  & 97.68  & 0.11  & 11.14  \\
					62 & 200 & 76.25  & 1.16  & 40.99  & 86.58  & 0.48  & 16.57  & 83.45  & 0.61  & 20.79  & 91.66  & 0.23  & 20.17  & 92.75  & 0.23  & 19.89  & 96.15  & 0.10  & 18.34  & 98.29  & 0.09  & 14.40  \\
					63 & 200 & 76.59  & 1.17  & 42.08  & 86.45  & 0.47  & 14.81  & 83.10  & 0.63  & 18.89  & 92.26  & 0.21  & 23.59  & 93.28  & 0.21  & 21.79  & 95.72  & 0.09  & 24.38  & 98.22  & 0.08  & 16.07  \\
					64 & 200 & 76.38  & 1.17  & 50.52  & 86.24  & 0.46  & 17.90  & 82.93  & 0.64  & 22.10  & 92.43  & 0.21  & 20.39  & 93.54  & 0.21  & 23.26  & 96.53  & 0.10  & 18.14  & 98.41  & 0.10  & 14.86  \\
					65 & 200 & 76.42  & 1.26  & 57.39  & 85.63  & 0.55  & 22.95  & 82.63  & 0.67  & 24.57  & 90.82  & 0.26  & 21.10  & 93.30  & 0.26  & 21.96  & 93.80  & 0.10  & 16.18  & 98.06  & 0.10  & 13.29  \\
					66 & 200 & 75.20  & 1.37  & 41.12  & 85.06  & 0.60  & 14.68  & 82.38  & 0.76  & 26.99  & 89.81  & 0.30  & 17.54  & 92.35  & 0.30  & 20.28  & 93.83  & 0.16  & 14.06  & 98.03  & 0.16  & 16.59  \\
					67 & 200 & 77.85  & 1.24  & 38.12  & 86.03  & 0.54  & 18.53  & 83.54  & 0.71  & 25.85  & 91.27  & 0.28  & 15.76  & 93.49  & 0.28  & 21.04  & 94.56  & 0.14  & 14.78  & 98.10  & 0.14  & 14.00  \\
					68 & 200 & 76.36  & 1.30  & 58.19  & 85.81  & 0.57  & 21.17  & 83.55  & 0.72  & 22.62  & 90.99  & 0.26  & 17.87  & 93.30  & 0.26  & 15.68  & 93.08  & 0.13  & 16.42  & 97.70  & 0.13  & 17.40  \\
					69 & 200 & 77.27  & 1.25  & 30.18  & 85.85  & 0.53  & 18.48  & 82.74  & 0.69  & 30.71  & 91.17  & 0.24  & 26.31  & 93.58  & 0.24  & 15.89  & 93.61  & 0.10  & 15.57  & 97.93  & 0.10  & 13.31  \\
					70 & 200 & 75.38  & 1.28  & 40.58  & 85.76  & 0.53  & 15.25  & 81.98  & 0.75  & 34.35  & 91.12  & 0.28  & 24.88  & 92.43  & 0.28  & 18.87  & 94.65  & 0.14  & 37.96  & 97.76  & 0.14  & 42.40  \\
					71 & 200 & 76.69  & 1.26  & 61.69  & 85.53  & 0.49  & 18.21  & 83.28  & 0.64  & 20.29  & 91.42  & 0.23  & 19.75  & 92.56  & 0.23  & 18.28  & 95.78  & 0.09  & 28.51  & 98.40  & 0.08  & 16.40  \\
					72 & 200 & 76.32  & 1.23  & 53.02  & 85.00  & 0.47  & 19.14  & 82.14  & 0.65  & 28.34  & 90.98  & 0.20  & 33.60  & 92.29  & 0.20  & 19.32  & 95.04  & 0.07  & 31.27  & 98.17  & 0.07  & 45.08  \\
					73 & 200 & 77.39  & 1.13  & 40.24  & 86.60  & 0.46  & 12.83  & 84.05  & 0.57  & 25.28  & 92.18  & 0.20  & 18.07  & 93.25  & 0.20  & 18.56  & 95.95  & 0.08  & 18.01  & 98.35  & 0.08  & 13.61  \\
					\botrule
				\end{tabular}
			}
		\end{minipage}
	\end{center}
\end{sidewaystable}

\section{Conclusion}\label{sec6}

In this paper, tighter upper bounds of generation limits and ramping constraints are constructed by introducing more binary variables to represent unit states. A multi-period locally ideal formulation based on sliding window is proposed. It is worth noting that different window sizes can be set for different units in the system. In general, larger sliding window will generate tighter formulation. However, the number of variables and constraints increase dramatically for large window, which will increase the computational burden. In practical application, we can choose appropriate window size for different units in the system, so as to obtain an efficient model that can find a high-quality solution in a relatively short time.

A method to identify the facets of constraint polytope is provided in this paper. For models with large window, only the facet constraints and the necessary physical constraints will be keep. This reduces memory requirements and improves the computational efficiency without reducing the tightness.

In order to improve the compactness of MP-1 models, we make a series of adjustments to the binary variables of MP-1, and finally get compact models MP-3 with as few binary variables as possible without reducing the tightness of the models. The experimental results show that MP-3 models have excellent performances in computational efficiency.

To further tighten the model, historical status is taken into consideration in upper bounds of generation limits and ramping constraints. This plays an important role in the tightness of models for the first data set. However, the addition of historical status does little to improve the tightness of models for the second data set, and the compactness is greatly reduced in the meantime.

Although our MP-3 fomulations are much more complex than the state-of-the-art MILP formulations, they have excellent performances in computational efficiency. Even if the system contains more than 1000 units, they still can get the optimal solution satisfying the accuracy requirement in 600 seconds.

On the other hand, studying the influence of sliding step size of sliding window on the computational efficiency of the model will be one of the future research directions. Although we have found tight multi-period models with high efficiency, we are still unable to represent convex hull of feasible solutions for the single-unit problem. In order to represent the convex hull, we need a large amount of (probably $O(2^M)$) inequalities without explicit physical significance, such as ramping constraints with ``$P_{i,t-1}-P_{i,t}+P_{i,t+1}$" as left hand side expression.

\vspace{1ex}
{\noindent{\bf{Acknowledgments} }
	The work of LF Yang and SF Chen were supported by the Natural Science Foundation of Guangxi (2020GXNSFAA297173, 2020GXNSFDA238017, 2021GXNSFBA075012), and Natural Science Foundation of China (51767003), and in part by the Thousands of Young and Middle-Aged Backbone Teachers Training Program for Guangxi Higher Education [Education Department of Guangxi (2017)]. The work of ZY Dong was partially supported by funding from the UNSW Digital Grid Futures Institute, UNSW, Sydney, under a cross disciplinary fund scheme. We would like to thank Wei Li of Guangxi University for his valuable discussions.}

\section*{Appendix 1: Some feasible points to the M-period polytope with partial constraints}\label{appendix1}

Theorem 3.6 of §$I.4.3$ in \cite{wolsey1988integer} provides a way to identify whether a valid inequality is a facet. This technique is also used in \cite{damci2016polyhedral}. First of all, we construct some feasible points to the M-period polytope with partial constraints in Appendix 1. We let $\varepsilon$ be a very small number greater than zero and let $[r_1,r_2]$ be one member of the set $A_m$. According to equality (\ref{eq:new-initial-status-1}), there is one and only one $[m,k]\in A_m$ for $m\in [0,\min(U,T-M+1)]$ such that $\tau^m_{m,k}=1$. The values of those variables whose values are not given are equal to zero.
\setcounter{equation}{0}
\renewcommand{\theequation}{A\arabic{equation}}
\begin{footnotesize}
\begin{align}
	&\tau^m_{h,k}=0,\forall [h,k]\in A_m;\widetilde{P}_r=0,\forall r\in [m,m+M-1]\label{p:1}\\
	&\tau^m_{m,m+M-1}=1;\widetilde{P}_r=1,\forall r\in[m,m+M-1]\label{p:2}\\
	&\tau^m_{r_1,r_2}=1;\widetilde{P}_r=0,\forall r\in[r_1,r_2]\label{p:3}\\
	&\tau^m_{r_1,r_2}=1;\widetilde{P}_{r_0}=\varepsilon;\widetilde{P}_r=0,\forall r\in[r_1,r_2],r\ne r_0\label{p:4}\\
	&\tau^m_{m,U}=1;\widetilde{P}_r=\max(0,\widetilde{P}_0-r\widetilde{P}_{down}),\forall r\in [m,U];\tau^m_{r_1,r_2}=1;\widetilde{P}_r=0,\forall r\in [r_1,r_2]\label{p:5}\\
	&\tau^m_{m,U}=1;\widetilde{P}_r=\max(0,\widetilde{P}_0-r\widetilde{P}_{down}),\forall r\in [m,U];\tau^m_{r_1,r_2}=1;\widetilde{P}_{r_0}=\varepsilon,r_1\le r_0\le r_2\label{p:6}\\
	&\tau^m_{m,r_2}=1;\widetilde{P}_{r_3}=\min[1,\widetilde{P}_{shut}+(r_2-r_3)\widetilde{P}_{down}];\notag\\
	&\widetilde{P}_r=\widetilde{P}_{r_3}-\max(0,r-r_3)\times\frac{\max(0,\widetilde{P}_{r_3}-\widetilde{P}_{shut})}{r_2-r_3},\forall r\in [m,r_2]\label{p:7}\\
	&\tau^m_{m,r_2}=1;\widetilde{P}_{r_3}=\min[1,\widetilde{P}_0+r_3\widetilde{P}_{up},\widetilde{P}_{shut}+(r_2-r_3)\widetilde{P}_{down}];\widetilde{P}_r=\widetilde{P}_{r_3}-\max(0,r_3-r)\notag\\
	&\times\frac{\widetilde{P}_{r_3}-\widetilde{P}_0}{r_3}-\max(0,r-r_3)\times\frac{\max(0,\widetilde{P}_{r_3}-\widetilde{P}_{shut})}{r_2-r_3},\forall r\in [m,r_2]\label{p:8}\\
	&\tau^m_{m,r_2}=1;\widetilde{P}_{r_3}=\min[1,\widetilde{P}_0+r_3\widetilde{P}_{up},\widetilde{P}_{shut}+(r_2-r_3)\widetilde{P}_{down}];\widetilde{P}_{r_0}=\lvert\widetilde{P}_{r_3}-\max(0,r_3-r_0)\notag\\
	&\times\frac{\widetilde{P}_{r_3}-\widetilde{P}_0}{r_3}-\max(0,r_0-r_3)\times\frac{\max(0,\widetilde{P}_{r_3}-\widetilde{P}_{shut})}{r_2-r_3}-\varepsilon\rvert;\widetilde{P}_r=\widetilde{P}_{r_3}-\max(0,r_3-r)\notag\\
	&\times\frac{\widetilde{P}_{r_3}-\widetilde{P}_0}{r_3}-\max(0,r-r_3)\times\frac{\max(0,\widetilde{P}_{r_3}-\widetilde{P}_{shut})}{r_2-r_3},\forall r\in [m,r_2],r\ne r_0\label{p:9}\\
	&\tau^m_{m,r_2}=1;\widetilde{P}_r=\max(0,\widetilde{P}_0-r\widetilde{P}_{down}),\forall r\in [m,r_1];\widetilde{P}_r=\max(0,\widetilde{P}_0-r_1\widetilde{P}_{down})-(r-r_1)\notag\\
	&\times\frac{\max[0,\max(0,\widetilde{P}_0-r_1\widetilde{P}_{down})-\widetilde{P}_{shut}]}{r_2-r_1},\forall r\in [r_1,r_2]\label{p:10}\\
	&\widetilde{P}_{r_0}=\lvert\max(0,\widetilde{P}_0-r_1\widetilde{P}_{down})-(r_0-r_1)\times\frac{\max[0,\max(0,\widetilde{P}_0-r_1\widetilde{P}_{down})-\widetilde{P}_{shut}]}{r_2-r_1}-\varepsilon\rvert;\notag\\
	&\tau^m_{m,r_2}=1;\widetilde{P}_r=\max(0,\widetilde{P}_0-r\widetilde{P}_{down}),\forall r\in[m,r_1],r\ne r_0;\widetilde{P}_r=\max(0,\widetilde{P}_0-r_1\widetilde{P}_{down})-\notag\\
	&(r-r_1)\times\frac{\max[0,\max(0,\widetilde{P}_0-r_1\widetilde{P}_{down})-\widetilde{P}_{shut}]}{r_2-r_1},\forall r\in [r_1,r_2],r\ne r_0\label{p:11}\\
	&\tau^m_{r_1,r_2}=1;\widetilde{P}_{r_3}=\min[1,\widetilde{P}_{start}+(r_3-r_1)\widetilde{P}_{up},\widetilde{P}_{shut}+(r_2-r_3)\widetilde{P}_{down}];\notag\\
	&\widetilde{P}_r=\widetilde{P}_{r_3}-\max(0,r_3-r)\times\frac{\max(0,\widetilde{P}_{r_3}-\widetilde{P}_{start})}{r_3-r_1}-\max(0,r-r_3)\times\frac{\max(0,\widetilde{P}_{r_3}-\widetilde{P}_{shut})}{r_2-r_3},\notag\\
	&\forall r\in [r_1,r_2]\label{p:12}\\
	&\tau^m_{r_1,r_2}=1;\widetilde{P}_{r_3}=\min[1,\widetilde{P}_{start}+(r_3-r_1)\widetilde{P}_{up},\widetilde{P}_{shut}+(r_2-r_3)\widetilde{P}_{down}];\widetilde{P}_{r_0}=\lvert\widetilde{P}_{r_3}-\notag\\
	&\max(0,r_3-r_0)\times\frac{\max(0,\widetilde{P}_{r_3}-\widetilde{P}_{start})}{r_3-r_1}-\max(0,r_0-r_3)\times\frac{\max(0,\widetilde{P}_{r_3}-\widetilde{P}_{shut})}{r_2-r_3}-\varepsilon\rvert;\notag\\
	&\widetilde{P}_r=\widetilde{P}_{r_3}-\max(0,r_3-r)\times\frac{\max(0,\widetilde{P}_{r_3}-\widetilde{P}_{start})}{r_3-r_1}-\max(0,r-r_3)\times\frac{\max(0,\widetilde{P}_{r_3}-\widetilde{P}_{shut})}{r_2-r_3},\notag\\
	&\forall r\in [r_1,r_2],r\ne r_0\label{p:13}\\
	&\tau^m_{m,U}=1;\widetilde{P}_r=\max(0,\widetilde{P}_0-r\widetilde{P}_{down}),\forall r\in[m,U];\tau^m_{r_1,r_2}=1;\notag\\
	&\widetilde{P}_{r_3}=\min[1,\widetilde{P}_{start}+(r_3-r_1)\widetilde{P}_{up},\widetilde{P}_{shut}+(r_2-r_3)\widetilde{P}_{down}];\widetilde{P}_r=\widetilde{P}_{r_3}-\max(0,r_3-r)\notag\\
	&\times\frac{\max(0,\widetilde{P}_{r_3}-\widetilde{P}_{start})}{r_3-r_1}-\max(0,r-r_3)\times\frac{\max(0,\widetilde{P}_{r_3}-\widetilde{P}_{shut})}{r_2-r_3},\forall r\in [r_1,r_2]\label{p:14}\\
	&\tau^m_{m,U}=1;\widetilde{P}_r=\max(0,\widetilde{P}_0-r\widetilde{P}_{down}),\forall r\in[m,U];\tau^m_{r_1,r_2}=1;\widetilde{P}_{r_3}=\min[1,\widetilde{P}_{start}+\notag\\
	&(r_3-r_1)\widetilde{P}_{up},\widetilde{P}_{shut}+(r_2-r_3)\widetilde{P}_{down}];\widetilde{P}_{r_0}=\lvert\widetilde{P}_{r_3}-\max(0,r_3-r_0)\times\frac{\max(0,\widetilde{P}_{r_3}-\widetilde{P}_{start})}{r_3-r_1}\notag\\
	&-\max(0,r_0-r_3)\times\frac{\max(0,\widetilde{P}_{r_3}-\widetilde{P}_{shut})}{r_2-r_3}-\varepsilon\rvert;\widetilde{P}_r=\widetilde{P}_{r_3}-\max(0,r_3-r)\notag\\
	&\times\frac{\max(0,\widetilde{P}_{r_3}-\widetilde{P}_{start})}{r_3-r_1}-\max(0,r-r_3)\times\frac{\max(0,\widetilde{P}_{r_3}-\widetilde{P}_{shut})}{r_2-r_3},\forall r\in [r_1,r_2],r\ne r_0\label{p:15}\\
	&\tau^m_{m,U}=1;\widetilde{P}_{r_3}=\min[1,\widetilde{P}_0+r_3\widetilde{P}_{up},\widetilde{P}_{shut}+(U-r_3)\widetilde{P}_{down}];\widetilde{P}_r=\widetilde{P}_{r_3}-\max(0,r_3-r)\notag\\
	&\times\frac{\widetilde{P}_{r_3}-\widetilde{P}_0}{r_3}-\max(0,r-r_3)\times\frac{\max(0,\widetilde{P}_{r_3}-\widetilde{P}_{shut})}{U-r_3},\forall r\in[m,U];\tau^m_{r_1,r_2}=1;\notag\\
	&\widetilde{P}_{r_4}=\min[1,\widetilde{P}_{start}+(r_4-r_1)\widetilde{P}_{up},\widetilde{P}_{shut}+(r_2-r_4)\widetilde{P}_{down}];\widetilde{P}_r=\widetilde{P}_{r_4}-\max(0,r_4-r)\notag\\
	&\times\frac{\max(0,\widetilde{P}_{r_4}-\widetilde{P}_{start})}{r_4-r_1}-\max(0,r-r_4)\times\frac{\max(0,\widetilde{P}_{r_4}-\widetilde{P}_{shut})}{r_2-r_4},\forall r\in [r_1,r_2]\label{p:16}\\
	&\tau^m_{m,U}=1;\widetilde{P}_{r_3}=\min[1,\widetilde{P}_0+r_3\widetilde{P}_{up},\widetilde{P}_{shut}+(U-r_3)\widetilde{P}_{down}];\notag\\
	&\widetilde{P}_r=\widetilde{P}_{r_3}-\max(0,r_3-r)\times\frac{\widetilde{P}_{r_3}-\widetilde{P}_0}{r_3}-\max(0,r-r_3)\times\frac{\max(0,\widetilde{P}_{r_3}-\widetilde{P}_{shut})}{U-r_3},\forall r\in[m,U];\notag\\
	&\tau^m_{r_1,r_2}=1;\widetilde{P}_r=0,\forall r\in [r_1,r_2]\label{p:17}\\
	&\tau^m_{r_1,m+M-1}=1;\widetilde{P}_{r_3}=\min[1,\widetilde{P}_{start}+(r_3-r_1)\widetilde{P}_{up}];\notag\\
	&\widetilde{P}_r=\widetilde{P}_{r_3}-\max(0,r_3-r)\times\frac{\max(0,\widetilde{P}_{r_3}-\widetilde{P}_{start})}{r_3-r_1},\forall r\in[r_1,m+M-1]\label{p:18}\\
	&\tau^m_{r_1,m+M-1}=1;\widetilde{P}_{r_3}=\min[1,\widetilde{P}_{start}+(r_3-r_1)\widetilde{P}_{up}];\notag\\
	&\widetilde{P}_{r_0}=\lvert\widetilde{P}_{r_3}-\max(0,r_3-r_0)\times\frac{\max(0,\widetilde{P}_{r_3}-\widetilde{P}_{start})}{r_3-r_1}-\varepsilon\rvert;\notag\\
	&\widetilde{P}_r=\widetilde{P}_{r_3}-\max(0,r_3-r)\times\frac{\max(0,\widetilde{P}_{r_3}-\widetilde{P}_{start})}{r_3-r_1},\forall r\in[r_1,m+M-1]\label{p:19}\\
	&\tau^m_{m,U}=1;\widetilde{P}_r=\max(0,\widetilde{P}_0-r\widetilde{P}_{down}),\forall r\in[m,U];\notag\\
	&\tau^m_{r_1,m+M-1}=1;\widetilde{P}_{r_3}=\min[1,\widetilde{P}_{start}+(r_3-r_1)\widetilde{P}_{up}];\notag\\
	&\widetilde{P}_r=\widetilde{P}_{r_3}-\max(0,r_3-r)\times\frac{\max(0,\widetilde{P}_{r_3}-\widetilde{P}_{start})}{r_3-r_1},\forall r\in[r_1,m+M-1]\label{p:20}\\
	&\tau^m_{m,U}=1;\widetilde{P}_r=\max(0,\widetilde{P}_0-r\widetilde{P}_{down}),\forall r\in[m,U];\widetilde{P}_{r_3}=\min[1,\widetilde{P}_{start}+(r_3-r_1)\widetilde{P}_{up}];\notag\\
	&\tau^m_{r_1,m+M-1}=1;\widetilde{P}_{r_0}=\lvert\widetilde{P}_{r_3}-\max(0,r_3-r_0)\times\frac{\max(0,\widetilde{P}_{r_3}-\widetilde{P}_{start})}{r_3-r_1}-\varepsilon\rvert;\notag\\
	&\widetilde{P}_r=\widetilde{P}_{r_3}-\max(0,r_3-r)\times\frac{\max(0,\widetilde{P}_{r_3}-\widetilde{P}_{start})}{r_3-r_1},\forall r\in[r_1,m+M-1],r\ne r_0\label{p:21}\\
	&\tau^m_{m,m+M-1}=1;\widetilde{P}_{r_3}=\min(1,\widetilde{P}_0+r_3\widetilde{P}_{up});\notag\\
	&\widetilde{P}_r=\widetilde{P}_{r_3}-\max(0,r_3-r)\times\frac{\widetilde{P}_{r_3}-\widetilde{P}_0}{r_3},\forall r\in[m,m+M-1]\label{p:22}\\
	&\tau^m_{m,m+M-1}=1;\widetilde{P}_{r_3}=\min(1,\widetilde{P}_0+r_3\widetilde{P}_{up});\widetilde{P}_{r_0}=\lvert\widetilde{P}_{r_3}-\max(0,r_3-r_0)\times\frac{\widetilde{P}_{r_3}-\widetilde{P}_0}{r_3}-\varepsilon\rvert;\notag\\
	&\widetilde{P}_r=\widetilde{P}_{r_3}-\max(0,r_3-r)\times\frac{\widetilde{P}_{r_3}-\widetilde{P}_0}{r_3},\forall r\in[m,m+M-1],r\ne r_0\label{p:23}\\
	&\tau^m_{m,m+M-1}=1;\widetilde{P}_r=\max[0,1-\max(0,r_3-r)\times\widetilde{P}_{up}],\forall r\in[m,m+M-1]\label{p:24}\\
	&\tau^m_{m,m+M-1}=1;\widetilde{P}_r=\max[0,1-\max(0,r_3-r)\times\widetilde{P}_{up}-\varepsilon],\forall r\in[m,m+M-1]\label{p:25}\\
	&\tau^m_{m,m+M-1}=1;\widetilde{P}_r=\max[0,1-\max(0,r-r_3)\times\widetilde{P}_{down}],\forall r\in[m,m+M-1]\label{p:26}\\
	&\tau^m_{m,m+M-1}=1;\widetilde{P}_r=\max[0,1-\max(0,r-r_3)\times\widetilde{P}_{down}-\varepsilon],\forall r\in[m,m+M-1]\label{p:27}\\
	&\tau^m_{r_1,r_2}=1;\widetilde{P}_{r_3}=\min[1,\widetilde{P}_{start}+(r_3-r_1)\widetilde{P}_{up},\widetilde{P}_{shut}+(r_2-r_3)\widetilde{P}_{down}];\notag\\
	&\widetilde{P}_r=\widetilde{P}_{r_3}-\max(0,r_3-r)\times\frac{\widetilde{P}_{ramp}}{a}-\max(0,r-r_3)\times\frac{\max(0,\widetilde{P}_{r_3}-\widetilde{P}_{shut})}{r_2-r_3},\forall r\in [r_4,r_2];\notag\\
	&\widetilde{P}_r=\widetilde{P}_{r_4}-(r_4-r)\times\frac{\max(0,\widetilde{P}_{r_4}-\widetilde{P}_{start})}{r_4-r_1},r\in [r_1,r_4-1]\label{p:28}\\
	&\tau^m_{m,U}=1;\widetilde{P}_r=\max(0,\widetilde{P}_0-r\widetilde{P}_{down}),\forall r\in[m,U];\notag\\
	&\tau^m_{r_1,r_2}=1;\widetilde{P}_{r_3}=\min[1,\widetilde{P}_{start}+(r_3-r_1)\widetilde{P}_{up},\widetilde{P}_{shut}+(r_2-r_3)\widetilde{P}_{down}];\notag\\
	&\widetilde{P}_r=\widetilde{P}_{r_3}-\max(0,r_3-r)\times\frac{\widetilde{P}_{ramp}}{a}-\max(0,r-r_3)\times\frac{\max(0,\widetilde{P}_{r_3}-\widetilde{P}_{shut})}{r_2-r_3},\forall r\in [r_4,r_2];\notag\\
	&\widetilde{P}_r=\widetilde{P}_{r_4}-(r_4-r)\times\frac{\max(0,\widetilde{P}_{r_4}-\widetilde{P}_{start})}{r_4-r_1},\forall r\in [r_1,r_4-1]\label{p:29}\\
	&\tau^m_{r_1,r_2}=1;\widetilde{P}_{r_3}=\min[1,\widetilde{P}_{start}+(r_3-r_1)\widetilde{P}_{up},\widetilde{P}_{shut}+(r_2-r_3)\widetilde{P}_{down}];\notag\\
	&\widetilde{P}_r=\widetilde{P}_{r_3}-\max(0,r-r_3)\times\frac{\widetilde{P}_{ramp}}{a}-\max(0,r_3-r)\times\frac{\max(0,\widetilde{P}_{r_3}-\widetilde{P}_{start})}{r_3-r_1},\forall r\in [r_1,r_4];\notag\\
	&\widetilde{P}_r=\widetilde{P}_{r_4}-(r-r_4)\times\frac{\max(0,\widetilde{P}_{r_4}-\widetilde{P}_{shut})}{r_2-r_4},\forall r\in [r_4+1,r_2]\label{p:30}\\
	&\tau^m_{m,U}=1;\widetilde{P}_r=\max(0,\widetilde{P}_0-r\widetilde{P}_{down}),\forall r\in[m,U];\notag\\
	&\tau^m_{r_1,r_2}=1;\widetilde{P}_{r_3}=\min[1,\widetilde{P}_{start}+(r_3-r_1)\widetilde{P}_{up},\widetilde{P}_{shut}+(r_2-r_3)\widetilde{P}_{down}];\notag\\
	&\widetilde{P}_r=\widetilde{P}_{r_3}-\max(0,r-r_3)\times\frac{\widetilde{P}_{ramp}}{a}-\max(0,r_3-r)\times\frac{\max(0,\widetilde{P}_{r_3}-\widetilde{P}_{start})}{r_3-r_1},\forall r\in [r_1,r_4];\notag\\
	&\widetilde{P}_r=\widetilde{P}_{r_4}-\max(0,r-r_4)\times\frac{\max(0,\widetilde{P}_{r_4}-\widetilde{P}_{shut})}{r_2-r_4},\forall r\in [r_4+1,r_2]\label{p:31}\\
	&\tau^m_{r_1,m+M-1}=1;\widetilde{P}_{r_3}=\min[1,\widetilde{P}_{start}+(r_3-r_1)\widetilde{P}_{up}];\notag\\
	&\widetilde{P}_r=\max[0,\widetilde{P}_{r_3}-\max(0,r_3-r)\times\frac{\widetilde{P}_{ramp}}{a}],\forall r\in [r_1,m+M-1]\label{p:32}\\
	&\tau^m_{r_1,m+M-1}=1;\widetilde{P}_{r_3}=\min[1,\widetilde{P}_{start}+(r_3-r_1)\widetilde{P}_{up}]-\varepsilon;\notag\\
	&\widetilde{P}_r=\max[0,\widetilde{P}_{r_3}-\max(0,r_3-r)\times\frac{\widetilde{P}_{ramp}}{a}],\forall r\in [r_1,m+M-1]\label{p:33}\\
	&\tau^m_{m,U}=1;\widetilde{P}_r=\max(0,\widetilde{P}_0-r\widetilde{P}_{down}),\forall r\in[m,U];\notag\\
	&\tau^m_{r_1,m+M-1}=1;\widetilde{P}_{r_3}=\min[1,\widetilde{P}_{start}+(r_3-r_1)\widetilde{P}_{up}];\notag\\
	&\widetilde{P}_r=\max[0,\widetilde{P}_{r_3}-\max(0,r_3-r)\times\frac{\widetilde{P}_{ramp}}{a}],\forall r\in [r_1,m+M-1]\label{p:34}\\
	&\tau^m_{m,U}=1;\widetilde{P}_r=\max(0,\widetilde{P}_0-r\widetilde{P}_{down}),\forall r\in[m,U];\notag\\
	&\tau^m_{r_1,m+M-1}=1;\widetilde{P}_{r_3}=\min[1,\widetilde{P}_{start}+(r_3-r_1)\widetilde{P}_{up}]-\varepsilon;\notag\\
	&\widetilde{P}_r=\max[0,\widetilde{P}_{r_3}-\max(0,r_3-r)\times\frac{\widetilde{P}_{ramp}}{a}],\forall r\in [r_1,m+M-1]\label{p:35}\\
	&\tau^m_{r_1,m+M-1}=1;\widetilde{P}_{r_3}=\min[1,\widetilde{P}_{start}+(r_3-r_1)\widetilde{P}_{up}];\notag\\
	&\widetilde{P}_r=\max[0,\widetilde{P}_{r_3}-\max(0,r-r_3)\times\frac{\widetilde{P}_{ramp}}{a}],\forall r\in [r_3+1,m+M-1];\notag\\
	&\widetilde{P}_r=\widetilde{P}_{r_3}-(r_3-r)\times\frac{\max(0,\widetilde{P}_{r_3}-\widetilde{P}_{start})}{r_3-r_1},\forall r\in[r_1,r_3-1]\label{p:36}\\
	&\tau^m_{r_1,m+M-1}=1;\widetilde{P}_{r_3}=\min[1,\widetilde{P}_{start}+(r_3-r_1)\widetilde{P}_{up}]-\varepsilon;\notag\\
	&\widetilde{P}_r=\max[0,\widetilde{P}_{r_3}-\max(0,r-r_3)\times\frac{\widetilde{P}_{ramp}}{a}],\forall r\in [r_3+1,m+M-1];\notag\\
	&\widetilde{P}_r=\widetilde{P}_{r_3}-(r_3-r)\times\frac{\max(0,\widetilde{P}_{r_3}-\widetilde{P}_{start})}{r_3-r_1},\forall r\in[r_1,r_3-1]\label{p:37}\\
	&\tau^m_{m,U}=1;\widetilde{P}_r=\max(0,\widetilde{P}_0-r\widetilde{P}_{down}),\forall r\in[m,U];\notag\\
	&\tau^m_{r_1,m+M-1}=1;\widetilde{P}_{r_3}=\min[1,\widetilde{P}_{start}+(r_3-r_1)\widetilde{P}_{up}];\notag\\
	&\widetilde{P}_r=\max[0,\widetilde{P}_{r_3}-\max(0,r-r_3)\times\frac{\widetilde{P}_{ramp}}{a}],\forall r\in [r_3+1,m+M-1];\notag\\
	&\widetilde{P}_r=\widetilde{P}_{r_3}-(r_3-r)\times\frac{\max(0,\widetilde{P}_{r_3}-\widetilde{P}_{start})}{r_3-r_1},\forall r\in[r_1,r_3-1]\label{p:38}\\
	&\tau^m_{m,U}=1;\widetilde{P}_r=\max(0,\widetilde{P}_0-r\widetilde{P}_{down}),\forall r\in[m,U];\notag\\
	&\tau^m_{r_1,m+M-1}=1;\widetilde{P}_{r_3}=\min[1,\widetilde{P}_{start}+(r_3-r_1)\widetilde{P}_{up}]-\varepsilon;\notag\\
	&\widetilde{P}_r=\max[0,\widetilde{P}_{r_3}-\max(0,r-r_3)\times\frac{\widetilde{P}_{ramp}}{a}],\forall r\in [r_3+1,m+M-1];\notag\\
	&\widetilde{P}_r=\widetilde{P}_{r_3}-(r_3-r)\times\frac{\max(0,\widetilde{P}_{r_3}-\widetilde{P}_{start})}{r_3-r_1},\forall r\in[r_1,r_3-1]\label{p:39}\\
	&\tau^m_{m,r_2}=1;\widetilde{P}_{r_3}=\min[1,\widetilde{P}_{shut}+(r_2-r_3)\widetilde{P}_{down}];\notag\\
	&\widetilde{P}_r=\widetilde{P}_{r_3}-\max(0,r-r_3)\times\frac{\widetilde{P}_{ramp}}{a},\forall r\in[m,r_4];\notag\\
	&\widetilde{P}_r=\widetilde{P}_{r_4}-(r-r_4)\times\frac{\max(0,\widetilde{P}_{r_4}-\widetilde{P}_{shut})}{r_2-r_4},\forall r\in[r_4+1,r_2]\label{p:40}\\
	&\tau^m_{m,r_2}=1;\widetilde{P}_{r_3}=\min[1,\widetilde{P}_0+r_3\widetilde{P}_{up},\widetilde{P}_{shut}+(r_2-r_3)\widetilde{P}_{down}];\notag\\
	&\widetilde{P}_r=\widetilde{P}_{r_3}-\max(0,r_3-r)\times\frac{\widetilde{P}_{ramp}}{a}-\max(0,r-r_3)\times\frac{\max(0,\widetilde{P}_{r_3}-\widetilde{P}_{shut})}{r_2-r_3},\forall r\in[r_4,r_2];\notag\\
	&\widetilde{P}_r=\widetilde{P}_0-r\times\frac{\widetilde{P}_0-\widetilde{P}_{r_4}}{r_4},\forall r\in[m,r_4-1]\label{p:41}\\
	&\tau^m_{m,r_2}=1;\widetilde{P}_{r_3}=\min[1,\widetilde{P}_0+r_3\widetilde{P}_{up},\widetilde{P}_{shut}+(r_2-r_3)\widetilde{P}_{down}];\widetilde{P}_{r_0}=\lvert\widetilde{P}_{r_3}-\max(0,r_3-r_0)\notag\\
	&\times\frac{\widetilde{P}_{ramp}}{a}-\max(0,r_0-r_3)\times\frac{\max(0,\widetilde{P}_{r_3}-\widetilde{P}_{shut})}{r_2-r_3}-\varepsilon\rvert,r_4<r_0\le r_2;\widetilde{P}_r=\widetilde{P}_{r_3}-\notag\\
	&\max(0,r_3-r)\times\frac{\widetilde{P}_{ramp}}{a}-\max(0,r-r_3)\times\frac{\max(0,\widetilde{P}_{r_3}-\widetilde{P}_{shut})}{r_2-r_3},\forall r\in[r_4,r_2],r\ne r_0;\notag\\
	&\widetilde{P}_r=\widetilde{P}_0-r\times\frac{\widetilde{P}_0-\widetilde{P}_{r_4}}{r_4},\forall r\in[1,r_4-1]\label{p:42}\\
	&\tau^m_{m,r_2}=1;\widetilde{P}_{r_3}=\min[1,\widetilde{P}_0+r_3\widetilde{P}_{up},\widetilde{P}_{shut}+(r_2-r_3)\widetilde{P}_{down}];\widetilde{P}_r=\max[0,\widetilde{P}_{r_3}-(r-r_3)\notag\\
	&\times\frac{\widetilde{P}_{ramp}}{a}],\forall r\in[r_3+1,m+M-1];\widetilde{P}_r=\widetilde{P}_0-r\times\frac{\widetilde{P}_0-\widetilde{P}_{r_3}}{r_3},\forall r\in[m,r_3-1]\label{p:43}\\
	&\tau^m_{m,m+M-1}=1;\widetilde{P}_{r_3}=\min(1,\widetilde{P}_0+r_3\widetilde{P}_{up});\widetilde{P}_r=\widetilde{P}_{r_3}-\max(0,r_3-r)\times\frac{\widetilde{P}_{ramp}}{a},\notag\\
	&\forall r\in[r_4,m+M-1];\widetilde{P}_r=\widetilde{P}_0-r\times\frac{\widetilde{P}_0-\widetilde{P}_{r_4}}{r_4},\forall r\in[m,r_4-1]\label{p:44}\\
	&\tau^m_{m,m+M-1}=1;\widetilde{P}_{r_3}=\min(1,\widetilde{P}_0+r_3\widetilde{P}_{up});\widetilde{P}_{r_0}=\lvert\widetilde{P}_{r_3}-\max(0,r_3-r_0)\times\frac{\widetilde{P}_{ramp}}{a}-\varepsilon\rvert,\notag\\
	&r_4<r_0\le m+M-1;\widetilde{P}_r=\widetilde{P}_{r_3}-\max(0,r_3-r)\times\frac{\widetilde{P}_{ramp}}{a},\forall r\in[r_4,m+M-1],r\ne r_0;\notag\\
	&\widetilde{P}_r=\widetilde{P}_0-r\times\frac{\widetilde{P}_0-\widetilde{P}_{r_4}}{r_4},\forall r\in[m,r_4-1],r\ne r_0\label{p:45}\\
	&\tau^m_{m,m+M-1}=1;\widetilde{P}_{r_3}=\min(1,\widetilde{P}_0+r_3\widetilde{P}_{up})-\varepsilon;\widetilde{P}_r=\widetilde{P}_{r_3}-\max(0,r_3-r)\times\frac{\widetilde{P}_{ramp}}{a},\notag\\
	&\forall r\in[r_4,m+M-1];\widetilde{P}_r=\widetilde{P}_0-r\times\frac{\widetilde{P}_0-\widetilde{P}_{r_4}}{r_4},\forall r\in[m,r_4-1]\label{p:46}\\
	&\tau^m_{m,m+M-1}=1;\widetilde{P}_{r_3}=\min(1,\widetilde{P}_0+r_3\widetilde{P}_{up})-\varepsilon;\widetilde{P}_{r_0}=\widetilde{P}_{r_3}-\max(0,r_3-r_0)\times\frac{\widetilde{P}_{ramp}}{a}-\varepsilon,\notag\\
	&r_4<r_0\le m+M-1;\widetilde{P}_r=\widetilde{P}_{r_3}-\max(0,r_3-r)\times\frac{\widetilde{P}_{ramp}}{a},\forall r\in[r_4,m+M-1],r\ne r_0;\notag\\
	&\widetilde{P}_{r_0}=\widetilde{P}_0-r_0\times\frac{\widetilde{P}_0-\widetilde{P}_{r_4}}{r_4}-\varepsilon,m\le r_0<r_4;\notag\\
	&\widetilde{P}_r=\widetilde{P}_0-r\times\frac{\widetilde{P}_0-\widetilde{P}_{r_4}}{r_4},\forall r\in[m,r_4-1],r\ne r_0\label{p:47}\\
	&\tau^m_{m,U}=1;\widetilde{P}_{r_3}=\min[1,\widetilde{P}_0+r_3\widetilde{P}_{up},\widetilde{P}_{shut}+(U-r_3)\widetilde{P}_{down}];\notag\\
	&\widetilde{P}_r=\widetilde{P}_{r_3}-\max(0,r_3-r)\times\frac{\widetilde{P}_{ramp}}{a}-\max(0,r-r_3)\times\frac{\max(0,\widetilde{P}_{r_3}-\widetilde{P}_{shut})}{U-r_3},\forall r\in[r_4,U];\notag\\
	&\widetilde{P}_r=\widetilde{P}_0-r\times\frac{\widetilde{P}_0-\widetilde{P}_{r_4}}{r_4},\forall r\in[m,r_4-1];\tau^m_{r_1,r_2}=1;\widetilde{P}_r=0,\forall r\in[r_1,r_2]\label{p:48}\\
	&\tau^m_{m,U}=1;\widetilde{P}_{r_3}=\min[1,\widetilde{P}_0+r_3\widetilde{P}_{up},\widetilde{P}_{shut}+(U-r_3)\widetilde{P}_{down}];\notag\\
	&\widetilde{P}_r=\max[0,\widetilde{P}_{r_3}-(r-r_3)\times\frac{\widetilde{P}_{ramp}}{a}],\forall r\in[r_3,U];\notag\\
	&\widetilde{P}_r=\widetilde{P}_0-r\times\frac{\widetilde{P}_0-\widetilde{P}_{r_3}}{r_3},\forall r\in[m,r_3-1];\tau^m_{r_1,r_2}=1;\widetilde{P}_r=0,\forall r\in[r_1,r_2]\label{p:49}\\
	&\tau^m_{m,m+M-1}=1;\widetilde{P}_{r_3}=\min(1,\widetilde{P}_0+r_3\widetilde{P}_{up});\notag\\
	&\widetilde{P}_r=\max[0,\widetilde{P}_{r_3}-(r-r_3)\times\frac{\widetilde{P}_{ramp}}{a}],\forall r\in[r_3+1,r_4];\notag\\
	&\widetilde{P}_r=\widetilde{P}_{r_4},\forall r\in[r_4+1,m+M-1];\widetilde{P}_r=\widetilde{P}_0-r\times\frac{\widetilde{P}_0-\widetilde{P}_{r_3}}{r_3},\forall r\in[m,r_3-1]\label{p:50}\\
	&\tau^m_{m,m+M-1}=1;\widetilde{P}_{r_3}=\min(1,\widetilde{P}_0+r_3\widetilde{P}_{up});\widetilde{P}_{r_0}=\max[0,\widetilde{P}_{r_3}-(r_0-r_3)\times\frac{\widetilde{P}_{ramp}}{a}]+\varepsilon,\notag\\
	&r_3<r_0<r_4;\widetilde{P}_r=\max[0,\widetilde{P}_{r_3}-(r-r_3)\times\frac{\widetilde{P}_{ramp}}{a}],\forall r\in[r_3+1,r_4],r\ne r_0;\notag\\
	&\widetilde{P}_{r_0}=\widetilde{P}_{r_4}+\varepsilon,r_4<r_0\le m+M-1;\widetilde{P}_r=\widetilde{P}_{r_4},\forall r\in[r_4+1,m+M-1],r\ne r_0;\notag\\
	&\widetilde{P}_r=\widetilde{P}_0-r\times\frac{\widetilde{P}_0-\widetilde{P}_{r_3}}{r_3},\forall r\in[m,r_3-1]\label{p:51}\\
	&\tau^m_{m,m+M-1}=1;\widetilde{P}_{r_3}=\min(1,\widetilde{P}_0+r_3\widetilde{P}_{up})-\varepsilon;\notag\\
	&\widetilde{P}_r=\max[0,\widetilde{P}_{r_3}-(r-r_3)\times\frac{\widetilde{P}_{ramp}}{a}],\forall r\in[r_3+1,r_4];\notag\\
	&\widetilde{P}_r=\widetilde{P}_{r_4},\forall r\in[r_4+1,m+M-1];\widetilde{P}_r=\widetilde{P}_0-r\times\frac{\widetilde{P}_0-\widetilde{P}_{r_3}}{r_3},\forall r\in[m,r_3-1]\label{p:52}\\
	&\tau^m_{m,m+M-1}=1;\widetilde{P}_{r_3}=\min(1,\widetilde{P}_0+r_3\widetilde{P}_{up})-\varepsilon;\widetilde{P}_r=\max[0,\widetilde{P}_{r_3}-(r-r_3)\times\frac{\widetilde{P}_{ramp}}{a}],\notag\\
	&\forall r\in[r_3+1,r_4];\widetilde{P}_r=\widetilde{P}_{r_4},\forall r\in[r_4+1,m+M-1];\widetilde{P}_{r_0}=\widetilde{P}_0-r_0\times\frac{\widetilde{P}_0-\widetilde{P}_{r_3}}{r_3}-\varepsilon,\notag\\
	&m\le r_0<r_3;\widetilde{P}_r=\widetilde{P}_0-r\times\frac{\widetilde{P}_0-\widetilde{P}_{r_3}}{r_3},\forall r\in[m,r_3-1],r\ne r_0\label{p:53}\\
	&\tau^m_{m,m+M-1}=1;\widetilde{P}_{r_3}=\max(0,\widetilde{P}_0-r_3\widetilde{P}_{down});\notag\\
	&\widetilde{P}_r=\widetilde{P}_{r_3}+\max(0,r_3-r)\times\frac{\widetilde{P}_0-\widetilde{P}_{r_3}}{r_3},\forall r\in[m,m+M-1]\label{p:54}\\
	&\tau^m_{m,m+M-1}=1;\widetilde{P}_{r_3}=\max(0,\widetilde{P}_0-r_3\widetilde{P}_{down});\widetilde{P}_{r_0}=\lvert\widetilde{P}_{r_3}+\max(0,r_3-r_0)\times\frac{\widetilde{P}_0-\widetilde{P}_{r_3}}{r_3}\notag\\
	&-\varepsilon\rvert;\widetilde{P}_r=\widetilde{P}_{r_3}+\max(0,r_3-r)\times\frac{\widetilde{P}_0-\widetilde{P}_{r_3}}{r_3},\forall r\in[m,m+M-1],r\ne r_0\label{p:55}\\
	&\tau^m_{m,r_2}=1;\widetilde{P}_r=\widetilde{P}_0-r\times\frac{\max(0,\widetilde{P}_0-\widetilde{P}_{shut})}{r_2},\forall r\in[m,r_2]\label{p:56}\\
	&\tau^m_{m,r_2}=1;\widetilde{P}_{r_0}=\lvert\widetilde{P}_0-r_0\times\frac{\max(0,\widetilde{P}_0-\widetilde{P}_{shut})}{r_2}-\varepsilon\rvert;\notag\\
	&\widetilde{P}_r=\widetilde{P}_0-r\times\frac{\max(0,\widetilde{P}_0-\widetilde{P}_{shut})}{r_2},\forall r\in[m,r_2],r\ne r_0\label{p:57}
\end{align}
\end{footnotesize}

\section*{Appendix 2: Proof of proposition for history-denpendent polytope}
\setcounter{equation}{0}
\renewcommand{\theequation}{B\arabic{equation}}
\setcounter{proposition}{0}

\begin{proof}[\textbf{Proof of Proposition \upshape\ref{prp1}}]
	If all points of $\mathcal{P}_m^I$ that are tight at inequality (\ref{eq:tight-max-output-0}) satisfy
	\begin{equation}
		\sum^{m+M-1}_{j=m}\omega_j\widetilde{P}_j+\sum\nolimits_{[h,k]\in A_m}\varphi_{h,k}\tau^m_{h,k}=\lambda_0\label{pr:1}
	\end{equation}
	then
	\begin{enumerate}
		\item $\lambda_0=0$.\\
		Point (\ref{p:1}) satisfies inequality (\ref{eq:tight-max-output-0}) at equality. If Point (\ref{p:1}) satisfies equality (\ref{pr:1}), then $\lambda_0=0$.
		\item $\omega_j=0,j\in[m,m+M-1]\backslash\{t\}$.\\
		Consider the following two cases:
		\begin{enumerate}
			\item $m\le j<t$.\\
			Consider point (\ref{p:3}) and point (\ref{p:4}) with $r_1=m$, $r_2=t-1$, $r_0=j$. Both points are valid and satisfy inequality (\ref{eq:tight-max-output-0}) at equality. Because both points satisfy equality (\ref{pr:1}), we get $\omega_j\times 0=\omega_j\varepsilon$. Hence $\omega_j=0$.
			\item $t<j\le m+M-1$.\\
			Consider point (\ref{p:3}) and point (\ref{p:4}) with $r_1=t+1$, $r_2=m+M-1$, $r_0=j$. Both points are valid and satisfy inequality (\ref{eq:tight-max-output-0}) at equality. Because both points satisfy equality (\ref{pr:1}), we get $\omega_j\times 0=\omega_j\varepsilon$. Hence $\omega_j=0$.
		\end{enumerate}
	    \item $\varphi_{h,k}=0,t\notin[h,k]$.\\
	    If $[h,k]\in A_m$ exists such that $t\notin[h,k]$, we consider point (\ref{p:3}) with $r_1=h$, $r_2=k$. Clearly this point is valid and satisfies inequality (\ref{eq:tight-max-output-0}) at equality. If Point (\ref{p:3}) satisfies equality (\ref{pr:1}), then $\varphi_{h,k}=\lambda_0$. We have shown that $\lambda_0=0$ in part 1. Hence $\varphi_{h,k}=0$.
	    \item $\varphi_{h,k}=-\omega_t\times\min[1,\widetilde{P}_{start}+(t-h)\widetilde{P}_{up},\widetilde{P}_{shut}+(k-t)\widetilde{P}_{down}],m<h\le t\le k<m+M-1$.\\
	    If $[h,k]\in A_m$ exists such that $m<h\le t\le k<m+M-1$, we consider point (\ref{p:12}) with $r_1=h$, $r_2=k$, $r_3=t$. Clearly this point is valid and satisfies inequality (\ref{eq:tight-max-output-0}) at equality. We have shown that $\omega_j=0,j\in[m,m+M-1]\backslash\{t\}$ in part 2. If Point (\ref{p:12}) satisfies equality (\ref{pr:1}), then we get $\omega_t\times\min[1,\widetilde{P}_{start}+(t-h)\widetilde{P}_{up},\widetilde{P}_{shut}+(k-t)\widetilde{P}_{down}]+\varphi_{h,k}=\lambda_0$. Hence, $\varphi_{h,k}=-\omega_t\times\min[1,\widetilde{P}_{start}+(t-h)\widetilde{P}_{up},\widetilde{P}_{shut}+(k-t)\widetilde{P}_{down}]$.
	    \item $\varphi_{h,m+M-1}=-\omega_t\times\min[1,\widetilde{P}_{start}+(t-h)\widetilde{P}_{up}],m<h\le t$.\\
	    If $[h,m+M-1]\in A_m$ exists such that $m<h\le t$, we consider point (\ref{p:18}) with $r_1=h$, $r_3=t$. Clearly this point is valid and satisfies inequality (\ref{eq:tight-max-output-0}) at equality. We have shown that $\omega_j=0,j\in[m,m+M-1]\backslash\{t\}$ in part 2. If Point (\ref{p:18}) satisfies equality (\ref{pr:1}), then we get $\omega_t\times\min[1,\widetilde{P}_{start}+(t-h)\widetilde{P}_{up}]+\varphi_{h,k}=\lambda_0$. Hence, $\varphi_{h,k}=-\omega_t\times\min[1,\widetilde{P}_{start}+(t-h)\widetilde{P}_{up}]$.
	    \item $\varphi_{m,k}=-\omega_t\times\min[1,\widetilde{P}_{shut}+(k-t)\widetilde{P}_{down}],t\le k<m+M-1$.\\
	    If $[m,k]\in A_m$ exists such that $t\le k<m+M-1$, we consider point (\ref{p:7}) with $r_2=k$, $r_3=t$. Clearly this point is valid and satisfies inequality (\ref{eq:tight-max-output-0}) at equality. We have shown that $\omega_j=0,j\in[m,m+M-1]\backslash\{t\}$ in part 2. If Point (\ref{p:7}) satisfies equality (\ref{pr:1}), then we get $\omega_t\times\min[1,\widetilde{P}_{shut}+(k-t)\widetilde{P}_{down}]+\varphi_{h,k}=\lambda_0$. Hence, $\varphi_{h,k}=-\omega_t\times\min[1,\widetilde{P}_{shut}+(k-t)\widetilde{P}_{down}]$.
	    \item $\varphi_{m,m+M-1}=-\omega_t$.\\
	    If $[m,m+M-1]\in A_m$ exists, we consider point (\ref{p:2}). Clearly this point is valid and satisfies inequality (\ref{eq:tight-max-output-0}) at equality. We have shown that $\omega_j=0,j\in[m,m+M-1]\backslash\{t\}$ in part 2. If Point (\ref{p:2}) satisfies equality (\ref{pr:1}), then we get $\omega_t+\varphi_{h,k}=\lambda_0$. Hence, $\varphi_{h,k}=-\omega_t$.
	\end{enumerate}
According to theorem 3.6 of §$I.4.3$ in \cite{wolsey1988integer}, we get Proposition \ref{prp1}.
\end{proof}

\begin{proof}[\textbf{Proof of Proposition \upshape\ref{prp2}}]
	If all points of $\mathcal{Q}_m^I$ that are tight at inequality (\ref{eq:tight-max-output-0}) satisfy
	\begin{equation}
		\sum^{m+M-1}_{j=m}\omega_j\widetilde{P}_j+\sum\nolimits_{[h,k]\in A_m}\varphi_{h,k}\tau^m_{h,k}=\lambda_0\label{pr:2}
	\end{equation}
	then
	\begin{enumerate}
		\item $\lambda_0=0$.\\
		Point (\ref{p:1}) satisfies inequality (\ref{eq:tight-max-output-0}) at equality. If Point (\ref{p:1}) satisfies equality (\ref{pr:2}), then $\lambda_0=0$.
		\item $\omega_j=0,j\in[m,m+M-1]\backslash\{t\}$.\\
		Consider the following two cases:
		\begin{enumerate}
			\item $m\le j<t$.\\
			Consider point (\ref{p:3}) and point (\ref{p:4}) with $r_1=m$, $r_2=t-1$, $r_0=j$. Both points are valid and satisfy inequality (\ref{eq:tight-max-output-0}) at equality. Because both points satisfy equality (\ref{pr:2}), we get $\omega_j\times 0=\omega_j\varepsilon$. Hence $\omega_j=0$.
			\item $t<j\le m+M-1$.\\
			Consider point (\ref{p:3}) and point (\ref{p:4}) with $r_1=t+1$, $r_2=m+M-1$, $r_0=j$. Both points are valid and satisfy inequality (\ref{eq:tight-max-output-0}) at equality. Because both points satisfy equality (\ref{pr:2}), we get $\omega_j\times 0=\omega_j\varepsilon$. Hence $\omega_j=0$.
		\end{enumerate}
		\item $\varphi_{h,k}=0,t\notin[h,k]$.\\
		If $[h,k]\in A_m$ exists such that $t\notin[h,k]$, we consider point (\ref{p:3}) with $r_1=h$, $r_2=k$. Clearly this point is valid and satisfies inequality (\ref{eq:tight-max-output-0}) at equality. If Point (\ref{p:3}) satisfies equality (\ref{pr:2}), then $\varphi_{h,k}=\lambda_0$. We have shown that $\lambda_0=0$ in part 1. Hence $\varphi_{h,k}=0$.
		\item $\varphi_{h,k}=-\omega_t\times\min[1,\widetilde{P}_{start}+(t-h)\widetilde{P}_{up},\widetilde{P}_{shut}+(k-t)\widetilde{P}_{down}],m<h\le t\le k<m+M-1$.\\
		If $[h,k]\in A_m$ exists such that $m<h\le t\le k<m+M-1$, we consider point (\ref{p:12}) with $r_1=h$, $r_2=k$, $r_3=t$. Clearly this point is valid and satisfies inequality (\ref{eq:tight-max-output-0}) at equality. We have shown that $\omega_j=0,j\in[m,m+M-1]\backslash\{t\}$ in part 2. If Point (\ref{p:12}) satisfies equality (\ref{pr:2}), then we get $\omega_t\times\min[1,\widetilde{P}_{start}+(t-h)\widetilde{P}_{up},\widetilde{P}_{shut}+(k-t)\widetilde{P}_{down}]+\varphi_{h,k}=\lambda_0$. Hence, $\varphi_{h,k}=-\omega_t\times\min[1,\widetilde{P}_{start}+(t-h)\widetilde{P}_{up},\widetilde{P}_{shut}+(k-t)\widetilde{P}_{down}]$.
		\item $\varphi_{h,m+M-1}=-\omega_t\times\min[1,\widetilde{P}_{start}+(t-h)\widetilde{P}_{up}],m<h\le t$.\\
		If $[h,m+M-1]\in A_m$ exists such that $m<h\le t$, we consider point (\ref{p:18}) with $r_1=h$, $r_3=t$. Clearly this point is valid and satisfies inequality (\ref{eq:tight-max-output-0}) at equality. We have shown that $\omega_j=0,j\in[m,m+M-1]\backslash\{t\}$ in part 2. If Point (\ref{p:18}) satisfies equality (\ref{pr:2}), then we get $\omega_t\times\min[1,\widetilde{P}_{start}+(t-h)\widetilde{P}_{up}]+\varphi_{h,k}=\lambda_0$. Hence, $\varphi_{h,k}=-\omega_t\times\min[1,\widetilde{P}_{start}+(t-h)\widetilde{P}_{up}]$.
		\item $\varphi_{m,k}=-\omega_t\times\min[1,\widetilde{P}_{shut}+(k-t)\widetilde{P}_{down}],t\le k<m+M-1$.\\
		If $[m,k]\in A_m$ exists such that $t\le k<m+M-1$, we consider point (\ref{p:7}) with $r_2=k$, $r_3=t$. Clearly this point is valid and satisfies inequality (\ref{eq:tight-max-output-0}) at equality. We have shown that $\omega_j=0,j\in[m,m+M-1]\backslash\{t\}$ in part 2. If Point (\ref{p:7}) satisfies equality (\ref{pr:2}), then we get $\omega_t\times\min[1,\widetilde{P}_{shut}+(k-t)\widetilde{P}_{down}]+\varphi_{h,k}=\lambda_0$. Hence, $\varphi_{h,k}=-\omega_t\times\min[1,\widetilde{P}_{shut}+(k-t)\widetilde{P}_{down}]$.
		\item $\varphi_{m,m+M-1}=-\omega_t$.\\
		If $[m,m+M-1]\in A_m$ exists, we consider point (\ref{p:2}). Clearly this point is valid and satisfies inequality (\ref{eq:tight-max-output-0}) at equality. We have shown that $\omega_j=0,j\in[m,m+M-1]\backslash\{t\}$ in part 2. If Point (\ref{p:2}) satisfies equality (\ref{pr:2}), then we get $\omega_t+\varphi_{h,k}=\lambda_0$. Hence, $\varphi_{h,k}=-\omega_t$.
	\end{enumerate}
	According to theorem 3.6 of §$I.4.3$ in \cite{wolsey1988integer}, we get Proposition \ref{prp2}.
\end{proof}

\begin{proof}[\textbf{Proof of Proposition \upshape\ref{prp3}}]
	Necessity: For contradiction we assume that $a\ge \frac{1}{\widetilde{P}_{up}}$ . We have $\min[1,a\widetilde{P}_{up},\widetilde{P}_{shut}+(k-t)\widetilde{P}_{down}]=\min[1,\widetilde{P}_{shut}+(k-t)\widetilde{P}_{down}]$ and $\min(1,a\widetilde{P}_{up})=1$. Then inequality (\ref{eq:tight-ramp-up-0}) can be written as $\widetilde{P}_t-\widetilde{P}_{t-a}\le \sum\nolimits_{\{[h,k]\in A_m,h\le t-a<t\le k<m+M-1\}}\tau^m_{h,k}\min[1,\widetilde{P}_{shut}+(k-t)\widetilde{P}_{down}]+\sum\nolimits_{\{[h,m+M-1]\in A_m,h\le t-a\}}\tau^m_{h,m+M-1}+\sum\nolimits_{\{[h,k]\in A_m,t-a<h\le t\le k<m+M-1\}}\tau^m_{h,k} \min[1,\widetilde{P}_{start}+(t-h)\widetilde{P}_{up},\widetilde{P}_{shut}+(k-t)\widetilde{P}_{down}]+\sum\nolimits_{\{[h,m+M-1]\in A_m,t-a<h\le t\}}\tau^m_{h,m+M-1}\min[1,\widetilde{P}_{start}+(t-h)\widetilde{P}_{up}]$. It is dominated by inequalities (\ref{eq:tight-max-output-0}) and $\widetilde{P}_{t-a}\ge 0$.\\
	Sufficiency: If all points of $\mathcal{Q}_m^I$ that are tight at inequality (\ref{eq:tight-ramp-up-0}) satisfy
	\begin{equation}
		\sum^{m+M-1}_{j=m}\omega_j\widetilde{P}_j+\sum\nolimits_{[h,k]\in A_m}\varphi_{h,k}\tau^m_{h,k}=\lambda_0\label{pr:3}
	\end{equation}
	then
	\begin{enumerate}
		\item $\lambda_0=0$.\\
		Point (\ref{p:1}) satisfies inequality (\ref{eq:tight-ramp-up-0}) at equality. If Point (\ref{p:1}) satisfies equality (\ref{pr:3}), then $\lambda_0=0$.
		\item $\omega_j=0,j\in[m,m+M-1]\backslash\{t-a,t\}$.\\
		Consider the following two cases:
		\begin{enumerate}
			\item $j\in[m,t-a-1]\cup[t-a+1,t-1]$.\\
			Consider point (\ref{p:3}) and point (\ref{p:4}) with $r_1=m$, $r_2=t-1$, $r_0=j$. Both points are valid and satisfy inequality (\ref{eq:tight-ramp-up-0}) at equality. Because both points satisfy equality (\ref{pr:3}), we get $\omega_j\times 0=\omega_j\varepsilon$. Hence $\omega_j=0$.
			\item $j\in[t+1,m+M-1]$.\\
			Consider point (\ref{p:3}) and point (\ref{p:4}) with $r_1=t+1$, $r_2=m+M-1$, $r_0=j$. Both points are valid and satisfy inequality (\ref{eq:tight-ramp-up-0}) at equality. Because both points satisfy equality (\ref{pr:3}), we get $\omega_j\times 0=\omega_j\varepsilon$. Hence $\omega_j=0$.
		\end{enumerate}
	    \item $\omega_{t-a}=-\omega_t$.\\
	    If $a<\frac{1}{\widetilde{P}_{up}}$, we get $a\widetilde{P}_{up}<1$. We consider point (\ref{p:24}) and point (\ref{p:25}) with $r_3=t$. Both points are valid and satisfy inequality (\ref{eq:tight-ramp-up-0}) at equality. We have shown that $\omega_j=0$, $j\in[m,m+M-1]\backslash\{t-a,t\}$ in part 2. Because both points satisfy equality (\ref{pr:3}), we get $\omega_{t-a} (1-a\widetilde{P}_{up})+\omega_t=\omega_{t-a} (1-a\widetilde{P}_{up}-\varepsilon)+\omega_t (1-\varepsilon)$. Hence $\omega_{t-a}=-\omega_t$.
		\item $\varphi_{h,k}=0,t\notin[h,k]$.\\
		If $[h,k]\in A_m$ exists such that $t\notin[h,k]$, we consider point (\ref{p:3}) with $r_1=h$, $r_2=k$. Clearly this point is valid and satisfies inequality (\ref{eq:tight-ramp-up-0}) at equality. If Point (\ref{p:3}) satisfies equality (\ref{pr:3}), then $\varphi_{h,k}=\lambda_0$. We have shown that $\lambda_0=0$ in part 1. Hence $\varphi_{h,k}=0$.
		\item $\varphi_{h,k}=-\omega_t\times\min[1,a\widetilde{P}_{up},\widetilde{P}_{shut}+(k-t)\widetilde{P}_{down}],h\le t-a<t\le k<m+M-1$.\\
		If $[h,k]\in A_m$ exists such that $h\le t-a<t\le k<m+M-1$, we consider point (\ref{p:28}) with $r_1=h$, $r_2=k$, $r_3=t$, $r_4=t-a$, $\widetilde{P}_{ramp}=\min[1,a\widetilde{P}_{up},\widetilde{P}_{shut}+(k-t)\widetilde{P}_{down}]$. Clearly this point is valid and satisfies inequality (\ref{eq:tight-ramp-up-0}) at equality. We have shown that $\omega_j=0,j\in[m,m+M-1]\backslash\{t-a,t\}$ and $\omega_{t-a}=-\omega_t$ above. If Point (\ref{p:28}) satisfies equality (\ref{pr:3}), then we get $\omega_{t-a}\{\min[1,\widetilde{P}_{start}+(t-h)\widetilde{P}_{up},\widetilde{P}_{shut}+(k-t)\widetilde{P}_{down}]-\min[1,a\widetilde{P}_{up},\widetilde{P}_{shut}+(k-t)\widetilde{P}_{down}]\}+\omega_t\times\min[1,\widetilde{P}_{start}+(t-h)\widetilde{P}_{up},\widetilde{P}_{shut}+(k-t)\widetilde{P}_{down}]+\varphi_{h,k}=\lambda_0$. Hence, $\varphi_{h,k}=-\omega_t\times\min[1,a\widetilde{P}_{up},\widetilde{P}_{shut}+(k-t)\widetilde{P}_{down}]$.
		\item $\varphi_{h,m+M-1}=-\omega_t\times\min(1,a\widetilde{P}_{up}),h\le t-a$.\\
		If $[h,m+M-1]\in A_m$ exists such that $h\le t-a$, we consider point (\ref{p:32}) with $r_1=h$, $r_3=t$, $r_4=t-a$, $\widetilde{P}_{ramp}=\min(1,a\widetilde{P}_{up})$. Clearly this point is valid and satisfies inequality (\ref{eq:tight-ramp-up-0}) at equality. We have shown that $\omega_j=0,j\in[m,m+M-1]\backslash\{t-a,t\}$ and $\omega_{t-a}=-\omega_t$ above. If Point (\ref{p:32}) satisfies equality (\ref{pr:3}), then we get $\omega_{t-a}\{\min[1,\widetilde{P}_{start}+(t-h)\widetilde{P}_{up}]-\min(1,a\widetilde{P}_{up})\}+\omega_t\times\min[1,\widetilde{P}_{start}+(t-h)\widetilde{P}_{up}]+\varphi_{h,k}=\lambda_0$. Hence, $\varphi_{h,k}=-\omega_t\times\min(1,a\widetilde{P}_{up})$.
		\item $\varphi_{h,k}=-\omega_t\times\min[1,\widetilde{P}_{start}+(t-h)\widetilde{P}_{up},\widetilde{P}_{shut}+(k-t)\widetilde{P}_{down}],t-a<h\le t\le k<m+M-1$.\\
		If $[h,k]\in A_m$ exists such that $t-a<h\le t\le k<m+M-1$, we consider point (\ref{p:12}) with $r_1=h$, $r_2=k$, $r_3=t$. Clearly this point is valid and satisfies inequality (\ref{eq:tight-ramp-up-0}) at equality. We have shown that $\omega_j=0,j\in[m,m+M-1]\backslash\{t-a,t\}$ in part 2. If Point (\ref{p:12}) satisfies equality (\ref{pr:3}), then we get $\omega_t\times\min[1,\widetilde{P}_{start}+(t-h)\widetilde{P}_{up},\widetilde{P}_{shut}+(k-t)\widetilde{P}_{down}]+\varphi_{h,k}=\lambda_0$. Hence, $\varphi_{h,k}=-\omega_t\times\min[1,\widetilde{P}_{start}+(t-h)\widetilde{P}_{up},\widetilde{P}_{shut}+(k-t)\widetilde{P}_{down}]$.
		\item $\varphi_{h,m+M-1}=-\omega_t\times\min[1,\widetilde{P}_{start}+(t-h)\widetilde{P}_{up}],t-a<h\le t$.\\
		If $[h,m+M-1]\in A_m$ exists such that $t-a<h\le t$, we consider point (\ref{p:18}) with $r_1=h$, $r_3=t$. Clearly this point is valid and satisfies inequality (\ref{eq:tight-ramp-up-0}) at equality. We have shown that $\omega_j=0,j\in[m,m+M-1]\backslash\{t-a,t\}$ in part 2. If Point (\ref{p:18}) satisfies equality (\ref{pr:3}), then we get $\omega_t\times\min[1,\widetilde{P}_{start}+(t-h)\widetilde{P}_{up}]+\varphi_{h,k}=\lambda_0$. Hence, $\varphi_{h,k}=-\omega_t\times\min[1,\widetilde{P}_{start}+(t-h)\widetilde{P}_{up}]$.
	\end{enumerate}
	According to theorem 3.6 of §$I.4.3$ in \cite{wolsey1988integer}, we get Proposition \ref{prp3}.
\end{proof}

\begin{proof}[\textbf{Proof of Proposition \upshape\ref{prp4}}]
	Necessity: For contradiction we assume that $a\ge \frac{1}{\widetilde{P}_{down}}$. We have $\min[1,\widetilde{P}_{start}+(t-a-h)\widetilde{P}_{up},a\widetilde{P}_{down}]=\min[1,\widetilde{P}_{start}+(t-a-h)\widetilde{P}_{up}]$ and $\min(1,a\widetilde{P}_{down})=1$. Then inequality (\ref{eq:tight-ramp-down-0}) can be written as $\widetilde{P}_{t-a}-\widetilde{P}_t\le \sum\nolimits_{\{[h,k]\in A_m,m<h\le t-a<t\le k\}}\tau^m_{h,k}\min[1,\widetilde{P}_{start}+(t-a-h)\widetilde{P}_{up}]+\sum\nolimits_{\{[m,k]\in A_m,t\le k\}}\tau^m_{m,k}+\sum\nolimits_{\{[h,k]\in A_m,m<h\le t-a\le k<t\}}\tau^m_{h,k} \min[1,\widetilde{P}_{start}+(t-a-h)\widetilde{P}_{up},\widetilde{P}_{shut}+(k-t+a)\widetilde{P}_{down}]+\sum\nolimits_{\{[m,k]\in A_m,t-a\le k<t\}}\tau^m_{m,k}\min[1,\widetilde{P}_{shut}+(k-t+a)\widetilde{P}_{down}]$. It is dominated by inequalities (\ref{eq:tight-max-output-0}) and $\widetilde{P}_t\ge 0$.\\
	Sufficiency: If all points of $\mathcal{Q}_m^I$ that are tight at inequality (\ref{eq:tight-ramp-down-0}) satisfy
	\begin{equation}
		\sum^{m+M-1}_{j=m}\omega_j\widetilde{P}_j+\sum\nolimits_{[h,k]\in A_m}\varphi_{h,k}\tau^m_{h,k}=\lambda_0\label{pr:4}
	\end{equation}
	then
	\begin{enumerate}
		\item $\lambda_0=0$.\\
		Point (\ref{p:1}) satisfies inequality (\ref{eq:tight-ramp-down-0}) at equality. If Point (\ref{p:1}) satisfies equality (\ref{pr:4}), then $\lambda_0=0$.
		\item $\omega_j=0,j\in[m,m+M-1]\backslash\{t-a,t\}$.\\
		Consider the following two cases:
		\begin{enumerate}
			\item $j\in[t-a+1,t-1]\cup[t+1,m+M-1]$.\\
			Consider point (\ref{p:3}) and point (\ref{p:4}) with $r_1=t-a+1$, $r_2=m+M-1$, $r_0=j$. Both points are valid and satisfy inequality (\ref{eq:tight-ramp-down-0}) at equality. Because both points satisfy equality (\ref{pr:4}), we get $\omega_j\times 0=\omega_j\varepsilon$. Hence $\omega_j=0$.
			\item $j\in[m,t-a-1]$.\\
			Consider point (\ref{p:3}) and point (\ref{p:4}) with $r_1=m$, $r_2=t-a-1$, $r_0=j$. Both points are valid and satisfy inequality (\ref{eq:tight-ramp-down-0}) at equality. Because both points satisfy equality (\ref{pr:4}), we get $\omega_j\times 0=\omega_j\varepsilon$. Hence $\omega_j=0$.
		\end{enumerate}
		\item $\omega_{t-a}=-\omega_t$.\\
		If $a<\frac{1}{\widetilde{P}_{down}}$, we get $a\widetilde{P}_{down}<1$. We consider point (\ref{p:26}) and point (\ref{p:27}) with $r_3=t-a$. Both points are valid and satisfy inequality (\ref{eq:tight-ramp-down-0}) at equality. We have shown that $\omega_j=0$, $j\in[m,m+M-1]\backslash\{t-a,t\}$ in part 2. Because both points satisfy equality (\ref{pr:4}), we get $\omega_{t-a} +\omega_t(1-a\widetilde{P}_{down})=\omega_{t-a}(1-\varepsilon)+\omega_t (1-a\widetilde{P}_{down}-\varepsilon)$. Hence $\omega_{t-a}=-\omega_t$.
		\item $\varphi_{h,k}=0,t-a\notin[h,k]$.\\
		If $[h,k]\in A_m$ exists such that $t-a\notin[h,k]$, we consider point (\ref{p:3}) with $r_1=h$, $r_2=k$. Clearly this point is valid and satisfies inequality (\ref{eq:tight-ramp-down-0}) at equality. If Point (\ref{p:3}) satisfies equality (\ref{pr:4}), then $\varphi_{h,k}=\lambda_0$. We have shown that $\lambda_0=0$ in part 1. Hence $\varphi_{h,k}=0$.
		\item $\varphi_{h,k}=\omega_t\times\min[1,\widetilde{P}_{start}+(t-a-h)\widetilde{P}_{up},a\widetilde{P}_{down}],m<h\le t-a<t\le k$.\\
		If $[h,k]\in A_m$ exists such that $m<h\le t-a<t\le k$, we consider point (\ref{p:30}) with $r_1=h$, $r_2=k$, $r_3=t-a$, $r_4=t$, $\widetilde{P}_{ramp}=\min[1,\widetilde{P}_{start}+(t-a-h)\widetilde{P}_{up},a\widetilde{P}_{down}]$. Clearly this point is valid and satisfies inequality (\ref{eq:tight-ramp-down-0}) at equality. We have shown that $\omega_j=0,j\in[m,m+M-1]\backslash\{t-a,t\}$ and $\omega_{t-a}=-\omega_t$ above. If Point (\ref{p:30}) satisfies equality (\ref{pr:4}), then we get $\omega_t\{\min[1,\widetilde{P}_{start}+(t-a-h)\widetilde{P}_{up},\widetilde{P}_{shut}+(k-t+a)\widetilde{P}_{down}]-\min[1,\widetilde{P}_{start}+(t-a-h)\widetilde{P}_{up},a\widetilde{P}_{down}]\}+\omega_{t-a}\times\min[1,\widetilde{P}_{start}+(t-a-h)\widetilde{P}_{up},\widetilde{P}_{shut}+(k-t+a)\widetilde{P}_{down}]+\varphi_{h,k}=\lambda_0$. Hence, $\varphi_{h,k}=\omega_t\times\min[1,\widetilde{P}_{start}+(t-a-h)\widetilde{P}_{up},a\widetilde{P}_{down}]$.
		\item $\varphi_{m,k}=\omega_t\times\min(1,a\widetilde{P}_{down}),t\le k$.\\
		If $[m,k]\in A_m$ exists such that $t\le k$, we consider point (\ref{p:40}) with $r_1=h$, $r_3=t-a$, $r_4=t$, $\widetilde{P}_{ramp}=\min(1,a\widetilde{P}_{down})$. Clearly this point is valid and satisfies inequality (\ref{eq:tight-ramp-down-0}) at equality. We have shown that $\omega_j=0,j\in[m,m+M-1]\backslash\{t-a,t\}$ and $\omega_{t-a}=-\omega_t$ above. If Point (\ref{p:40}) satisfies equality (\ref{pr:4}), then we get $\omega_t\{\min[1,\widetilde{P}_{shut}+(k-t+a)\widetilde{P}_{down}]-\min(1,a\widetilde{P}_{down})\}+\omega_{t-a}\times\min[1,\widetilde{P}_{shut}+(k-t+a)\widetilde{P}_{down}]+\varphi_{h,k}=\lambda_0$. Hence, $\varphi_{h,k}=\omega_t\times\min(1,a\widetilde{P}_{down})$.
		\item $\varphi_{h,k}=\omega_t\times\min[1,\widetilde{P}_{start}+(t-a-h)\widetilde{P}_{up},\widetilde{P}_{shut}+(k-t+a)\widetilde{P}_{down}],m<h\le t-a\le k<t$.\\
		If $[h,k]\in A_m$ exists such that $m<h\le t-a\le k<t$, we consider point (\ref{p:12}) with $r_1=h$, $r_2=k$, $r_3=t-a$. Clearly this point is valid and satisfies inequality (\ref{eq:tight-ramp-down-0}) at equality. We have shown that $\omega_j=0,j\in[m,m+M-1]\backslash\{t-a,t\}$ and $\omega_{t-a}=-\omega_t$ above. If Point (\ref{p:12}) satisfies equality (\ref{pr:4}), then we get $\omega_{t-a}\times\min[1,\widetilde{P}_{start}+(t-a-h)\widetilde{P}_{up},\widetilde{P}_{shut}+(k-t+a)\widetilde{P}_{down}]+\varphi_{h,k}=\lambda_0$. Hence, $\varphi_{h,k}=\omega_t\times\min[1,\widetilde{P}_{start}+(t-a-h)\widetilde{P}_{up},\widetilde{P}_{shut}+(k-t+a)\widetilde{P}_{down}]$.
		\item $\varphi_{m,k}=\omega_t\times\min[1,\widetilde{P}_{shut}+(k-t+a)\widetilde{P}_{down}],t-a\le k<t$.\\
		If $[m,k]\in A_m$ exists such that $t-a\le k<t$, we consider point (\ref{p:7}) with $r_1=h$, $r_3=t-a$. Clearly this point is valid and satisfies inequality (\ref{eq:tight-ramp-down-0}) at equality. We have shown that $\omega_j=0,j\in[m,m+M-1]\backslash\{t-a,t\}$ and $\omega_{t-a}=-\omega_t$ above. If Point (\ref{p:7}) satisfies equality (\ref{pr:4}), then we get $\omega_{t-a}\times\min[1,\widetilde{P}_{shut}+(k-t+a)\widetilde{P}_{down}]+\varphi_{h,k}=\lambda_0$. Hence, $\varphi_{h,k}=\omega_t\times\min[1,\widetilde{P}_{shut}+(k-t+a)\widetilde{P}_{down}]$.
	\end{enumerate}
	According to theorem 3.6 of §$I.4.3$ in \cite{wolsey1988integer}, we get Proposition \ref{prp4}.
\end{proof}

\begin{proof}[\textbf{Proof of Proposition \upshape\ref{prp5}}]
	If all points of $\widetilde{\mathcal{P}}_m^I$ that are tight at inequality (\ref{eq:tight-min-output-1}) satisfy
	\begin{equation}
		\sum^{m+M-1}_{j=\max(1,m)}\omega_j\widetilde{P}_j+\sum\nolimits_{[h,k]\in A_m}\varphi_{h,k}\tau^m_{h,k}=\lambda_0\label{pr:5}
	\end{equation}
	then
	\begin{enumerate}
		\item $\lambda_0=0$ when $m>U$.\\
		If $m>U$, point (\ref{p:1}) is valid and satisfies inequality (\ref{eq:tight-min-output-1}) at equality. If Point (\ref{p:1}) satisfies equality (\ref{pr:5}), then $\lambda_0=0$.
		\item $\omega_j=0,j\in[\max(1,m),m+M-1]\backslash\{t\}$.\\
		Consider point (\ref{p:54}) and point (\ref{p:55}) with $r_3=t$, $r_0=j$. Both points are valid and satisfy inequality (\ref{eq:tight-min-output-1}) at equality. Because both points satisfy equality (\ref{pr:5}), we get $\omega_j[\widetilde{P}_t+\max(0,t-j)\times\frac{\widetilde{P}_0-\widetilde{P}_t}{t}]=\omega_j\lvert\widetilde{P}_t+\max(0,t-j)\times\frac{\widetilde{P}_0-\widetilde{P}_t}{t}-\varepsilon\rvert$. Hence $\omega_j=0$.
		\item $\varphi_{m,k}=\lambda_0,t\notin[m,k]$.\\
		If $[m,k]\in A_m$ exists such that $t\notin[m,k]$, we consider point (\ref{p:10}) with $r_1=r_2=k$. Clearly this point is valid and satisfies inequality (\ref{eq:tight-min-output-1}) at equality. We have shown that $\omega_j=0$, $j\in[\max(1,m),m+M-1]\backslash\{t\}$ in part 2. If Point (\ref{p:10}) satisfies equality (\ref{pr:5}), then $\varphi_{m,k}=\lambda_0$.
		\item $\varphi_{m,k}=\lambda_0-\omega_t\times\max(0,\widetilde{P}_0-t\widetilde{P}_{down}),t\in[m,k]$.\\
		If $[m,k]\in A_m$ exists such that $t\in[m,k]$, we consider point (\ref{p:10}) with $r_1=r_2=k$. Clearly this point is valid and satisfies inequality (\ref{eq:tight-min-output-1}) at equality. We have shown that $\omega_j=0$, $j\in[\max(1,m),m+M-1]\backslash\{t\}$ in part 2. If Point (\ref{p:10}) satisfies equality (\ref{pr:5}), then $\varphi_{m,k}=\lambda_0-\omega_t\times\max(0,\widetilde{P}_0-t\widetilde{P}_{down})$.
		\item $\varphi_{h,k}=0,h\ne m$.\\
		If $[h,k]\in A_m,h\ne m$ exists, we consider the following two cases:
		\begin{enumerate}
			\item $t\le U$.\\
			Consider point (\ref{p:5}) with $r_1=h$, $r_2=k$. Clearly this point satisfies inequality (\ref{eq:tight-min-output-1}) at equality. We have shown that $\omega_j=0$, $j\in[\max(1,m),m+M-1]\backslash\{t\}$ and $\varphi_{m,k}=\lambda_0-\omega_t\times\max(0,\widetilde{P}_0-t\widetilde{P}_{down})$, $t\in[m,k]$ above. If Point (\ref{p:5}) satisfies equality (\ref{pr:5}), then we get $\omega_t\times\max(0,\widetilde{P}_0-t\widetilde{P}_{down})+\varphi_{m,U}+\varphi_{h,k}=\lambda_0$. Hence, $\varphi_{h,k}=0$.
			\item $t>U$.\\
			If $m\le U$, we consider point (\ref{p:5}) with $r_1=h$, $r_2=k$. Clearly this point satisfies inequality (\ref{eq:tight-min-output-1}) at equality. We have shown that $\omega_j=0$, $j\in[\max(1,m),m+M-1]\backslash\{t\}$ and $\varphi_{m,k}=\lambda_0$, $t\notin[m,k]$ above. If Point (\ref{p:5}) satisfies equality (\ref{pr:5}), then we get $\varphi_{m,U}+\varphi_{h,k}=\lambda_0$. Hence, $\varphi_{h,k}=0$. If $m>U$, we consider point (\ref{p:3}) with $r_1=h$, $r_2=k$. Clearly this point satisfies inequality (\ref{eq:tight-min-output-1}) at equality. We have shown that $\lambda_0=0$ when $m>U$ in part 1. If Point (\ref{p:3}) satisfies equality (\ref{pr:5}), then we get $\varphi_{h,k}=\lambda_0$. Hence, $\varphi_{h,k}=0$.
		\end{enumerate}
	\end{enumerate}
	According to theorem 3.6 of §$I.4.3$ in \cite{wolsey1988integer}, we get Proposition \ref{prp5}.
\end{proof}

\begin{proof}[\textbf{Proof of Proposition \upshape\ref{prp6}}]
	If all points of $\widetilde{\mathcal{P}}_m^I$ that are tight at inequality (\ref{eq:tight-max-output-1}) satisfy
	\begin{equation}
		\sum^{m+M-1}_{j=\max(1,m)}\omega_j\widetilde{P}_j+\sum\nolimits_{[h,k]\in A_m}\varphi_{h,k}\tau^m_{h,k}=\lambda_0\label{pr:6}
	\end{equation}
	then
	\begin{enumerate}
		\item $\lambda_0=0$ when $m>U$.\\
		If $m>U$, point (\ref{p:1}) is valid and satisfies inequality (\ref{eq:tight-max-output-1}) at equality. If Point (\ref{p:1}) satisfies equality (\ref{pr:6}), then $\lambda_0=0$.
		\item $\omega_j=0,j\in[\max(1,m),m+M-1]\backslash\{t\}$.\\
		Consider point (\ref{p:22}) and point (\ref{p:23}) with $r_3=t$, $r_0=j$. Both points are valid and satisfy inequality (\ref{eq:tight-max-output-1}) at equality. Because both points satisfy equality (\ref{pr:6}), we get $\omega_j[\widetilde{P}_t-\max(0,t-j)\times\frac{\widetilde{P}_t-\widetilde{P}_0}{t}]=\omega_j\lvert\widetilde{P}_t-\max(0,t-j)\times\frac{\widetilde{P}_t-\widetilde{P}_0}{t}-\varepsilon\rvert$. Hence $\omega_j=0$.
		\item $\varphi_{m,k}=\lambda_0,t\notin[m,k]$.\\
		If $[m,k]\in A_m$ exists such that $t\notin[m,k]$, we consider point (\ref{p:10}) with $r_1=m$, $r_2=k$. Clearly this point is valid and satisfies inequality (\ref{eq:tight-max-output-1}) at equality. We have shown that $\omega_j=0$, $j\in[\max(1,m),m+M-1]\backslash\{t\}$ in part 2. If Point (\ref{p:10}) satisfies equality (\ref{pr:6}), then $\varphi_{m,k}=\lambda_0$.
		\item $\varphi_{m,k}=\lambda_0-\omega_t\times\min[1,\widetilde{P}_0+t\widetilde{P}_{up},\widetilde{P}_{shut}+(k-t)\widetilde{P}_{down}],m\le t\le k<m+M-1$.\\
		If $[m,k]\in A_m$ exists such that $m\le t\le k<m+M-1$, we consider point (\ref{p:8}) with $r_2=k$, $r_3=t$. Clearly this point is valid and satisfies inequality (\ref{eq:tight-max-output-1}) at equality. We have shown that $\omega_j=0,j\in[\max(1,m),m+M-1]\backslash\{t\}$ in part 2. If Point (\ref{p:8}) satisfies equality (\ref{pr:6}), then we get $\omega_t\times\min[1,\widetilde{P}_0+t\widetilde{P}_{up},\widetilde{P}_{shut}+(k-t)\widetilde{P}_{down}]+\varphi_{m,k}=\lambda_0$. Hence, $\varphi_{m,k}=\lambda_0-\omega_t\times\min[1,\widetilde{P}_0+t\widetilde{P}_{up},\widetilde{P}_{shut}+(k-t)\widetilde{P}_{down}]$.
		\item $\varphi_{m,m+M-1}=\lambda_0-\omega_t\times\min(1,\widetilde{P}_0+t\widetilde{P}_{up})$.\\
		We consider point (\ref{p:22}) with $r_3=t$. Clearly this point is valid and satisfies inequality (\ref{eq:tight-max-output-1}) at equality. We have shown that $\omega_j=0,j\in[\max(1,m),m+M-1]\backslash\{t\}$ in part 2. If Point (\ref{p:22}) satisfies equality (\ref{pr:6}), then we get $\omega_t\times\min(1,\widetilde{P}_0+t\widetilde{P}_{up})+\varphi_{m,m+M-1}=\lambda_0$. Hence, $\varphi_{m,m+M-1}=\lambda_0-\omega_t\times\min(1,\widetilde{P}_0+t\widetilde{P}_{up})$.
		\item $\varphi_{h,k}=0,h\ne m,t\notin[h,k]$.\\
		If $[h,k]\in A_m,h\ne m$ exists such that $t\notin[h,k]$, we consider the following two cases:
		\begin{enumerate}
			\item $t\le U$.\\
			Consider point (\ref{p:16}) with $r_1=h$, $r_2=r_4=k$, $r_3=t$. Clearly this point satisfies inequality (\ref{eq:tight-max-output-1}) at equality. We have shown that $\omega_j=0$, $j\in[\max(1,m),m+M-1]\backslash\{t\}$ and $\varphi_{m,k}=\lambda_0-\omega_t\times\min[1,\widetilde{P}_0+t\widetilde{P}_{up},\widetilde{P}_{shut}+(k-t)\widetilde{P}_{down}]$, $m\le t\le k<m+M-1$ above. If Point (\ref{p:16}) satisfies equality (\ref{pr:6}), then we get $\omega_t\times\min[1,\widetilde{P}_0+t\widetilde{P}_{up},\widetilde{P}_{shut}+(k-t)\widetilde{P}_{down} ]+\varphi_{m,U}+\varphi_{h,k}=\lambda_0$. Hence, $\varphi_{h,k}=0$.
			\item $t>U$.\\
			If $m\le U$, we consider point (\ref{p:16}) with $r_1=h$, $r_2=r_4=k$, $r_3=U$. Clearly this point satisfies inequality (\ref{eq:tight-max-output-1}) at equality. We have shown that $\omega_j=0$, $j\in[\max(1,m),m+M-1]\backslash\{t\}$ and $\varphi_{m,k}=\lambda_0$, $t\notin[m,k]$ above. If Point (\ref{p:16}) satisfies equality (\ref{pr:6}), then we get $\varphi_{m,U}+\varphi_{h,k}=\lambda_0$. Hence, $\varphi_{h,k}=0$. If $m>U$, we consider point (\ref{p:3}) with $r_1=h$, $r_2=k$. Clearly this point satisfies inequality (\ref{eq:tight-max-output-1}) at equality. We have shown that $\lambda_0=0$ when $m>U$ in part 1. If Point (\ref{p:3}) satisfies equality (\ref{pr:6}), then we get $\varphi_{h,k}=\lambda_0$. Hence, $\varphi_{h,k}=0$.
		\end{enumerate}
		\item $\varphi_{h,k}=-\omega_t\times\min[1,\widetilde{P}_{start}+(t-h)\widetilde{P}_{up},\widetilde{P}_{shut}+(k-t)\widetilde{P}_{down}],m<h\le t\le k<m+M-1$.\\
		If $[h,k]\in A_m$ exists such that $m<h\le t\le k<m+M-1$, we consider the following two cases:
		\begin{enumerate}
			\item $m\le U$.\\
			Consider point (\ref{p:16}) with $r_1=h$, $r_2=k$, $r_3=U$, $r_4=t$. Clearly this point is valid and satisfies inequality (\ref{eq:tight-max-output-1}) at equality. We have shown that $\omega_j=0,j\in[\max(1,m),m+M-1]\backslash\{t\}$ and $\varphi_{m,k}=\lambda_0$, $t\notin[m,k]$ above. If Point (\ref{p:16}) satisfies equality (\ref{pr:6}), then we get $\omega_t\times\min[1,\widetilde{P}_{start}+(t-h)\widetilde{P}_{up},\widetilde{P}_{shut}+(k-t)\widetilde{P}_{down}]+\varphi_{m,U}+\varphi_{h,k}=\lambda_0$. Hence, $\varphi_{h,k}=-\omega_t\times\min[1,\widetilde{P}_{start}+(t-h)\widetilde{P}_{up},\widetilde{P}_{shut}+(k-t)\widetilde{P}_{down}]$.
			\item $m>U$.\\
			Consider point (\ref{p:12}) with $r_1=h$, $r_2=k$, $r_3=t$. Clearly this point is valid and satisfies inequality (\ref{eq:tight-max-output-1}) at equality. We have shown that $\omega_j=0,j\in[\max(1,m),m+M-1]\backslash\{t\}$ and $\lambda_0=0$ when $m>U$ above. If Point (\ref{p:12}) satisfies equality (\ref{pr:6}), then we get $\omega_t\times\min[1,\widetilde{P}_{start}+(t-h)\widetilde{P}_{up},\widetilde{P}_{shut}+(k-t)\widetilde{P}_{down}]+\varphi_{h,k}=\lambda_0$. Hence, $\varphi_{h,k}=-\omega_t\times\min[1,\widetilde{P}_{start}+(t-h)\widetilde{P}_{up},\widetilde{P}_{shut}+(k-t)\widetilde{P}_{down}]$.
		\end{enumerate}
		\item $\varphi_{h,m+M-1}=-\omega_t\times\min[1,\widetilde{P}_{start}+(t-h)\widetilde{P}_{up}],m<h\le t$.\\
		If $[h,m+M-1]\in A_m$ exists such that $m<h\le t$, we consider the following two cases:
		\begin{enumerate}
			\item $m\le U$.\\
			Consider point (\ref{p:20}) with $r_1=h$, $r_3=t$. Clearly this point is valid and satisfies inequality (\ref{eq:tight-max-output-1}) at equality. We have shown that $\omega_j=0,j\in[\max(1,m),m+M-1]\backslash\{t\}$ and $\varphi_{m,k}=\lambda_0$, $t\notin[m,k]$ above. If Point (\ref{p:20}) satisfies equality (\ref{pr:6}), then we get $\omega_t\times\min[1,\widetilde{P}_{start}+(t-h)\widetilde{P}_{up}]+\varphi_{m,U}+\varphi_{h,k}=\lambda_0$. Hence, $\varphi_{h,k}=-\omega_t\times\min[1,\widetilde{P}_{start}+(t-h)\widetilde{P}_{up}]$.
			\item $m>U$.\\
			Consider point (\ref{p:18}) with $r_1=h$, $r_3=t$. Clearly this point is valid and satisfies inequality (\ref{eq:tight-max-output-1}) at equality. We have shown that $\omega_j=0,j\in[\max(1,m),m+M-1]\backslash\{t\}$ and $\lambda_0=0$ when $m>U$ above. If Point (\ref{p:18}) satisfies equality (\ref{pr:6}), then we get $\omega_t\times\min[1,\widetilde{P}_{start}+(t-h)\widetilde{P}_{up}]+\varphi_{h,k}=\lambda_0$. Hence, $\varphi_{h,k}=-\omega_t\times\min[1,\widetilde{P}_{start}+(t-h)\widetilde{P}_{up}]$.
		\end{enumerate}
	\end{enumerate}
	According to theorem 3.6 of §$I.4.3$ in \cite{wolsey1988integer}, we get Proposition \ref{prp6}.
\end{proof}

\begin{proof}[\textbf{Proof of Proposition \upshape\ref{prp7}}]
	Necessity: For contradiction we assume that all of the conditions are not satisfied. Then we have:
	\begin{enumerate}
		\item $t>\max(1,m)$.
		\item $t\le K$.
		\item $\widetilde{P}_0-t\widetilde{P}_{down}\ge 0$.
	\end{enumerate}
	We have $\widetilde{P}_0-(t-1)\widetilde{P}_{down}\ge\widetilde{P}_{shut}$ or $t\le W$. If $\widetilde{P}_0-(t-1)\widetilde{P}_{down}=\widetilde{P}_{shut}$, the unit must remain operational at least until time period $t-1$ (i.e., $t-1\le U$). If $\widetilde{P}_0-(t-1)\widetilde{P}_{down}>\widetilde{P}_{shut}$ or $t\le W$, the unit must remain operational at least until time period $t$ (i.e., $t\le U$). According to the definition of $A_m$, inequalities (\ref{eq:tight-min-output-1})(\ref{eq:tight-ramp-down-1}) can be written as $\widetilde{P}_t\ge\sum\nolimits_{\{[m,k] \in A_m,t\le k\}}\tau^m_{m,k}(\widetilde{P}_0-t\widetilde{P}_{down})$ and $\widetilde{P}_{t-a}-\widetilde{P}_t\le\sum\nolimits_{\{[m,k] \in A_m,t-a<t\le k\}}\tau^m_{m,k}a\widetilde{P}_{down}+\sum\nolimits_{\{[m,t-1] \in A_m\}}\tau^m_{m,t-1}[\widetilde{P}_{shut}+(a-1)\widetilde{P}_{down}]$ respectively. Inequality (\ref{eq:tight-min-output-1}) with $t>m$ is dominated by inequalities (\ref{eq:tight-min-output-1}) with $t=m$ and (\ref{eq:tight-ramp-down-1}) with $t-a=m$. It should be noted that $K$ is equal to $U$ or $U+1$.\\
	Sufficiency: If all points of $\widetilde{\mathcal{P}}_m^I$ that are tight at inequality (\ref{eq:tight-min-output-1}) satisfy
	\begin{equation}
		\sum^{m+M-1}_{j=\max(1,m)}\omega_j\widetilde{P}_j+\sum\nolimits_{[h,k]\in A_m}\varphi_{h,k}\tau^m_{h,k}=\lambda_0\label{pr:7}
	\end{equation}
	then
	\begin{enumerate}
		\item $\lambda_0=0$ when $m>U$.\\
		If $m>U$, point (\ref{p:1}) is valid and satisfies inequality (\ref{eq:tight-min-output-1}) at equality. If Point (\ref{p:1}) satisfies equality (\ref{pr:7}), then $\lambda_0=0$.
		\item $\omega_j=0,j\in[\max(1,m),m+M-1]\backslash\{t\}$.\\
		\begin{enumerate}
			\item $t<j\le m+M-1$.\\
			Consider point (\ref{p:54}) and point (\ref{p:55}) with $r_3=t$, $r_0=j$. Both points are valid and satisfy inequality (\ref{eq:tight-min-output-1}) at equality. Because both points satisfy equality (\ref{pr:7}), we get $\omega_j\widetilde{P}_t=\omega_j\lvert\widetilde{P}_t-\varepsilon\rvert$. Hence $\omega_j=0$.
			\item $\max(1,m)\le j<t$.\\
			If $t=m$, we don’t need to consider this case. Otherwise, consider the following two cases:
			\begin{enumerate}
				\item $K<t$.\\
				We get $\frac{\widetilde{P}_0-\widetilde{P}_{shut}}{t-1}<\widetilde{P}_{down}$, $t-1\ge U$. Consider point (\ref{p:56}) and point (\ref{p:57}) with $r_2=t-1$, $r_0=j$. Both points are valid and satisfy inequality (\ref{eq:tight-min-output-1}) at equality. Because both points satisfy equality (\ref{pr:7}), we get $\omega_j[\widetilde{P}_0-j\times\frac{\max(0,\widetilde{P}_0-\widetilde{P}_{shut})}{t-1}]=\omega_j\lvert\widetilde{P}_0-j\times\frac{\max(0,\widetilde{P}_0-\widetilde{P}_{shut})}{t-1}-\varepsilon\rvert$. Hence $\omega_j=0$.
				\item $\frac{\widetilde{P}_0}{\widetilde{P}_{down}}<t$.\\
				We get $\widetilde{P}_0-t\widetilde{P}_{down}<0$. Consider point (\ref{p:54}) and point (\ref{p:55}) with $r_3=t$, $r_0=j$. Both points are valid and satisfy inequality (\ref{eq:tight-min-output-1}) at equality. Because both points satisfy equality (\ref{pr:7}), we get $\omega_j[(t-j)\times\frac{\widetilde{P}_0}{t}]=\omega_j\lvert(t-j)\times\frac{\widetilde{P}_0}{t}-\varepsilon\rvert$. Hence $\omega_j=0$.
			\end{enumerate}
		\end{enumerate}
		\item $\varphi_{m,k}=\lambda_0,t\notin[m,k]$.\\
		If $[m,k]\in A_m$ exists such that $t\notin[m,k]$, we consider point (\ref{p:10}) with $r_1=r_2=k$. Clearly this point is valid and satisfies inequality (\ref{eq:tight-min-output-1}) at equality. We have shown that $\omega_j=0$, $j\in[\max(1,m),m+M-1]\backslash\{t\}$ in part 2. If Point (\ref{p:10}) satisfies equality (\ref{pr:7}), then $\varphi_{m,k}=\lambda_0$.
		\item $\varphi_{m,k}=\lambda_0-\omega_t\times\max(0,\widetilde{P}_0-t\widetilde{P}_{down}),t\in[m,k]$.\\
		If $[m,k]\in A_m$ exists such that $t\in[m,k]$, we consider point (\ref{p:10}) with $r_1=r_2=k$. Clearly this point is valid and satisfies inequality (\ref{eq:tight-min-output-1}) at equality. We have shown that $\omega_j=0$, $j\in[\max(1,m),m+M-1]\backslash\{t\}$ in part 2. If Point (\ref{p:10}) satisfies equality (\ref{pr:7}), then $\varphi_{m,k}=\lambda_0-\omega_t\times\max(0,\widetilde{P}_0-t\widetilde{P}_{down})$.
		\item $\varphi_{h,k}=0,h\ne m$.\\
		If $[h,k]\in A_m,h\ne m$ exists, we consider the following two cases:
		\begin{enumerate}
			\item $t\le U$.\\
			Consider point (\ref{p:5}) with $r_1=h$, $r_2=k$. Clearly this point satisfies inequality (\ref{eq:tight-min-output-1}) at equality. We have shown that $\omega_j=0$, $j\in[\max(1,m),m+M-1]\backslash\{t\}$ and $\varphi_{m,k}=\lambda_0-\omega_t\times\max(0,\widetilde{P}_0-t\widetilde{P}_{down})$, $t\in[m,k]$ above. If Point (\ref{p:5}) satisfies equality (\ref{pr:7}), then we get $\omega_t\times\max(0,\widetilde{P}_0-t\widetilde{P}_{down})+\varphi_{m,U}+\varphi_{h,k}=\lambda_0$. Hence, $\varphi_{h,k}=0$.
			\item $t>U$.\\
			If $m\le U$, we consider point (\ref{p:5}) with $r_1=h$, $r_2=k$. Clearly this point satisfies inequality (\ref{eq:tight-min-output-1}) at equality. We have shown that $\omega_j=0$, $j\in[\max(1,m),m+M-1]\backslash\{t\}$ and $\varphi_{m,k}=\lambda_0$, $t\notin[m,k]$ above. If Point (\ref{p:5}) satisfies equality (\ref{pr:7}), then we get $\varphi_{m,U}+\varphi_{h,k}=\lambda_0$. Hence, $\varphi_{h,k}=0$. If $m>U$, we consider point (\ref{p:3}) with $r_1=h$, $r_2=k$. Clearly this point satisfies inequality (\ref{eq:tight-min-output-1}) at equality. We have shown that $\lambda_0=0$ when $m>U$ in part 1. If Point (\ref{p:3}) satisfies equality (\ref{pr:7}), then we get $\varphi_{h,k}=\lambda_0$. Hence, $\varphi_{h,k}=0$.
		\end{enumerate}
	\end{enumerate}
	According to theorem 3.6 of §$I.4.3$ in \cite{wolsey1988integer}, we get Proposition \ref{prp7}.
\end{proof}

\begin{proof}[\textbf{Proof of Proposition \upshape\ref{prp8}}]
	Necessity: For contradiction we assume that all of the conditions are not satisfied. Then we have:
	\begin{enumerate}
		\item $t>\max(1,m)$.
		\item $t\le K$.
        \item $\widetilde{P}_0+t\widetilde{P}_{up}\le 1$.
        \item $K\ge m+M-1$ or $\widetilde{P}_{shut}+(K-t)\widetilde{P}_{down}\ge\widetilde{P}_0+t\widetilde{P}_{up}$.
	\end{enumerate}
    We have $\widetilde{P}_0-(t-1)\widetilde{P}_{down}\ge\widetilde{P}_{shut}$ or $t\le W$. If $\widetilde{P}_0-(t-1)\widetilde{P}_{down}=\widetilde{P}_{shut}$, the unit must remain operational at least until time period $t-1$ (i.e., $t-1\le U$). If $\widetilde{P}_0-(t-1)\widetilde{P}_{down}>\widetilde{P}_{shut}$ or $t\le W$, the unit must remain operational at least until time period $t$ (i.e., $t\le U$). It should be noted that $K$ is equal to $U$ or $U+1$. We consider the following three cases:
	\begin{enumerate}
		\item $K\ge m+M-1,U\ge m+M-1$.\\
		 According to the definition of $A_m$ and (\ref{eq:new-initial-status-1}), inequality (\ref{eq:tight-ramp-up-1}) can be written as $\widetilde{P}_t-\widetilde{P}_{t-a}\le a\widetilde{P}_{up}$, and (\ref{eq:tight-max-output-1}) can be written as $\widetilde{P}_t\le\widetilde{P}_0+t\widetilde{P}_{up}$. Inequality (\ref{eq:tight-max-output-1}) with $t>m$ is dominated by inequalities (\ref{eq:tight-max-output-1}) with $t=m$ and (\ref{eq:tight-ramp-up-1}) with $t-a=m$.
		\item $K\ge m+M-1,U=m+M-2$.\\
		We have $\widetilde{P}_{shut}+U\widetilde{P}_{down}=\widetilde{P}_0$. According to the definition of $A_m$, inequalities (\ref{eq:tight-max-output-1})(\ref{eq:tight-ramp-up-1}) can be written as $\widetilde{P}_t\le\tau^m_{m,U}[\widetilde{P}_{shut}+(U-t)\widetilde{P}_{down}]+\tau^m_{m,m+M-1}(\widetilde{P}_0+t\widetilde{P}_{up})$ and $\widetilde{P}_t-\widetilde{P}_{t-a}\le-\tau^m_{m,U}a\widetilde{P}_{down}+\tau^m_{m,m+M-1}a\widetilde{P}_{up}$ respectively when $t\le U$. Inequalities (\ref{eq:tight-max-output-1})(\ref{eq:tight-ramp-up-1}) can be written as $\widetilde{P}_t\le\tau^m_{m,m+M-1}(\widetilde{P}_0+t\widetilde{P}_{up})$ and $\widetilde{P}_t-\widetilde{P}_{t-a}\le\tau^m_{m,m+M-1}a\widetilde{P}_{up}-\tau^m_{m,U}[\widetilde{P}_0-(t-a)\widetilde{P}_{down}]$ respectively when $t>U$. Inequality (\ref{eq:tight-max-output-1}) with $t>m$ is dominated by inequalities (\ref{eq:tight-max-output-1}) with $t=m$ and (\ref{eq:tight-ramp-up-1}) with $t-a=m$.
		\item $\widetilde{P}_{shut}+(K-t)\widetilde{P}_{down}\ge\widetilde{P}_0+t\widetilde{P}_{up}$.\\
		We have $(K-1)\widetilde{P}_{down}>\widetilde{P}_0-\widetilde{P}_{shut}$. Thus, $t\le U=K=W$. According to the definition of $A_m$ and (\ref{eq:new-initial-status-1}), inequality (\ref{eq:tight-ramp-up-1}) can be written as $\widetilde{P}_t-\widetilde{P}_{t-a}\le a\widetilde{P}_{up}$, and (\ref{eq:tight-max-output-1}) can be written as $\widetilde{P}_t\le(\widetilde{P}_0+t\widetilde{P}_{up})$. Inequality (\ref{eq:tight-max-output-1}) with $t>m$ is dominated by inequalities (\ref{eq:tight-max-output-1}) with $t=m$ and (\ref{eq:tight-ramp-up-1}) with $t-a=m$.
	\end{enumerate}
	Sufficiency: If all points of $\widetilde{\mathcal{Q}}_m^I$ that are tight at inequality (\ref{eq:tight-max-output-1}) satisfy
	\begin{equation}
		\sum^{m+M-1}_{j=\max(1,m)}\omega_j\widetilde{P}_j+\sum\nolimits_{[h,k]\in A_m}\varphi_{h,k}\tau^m_{h,k}=\lambda_0\label{pr:8}
	\end{equation}
	then
	\begin{enumerate}
		\item $\lambda_0=0$ when $m>U$.\\
		If $m>U$, point (\ref{p:1}) is valid and satisfies inequality (\ref{eq:tight-max-output-1}) at equality. If Point (\ref{p:1}) satisfies equality (\ref{pr:8}), then $\lambda_0=0$.
		\item $\omega_j=0,j\in[\max(1,m),m+M-1]\backslash\{t\}$.\\
		Consider the following two cases:
		\begin{enumerate}
			\item $t<j\le m+M-1$.\\
			Consider point (\ref{p:22}) and point (\ref{p:23}) with $r_3=t$, $r_0=j$. Both points are valid and satisfy inequality (\ref{eq:tight-max-output-1}) at equality. Because both points satisfy equality (\ref{pr:8}), we get $\omega_j\times\min(1,\widetilde{P}_0+t\widetilde{P}_{up})=\omega_j\lvert\min(1,\widetilde{P}_0+t\widetilde{P}_{up})-\varepsilon\rvert$. Hence $\omega_j=0$.
			\item $\max(1,m)\le j<t$.\\
			If $t=m$, we don’t need to consider this case. Otherwise, consider the following two cases:
			\begin{enumerate}
				\item $\min(K,\frac{1-\widetilde{P}_0}{\widetilde{P}_{up}})<t$.\\
				If $t>K$, we get $\frac{\widetilde{P}_0-\widetilde{P}_{shut}}{t-1}<\widetilde{P}_{down}$. Consider point (\ref{p:10}) and point (\ref{p:11}) with $r_1=m$, $r_2=t-1$, $r_0=j$. Both points are valid and satisfy inequality (\ref{eq:tight-max-output-1}) at equality. Because both points satisfy equality (\ref{pr:8}), we get $\omega_j\{\max(0,\widetilde{P}_0-r_1\widetilde{P}_{down})-(j-r_1)\times\frac{\max[0,\max(0,\widetilde{P}_0-r_1\widetilde{P}_{down})-\widetilde{P}_{shut}]}{t-m-1}\}=\omega_j\lvert\max(0,\widetilde{P}_0-r_1\widetilde{P}_{down})-(j-r_1)\times\frac{\max[0,\max(0,\widetilde{P}_0-r_1\widetilde{P}_{down})-\widetilde{P}_{shut}]}{t-m-1}-\varepsilon\rvert$. Hence $\omega_j=0$. If $\frac{1-\widetilde{P}_0}{\widetilde{P}_{up}}<t$, we get $1<\widetilde{P}_0+t\widetilde{P}_{up}$. Consider point (\ref{p:22}) and point (\ref{p:23}) with $r_3=t$, $r_0=j$. Both points are valid and satisfy inequality (\ref{eq:tight-max-output-1}) at equality. Because both points satisfy equality (\ref{pr:8}), we get $\omega_j[1-(t-j)\times\frac{1-\widetilde{P}_0}{t}]=\omega_j \lvert1-(t-j)\times\frac{1-\widetilde{P}_0}{t}-\varepsilon\rvert$. Hence $\omega_j=0$.
				\item $K<m+M-1,\max(\frac{\widetilde{P}_{shut}-\widetilde{P}_0}{\widetilde{P}_{up}},\frac{K\widetilde{P}_{down}+\widetilde{P}_{shut}-\widetilde{P}_0}{\widetilde{P}_{up}+\widetilde{P}_{down}})<t<m+M-1$.\\
				We get $\max(K,t)<m+M-1$, $\widetilde{P}_{shut}+[\max(K,t)-t]\widetilde{P}_{down}<\widetilde{P}_0+t\widetilde{P}_{up}$. Consider point (\ref{p:8}) and point (\ref{p:9}) with $r_2=\max(K,t)$, $r_0=j$, $r_3=t$. Both points are valid and satisfy inequality (\ref{eq:tight-max-output-1}) at equality. Because both points satisfy equality (\ref{pr:8}), we get $\omega_j[\widetilde{P}_t-(t-j)\times\frac{\widetilde{P}_t-\widetilde{P}_0}{t}]=\omega_j [\widetilde{P}_t-(t-j)\times\frac{\widetilde{P}_t-\widetilde{P}_0}{t}-\varepsilon]$. Hence $\omega_j=0$.
			\end{enumerate}
		\end{enumerate}
		\item $\varphi_{m,k}=\lambda_0,t\notin[m,k]$.\\
		If $[m,k]\in A_m$ exists such that $t\notin[m,k]$, we consider point (\ref{p:10}) with $r_1=m$, $r_2=k$. Clearly this point is valid and satisfies inequality (\ref{eq:tight-max-output-1}) at equality. We have shown that $\omega_j=0$, $j\in[\max(1,m),m+M-1]\backslash\{t\}$ in part 2. If Point (\ref{p:10}) satisfies equality (\ref{pr:8}), then $\varphi_{m,k}=\lambda_0$.
		\item $\varphi_{m,k}=\lambda_0-\omega_t\times\min[1,\widetilde{P}_0+t\widetilde{P}_{up},\widetilde{P}_{shut}+(k-t)\widetilde{P}_{down}],m\le t\le k<m+M-1$.\\
		If $[m,k]\in A_m$ exists such that $m\le t\le k<m+M-1$, we consider point (\ref{p:8}) with $r_2=k$, $r_3=t$. Clearly this point is valid and satisfies inequality (\ref{eq:tight-max-output-1}) at equality. We have shown that $\omega_j=0,j\in[\max(1,m),m+M-1]\backslash\{t\}$ in part 2. If Point (\ref{p:8}) satisfies equality (\ref{pr:8}), then we get $\omega_t\times\min[1,\widetilde{P}_0+t\widetilde{P}_{up},\widetilde{P}_{shut}+(k-t)\widetilde{P}_{down}]+\varphi_{m,k}=\lambda_0$. Hence, $\varphi_{m,k}=\lambda_0-\omega_t\times\min[1,\widetilde{P}_0+t\widetilde{P}_{up},\widetilde{P}_{shut}+(k-t)\widetilde{P}_{down}]$.
		\item $\varphi_{m,m+M-1}=\lambda_0-\omega_t\times\min(1,\widetilde{P}_0+t\widetilde{P}_{up})$.\\
		We consider point (\ref{p:22}) with $r_3=t$. Clearly this point is valid and satisfies inequality (\ref{eq:tight-max-output-1}) at equality. We have shown that $\omega_j=0,j\in[\max(1,m),m+M-1]\backslash\{t\}$ in part 2. If Point (\ref{p:22}) satisfies equality (\ref{pr:8}), then we get $\omega_t\times\min(1,\widetilde{P}_0+t\widetilde{P}_{up})+\varphi_{m,m+M-1}=\lambda_0$. Hence, $\varphi_{m,m+M-1}=\lambda_0-\omega_t\times\min(1,\widetilde{P}_0+t\widetilde{P}_{up})$.
		\item $\varphi_{h,k}=0,h\ne m,t\notin[h,k]$.\\
		If $[h,k]\in A_m,h\ne m$ exists such that $t\notin[h,k]$, we consider the following two cases:
		\begin{enumerate}
			\item $t\le U$.\\
			Consider point (\ref{p:16}) with $r_1=h$, $r_2=r_4=k$, $r_3=t$. Clearly this point satisfies inequality (\ref{eq:tight-max-output-1}) at equality. We have shown that $\omega_j=0$, $j\in[\max(1,m),m+M-1]\backslash\{t\}$ and $\varphi_{m,k}=\lambda_0-\omega_t\times\min[1,\widetilde{P}_0+t\widetilde{P}_{up},\widetilde{P}_{shut}+(k-t)\widetilde{P}_{down}]$, $m\le t\le k<m+M-1$ above. If Point (\ref{p:16}) satisfies equality (\ref{pr:8}), then we get $\omega_t\times\min[1,\widetilde{P}_0+t\widetilde{P}_{up},\widetilde{P}_{shut}+(U-t)\widetilde{P}_{down} ]+\varphi_{m,U}+\varphi_{h,k}=\lambda_0$. Hence, $\varphi_{h,k}=0$.
			\item $t>U$.\\
			If $m\le U$, we consider point (\ref{p:16}) with $r_1=h$, $r_2=r_4=k$, $r_3=U$. Clearly this point satisfies inequality (\ref{eq:tight-max-output-1}) at equality. We have shown that $\omega_j=0$, $j\in[\max(1,m),m+M-1]\backslash\{t\}$ and $\varphi_{m,k}=\lambda_0$, $t\notin[m,k]$ above. If Point (\ref{p:16}) satisfies equality (\ref{pr:8}), then we get $\varphi_{m,U}+\varphi_{h,k}=\lambda_0$. Hence, $\varphi_{h,k}=0$. If $m>U$, we consider point (\ref{p:3}) with $r_1=h$, $r_2=k$. Clearly this point satisfies inequality (\ref{eq:tight-max-output-1}) at equality. We have shown that $\lambda_0=0$ when $m>U$ in part 1. If Point (\ref{p:3}) satisfies equality (\ref{pr:8}), then we get $\varphi_{h,k}=\lambda_0$. Hence, $\varphi_{h,k}=0$.
		\end{enumerate}
		\item $\varphi_{h,k}=-\omega_t\times\min[1,\widetilde{P}_{start}+(t-h)\widetilde{P}_{up},\widetilde{P}_{shut}+(k-t)\widetilde{P}_{down}],m<h\le t\le k<m+M-1$.\\
		If $[h,k]\in A_m$ exists such that $m<h\le t\le k<m+M-1$, we consider the following two cases:
		\begin{enumerate}
			\item $m\le U$.\\
			Consider point (\ref{p:16}) with $r_1=h$, $r_2=k$, $r_3=U$, $r_4=t$. Clearly this point is valid and satisfies inequality (\ref{eq:tight-max-output-1}) at equality. We have shown that $\omega_j=0,j\in[\max(1,m),m+M-1]\backslash\{t\}$ and $\varphi_{m,k}=\lambda_0$, $t\notin[m,k]$ above. If Point (\ref{p:16}) satisfies equality (\ref{pr:8}), then we get $\omega_t\times\min[1,\widetilde{P}_{start}+(t-h)\widetilde{P}_{up},\widetilde{P}_{shut}+(k-t)\widetilde{P}_{down}]+\varphi_{m,U}+\varphi_{h,k}=\lambda_0$. Hence, $\varphi_{h,k}=-\omega_t\times\min[1,\widetilde{P}_{start}+(t-h)\widetilde{P}_{up},\widetilde{P}_{shut}+(k-t)\widetilde{P}_{down}]$.
			\item $m>U$.\\
			Consider point (\ref{p:12}) with $r_1=h$, $r_2=k$, $r_3=t$. Clearly this point is valid and satisfies inequality (\ref{eq:tight-max-output-1}) at equality. We have shown that $\omega_j=0,j\in[\max(1,m),m+M-1]\backslash\{t\}$ and $\lambda_0=0$ when $m>U$ above. If Point (\ref{p:12}) satisfies equality (\ref{pr:8}), then we get $\omega_t\times\min[1,\widetilde{P}_{start}+(t-h)\widetilde{P}_{up},\widetilde{P}_{shut}+(k-t)\widetilde{P}_{down}]+\varphi_{h,k}=\lambda_0$. Hence, $\varphi_{h,k}=-\omega_t\times\min[1,\widetilde{P}_{start}+(t-h)\widetilde{P}_{up},\widetilde{P}_{shut}+(k-t)\widetilde{P}_{down}]$.
		\end{enumerate}
		\item $\varphi_{h,m+M-1}=-\omega_t\times\min[1,\widetilde{P}_{start}+(t-h)\widetilde{P}_{up}],m<h\le t$.\\
		If $[h,m+M-1]\in A_m$ exists such that $m<h\le t$, we consider the following two cases:
		\begin{enumerate}
			\item $m\le U$.\\
			Consider point (\ref{p:20}) with $r_1=h$, $r_3=t$. Clearly this point is valid and satisfies inequality (\ref{eq:tight-max-output-1}) at equality. We have shown that $\omega_j=0,j\in[\max(1,m),m+M-1]\backslash\{t\}$ and $\varphi_{m,k}=\lambda_0$, $t\notin[m,k]$ above. If Point (\ref{p:20}) satisfies equality (\ref{pr:8}), then we get $\omega_t\times\min[1,\widetilde{P}_{start}+(t-h)\widetilde{P}_{up}]+\varphi_{m,U}+\varphi_{h,k}=\lambda_0$. Hence, $\varphi_{h,k}=-\omega_t\times\min[1,\widetilde{P}_{start}+(t-h)\widetilde{P}_{up}]$.
			\item $m>U$.\\
			Consider point (\ref{p:18}) with $r_1=h$, $r_3=t$. Clearly this point is valid and satisfies inequality (\ref{eq:tight-max-output-1}) at equality. We have shown that $\omega_j=0,j\in[\max(1,m),m+M-1]\backslash\{t\}$ and $\lambda_0=0$ when $m>U$ above. If Point (\ref{p:18}) satisfies equality (\ref{pr:8}), then we get $\omega_t\times\min[1,\widetilde{P}_{start}+(t-h)\widetilde{P}_{up}]+\varphi_{h,k}=\lambda_0$. Hence, $\varphi_{h,k}=-\omega_t\times\min[1,\widetilde{P}_{start}+(t-h)\widetilde{P}_{up}]$.
		\end{enumerate}
	\end{enumerate}
	According to theorem 3.6 of §$I.4.3$ in \cite{wolsey1988integer}, we get Proposition \ref{prp8}.
\end{proof}

\begin{proof}[\textbf{Proof of Proposition \upshape\ref{prp9}}]
	Necessity: For contradiction we assume that one of condition $\mathcal{M}_1$ and condition $\mathcal{M}_2$ is not satisfied.\\
	If condition $\mathcal{M}_1$ is not satisfied. We have
	\begin{enumerate}
		\item $a>1$.
		\item $t\le K$.
		\item $a\widetilde{P}_{up}\le 1-\max[0,\widetilde{P}_0-(t-a)\widetilde{P}_{down}]$.
		\item $K\ge m+M-1$ or $a\widetilde{P}_{up}\le\widetilde{P}_{shut}+(K-t)\widetilde{P}_{down}-\max[0,\widetilde{P}_0-(t-a)\widetilde{P}_{down}]$.
		\item $a\le\underline{T}_{on}$ or $t-a<U+\underline{T}_{off}$.
        \item $\widetilde{P}_{start}+(a-1)\widetilde{P}_{up}\le1$ or $t-a<U+\underline{T}_{off}$.
		\item $t\ge m+M-1$ or $t-a<U+\underline{T}_{off}$ or $\widetilde{P}_{start}+(a-1)\widetilde{P}_{up}\le\widetilde{P}_{shut}+[\max(t,t-a+\underline{T}_{on})-t]\widetilde{P}_{down}$.
	\end{enumerate}
    We have $\widetilde{P}_0-(t-1)\widetilde{P}_{down}\ge\widetilde{P}_{shut}$ or $t\le W$. If $\widetilde{P}_0-(t-1)\widetilde{P}_{down}=\widetilde{P}_{shut}$, the unit must remain operational at least until time period $t-1$ (i.e., $t-1\le U$). If $\widetilde{P}_0-(t-1)\widetilde{P}_{down}>\widetilde{P}_{shut}$ or $t\le W$, the unit must remain operational at least until time period $t$ (i.e., $t\le U$). It should be noted that $K$ is equal to $U$ or $U+1$. We have $t-a<t\le K\le U+1\le U+\underline{T}_{off}$. We consider the following three cases:
	\begin{enumerate}
		\item $K\ge m+M-1,U\ge m+M-1$.\\
		According to the definition of $A_m$ and (\ref{eq:new-initial-status-1}), inequality (\ref{eq:tight-ramp-up-1}) can be written as $\widetilde{P}_t-\widetilde{P}_{t-a}\le a\widetilde{P}_{up}$. Inequality (\ref{eq:tight-ramp-up-1}) with $a>1$ is dominated by inequalities (\ref{eq:tight-ramp-up-1}) with $a=1$.
		\item $K\ge m+M-1,U=m+M-2$.\\
		We have $\widetilde{P}_{shut}+U\widetilde{P}_{down}=\widetilde{P}_0$. According to the definition of $A_m$, inequality (\ref{eq:tight-ramp-up-1}) can be written as $\widetilde{P}_t-\widetilde{P}_{t-a}\le-\tau^m_{m,U}a\widetilde{P}_{down}+\tau^m_{m,m+M-1}a\widetilde{P}_{up}$ when $t\le U$. (\ref{eq:tight-ramp-up-1}) can be written as $\widetilde{P}_t-\widetilde{P}_{t-a}\le\tau^m_{m,m+M-1}a\widetilde{P}_{up}-\tau^m_{m,U}[\widetilde{P}_0-(t-a)\widetilde{P}_{down}]$ when $t>U$. Inequality (\ref{eq:tight-ramp-up-1}) with $a>1$ is dominated by inequalities (\ref{eq:tight-ramp-up-1}) with $a=1$.
		\item $a\widetilde{P}_{up}\le\widetilde{P}_{shut}+(K-t)\widetilde{P}_{down}-\max[0,\widetilde{P}_0-(t-a)\widetilde{P}_{down}]$.\\
		We have $(K-1)\widetilde{P}_{down}>\widetilde{P}_0-\widetilde{P}_{shut}$. Thus, $t\le U=K=W$. According to the definition of $A_m$ and (\ref{eq:new-initial-status-1}), inequality (\ref{eq:tight-ramp-up-1}) can be written as $\widetilde{P}_t-\widetilde{P}_{t-a}\le a\widetilde{P}_{up}$. Inequality (\ref{eq:tight-ramp-up-1}) with $a>1$ is dominated by inequalities (\ref{eq:tight-ramp-up-1}) with $a=1$.
	\end{enumerate}
    If condition $\mathcal{M}_2$ is not satisfied. We have
    \begin{enumerate}
    	\item $t-a>0$.
    	\item $a\widetilde{P}_{up}\ge 1-\max[0,\widetilde{P}_0-(t-a)\widetilde{P}_{down}]$.
    	\item $a\widetilde{P}_{up}\ge 1$ or $t-a\le U+\underline{T}_{off}$.
    \end{enumerate}
    If $a\widetilde{P}_{up}\ge 1$, inequality (\ref{eq:tight-ramp-up-1}) can be written as $\widetilde{P}_t-\widetilde{P}_{t-a}\le\sum\nolimits_{\{[m,k] \in A_m,t\le k<m+M-1\}}\tau^m_{m,k}\{\min[1,\widetilde{P}_{shut}+(k-t)\widetilde{P}_{down}]-\max[0,\widetilde{P}_0-(t-a)\widetilde{P}_{down}]\}+\sum\nolimits_{\{[m,m+M-1] \in A_m\}}\tau^m_{m,m+M-1}\{1-\max[0,\widetilde{P}_0-(t-a)\widetilde{P}_{down}]\}-\sum\nolimits_{\{[m,k] \in A_m,t-a\le k<t\}}\tau^m_{m,k}\max[0,\widetilde{P}_0-(t-a)\widetilde{P}_{down}]+\sum\nolimits_{\{[h,k] \in A_m,t-a<h\le t\le k<m+M-1\}}\tau^m_{h,k}\min[1,\widetilde{P}_{start}+(t-h)\widetilde{P}_{up},\widetilde{P}_{shut}+(k-t)\widetilde{P}_{down}]+\sum\nolimits_{\{[h,m+M-1]\in A_m,t-a<h\le t\}}\tau^m_{h,m+M-1}\min[1,\widetilde{P}_{start}+(t-h)\widetilde{P}_{up}]+\sum\nolimits_{\{[h,k] \in A_m,U+\underline{T}_{off}<h\le t-a<t\le k<m+M-1\}}\tau^m_{h,k}\min\{1,\widetilde{P}_{shut}+(k-t)\widetilde{P}_{down}\}+\sum\nolimits_{\{[h,m+M-1] \in A_m,U+\underline{T}_{off}<h\le t-a\}}\tau^m_{h,m+M-1}$. If $t-a\le U+\underline{T}_{off}$, inequality (\ref{eq:tight-ramp-up-1}) can be written as $\widetilde{P}_t-\widetilde{P}_{t-a}\le\sum\nolimits_{\{[m,k] \in A_m,t\le k<m+M-1\}}\tau^m_{m,k}\{\min[1,\widetilde{P}_{shut}+(k-t)\widetilde{P}_{down}]-\max[0,\widetilde{P}_0-(t-a)\widetilde{P}_{down}]\}+\sum\nolimits_{\{[m,m+M-1] \in A_m\}}\tau^m_{m,m+M-1}\{1-\max[0,\widetilde{P}_0-(t-a)\widetilde{P}_{down}]\}-\sum\nolimits_{\{[m,k] \in A_m,t-a\le k<t\}}\tau^m_{m,k}\max[0,\widetilde{P}_0-(t-a)\widetilde{P}_{down}]+\sum\nolimits_{\{[h,k] \in A_m,t-a<h\le t\le k<m+M-1\}}\tau^m_{h,k}\min[1,\widetilde{P}_{start}+(t-h)\widetilde{P}_{up},\widetilde{P}_{shut}+(k-t)\widetilde{P}_{down}]+\sum\nolimits_{\{[h,m+M-1]\in A_m,t-a<h\le t\}}\tau^m_{h,m+M-1}\min[1,\widetilde{P}_{start}+(t-h)\widetilde{P}_{up}]$. Inequality (\ref{eq:tight-ramp-up-1}) is dominated by inequalities (\ref{eq:tight-max-output-1}) and (\ref{eq:tight-min-output-1}).\\
	Sufficiency: If all points of $\widetilde{\mathcal{Q}}_m^I$ that are tight at inequality (\ref{eq:tight-ramp-up-1}) satisfy
	\begin{equation}
		\sum^{m+M-1}_{j=\max(1,m)}\omega_j\widetilde{P}_j+\sum\nolimits_{[h,k]\in A_m}\varphi_{h,k}\tau^m_{h,k}=\lambda_0\label{pr:9}
	\end{equation}
	then
	\begin{enumerate}
		\item $\lambda_0=0$ when $m>U$.\\
		If $m>U$, point (\ref{p:1}) satisfies inequality (\ref{eq:tight-ramp-up-1}) at equality. We get $\widetilde{P}_t=0$. If Point (\ref{p:1}) satisfies equality (\ref{pr:9}), then $\lambda_0=0$.
		\item $\omega_j=0,j\in[m,m+M-1]\backslash\{t-a,t\}$.\\
		There are three cases to be considered.
		\begin{enumerate}
			\item $t<j\le m+M-1$.\\
			Consider point (\ref{p:44}) and point (\ref{p:45}) with $r_3=t$, $r_4=t-a$, $r_0=j$, $\widetilde{P}_{ramp}=\min\{1-\max[0,\widetilde{P}_0-(t-a)\widetilde{P}_{down}],a\widetilde{P}_{up}\}$. Both points are valid and satisfy inequality (\ref{eq:tight-ramp-up-1}) at equality. Because both points satisfy equality (\ref{pr:9}), we get $\omega_j\times\min(1,\widetilde{P}_0+t\widetilde{P}_{up})=\omega_j\lvert\min(1,\widetilde{P}_0+t\widetilde{P}_{up})-\varepsilon\rvert$. Hence $\omega_j=0$.
			\item $t-a<j<t$.\\
			If $a=1$, we don't need to consider this case. Otherwise, consider the following six cases:
			\begin{enumerate}
				\item 	$t>K$.\\
				We get $\frac{\widetilde{P}_0-\widetilde{P}_{shut}}{t-1}<\widetilde{P}_{down}$. Consider point (\ref{p:10}) and point (\ref{p:11}) with $r_1=t-a$, $r_2=t-1$, $r_0=j$. Both points are valid and satisfy inequality (\ref{eq:tight-ramp-up-1}) at equality. Because both points satisfy equality (\ref{pr:9}), we get $\omega_j[\widetilde{P}_{t-a}-(j-t+a)\times\frac{\max(0,\widetilde{P}_{t-a}-\widetilde{P}_{shut})}{a-1}]=\omega_j\lvert\widetilde{P}_{t-a}-(j-t+a)\times\frac{\max(0,\widetilde{P}_{t-a}-\widetilde{P}_{shut})}{a-1}-\varepsilon\rvert$. Hence $\omega_j=0$.
				\item $\min(\frac{1}{\widetilde{P}_{up}},\frac{1-\widetilde{P}_0+t\widetilde{P}_{down}}{\widetilde{P}_{up}+\widetilde{P}_{down}})<a$.\\
				We get $1-\max[0,\widetilde{P}_0-(t-a)\widetilde{P}_{down}]<a\widetilde{P}_{up}$. Consider point (\ref{p:44}) and point (\ref{p:45}) with $r_3=t$, $r_4=t-a$, $r_0=j$, $\widetilde{P}_{ramp}=1-\max[0,\widetilde{P}_0-(t-a)\widetilde{P}_{down}]$. Both points are valid and satisfy inequality (\ref{eq:tight-ramp-up-1}) at equality. Because both points satisfy equality (\ref{pr:9}), we get $\omega_j\{1-(t-j)\times\frac{1-\max[0,\widetilde{P}_0-(t-a)\widetilde{P}_{down}]}{a}\}=\omega_j\lvert1-(t-j)\times\frac{1-\max[0,\widetilde{P}_0-(t-a)\widetilde{P}_{down}]}{a}-\varepsilon\rvert$. Hence $\omega_j=0$.
				\item $\underline{T}_{on}<a\le t-U-\underline{T}_{off}$.\\
				If  $m\le U$, we consider point (\ref{p:5}) and point (\ref{p:6}), if  $m>U$, we consider point (\ref{p:3}) and point (\ref{p:4}). We let $r_1=t-a+1$, $r_2=t-1$, $r_0=j$. Both points are valid and satisfy inequality (\ref{eq:tight-ramp-up-1}) at equality. Because both points satisfy equality (\ref{pr:9}), we get $\omega_j\times 0=\omega_j\varepsilon$. Hence $\omega_j=0$.
				\item $\frac{1-\widetilde{P}_{start}}{\widetilde{P}_{up}}+1<a\le t-U-\underline{T}_{off}$.\\
				We get $1<\widetilde{P}_{start}+(a-1)\widetilde{P}_{up}$, $U+\underline{T}_{{off}}\le t-a$. If $m\le U$, we consider point (\ref{p:20}) and point (\ref{p:21}), if  $m>U$, we consider point (\ref{p:18}) and point (\ref{p:19}). We let $r_1=t-a+1$, $r_3=t$, $r_0=j$. Both points are valid and satisfy inequality (\ref{eq:tight-ramp-up-1}) at equality. Because both points satisfy equality (\ref{pr:9}), we get $\omega_j\times[1-(t-j)\times\frac{1-\widetilde{P}_{start}}{a-1}]=\omega_j\times\lvert1-(t-j)\times\frac{1-\widetilde{P}_{start}}{a-1}-\varepsilon\rvert$. Hence $\omega_j=0$.
				\item $\max(t,K)<m+M-1$,\\
				$\min\{\frac{\widetilde{P}_{shut}+[K-t]^+ \widetilde{P}_{down}}{\widetilde{P}_{up}},\frac{\widetilde{P}_{shut}+\max(t,K)\times\widetilde{P}_{down}-\widetilde{P}_0}{\widetilde{P}_{up}+\widetilde{P}_{down}}\}<a$.\\
				We get $\widetilde{P}_{shut}+[\max(t,K)-t]\widetilde{P}_{down}-\max[0,\widetilde{P}_0-(t-a)\widetilde{P}_{down}]<a\widetilde{P}_{up}$. Consider point (\ref{p:41}) and point (\ref{p:42}) with $r_2=\max(t,K)$, $r_3=t$, $r_4=t-a$, $r_0=j$, $\widetilde{P}_{ramp}=\min\{1,\widetilde{P}_{shut}+[\max(t,K)-t]\widetilde{P}_{down}\}-\max[0,\widetilde{P}_0-(t-a)\widetilde{P}_{down}]>-a\widetilde{P}_{down}$. Both points are valid and satisfy inequality (\ref{eq:tight-ramp-up-1}) at equality. Because both points satisfy equality (\ref{pr:9}), we get $\omega_j[\widetilde{P}_t-(t-j)\times\frac{\widetilde{P}_{ramp}}{a}]=\omega_j\lvert\widetilde{P}_t-(t-j)\times\frac{\widetilde{P}_{ramp}}{a}-\varepsilon\rvert$. Hence $\omega_j=0$.
				\item $t<m+M-1,\max[\frac{\widetilde{P}_{shut}-\widetilde{P}_{start}}{\widetilde{P}_{up}}+1,\frac{(\underline{T}_{on}-1)\widetilde{P}_{down}+\widetilde{P}_{shut}-\widetilde{P}_{start}}{\widetilde{P}_{up}+\widetilde{P}_{down}}+1,t+\underline{T}_{on}-m-M+1]<a\le t-U-\underline{T}_{off}$.\\
				We get $\max(t,t-a+\underline{T}_{on})<m+M-1$, $t-a\ge U+\underline{T}_{off}$, $\widetilde{P}_{shut}+[\max(t,t-a+\underline{T}_{on})-t]\widetilde{P}_{down}<\widetilde{P}_{start}+(a-1)\widetilde{P}_{up}$. If $m\le U$, we consider point (\ref{p:14}) and point (\ref{p:15}), if $m>U$, we consider point (\ref{p:12}) and point (\ref{p:13}). We let $r_1=t-a+1$, $r_2=\max(t,t-a+\underline{T}_{on})$, $r_3=U$, $r_4=t$, $r_0=j$. Both points are valid and satisfy inequality (\ref{eq:tight-ramp-up-1}) at equality. Because both points satisfy equality (\ref{pr:9}), we get $\omega_j\times[\widetilde{P}_t-(t-j)\times\frac{\max(0,\widetilde{P}_t-\widetilde{P}_{start})}{a-1}]=\omega_j\times\lvert\widetilde{P}_t-(t-j)\times\frac{\max(0,\widetilde{P}_t-\widetilde{P}_{start})}{a-1}-\varepsilon\rvert$. Hence $\omega_j=0$.
			\end{enumerate}
		    \item $\max(1,m)\le j<t-a$.\\
		    If  $t-a=0$ or $t-a=\max(1,m)$, we don’t need to consider this case. Otherwise, consider the following two cases:
		    \begin{enumerate}
		    	\item $a<\min(\frac{1}{\widetilde{P}_{up}},\frac{1-\widetilde{P}_0+t\widetilde{P}_{down}}{\widetilde{P}_{up}+\widetilde{P}_{down}})$.\\
		    	We get $a\widetilde{P}_{up}<1-\max[0,\widetilde{P}_0-(t-a)\widetilde{P}_{down}]$. Consider point (\ref{p:46}) and point (\ref{p:47}) with $r_3=t$, $r_4=t-a$, $r_0=j$, $\widetilde{P}_{ramp}=a\widetilde{P}_{up}$. Both points are valid and satisfy inequality (\ref{eq:tight-ramp-up-1}) at equality. Because both points satisfy equality (\ref{pr:9}), we get $\omega_j (\widetilde{P}_0-j\times\frac{\widetilde{P}_0-\widetilde{P}_{t-a}}{t-a})=\omega_j (\widetilde{P}_0-j\times\frac{\widetilde{P}_0-\widetilde{P}_{t-a}}{t-a}-\varepsilon)$. Hence $\omega_j=0$.
		    	\item $a<\min(\frac{1}{\widetilde{P}_{up}},t-U-\underline{T}_{off})$.\\
		    	We get $t-a>U+\underline{T}_{off}\ge U+1\ge K$. Then we have $\frac{\widetilde{P}_0-\widetilde{P}_{shut}}{t-a-1}<\widetilde{P}_{down}$. Consider point (\ref{p:10}) and point (\ref{p:11}) with $r_1=m$, $r_2=t-a-1$, $r_0=j$. Both points are valid and satisfy inequality (\ref{eq:tight-ramp-up-1}) at equality. Because both points satisfy equality (\ref{pr:9}), we get $\omega_j\{\max(0,\widetilde{P}_0-r_1\widetilde{P}_{down})-(j-r_1)\times\frac{\max[0,\max(0,\widetilde{P}_0-r_1\widetilde{P}_{down})-\widetilde{P}_{shut}]}{t-a-m-1}\}=\omega_j\lvert\max(0,\widetilde{P}_0-r_1\widetilde{P}_{down})-(j-r_1)\times\frac{\max[0,\max(0,\widetilde{P}_0-r_1\widetilde{P}_{down})-\widetilde{P}_{shut}]}{t-a-m-1}-\varepsilon\rvert$. Hence $\omega_j=0$.
		    \end{enumerate}
		\end{enumerate}
	    \item $\omega_{t-a}=-\omega_t$  when $t-a\ne 0$.\\
	    Consider the following two cases:
	    \begin{enumerate}
	    	\item $a<\min(\frac{1}{\widetilde{P}_{up}},\frac{1-\widetilde{P}_0+t\widetilde{P}_{down}}{\widetilde{P}_{up}+\widetilde{P}_{down}})$.\\
	    	We get $a\widetilde{P}_{up}<1-\max[0,\widetilde{P}_0-(t-a)\widetilde{P}_{down}]$. Consider point (\ref{p:44}) and point (\ref{p:46}) with $r_3=t$, $r_4=t-a$, $r_0=j$, $\widetilde{P}_{ramp}=a\widetilde{P}_{up}$. Both points are valid and satisfy inequality (\ref{eq:tight-ramp-up-1}) at equality. We have shown that $\omega_j=0$, $j\in[\max(1,m),m+M-1]\{t-a,t\}$ in part 2. Because both points satisfy equality (\ref{pr:9}), we get $\omega_{t-a}[\min(1,\widetilde{P}_0+t\widetilde{P}_{up})-a\widetilde{P}_{up}]+\omega_t\times\min(1,\widetilde{P}_0+t\widetilde{P}_{up})=\omega_{t-a}[\min(1,\widetilde{P}_0+t\widetilde{P}_{up})-\varepsilon-a\widetilde{P}_{up}]+\omega_t[\min(1,\widetilde{P}_0+t\widetilde{P}_{up})-\varepsilon]$. Hence $\omega_{t-a}=-\omega_t$.
	    	\item $a<\min(\frac{1}{\widetilde{P}_{up}},t-U-\underline{T}_{off})$.\\
	    	We get $a\widetilde{P}_{up}<1$, $t-a\ge U+\underline{T}_{off}+1.$ If $m\le U$, we consider point (\ref{p:34}) and point (\ref{p:35}), if  $m>U$, we consider point (\ref{p:32}) and point (\ref{p:33}). We let $r_1=U+\underline{T}_{off}+1$, $r_3=t$, $r_4=t-a$, $r_0=j$, $\widetilde{P}_{ramp}=a\widetilde{P}_{up}$. Both points are valid and satisfy inequality (\ref{eq:tight-ramp-up-1}) at equality. We have shown that $\omega_j=0$, $j\in[\max(1,m),m+M-1]\{t-a,t\}$ in part 2. Because both points satisfy equality (\ref{pr:9}), we get $\omega_{t-a}\{\min[1,\widetilde{P}_{start}+(t-U-\underline{T}_{off}-1)\widetilde{P}_{up}]-a\widetilde{P}_{up}\}+\omega_t\times\min[1,\widetilde{P}_{start}+(t-U-\underline{T}_{off}-1)\widetilde{P}_{up}]=\omega_{t-a}\{\min[1,\widetilde{P}_{start}+(t-U-\underline{T}_{off}-1)\widetilde{P}_{up}]-\varepsilon-a\widetilde{P}_{up}\}+\omega_t\{\min[1,\widetilde{P}_{start}+(t-U-\underline{T}_{off}-1)\widetilde{P}_{up}]-\varepsilon\}$. Hence $\omega_{t-a}=-\omega_t$.
	    \end{enumerate}
	    \item 	$\varphi_{m,k}=\lambda_0,k<t-a$.\\
	    If $[m,k]\in A_m$ exists such that $k<t-a$, consider point (\ref{p:10}) with $r_1=r_2=k$. Clearly this point is valid and satisfies inequality (\ref{eq:tight-ramp-up-1}) at equality. We have shown that $\omega_j=0$, $j\in[\max(1,m),m+M-1]\{t-a,t\}$ above. Because point (\ref{p:10}) satisfies equality (\ref{pr:9}), we get $\varphi_{m,k}=\lambda_0$.
        \item 	$\varphi_{m,k}=\lambda_0+\omega_t\times\max[0,\widetilde{P}_0-(t-a)\widetilde{P}_{down}],t-a\le k<t$.\\
        If $[m,k]\in A_m$ exists such that $t-a\le k<t$, consider point (\ref{p:10}) with $r_1=r_2=k$. Clearly this point is valid and satisfies inequality (\ref{eq:tight-ramp-up-1}) at equality. We have shown that $\omega_j=0$, $j\in[\max(1,m),m+M-1]\{t-a,t\},\omega_{t-a}=-\omega_t$ above. Because point (\ref{p:10}) satisfies equality (\ref{pr:9}), we get $\omega_{t-a}\times\max[0,\widetilde{P}_0-(t-a)\widetilde{P}_{down}]+\varphi_{m,k}=\lambda_0$. Hence, $\varphi_{m,k}=\lambda_0+\omega_t\times\max[0,\widetilde{P}_0-(t-a)\widetilde{P}_{down}]$.
        \item $\varphi_{m,k}=\lambda_0-\omega_t\times\min\{1-\max[0,\widetilde{P}_0-(t-a)\widetilde{P}_{down}],a\widetilde{P}_{up},\widetilde{P}_{shut}+(k-t)\widetilde{P}_{down}-\max[0,\widetilde{P}_0-(t-a)\widetilde{P}_{down}]\},t\le k<m+M-1$.\\
        If $[m,k]\in A_m $ exists such that $t\le k<m+M-1$, we consider point (\ref{p:41}) with $r_2=k$, $r_3=t$, $r_4=t-a$, $\widetilde{P}_{ramp}=\min\{1-\max[0,\widetilde{P}_0-(t-a)\widetilde{P}_{down}],a\widetilde{P}_{up},\widetilde{P}_{shut}+(k-t)\widetilde{P}_{down}-\max[0,\widetilde{P}_0-(t-a)\widetilde{P}_{down}]\}$. Clearly this point satisfies inequality (\ref{eq:tight-ramp-up-1}) at equality. We have shown that $\omega_j=0$, $j\in[\max(1,m),m+M-1]\{t-a,t\}$ and $\omega_{t-a}=-\omega_t$ above. If Point (\ref{p:41}) satisfies equality (\ref{pr:9}), then we get $\omega_{t-a}\times\{\min[1,\widetilde{P}_0+t\widetilde{P}_{up},\widetilde{P}_{shut}+(k-t)\widetilde{P}_{down}]-\min\{1-\max[0,\widetilde{P}_0-(t-a)\widetilde{P}_{down}],a\widetilde{P}_{up},\widetilde{P}_{shut}+(k-t)\widetilde{P}_{down}-\max[0,\widetilde{P}_0-(t-a)\widetilde{P}_{down}]\}\}+\omega_t\times\min[1,\widetilde{P}_0+t\widetilde{P}_{up},\widetilde{P}_{shut}+(k-t)\widetilde{P}_{down}]+\varphi_{m,k}=\lambda_0$. Hence, $\varphi_{m,k}=\lambda_0-\omega_t\times\min\{1-\max[0,\widetilde{P}_0-(t-a)\widetilde{P}_{down}],a\widetilde{P}_{up},\widetilde{P}_{shut}+(k-t)\widetilde{P}_{down}-\max[0,\widetilde{P}_0-(t-a)\widetilde{P}_{down}]\}$.
         \item $\varphi_{m,m+M-1}=\lambda_0-\omega_t\times\min\{1-\max[0,\widetilde{P}_0-(t-a)\widetilde{P}_{down}],a\widetilde{P}_{up}\}$.\\
         If $[m,m+M-1]\in A_m$ exists, we consider point (\ref{p:44}) with $r_3=t$, $r_4=t-a$, $\widetilde{P}_{ramp}=\min\{1-\max[0,\widetilde{P}_0-(t-a)\widetilde{P}_{down}],a\widetilde{P}_{up}\}$. Clearly this point satisfies inequality (\ref{eq:tight-ramp-up-1}) at equality. We have shown that $\omega_j=0$, $j\in[\max(1,m),m+M-1]\{t-a,t\}$ and $\omega_{t-a}=-\omega_t$ above. If Point (\ref{p:44}) satisfies equality (\ref{pr:9}), then we get $\omega_{t-a}\times\{\min(1,\widetilde{P}_0+t\widetilde{P}_{up})-\min\{1-\max[0,\widetilde{P}_0-(t-a)\widetilde{P}_{down}],a\widetilde{P}_{up}\}\}+\omega_t\times\min(1,\widetilde{P}_0+t\widetilde{P}_{up})+\varphi_{m,m+M-1}=\lambda_0$. Hence, $\varphi_{m,m+M-1}=\lambda_0-\omega_t\times\min\{1-\max[0,\widetilde{P}_0-(t-a)\widetilde{P}_{down}],a\widetilde{P}_{up}\}$.
         \item 	$\varphi_{h,k}=0$, $h\ne m$, $t\notin[h,k]$.\\
         If $[h,k]\in A_m$, $h\ne m$ exists such that $t\notin[h,k]$, we consider the following three cases:
         \begin{enumerate}
         	\item $U<t-a$.\\
         	If $m\le U$, consider point (\ref{p:16}) with $r_1=h$, $r_2=r_4=k$, $r_3=U$. Clearly this point is valid and satisfies inequality (\ref{eq:tight-ramp-up-1}) at equality. We have shown that $\omega_j=0$, $j\in[\max(1,m),m+M-1]\{t-a,t\}$, $\varphi_{m,k}=\lambda_0$, $k<t-a$ above. Because point (\ref{p:16}) satisfies equality (\ref{pr:9}), we get $\varphi_{m,U}+\varphi_{h,k}=\lambda_0$. Hence, $\varphi_{h,k}=0$. If $m>U$, consider point (\ref{p:3}) with $r_1=h$, $r_2=k$. Clearly this point is valid and satisfies inequality (\ref{eq:tight-ramp-up-1}) at equality. We have shown that $\lambda_0=0$ when $m>U$ above. Because point (\ref{p:3}) satisfies equality (\ref{pr:9}), we get $\varphi_{h,k}=\lambda_0$. Hence, $\varphi_{h,k}=0$.
            \item $t-a\le U<t$.\\
             Consider point (\ref{p:14}) with $r_1=h$, $r_2=r_3=k$. Clearly this point is valid and satisfies inequality (\ref{eq:tight-ramp-up-1}) at equality. We have shown that $\omega_j=0$, $j\in[\max(1,m),m+M-1]\{t-a,t\}$, $\omega_{t-a}=-\omega_t$, $\varphi_{m,k}=\lambda_0+\omega_t\times\max[0,\widetilde{P}_0-(t-a)\widetilde{P}_{down}]$, $t-a\le k<t$ above. Because point (\ref{p:14}) satisfies equality (\ref{pr:9}), we get $\omega_{t-a}\times\max[0,\widetilde{P}_0-(t-a)\widetilde{P}_{down}]+\varphi_{m,U}+\varphi_{h,k}=\lambda_0$. Hence, $\varphi_{h,k}=0$.
             \item $t-a<t\le U$.\\
             Consider point (\ref{p:48}) with $r_1=h$, $r_2=k$, $r_3=t$, $r_4=t-a$, $\widetilde{P}_{ramp}=\min\{1-\max[0,\widetilde{P}_0-(t-a)\widetilde{P}_{down}],a\widetilde{P}_{up},\widetilde{P}_{shut}+(U-t)\widetilde{P}_{down}-\max[0,\widetilde{P}_0-(t-a)\widetilde{P}_{down}]\}$. Clearly this point is valid and satisfies inequality (\ref{eq:tight-ramp-up-1}) at equality. We have shown that $\omega_j=0$, $j\in[\max(1,m),m+M-1]\{t-a,t\}$, $\omega_{t-a}=-\omega_t$, $\varphi_{m,k}=\lambda_0-\omega_t\times\min\{1-\max[0,\widetilde{P}_0-(t-a)\widetilde{P}_{down}],a\widetilde{P}_{up},\widetilde{P}_{shut}+(k-t)\widetilde{P}_{down}-\max[0,\widetilde{P}_0-(t-a)\widetilde{P}_{down}]\}$, $m\le t-a<t\le k<m+M-1$ above. Because point (\ref{p:48}) satisfies equality (\ref{pr:9}), we get $\omega_{t-a}\times\{\min[1,\widetilde{P}_0+t\widetilde{P}_{up},\widetilde{P}_{shut}+(U-t)\widetilde{P}_{down}]-\min\{1-\max[0,\widetilde{P}_0-(t-a)\widetilde{P}_{down}],a\widetilde{P}_{up},\widetilde{P}_{shut}+(U-t)\widetilde{P}_{down}-\max[0,\widetilde{P}_0-(t-a)\widetilde{P}_{down}]\}\}+\omega_t\times\min[1,\widetilde{P}_0+t\widetilde{P}_{up},\widetilde{P}_{shut}+(U-t)\widetilde{P}_{down}]+\varphi_{m,U}+\varphi_{h,k}=\lambda_0$. Hence, $\varphi_{h,k}=0$.
         \end{enumerate}
         \item 	$\varphi_{h,k}=-\omega_t\times\min[1,\widetilde{P}_{start}+(t-h)\widetilde{P}_{up},\widetilde{P}_{shut}+(k-t)\widetilde{P}_{down}],t-a<h\le t\le k<m+M-1$.\\
         If $[h,k]\in A_m$ exists such that  $t-a<h\le t\le k<m+M-1$, we consider the following two cases:
         \begin{enumerate}
         	\item $U<t-a$.\\
         	If $m\le U$, consider point (\ref{p:14}) with $r_1=h$, $r_2=k$, $r_3=t$. We get $\widetilde{P}_t=\min[1,\widetilde{P}_{start}+(t-h)\widetilde{P}_{up},\widetilde{P}_{shut}+(k-t)\widetilde{P}_{down}]$. Clearly this point is valid and satisfies inequality (\ref{eq:tight-ramp-up-1}) at equality. We have shown that $\omega_j=0$, $j\in[\max(1,m),m+M-1]\{t-a,t\}$, $\varphi_{m,k}=\lambda_0$, $k<t-a$ above. Because point (\ref{p:14}) satisfies equality (\ref{pr:9}), we get $\omega_t\times\min[1,\widetilde{P}_{start}+(t-h)\widetilde{P}_{up},\widetilde{P}_{shut}+(k-t)\widetilde{P}_{down}]+\varphi_{m,U}+\varphi_{h,k}=\lambda_0$. Hence, $\varphi_{h,k}=-\omega_t\times\min[1,\widetilde{P}_{start}+(t-h)\widetilde{P}_{up},\widetilde{P}_{shut}+(k-t)\widetilde{P}_{down}]$. If $m>U$, consider point (\ref{p:12}) with $r_1=h$, $r_2=k$, $r_3=t$. Clearly this point is valid and satisfies inequality (\ref{eq:tight-ramp-up-1}) at equality. We have shown that $\omega_j=0$, $j\in[\max(1,m),m+M-1]\{t-a,t\}$, $\lambda_0=0$ when $m>U$ above. Because point (\ref{p:12}) satisfies equality (\ref{pr:9}), we get $\omega_t\times\min[1,\widetilde{P}_{start}+(t-h)\widetilde{P}_{up},\widetilde{P}_{shut}+(k-t)\widetilde{P}_{down}]+\varphi_{h,k}=\lambda_0$. Hence, $\varphi_{h,k}=-\omega_t\times\min[1,\widetilde{P}_{start}+(t-h)\widetilde{P}_{up},\widetilde{P}_{shut}+(k-t)\widetilde{P}_{down}]$.
            \item $t-a\le U$.\\
            Consider point (\ref{p:14}) with $r_1=h$, $r_2=k$, $r_3=t$. We get $\widetilde{P}_t=\min[1,\widetilde{P}_{start}+(t-h)\widetilde{P}_{up},\widetilde{P}_{shut}+(k-t)\widetilde{P}_{down}]$. Clearly this point is valid and satisfies inequality (\ref{eq:tight-ramp-up-1}) at equality. We have shown that $\omega_j=0$, $j\in[\max(1,m),m+M-1]\{t-a,t\}$, $\omega_{t-a}=-\omega_t$, $\varphi_{m,k}=\lambda_0+\omega_t\times\max[0,\widetilde{P}_0-(t-a)\widetilde{P}_{down}]$, $t-a\le k<t$ above. Because point (\ref{p:14}) satisfies equality (\ref{pr:9}), we get $\omega_{t-a}\times\max[0,\widetilde{P}_0-(t-a)\widetilde{P}_{down}]+\omega_t\times\min[1,\widetilde{P}_{start}+(t-h)\widetilde{P}_{up},\widetilde{P}_{shut}+(k-t)\widetilde{P}_{down}]+\varphi_{m,U}+\varphi_{h,k}=\lambda_0$. Hence, $\varphi_{h,k}=-\omega_t\times\min[1,\widetilde{P}_{start}+(t-h)\widetilde{P}_{up},\widetilde{P}_{shut}+(k-t)\widetilde{P}_{down}]$.
         \end{enumerate}
         \item 	$\varphi_{h,m+M-1}=-\omega_t\times\min[1,\widetilde{P}_{start}+(t-h)\widetilde{P}_{up}],t-a<h\le t$.\\
         If $[h,m+M-1]\in A_m$ exists such that  $t-a<h\le t$, we consider the following two cases:
         \begin{enumerate}
         	\item $U<t-a$.\\
         	If $m\le U$, consider point (\ref{p:20}) with $r_1=h$, $r_3=t$. Clearly this point is valid and satisfies inequality (\ref{eq:tight-ramp-up-1}) at equality. We have shown that $\omega_j=0$, $j\in[\max(1,m),m+M-1]\{t-a,t\}$, $\varphi_{m,k}=\lambda_0$, $k<t-a$ above. Because point (\ref{p:20}) satisfies equality (\ref{pr:9}), we get $\omega_t\times\min[1,\widetilde{P}_{start}+(t-h)\widetilde{P}_{up}]+\varphi_{m,U}+\varphi_{h,m+M-1}=\lambda_0$. Hence, $\varphi_{h,m+M-1}=-\omega_t\times\min[1,\widetilde{P}_{start}+(t-h)\widetilde{P}_{up}]$. If $m>U$, consider point (\ref{p:18}) with $r_1=h$, $r_2=k$, $r_3=t$. Clearly this point is valid and satisfies inequality (\ref{eq:tight-ramp-up-1}) at equality. We have shown that $\omega_j=0$, $j\in[\max(1,m),m+M-1]\{t-a,t\}$, $\lambda_0=0$ when $m>U$ above. Because point (\ref{p:18}) satisfies equality (\ref{pr:9}), we get $\omega_t\times\min[1,\widetilde{P}_{start}+(t-h)\widetilde{P}_{up}]+\varphi_{h,m+M-1}=\lambda_0$. Hence, $\varphi_{h,m+M-1}=-\omega_t\times\min[1,\widetilde{P}_{start}+(t-h)\widetilde{P}_{up}]$.
         	\item $t-a\le U$.\\
         	Consider point (\ref{p:20}) with $r_1=h$, $r_3=t$. We get $\widetilde{P}_t=\min[1,\widetilde{P}_{start}+(t-h)\widetilde{P}_{up}]$. Clearly this point is valid and satisfies inequality (\ref{eq:tight-ramp-up-1}) at equality. We have shown that $\omega_j=0$, $j\in[\max(1,m),m+M-1]\{t-a,t\}$, $\omega_{t-a}=-\omega_t$, $\varphi_{m,k}=\lambda_0+\omega_t\times\max[0,\widetilde{P}_0-(t-a)\widetilde{P}_{down}]$, $t-a\le k<t$ above. Because point (\ref{p:20}) satisfies equality (\ref{pr:9}), we get $\omega_{t-a}\times\max[0,\widetilde{P}_0-(t-a)\widetilde{P}_{down}]+\omega_t\times\min[1,\widetilde{P}_{start}+(t-h)\widetilde{P}_{up}]+\varphi_{m,U}+\varphi_{h,m+M-1}=\lambda_0$. Hence, $\varphi_{h,m+M-1}=-\omega_t\times\min[1,\widetilde{P}_{start}+(t-h)\widetilde{P}_{up}]$.
         \end{enumerate}
         \item 	$\varphi_{h,k}=-\omega_t\times\min[1,a\widetilde{P}_{up},\widetilde{P}_{shut}+(k-t)\widetilde{P}_{down}]$, $m<h\le t-a<t\le k<m+M-1$.\\
         If $[h,k]\in A_m$ exists such that $m<h\le t-a<t\le k<m+M-1$, we consider the following two cases:
         \begin{enumerate}
         	\item $m\le U$.\\
         	We consider point (\ref{p:29}) with $r_1=h$, $r_2=k$, $r_3=t$, $r_4=t-a$, $\widetilde{P}_{ramp}=\min[1,a\widetilde{P}_{up},\widetilde{P}_{shut}+(k-t)\widetilde{P}_{down}]$. Clearly this point satisfies inequality (\ref{eq:tight-ramp-up-1}) at equality. We have shown that $\omega_j=0$, $j\in[\max(1,m),m+M-1]\{t-a,t\}$, $\omega_{t-a}=-\omega_t$, $\varphi_{m,k}=\lambda_0$, $k<t-a$ above. If Point (\ref{p:29}) satisfies equality (\ref{pr:9}), then we get $\omega_{t-a}\{\min[1,\widetilde{P}_{start}+(t-h)\widetilde{P}_{up},\widetilde{P}_{shut}+(k-t)\widetilde{P}_{down}]-\min[1,a\widetilde{P}_{up},\widetilde{P}_{shut}+(k-t)\widetilde{P}_{down}]\}+\omega_t\times\min[1,\widetilde{P}_{start}+(t-h)\widetilde{P}_{up},\widetilde{P}_{shut}+(k-t)\widetilde{P}_{down}]+\varphi_{m,U}+\varphi_{h,k}=\lambda_0$. Hence, $\varphi_{h,k}=-\omega_t\times\min[1,a\widetilde{P}_{up},\widetilde{P}_{shut}+(k-t)\widetilde{P}_{down}]$.
         	\item  $m>U$.\\
         	We consider point (\ref{p:28}) with $r_1=h$, $r_2=k$, $r_3=t$, $r_4=t-a$, $\widetilde{P}_{ramp}=\min[1,a\widetilde{P}_{up},\widetilde{P}_{shut}+(k-t)\widetilde{P}_{down}]$. Clearly this point satisfies inequality (\ref{eq:tight-ramp-up-1}) at equality. We have shown that $\omega_j=0$, $j\in[\max(1,m),m+M-1]\{t-a,t\}$, $\omega_{t-a}=-\omega_t$, $\lambda_0=0$ when $m>U$ above. If Point (\ref{p:28}) satisfies equality (\ref{pr:9}), then we get $\omega_{t-a}\{\min[1,\widetilde{P}_{start}+(t-h)\widetilde{P}_{up},\widetilde{P}_{shut}+(k-t)\widetilde{P}_{down}]-\min[1,a\widetilde{P}_{up},\widetilde{P}_{shut}+(k-t)\widetilde{P}_{down}]\}+\omega_t\times\min[1,\widetilde{P}_{start}+(t-h)\widetilde{P}_{up},\widetilde{P}_{shut}+(k-t)\widetilde{P}_{down}]+\varphi_{h,k}=\lambda_0$. Hence, $\varphi_{h,k}=-\omega_t\times\min[1,a\widetilde{P}_{up},\widetilde{P}_{shut}+(k-t)\widetilde{P}_{down}]$.
         \end{enumerate}
         \item 	$\varphi_{h,m+M-1}=-\omega_t\times\min(1,a\widetilde{P}_{up})$, $m<h\le t-a$.\\
         If $[h,m+M-1]\in A_m$ exists such that $m<h\le t-a$, we consider the following two cases:
         \begin{enumerate}
         	\item $m\le U$.\\
         	We consider point (\ref{p:34}) with $r_1=h$, $r_3=t$, $r_4=t-a$, $\widetilde{P}_{ramp}=\min(1,a\widetilde{P}_{up})$. Clearly this point satisfies inequality (\ref{eq:tight-ramp-up-1}) at equality. We have shown that $\omega_j=0$, $j\in[\max(1,m),m+M-1]\{t-a,t\}$, $\omega_{t-a}=-\omega_t$, $\varphi_{m,k}=\lambda_0$, $k<t-a$ above. If Point (\ref{p:34}) satisfies equality (\ref{pr:9}), then we get $\omega_{t-a}\{\min[1,\widetilde{P}_{start}+(t-h)\widetilde{P}_{up}]-\min(1,a\widetilde{P}_{up})\}+\omega_t\times\min[1,\widetilde{P}_{start}+(t-h)\widetilde{P}_{up}]+\varphi_{m,U}+\varphi_{h,m+M-1}=\lambda_0$. Hence, $\varphi_{h,m+M-1}=-\omega_t\times\min(1,a\widetilde{P}_{up})$.
         	\item  $m>U$.\\
         	We consider point (\ref{p:32}) with $r_1=h$, $r_3=t$, $r_4=t-a$, $\widetilde{P}_{ramp}=\min(1,a\widetilde{P}_{up})$. Clearly this point satisfies inequality (\ref{eq:tight-ramp-up-1}) at equality. We have shown that $\omega_j=0$, $j\in[\max(1,m),m+M-1]\{t-a,t\}$, $\omega_{t-a}=-\omega_t$, $\lambda_0=0$ when $m>U$ above. If Point (\ref{p:32}) satisfies equality (\ref{pr:9}), then we get $\omega_{t-a}\{\min[1,\widetilde{P}_{start}+(t-h)\widetilde{P}_{up}]-\min(1,a\widetilde{P}_{up})\}+\omega_t\times\min[1,\widetilde{P}_{start}+(t-h)\widetilde{P}_{up}]+\varphi_{h,m+M-1}=\lambda_0$. Hence, $\varphi_{h,m+M-1}=-\omega_t\times\min(1,a\widetilde{P}_{up})$.
         \end{enumerate}
	\end{enumerate}
	According to theorem 3.6 of §$I.4.3$ in \cite{wolsey1988integer}, we get Proposition \ref{prp9}.
\end{proof}

\begin{proof}[\textbf{Proof of Proposition \upshape\ref{prp10}}]
	Necessity: For contradiction we assume that one of condition $\mathcal{M}_1$ and condition $\mathcal{M}_2$ is not satisfied.\\
	If condition $\mathcal{M}_3$ is not satisfied. We have
	\begin{enumerate}
		\item $a>1$.
		\item $t-a<U+\underline{T}_{off}$.
		\item $a\widetilde{P}_{down}\le\min[1,\widetilde{P}_0+(t-a)\widetilde{P}_{up}]$.
		\item $t\le U$ or $\widetilde{P}_{shut}+(a-1)\widetilde{P}_{down}\le \min[1,\widetilde{P}_0+(t-a)\widetilde{P}_{up}]$.
	\end{enumerate}
	We consider the following two cases:
	\begin{enumerate}
		\item $t\le U$.\\
		According to the definition of $A_m$ and (\ref{eq:new-initial-status-1}), inequality (\ref{eq:tight-ramp-down-1}) can be written as $\widetilde{P}_{t-a}-\widetilde{P}_t\le a\widetilde{P}_{down}$. Inequality (\ref{eq:tight-ramp-down-1}) with $a>1$ is dominated by inequalities (\ref{eq:tight-ramp-down-1}) with $a=1$.
		\item $\widetilde{P}_{shut}+(a-1)\widetilde{P}_{down}\le \min[1,\widetilde{P}_0+(t-a)\widetilde{P}_{up}]$.\\
		According to the definition of $A_m$ and (\ref{eq:new-initial-status-1}), inequality (\ref{eq:tight-ramp-down-1}) can be written as $\widetilde{P}_{t-a}-\widetilde{P}_t\le\sum\nolimits_{\{[m,k]\in A_m,t\le k\}}\tau^m_{m,k}a\widetilde{P}_{down}+\sum\nolimits_{\{[m,k] \in A_m,t-a\le k<t\}}\tau^m_{m,k}[\widetilde{P}_{shut}+(k-t+a)\widetilde{P}_{down}]$. Inequality (\ref{eq:tight-ramp-down-1}) with $a>1$ is dominated by inequalities (\ref{eq:tight-ramp-down-1}) with $a=1$.
	\end{enumerate}
	If condition $\mathcal{M}_4$ is not satisfied. We have
	\begin{enumerate}
		\item $t-a>0$.
		\item $a\widetilde{P}_{down}\ge \min[1,\widetilde{P}_0+(t-a)\widetilde{P}_{up}]$.
		\item $a\widetilde{P}_{down}\ge\min[1,\widetilde{P}_{start}+(t-a-U-\underline{T}_{off}-1)\widetilde{P}_{up}]$ or $t-a\le U+\underline{T}_{off}$.
	\end{enumerate}
	If $a\widetilde{P}_{down}\ge\min[1,\widetilde{P}_{start}+(t-a-U-\underline{T}_{off}-1)\widetilde{P}_{up}]$, inequality (\ref{eq:tight-ramp-down-1}) can be written as $\widetilde{P}_{t-a}-\widetilde{P}_t\le\sum\nolimits_{\{[h,k] \in A_m,U+\underline{T}_{off}<h\le t-a<t\le k\}}\tau^m_{h,k}\min[1,\widetilde{P}_{start}+(t-a-h)\widetilde{P}_{up}]+\sum\nolimits_{\{[h,k] \in A_m,U+\underline{T}_{off}<h\le t-a\le k<t\}}\tau^m_{h,k}\min[1,\widetilde{P}_{start}+(t-a-h)\widetilde{P}_{up},\widetilde{P}_{shut}+(k-t+a)\widetilde{P}_{down}]+\sum\nolimits_{\{[m,k] \in A_m,t\le k\}}\tau^m_{m,k}\min[1,\widetilde{P}_0+(t-a)\widetilde{P}_{up}]+\sum\nolimits_{\{[m,k] \in A_m,t-a\le k<t\}}\tau^m_{m,k}\min[1,\widetilde{P}_0+(t-a)\widetilde{P}_{up},\widetilde{P}_{shut}+(k-t+a)\widetilde{P}_{down}]$. If $t-a\le U+\underline{T}_{off}$, inequality (\ref{eq:tight-ramp-down-1}) can be written as $\widetilde{P}_{t-a}-\widetilde{P}_t\le\sum\nolimits_{\{[m,k] \in A_m,t\le k\}}\tau^m_{m,k}\min[1,\widetilde{P}_0+(t-a)\widetilde{P}_{up}]+\sum\nolimits_{\{[m,k] \in A_m,t-a\le k<t\}}\tau^m_{m,k}\min[1,\widetilde{P}_0+(t-a)\widetilde{P}_{up},\widetilde{P}_{shut}+(k-t+a)\widetilde{P}_{down}]$. Inequality (\ref{eq:tight-ramp-down-1}) is dominated by inequalities (\ref{eq:tight-max-output-1}) and (\ref{eq:tight-min-output-1}).\\
	Sufficiency: If all points of $\widetilde{\mathcal{Q}}_m^I$ that are tight at inequality (\ref{eq:tight-ramp-down-1}) satisfy
	\begin{equation}
		\sum^{m+M-1}_{j=\max(1,m)}\omega_j\widetilde{P}_j+\sum\nolimits_{[h,k]\in A_m}\varphi_{h,k}\tau^m_{h,k}=\lambda_0\label{pr:10}
	\end{equation}
	then
	\begin{enumerate}
		\item $\lambda_0=0$ when $m>U$.\\
		If $m>U$, point (\ref{p:1}) satisfies inequality (\ref{eq:tight-ramp-down-1}) at equality. If Point (\ref{p:1}) satisfies equality (\ref{pr:10}), then $\lambda_0=0$.
		\item $\omega_j=0,j\in[m,m+M-1]\backslash\{t-a,t\}$.\\
		There are three cases to be considered.
		\begin{enumerate}
			\item $t<j\le m+M-1$.\\
			Consider point (\ref{p:50}) and point (\ref{p:51}) with $r_3=t-a$, $r_4=t$, $r_0=j$, $\widetilde{P}_{ramp}=\min[1,\widetilde{P}_0+(t-a)\widetilde{P}_{up},a\widetilde{P}_{down}]$. Both points are valid and satisfy inequality (\ref{eq:tight-ramp-down-1}) at equality. Because both points satisfy equality (\ref{pr:10}), we get $\omega_j\widetilde{P}_t=\omega_j(\widetilde{P}_t+\varepsilon)$. Hence $\omega_j=0$.
			\item $t-a<j<t$.\\
			If $a=1$, we don't need to consider this case. Otherwise, consider the following six cases:
			\begin{enumerate}
				\item $a\le t-U-\underline{T}_{off}$.\\
				If  $m\le U$, we consider point (\ref{p:5}) and point (\ref{p:6}), if $m>U$, we consider point (\ref{p:3}) and point (\ref{p:4}). We let $r_1=t-a+1$, $r_2=m+M-1$, $r_0=j$. Both points are valid and satisfy inequality (\ref{eq:tight-ramp-down-1}) at equality. Because both points satisfy equality (\ref{pr:10}), we get $\omega_j\times 0=\omega_j\varepsilon$. Hence $\omega_j=0$.
				\item $\min(\frac{1}{\widetilde{P}_{down}},\frac{\widetilde{P}_0+t\widetilde{P}_{up}}{\widetilde{P}_{up}+\widetilde{P}_{down}})<a$.\\
				We get $\min[1,\widetilde{P}_0+(t-a)\widetilde{P}_{up}]<a\widetilde{P}_{down}$. Consider point (\ref{p:50}) and point (\ref{p:51}) with $r_3=t-a$, $r_4=t$, $r_0=j$, $\widetilde{P}_{ramp}=\min[1,\widetilde{P}_0+(t-a)\widetilde{P}_{up}]$. Both points are valid and satisfy inequality (\ref{eq:tight-ramp-down-1}) at equality. Because both points satisfy equality (\ref{pr:10}), we get $\omega_j\times\max[0,\widetilde{P}_{t-a}-(j-t+a)\times\frac{\widetilde{P}_{ramp}}{a}]=\omega_j\{\max[0,\widetilde{P}_{t-a}-(j-t+a)\times\frac{\widetilde{P}_{ramp}}{a}]+\varepsilon\}$. Hence $\omega_j=0$.
				\item $U<t,\min\{\frac{1-\widetilde{P}_{shut}}{\widetilde{P}_{down}}+1,\frac{\widetilde{P}_0+t\widetilde{P}_{up}+\widetilde{P}_{down}-\widetilde{P}_{shut}}{\widetilde{P}_{up}+\widetilde{P}_{down}}\}<a$.\\
				We get $\min[1,\widetilde{P}_0+(t-a)\widetilde{P}_{up}]<\widetilde{P}_{shut}+(a-1)\widetilde{P}_{down}$. Consider point (\ref{p:8}) and point (\ref{p:9}) with $r_2=t-1$, $r_3=t-a$, $r_0=j$. Both points are valid and satisfy inequality (\ref{eq:tight-ramp-down-1}) at equality. Because both points satisfy equality (\ref{pr:10}), we get $\omega_j[\widetilde{P}_{t-a}-(j-t+a)\times\frac{\max(0,\widetilde{P}_{t-a}-\widetilde{P}_{shut})}{a-1}]=\omega_j\lvert\widetilde{P}_{t-a}-(j-t+a)\times\frac{\max(0,\widetilde{P}_{t-a}-\widetilde{P}_{shut})}{a-1}-\varepsilon\rvert$. Hence $\omega_j=0$.
			\end{enumerate}
			\item $\max(1,m)\le j<t-a$.\\
			If  $t-a=0$ or $t-a=\max(1,m)$, we don’t need to consider this case. Otherwise, consider the following two cases:
			\begin{enumerate}
				\item $a<\min(\frac{1}{\widetilde{P}_{down}},\frac{\widetilde{P}_0+t\widetilde{P}_{up}}{\widetilde{P}_{up}+\widetilde{P}_{down}})$.\\
				We get $a\widetilde{P}_{down}<\min[1,\widetilde{P}_0+(t-a)\widetilde{P}_{up}]$. Consider point (\ref{p:52}) and point (\ref{p:53}) with $r_3=t-a$, $r_4=t$, $r_0=j$, $\widetilde{P}_{ramp}=a\widetilde{P}_{down}$. Both points are valid and satisfy inequality (\ref{eq:tight-ramp-down-1}) at equality. Because both points satisfy equality (\ref{pr:10}), we get $\omega_j (\widetilde{P}_0-j\times\frac{\widetilde{P}_0-\widetilde{P}_{t-a}}{t-a})=\omega_j (\widetilde{P}_0-j\times\frac{\widetilde{P}_0-\widetilde{P}_{t-a}}{t-a}-\varepsilon)$. Hence $\omega_j=0$.
				\item $a<\min(\frac{1}{\widetilde{P}_{down}},t-U-\underline{T}_{off},\frac{\widetilde{P}_{start}+(t-U-\underline{T}_{off}-1)\widetilde{P}_{up}}{\widetilde{P}_{up}+\widetilde{P}_{down}})$.\\
				We get $t-a>U+\underline{T}_{off}\ge U+1\ge K$. Then we have $\frac{\widetilde{P}_0-\widetilde{P}_{shut}}{t-a-1}<\widetilde{P}_{down}$. Consider point (\ref{p:10}) and point (\ref{p:11}) with $r_1=m$, $r_2=t-a-1$, $r_0=j$. Both points are valid and satisfy inequality (\ref{eq:tight-ramp-down-1}) at equality. Because both points satisfy equality (\ref{pr:10}), we get $\omega_j[\widetilde{P}_0-j\times\frac{\max(0,\widetilde{P}_0-\widetilde{P}_{shut})}{t-a-1}]=\omega_j\lvert\widetilde{P}_0-j\times\frac{\max(0,\widetilde{P}_0-\widetilde{P}_{shut})}{t-a-1}-\varepsilon\rvert$. Hence $\omega_j=0$.
			\end{enumerate}
		\end{enumerate}
		\item $\omega_{t-a}=-\omega_t$  when $t-a\ne 0$.\\
		Consider the following two cases:
		\begin{enumerate}
			\item $a<\min(\frac{1}{\widetilde{P}_{down}},\frac{\widetilde{P}_0+t\widetilde{P}_{up}}{\widetilde{P}_{up}+\widetilde{P}_{down}})$.\\
			We get $a\widetilde{P}_{down}<\min[1,\widetilde{P}_0+(t-a)\widetilde{P}_{up}]$. Consider point (\ref{p:50}) and point (\ref{p:52}) with $r_3=t-a$, $r_4=t$, $r_0=j$, $\widetilde{P}_{ramp}=a\widetilde{P}_{down}$. Both points are valid and satisfy inequality (\ref{eq:tight-ramp-down-1}) at equality. We have shown that $\omega_j=0$, $j\in[\max(1,m),m+M-1]\{t-a,t\}$ in part 2. Because both points satisfy equality (\ref{pr:10}), we get $\omega_{t-a}\times\min[1,\widetilde{P}_0+(t-a)\widetilde{P}_{up}]+\omega_t\{\min[1,\widetilde{P}_0+(t-a)\widetilde{P}_{up}]-a\widetilde{P}_{down}\}=\omega_{t-a}\{\min[1,\widetilde{P}_0+(t-a)\widetilde{P}_{up}]-\varepsilon\}+\omega_t\{\min[1,\widetilde{P}_0+(t-a)\widetilde{P}_{up}]-a\widetilde{P}_{down}-\varepsilon\}$. Hence $\omega_{t-a}=-\omega_t$.
			\item $a<\min(\frac{1}{\widetilde{P}_{down}},t-U-\underline{T}_{off},\frac{\widetilde{P}_{start}+(t-U-\underline{T}_{off}-1)\widetilde{P}_{up}}{\widetilde{P}_{up}+\widetilde{P}_{down}})$.\\
			We get $a\widetilde{P}_{down}<\min[1,\widetilde{P}_{start}+(t-a-U-\underline{T}_{off}-1)\widetilde{P}_{up}]$, $t-a\ge U+\underline{T}_{off}+1.$ If $m\le U$, we consider point (\ref{p:38}) and point (\ref{p:39}).  If  $m>U$, we consider point (\ref{p:36}) and point (\ref{p:37}). We let $r_1=U+\underline{T}_{off}+1$, $r_3=t-a$, $r_4=t$, $r_0=j$, $\widetilde{P}_{ramp}=a\widetilde{P}_{down}$. Both points are valid and satisfy inequality (\ref{eq:tight-ramp-down-1}) at equality. We have shown that $\omega_j=0$, $j\in[\max(1,m),m+M-1]\{t-a,t\}$ in part 2. Because both points satisfy equality (\ref{pr:10}), we get $\omega_{t-a}\times\min[1,\widetilde{P}_{start}+(t-a-U-\underline{T}_{off}-1)\widetilde{P}_{up}]+\omega_t\times\{\min[1,\widetilde{P}_{start}+(t-a-U-\underline{T}_{off}-1)\widetilde{P}_{up}]-a\widetilde{P}_{down}\}=\omega_{t-a}\{\min[1,\widetilde{P}_{start}+(t-a-U-\underline{T}_{off}-1)\widetilde{P}_{up}]-\varepsilon\}+\omega_t\times\{\min[1,\widetilde{P}_{start}+(t-a-U-\underline{T}_{off}-1)\widetilde{P}_{up}]-a\widetilde{P}_{down}-\varepsilon\}$. Hence $\omega_{t-a}=-\omega_t$.
		\end{enumerate}
		\item 	$\varphi_{m,k}=\lambda_0,k<t-a$.\\
		If $[m,k]\in A_m$ exists such that $k<t-a$, consider point (\ref{p:10}) with $r_1=r_2=k$. Clearly this point is valid and satisfies inequality (\ref{eq:tight-ramp-down-1}) at equality. We have shown that $\omega_j=0$, $j\in[\max(1,m),m+M-1]\{t-a,t\}$ above. Because point (\ref{p:10}) satisfies equality (\ref{pr:10}), we get $\varphi_{m,k}=\lambda_0$.
		\item 	$\varphi_{m,k}=\lambda_0+\omega_t\times\min[1,\widetilde{P}_0+(t-a)\widetilde{P}_{up},\widetilde{P}_{shut}+(k-t+a)\widetilde{P}_{down}],t-a\le k<t$.\\
		If $[m,k]\in A_m$ exists such that $t-a\le k<t$, consider point (\ref{p:8}) with $r_2=k,r_3=t-a$. Clearly this point is valid and satisfies inequality (\ref{eq:tight-ramp-down-1}) at equality. We have shown that $\omega_j=0$, $j\in[\max(1,m),m+M-1]\{t-a,t\},\omega_{t-a}=-\omega_t$ above. Because point (\ref{p:8}) satisfies equality (\ref{pr:10}), we get $\omega_{t-a}\times\min[1,\widetilde{P}_0+(t-a)\widetilde{P}_{up},\widetilde{P}_{shut}+(k-t+a)\widetilde{P}_{down}]+\varphi_{m,k}=\lambda_0$. Hence, $\varphi_{m,k}=\lambda_0+\omega_t\times\min[1,\widetilde{P}_0+(t-a)\widetilde{P}_{up},\widetilde{P}_{shut}+(k-t+a)\widetilde{P}_{down}]$.
		\item $\varphi_{m,k}=\lambda_0+\omega_t\times\min[1,\widetilde{P}_0+(t-a)\widetilde{P}_{up},a\widetilde{P}_{down}],m\le t-a<t\le k$.\\
		If $[m,k]\in A_m $ exists such that $m\le t-a<t\le k$, we consider point (\ref{p:43}) with $r_2=k$, $r_3=t-a$, $r_4=t$, $\widetilde{P}_{ramp}=\min[1,\widetilde{P}_0+(t-a)\widetilde{P}_{up},a\widetilde{P}_{down}]$. Clearly this point satisfies inequality (\ref{eq:tight-ramp-down-1}) at equality. We have shown that $\omega_j=0$, $j\in[\max(1,m),m+M-1]\{t-a,t\}$ and $\omega_{t-a}=-\omega_t$ above. If Point (\ref{p:43}) satisfies equality (\ref{pr:10}), then we get $\omega_{t-a}\times\min[1,\widetilde{P}_0+(t-a)\widetilde{P}_{up},\widetilde{P}_{shut}+(k-t+a)\widetilde{P}_{down}]+\omega_t\{\min[1,\widetilde{P}_0+(t-a)\widetilde{P}_{up},\widetilde{P}_{shut}+(k-t+a)\widetilde{P}_{down}]-\min[1,\widetilde{P}_0+(t-a)\widetilde{P}_{up},a\widetilde{P}_{down}]\}+\varphi_{m,k}=\lambda_0$. Hence, $\varphi_{m,k}=\lambda_0+\omega_t\times\min[1,\widetilde{P}_0+(t-a)\widetilde{P}_{up},a\widetilde{P}_{down}]$.
		\item 	$\varphi_{h,k}=0$, $h\ne m$, $t-a\notin[h,k]$.\\
		If $[h,k]\in A_m$, $h\ne m$ exists such that $t-a\notin[h,k]$, we consider the following three cases:
		\begin{enumerate}
			\item $U<t-a$.\\
			If $m\le U$, consider point (\ref{p:5}) with $r_1=h$, $r_2=k$. Clearly this point is valid and satisfies inequality (\ref{eq:tight-ramp-down-1}) at equality. We have shown that $\omega_j=0$, $j\in[\max(1,m),m+M-1]\{t-a,t\}$, $\varphi_{m,k}=\lambda_0$, $k<t-a$ above. Because point (\ref{p:5}) satisfies equality (\ref{pr:9}), we get $\varphi_{m,U}+\varphi_{h,k}=\lambda_0$. Hence, $\varphi_{h,k}=0$. If $m>U$, consider point (\ref{p:3}) with $r_1=h$, $r_2=k$. Clearly this point is valid and satisfies inequality (\ref{eq:tight-ramp-down-1}) at equality. We have shown that $\lambda_0=0$ when $m>U$ above. Because point (\ref{p:3}) satisfies equality (\ref{pr:10}), we get $\varphi_{h,k}=\lambda_0$. Hence, $\varphi_{h,k}=0$.
			\item $t-a\le U<t$.\\
			Consider point (\ref{p:17}) with $r_1=h$, $r_2=k$, $r_3=t-a$. Clearly this point is valid and satisfies inequality (\ref{eq:tight-ramp-down-1}) at equality. We have shown that $\omega_j=0$, $j\in[\max(1,m),m+M-1]\{t-a,t\}$, $\omega_{t-a}=-\omega_t$, $\varphi_{m,k}=\lambda_0+\omega_t\times\min[1,\widetilde{P}_0+(t-a)\widetilde{P}_{up},\widetilde{P}_{shut}+(k-t+a)\widetilde{P}_{down}]$, $t-a\le k<t$ above. Because point (\ref{p:17}) satisfies equality (\ref{pr:10}), we get $\omega_{t-a}\times\min[1,\widetilde{P}_0+(t-a)\widetilde{P}_{up},\widetilde{P}_{shut}+(U-t+a)\widetilde{P}_{down}]+\varphi_{m,U}+\varphi_{h,k}=\lambda_0$. Hence, $\varphi_{h,k}=0$.
			\item $t-a<t\le U$.\\
			Consider point (\ref{p:49}) with $r_1=h$, $r_2=k$, $r_3=t-a$, $r_4=t$, $\widetilde{P}_{ramp}=\min[1,\widetilde{P}_0+(t-a)\widetilde{P}_{up},a\widetilde{P}_{down}]$.  Clearly this point is valid and satisfies inequality (\ref{eq:tight-ramp-down-1}) at equality. We have shown that $\omega_j=0$, $j\in[\max(1,m),m+M-1]\{t-a,t\}$, $\omega_{t-a}=-\omega_t$, $\varphi_{m,k}=\lambda_0+\omega_t\times\min[1,\widetilde{P}_0+(t-a)\widetilde{P}_{up},a\widetilde{P}_{down}]$, $m\le t-a<t\le k\le m+M-1$ above. Because point (\ref{p:49}) satisfies equality (\ref{pr:10}), we get $\omega_{t-a}\times\min[1,\widetilde{P}_0+(t-a)\widetilde{P}_{up},\widetilde{P}_{shut}+(U-t+a)\widetilde{P}_{down}]+\omega_t\{\min[1,\widetilde{P}_0+(t-a)\widetilde{P}_{up},\widetilde{P}_{shut}+(U-t+a)\widetilde{P}_{down}]-\min[1,\widetilde{P}_0+(t-a)\widetilde{P}_{up},a\widetilde{P}_{down}]\}+\varphi_{m,U}+\varphi_{h,k}=\lambda_0$. Hence, $\varphi_{h,k}=0$.
		\end{enumerate}
		\item 	$\varphi_{h,k}=\omega_t\times\min[1,\widetilde{P}_{start}+(t-a-h)\widetilde{P}_{up},\widetilde{P}_{shut}+(k-t+a)\widetilde{P}_{down}],m<h\le t-a\le k<t$.\\
		If $[h,k]\in A_m$ exists such that  $m<h\le t-a\le k<t$, we consider the following two cases:
		\begin{enumerate}
			\item $m\le U$.\\
			Consider point (\ref{p:14}) with $r_1=h$, $r_2=k$, $r_3=t-a$. Clearly this point is valid and satisfies inequality (\ref{eq:tight-ramp-down-1}) at equality. We have shown that $\omega_j=0$, $j\in[\max(1,m),m+M-1]\{t-a,t\}$, $\varphi_{m,k}=\lambda_0$, $k<t-a$ above. Because point (\ref{p:14}) satisfies equality (\ref{pr:10}), we get $\omega_{t-a}\times\min[1,\widetilde{P}_{start}+(t-a-h)\widetilde{P}_{up},\widetilde{P}_{shut}+(k-t+a)\widetilde{P}_{down}]+\varphi_{m,U}+\varphi_{h,k}=\lambda_0$. Hence, $\varphi_{h,k}=\omega_t\times\min[1,\widetilde{P}_{start}+(t-a-h)\widetilde{P}_{up},\widetilde{P}_{shut}+(k-t+a)\widetilde{P}_{down}]$.
			\item $m>U$.\\
			 Consider point (\ref{p:12}) with $r_1=h$, $r_2=k$, $r_3=t-a$. Clearly this point is valid and satisfies inequality (\ref{eq:tight-ramp-down-1}) at equality. We have shown that $\omega_j=0$, $j\in[\max(1,m),m+M-1]\{t-a,t\}$, $\lambda_0=0$ when $m>U$ above. Because point (\ref{p:12}) satisfies equality (\ref{pr:10}), we get $\omega_{t-a}\times\min[1,\widetilde{P}_{start}+(t-a-h)\widetilde{P}_{up},\widetilde{P}_{shut}+(k-t+a)\widetilde{P}_{down}]+\varphi_{h,k}=\lambda_0$. Hence, $\varphi_{h,k}=\omega_t\times\min[1,\widetilde{P}_{start}+(t-a-h)\widetilde{P}_{up},\widetilde{P}_{shut}+(k-t+a)\widetilde{P}_{down}]$.
		\end{enumerate}
		\item 	$\varphi_{h,k}=\omega_t\times\min[1,\widetilde{P}_{start}+(t-a-h)\widetilde{P}_{up},a\widetilde{P}_{down}]$, $m<h\le t-a<t\le k\le m+M-1$.\\
		If $[h,k]\in A_m$ exists such that $m<h\le t-a<t\le k\le m+M-1$, we consider the following two cases:
		\begin{enumerate}
			\item $m\le U$.\\
			We consider point (\ref{p:31}) with $r_1=h$, $r_2=k$, $r_3=t-a$, $r_4=t$, $\widetilde{P}_{ramp}=\min[1,\widetilde{P}_{start}+(t-a-h)\widetilde{P}_{up},a\widetilde{P}_{down}]$. Clearly this point satisfies inequality (\ref{eq:tight-ramp-down-1}) at equality. We have shown that $\omega_j=0$, $j\in[\max(1,m),m+M-1]\{t-a,t\}$, $\omega_{t-a}=-\omega_t$, $\varphi_{m,k}=\lambda_0$, $k<t-a$ above. If Point (\ref{p:31}) satisfies equality (\ref{pr:10}), then we get $\omega_t\{\min[1,\widetilde{P}_{start}+(t-a-h)\widetilde{P}_{up},\widetilde{P}_{shut}+(k-t+a)\widetilde{P}_{down}]-\min[1,\widetilde{P}_{start}+(t-a-h)\widetilde{P}_{up},a\widetilde{P}_{down}]\}+\omega_{t-a}\times\min[1,\widetilde{P}_{start}+(t-a-h)\widetilde{P}_{up},\widetilde{P}_{shut}+(k-t+a)\widetilde{P}_{down}]+\varphi_{m,U}+\varphi_{h,k}=\lambda_0$. Hence, $\varphi_{h,k}=\omega_t\times\min[1,\widetilde{P}_{start}+(t-a-h)\widetilde{P}_{up},a\widetilde{P}_{down}]$.
			\item  $m>U$.\\
			We consider point (\ref{p:30}) with $r_1=h$, $r_2=k$, $r_3=t-a$, $r_4=t$, $\widetilde{P}_{ramp}=\min[1,\widetilde{P}_{start}+(t-a-h)\widetilde{P}_{up},a\widetilde{P}_{down}]$. Clearly this point satisfies inequality (\ref{eq:tight-ramp-down-1}) at equality. We have shown that $\omega_j=0$, $j\in[\max(1,m),m+M-1]\{t-a,t\}$, $\omega_{t-a}=-\omega_t$, $\lambda_0=0$ when $m>U$ above. If Point (\ref{p:30}) satisfies equality (\ref{pr:10}), then we get $\omega_t\{\min[1,\widetilde{P}_{start}+(t-a-h)\widetilde{P}_{up},\widetilde{P}_{shut}+(k-t+a)\widetilde{P}_{down}]-\min[1,\widetilde{P}_{start}+(t-a-h)\widetilde{P}_{up},a\widetilde{P}_{down}]\}+\omega_{t-a}\times\min[1,\widetilde{P}_{start}+(t-a-h)\widetilde{P}_{up},\widetilde{P}_{shut}+(k-t+a)\widetilde{P}_{down}]+\varphi_{h,k}=\lambda_0$. Hence, $\varphi_{h,k}=\omega_t\times\min[1,\widetilde{P}_{start}+(t-a-h)\widetilde{P}_{up},a\widetilde{P}_{down}]$.
		\end{enumerate}
	\end{enumerate}
	According to theorem 3.6 of §$I.4.3$ in \cite{wolsey1988integer}, we get Proposition \ref{prp10}.
\end{proof}

\bibliography{references}


\begin{thebibliography}{24}
\ifx \bisbn   \undefined \def \bisbn  #1{ISBN #1}\fi
\ifx \binits  \undefined \def \binits#1{#1}\fi
\ifx \bauthor  \undefined \def \bauthor#1{#1}\fi
\ifx \batitle  \undefined \def \batitle#1{#1}\fi
\ifx \bjtitle  \undefined \def \bjtitle#1{#1}\fi
\ifx \bvolume  \undefined \def \bvolume#1{\textbf{#1}}\fi
\ifx \byear  \undefined \def \byear#1{#1}\fi
\ifx \bissue  \undefined \def \bissue#1{#1}\fi
\ifx \bfpage  \undefined \def \bfpage#1{#1}\fi
\ifx \blpage  \undefined \def \blpage #1{#1}\fi
\ifx \burl  \undefined \def \burl#1{\textsf{#1}}\fi
\ifx \doiurl  \undefined \def \doiurl#1{\url{https://doi.org/#1}}\fi
\ifx \betal  \undefined \def \betal{\textit{et al.}}\fi
\ifx \binstitute  \undefined \def \binstitute#1{#1}\fi
\ifx \binstitutionaled  \undefined \def \binstitutionaled#1{#1}\fi
\ifx \bctitle  \undefined \def \bctitle#1{#1}\fi
\ifx \beditor  \undefined \def \beditor#1{#1}\fi
\ifx \bpublisher  \undefined \def \bpublisher#1{#1}\fi
\ifx \bbtitle  \undefined \def \bbtitle#1{#1}\fi
\ifx \bedition  \undefined \def \bedition#1{#1}\fi
\ifx \bseriesno  \undefined \def \bseriesno#1{#1}\fi
\ifx \blocation  \undefined \def \blocation#1{#1}\fi
\ifx \bsertitle  \undefined \def \bsertitle#1{#1}\fi
\ifx \bsnm \undefined \def \bsnm#1{#1}\fi
\ifx \bsuffix \undefined \def \bsuffix#1{#1}\fi
\ifx \bparticle \undefined \def \bparticle#1{#1}\fi
\ifx \barticle \undefined \def \barticle#1{#1}\fi
\bibcommenthead
\ifx \bconfdate \undefined \def \bconfdate #1{#1}\fi
\ifx \botherref \undefined \def \botherref #1{#1}\fi
\ifx \url \undefined \def \url#1{\textsf{#1}}\fi
\ifx \bchapter \undefined \def \bchapter#1{#1}\fi
\ifx \bbook \undefined \def \bbook#1{#1}\fi
\ifx \bcomment \undefined \def \bcomment#1{#1}\fi
\ifx \oauthor \undefined \def \oauthor#1{#1}\fi
\ifx \citeauthoryear \undefined \def \citeauthoryear#1{#1}\fi
\ifx \endbibitem  \undefined \def \endbibitem {}\fi
\ifx \bconflocation  \undefined \def \bconflocation#1{#1}\fi
\ifx \arxivurl  \undefined \def \arxivurl#1{\textsf{#1}}\fi
\csname PreBibitemsHook\endcsname

\bibitem{anjos2017unit}
\begin{barticle}
\bauthor{\bsnm{Anjos}, \binits{M.F.}},
\bauthor{\bsnm{Conejo}, \binits{A.J.}}, \betal:
\batitle{Unit commitment in electric energy systems}.
\bjtitle{Foundations and Trends in Electric Energy Systems}
\bvolume{1}(\bissue{4}),
\bfpage{220}--\blpage{310}
(\byear{2017})
\end{barticle}
\endbibitem

\bibitem{yang2017novel}
\begin{barticle}
\bauthor{\bsnm{Yang}, \binits{L.}},
\bauthor{\bsnm{Zhang}, \binits{C.}},
\bauthor{\bsnm{Jian}, \binits{J.}},
\bauthor{\bsnm{Meng}, \binits{K.}},
\bauthor{\bsnm{Xu}, \binits{Y.}},
\bauthor{\bsnm{Dong}, \binits{Z.}}:
\batitle{A novel projected two-binary-variable formulation for unit commitment
  in power systems}.
\bjtitle{Applied Energy}
\bvolume{187},
\bfpage{732}--\blpage{745}
(\byear{2017})
\end{barticle}
\endbibitem

\bibitem{carrion2006computationally}
\begin{barticle}
\bauthor{\bsnm{Carri{\'o}n}, \binits{M.}},
\bauthor{\bsnm{Arroyo}, \binits{J.M.}}:
\batitle{A computationally efficient mixed-integer linear formulation for the
  thermal unit commitment problem}.
\bjtitle{IEEE Transactions on power systems}
\bvolume{21}(\bissue{3}),
\bfpage{1371}--\blpage{1378}
(\byear{2006})
\end{barticle}
\endbibitem

\bibitem{frangioni2009computational}
\begin{barticle}
\bauthor{\bsnm{Frangioni}, \binits{A.}},
\bauthor{\bsnm{Gentile}, \binits{C.}}:
\batitle{A computational comparison of reformulations of the perspective
  relaxation: Socp vs. cutting planes}.
\bjtitle{Operations Research Letters}
\bvolume{37}(\bissue{3}),
\bfpage{206}--\blpage{210}
(\byear{2009})
\end{barticle}
\endbibitem

\bibitem{williams2013model}
\begin{bbook}
\bauthor{\bsnm{Williams}, \binits{H.P.}}:
\bbtitle{Model Building in Mathematical Programming}.
\bpublisher{John Wiley \& Sons},
\blocation{Chichester}
(\byear{2013})
\end{bbook}
\endbibitem

\bibitem{wolsey2020integer}
\begin{bbook}
\bauthor{\bsnm{Wolsey}, \binits{L.A.}}:
\bbtitle{Integer Programming},
\bedition{2}nd edn.
\bpublisher{John Wiley \& Sons},
\blocation{Hoboken}
(\byear{2020})
\end{bbook}
\endbibitem

\bibitem{wolsey2003strong}
\begin{barticle}
\bauthor{\bsnm{Wolsey}, \binits{L.A.}}:
\batitle{Strong formulations for mixed integer programs: valid inequalities and
  extended formulations}.
\bjtitle{Mathematical programming}
\bvolume{97}(\bissue{1}),
\bfpage{423}--\blpage{447}
(\byear{2003})
\end{barticle}
\endbibitem

\bibitem{garver1962power}
\begin{barticle}
\bauthor{\bsnm{Garver}, \binits{L.L.}}:
\batitle{Power generation scheduling by integer programming-development of
  theory}.
\bjtitle{Transactions of the American Institute of Electrical Engineers. Part
  III: Power Apparatus and Systems}
\bvolume{81}(\bissue{3}),
\bfpage{730}--\blpage{734}
(\byear{1962})
\end{barticle}
\endbibitem

\bibitem{frangioni2009tighter}
\begin{barticle}
\bauthor{\bsnm{Frangioni}, \binits{A.}},
\bauthor{\bsnm{Gentile}, \binits{C.}},
\bauthor{\bsnm{Lacalandra}, \binits{F.}}:
\batitle{Tighter approximated milp formulations for unit commitment problems}.
\bjtitle{IEEE Transactions on Power Systems}
\bvolume{24}(\bissue{1}),
\bfpage{105}--\blpage{113}
(\byear{2009})
\end{barticle}
\endbibitem

\bibitem{padberg1998location}
\begin{barticle}
\bauthor{\bsnm{Padberg}, \binits{M.W.}},
\bauthor{\bsnm{Rijal}, \binits{M.P.}}:
\batitle{Location, scheduling, design and integer programming}.
\bjtitle{Journal of The Operational Research Society}
\bvolume{49}(\bissue{1}),
\bfpage{95}--\blpage{95}
(\byear{1998})
\end{barticle}
\endbibitem

\bibitem{rajan2005minimum}
\begin{barticle}
\bauthor{\bsnm{Rajan}, \binits{D.}},
\bauthor{\bsnm{Takriti}, \binits{S.}}, \betal:
\batitle{Minimum up/down polytopes of the unit commitment problem with start-up
  costs}.
\bjtitle{IBM Res. Rep}
\bvolume{23628},
\bfpage{1}--\blpage{14}
(\byear{2005})
\end{barticle}
\endbibitem

\bibitem{gentile2017tight}
\begin{barticle}
\bauthor{\bsnm{Gentile}, \binits{C.}},
\bauthor{\bsnm{Morales-Espa{\~n}a}, \binits{G.}},
\bauthor{\bsnm{Ramos}, \binits{A.}}:
\batitle{A tight mip formulation of the unit commitment problem with start-up
  and shut-down constraints}.
\bjtitle{EURO Journal on Computational Optimization}
\bvolume{5}(\bissue{1-2}),
\bfpage{177}--\blpage{201}
(\byear{2017})
\end{barticle}
\endbibitem

\bibitem{damci2016polyhedral}
\begin{barticle}
\bauthor{\bsnm{Damc{\i}-Kurt}, \binits{P.}},
\bauthor{\bsnm{K{\"u}{\c{c}}{\"u}kyavuz}, \binits{S.}},
\bauthor{\bsnm{Rajan}, \binits{D.}},
\bauthor{\bsnm{Atamt{\"u}rk}, \binits{A.}}:
\batitle{A polyhedral study of production ramping}.
\bjtitle{Mathematical Programming}
\bvolume{158}(\bissue{1}),
\bfpage{175}--\blpage{205}
(\byear{2016})
\end{barticle}
\endbibitem

\bibitem{frangioni2015new}
\begin{botherref}
\oauthor{\bsnm{Frangioni}, \binits{A.}},
\oauthor{\bsnm{Gentile}, \binits{C.}}:
New mip formulations for the single-unit commitment problems with ramping
  constraints.
IASI Annual Research Reports.
  \url{http://www.iasi.cnr.it/new/publications.php/id_p/1/anno/0/id_autore/0/id_tipologia/6/rep/4414}
(2015).
Accessed 15 June 2015
\end{botherref}
\endbibitem

\bibitem{yang2021two}
\begin{barticle}
\bauthor{\bsnm{Yang}, \binits{L.}},
\bauthor{\bsnm{Li}, \binits{W.}},
\bauthor{\bsnm{Xu}, \binits{Y.}},
\bauthor{\bsnm{Zhang}, \binits{C.}},
\bauthor{\bsnm{Chen}, \binits{S.}}:
\batitle{Two novel locally ideal three-period unit commitment formulations in
  power systems}.
\bjtitle{Applied Energy}
\bvolume{284},
\bfpage{116081}
(\byear{2021})
\end{barticle}
\endbibitem

\bibitem{jabr2012tight}
\begin{barticle}
\bauthor{\bsnm{Jabr}, \binits{R.}}:
\batitle{Tight polyhedral approximation for mixed-integer linear programming
  unit commitment formulations}.
\bjtitle{IET Generation, Transmission \& Distribution}
\bvolume{6}(\bissue{11}),
\bfpage{1104}--\blpage{1111}
(\byear{2012})
\end{barticle}
\endbibitem

\bibitem{atakan2018state}
\begin{barticle}
\bauthor{\bsnm{Atakan}, \binits{S.}},
\bauthor{\bsnm{Lulli}, \binits{G.}},
\bauthor{\bsnm{Sen}, \binits{S.}}:
\batitle{A state transition mip formulation for the unit commitment problem}.
\bjtitle{IEEE Transactions on Power Systems}
\bvolume{33}(\bissue{1}),
\bfpage{736}--\blpage{748}
(\byear{2018})
\end{barticle}
\endbibitem

\bibitem{morales2013tight}
\begin{barticle}
\bauthor{\bsnm{Morales-Espa{\~n}a}, \binits{G.}},
\bauthor{\bsnm{Latorre}, \binits{J.M.}},
\bauthor{\bsnm{Ramos}, \binits{A.}}:
\batitle{Tight and compact milp formulation for the thermal unit commitment
  problem}.
\bjtitle{IEEE Transactions on Power Systems}
\bvolume{28}(\bissue{4}),
\bfpage{4897}--\blpage{4908}
(\byear{2013})
\end{barticle}
\endbibitem

\bibitem{ostrowski2012tight}
\begin{barticle}
\bauthor{\bsnm{Ostrowski}, \binits{J.}},
\bauthor{\bsnm{Anjos}, \binits{M.F.}},
\bauthor{\bsnm{Vannelli}, \binits{A.}}:
\batitle{Tight mixed integer linear programming formulations for the unit
  commitment problem}.
\bjtitle{IEEE Transactions on Power Systems}
\bvolume{27}(\bissue{1}),
\bfpage{39}--\blpage{46}
(\byear{2012})
\end{barticle}
\endbibitem

\bibitem{pan2016polyhedral}
\begin{botherref}
\oauthor{\bsnm{Pan}, \binits{K.}},
\oauthor{\bsnm{Guan}, \binits{Y.}}:
A polyhedral study of the integrated minimum-up/-down time and ramping
  polytope.
arXiv preprint arXiv:1604.02184
(2016)
\end{botherref}
\endbibitem

\bibitem{bacci2019new}
\begin{botherref}
\oauthor{\bsnm{Bacci}, \binits{T.}},
\oauthor{\bsnm{Frangioni}, \binits{A.}},
\oauthor{\bsnm{Gentile}, \binits{C.}},
\oauthor{\bsnm{Tavlaridis-Gyparakis}, \binits{K.}}:
New minlp formulations for the unit commitment problems with ramping
  constraints.
Optimization Online.
  \url{http://www.optimization-online.org/DB_HTML/2019/10/7426.html}
(2019).
Accessed 10 November 2019
\end{botherref}
\endbibitem

\bibitem{knueven2018ramping}
\begin{barticle}
\bauthor{\bsnm{Knueven}, \binits{B.}},
\bauthor{\bsnm{Ostrowski}, \binits{J.}},
\bauthor{\bsnm{Wang}, \binits{J.}}:
\batitle{The ramping polytope and cut generation for the unit commitment
  problem}.
\bjtitle{INFORMS Journal on Computing}
\bvolume{30}(\bissue{4}),
\bfpage{739}--\blpage{749}
(\byear{2018})
\end{barticle}
\endbibitem

\bibitem{wolsey1988integer}
\begin{bbook}
\bauthor{\bsnm{Nemhauser}, \binits{G.L.}},
\bauthor{\bsnm{Wolsey}, \binits{L.A.}}:
\bbtitle{Integer and Combinatorial Optimization}
vol. \bseriesno{15}.
\bpublisher{Wiley},
\blocation{New York}
(\byear{1988})
\end{bbook}
\endbibitem

\bibitem{dolan2002benchmarking}
\begin{barticle}
\bauthor{\bsnm{Dolan}, \binits{E.D.}},
\bauthor{\bsnm{Mor{\'e}}, \binits{J.J.}}:
\batitle{Benchmarking optimization software with performance profiles}.
\bjtitle{Mathematical programming}
\bvolume{91}(\bissue{2}),
\bfpage{201}--\blpage{213}
(\byear{2002})
\end{barticle}
\endbibitem

\end{thebibliography}

\end{document}